\newif\ifdraft

\newif\ifshowlongcomments
\documentclass{amsart}

\usepackage{mathrsfs,eucal}
\usepackage{mparhack}
\usepackage{xargs}
\usepackage{enumitem}
\usepackage[usenames,dvipsnames]{xcolor}
\usepackage{graphicx}
\usepackage{hyperref}
\usepackage{tikz}
\usepackage{amsmath, amsfonts, latexsym, array, amssymb, amscd, amsthm}
\usetikzlibrary{shapes,arrows,positioning,decorations.pathreplacing,arrows.meta,fadings,patterns}
\usepackage{tcolorbox}
\usepackage{mdframed}
\tcbuselibrary{skins,hooks}

\tcbset{enhanced,
attach boxed title to top center = {xshift=0mm,yshift=-2mm},
boxed title style={left=2pt,right=2pt,colback=black!10,boxsep=1pt},
coltitle=black,
boxsep=0pt,left=4pt,right=4pt
}

\newtcolorbox[auto counter]{mybox}[1][]{%
title=\thetcbcounter,
#1}

\newcommandx{\boxymcboxface}[4][1=,2=0pt,usedefault]{
\vspace*{-10pt}\vspace*{-#2}\tiny
\begin{mybox}[colback=#3,label={#1}]
\hspace{0pt}\noindent
#4
\end{mybox}
}

\newcommandx{\discuss}[3][1=,2=0pt,usedefault]%
{\marginpar{\hspace{0pt}\boxymcboxface[#1][#2]{SkyBlue!30}{#3}}}

\newcommandx{\toremove}[3][1=,2=0pt,usedefault]%
{\marginpar{\hspace{0pt}\boxymcboxface[#1][#2]{red!30}{#3}}}

\newcommandx{\addednow}[3][1=,2=0pt,usedefault]%
{\marginpar{\hspace{0pt}\boxymcboxface[#1][#2]{green!30}{#3}}}

\newcommandx{\todo}[3][1=,2=0pt,usedefault]%
{\marginpar{\hspace{0pt}\boxymcboxface[#1][#2]{yellow!30}{#3}}}

\newcommandx{\spinoff}[3][1=,2=0pt,usedefault]%
{\marginpar{\hspace{0pt}\boxymcboxface[#1][#2]{orange!60}{#3}}}

\usepackage[refpage,notocbasic]{nomencl}

\makenomenclature
\usepackage{imakeidx}
\makeindex[name=terms,title=Index of terminology]

\newcommand{\nome}[2]{\nomenclature{#1}{#2}}
\newcommand{\terms}[1]{\index[terms]{#1}}

\ifdraft
\usepackage{fancyhdr,datetime2}
\DTMsettimestyle{default}
\DTMsetup{showseconds=false}
\DTMsetup{showzone=false}

\fancypagestyle{firstpage}{%
	
	\fancyhf{}
	\fancyhead[C]{\color{red}\fbox{\textbf{DRAFT} (\DTMtoday\ at \DTMcurrenttime) -- NOT FOR DISTRIBUTION}}
	\fancyfoot[C]{\thepage}
}
\fi

\newcommand{\bE}{\mathbf{E}}
\newcommand{\bZ}{\mathbf{Z}}
\newcommand{\Ar}%{A_{\mathrm{ret}}}
{\mathcal{C}_A}
\newcommand{\ph}{\varphi}
\newcommand{\eps}{\varepsilon}
\newcommand{\RR}{\mathbb{R}}
\newcommand{\Mf}{\mathcal{M}_\mathcal{F}}
\newcommand{\MF}{\mathcal{M}_F}

\newcommand{\EE}{\mathcal{E}}
\newcommand{\EN}[1][N]{\mathcal{E}^{#1}}
\newcommand{\ulim}{\varlimsup}
\newcommand{\llim}{\varliminf}

\newcommand{\DDD}{\mathcal{D}}
\newcommand{\DDDH}{\mathcal{D}^\mathrm{H}}
\newcommand{\DDDBow}{\mathcal{D}^\mathrm{Bow}}

\newcommand{\BBB}{\mathcal{B}}
\newcommand{\GGG}{\mathcal{G}}

\DeclareMathOperator{\Leb}{Leb}

\DeclareMathOperator{\Per}{Per}

\newcommand{\Vt}{V}
\newcommand{\esssup}[1]{#1\mathrm{\text{-}ess\,sup\,}}
\newcommand{\essinf}[1]{#1\mathrm{\text{-}ess\,inf\,}}

\newcommand{\rmet}{\mathsf{r_{\DDD}}}
\newcommand{\rexp}{\mathsf{r_{exp}}}
\newcommand{\rbow}{\mathsf{r_{Bow}}}
\newcommand{\dbow}{d^{\mathrm{Bow}}}
\newcommand{\Dbow}{D^{\mathrm{Bow}}}
\newcommand{\dH}{d^{\mathrm{H}}}

\newcommand{\RRR}{\mathcal{R}}
\newcommand{\FFF}{\mathcal{F}}

\newcommand{\NN}{\mathbb{N}}
\newcommand{\ZZ}{\mathbb{Z}}

\newcommand{\bph}{\overline{\ph}}
\newcommand{\bPh}{\overline{\Phi}}

\newcommand{\barf}{\overline{f}}
\newcommand{\barfn}{\overline{f_N}}
\newcommand{\bnu}{\overline{\nu}}
\newcommand{\bnun}{\overline{\nu}_N}
\newcommand{\bq}{\mathbf{q}}
\newcommand{\bp}{\mathbf{p}}
\newcommand{\bx}{\mathbf{x}}

\newcommand{\bt}{\mathbf{t}}
\newcommand{\bw}{\mathbf{w}}
\newcommand{\bs}{\mathbf{s}}
\newcommand{\bS}{\mathbf{S}}
\newcommand{\btau}{\boldsymbol{\tau}}

\newcommand{\bz}{\mathbf{z}}

\newcommand{\emp}[1]%{\mathcal{E}_{#1}}
{\delta_{(#1)}}

\newcommand{\glu}{\mathfrak{g}}
\newcommand{\Sched}{\mathbf{S}}
\newcommand{\tps}{\psi^*}

\newcommand{\one}{\mathbf{1}}
\newcommand{\Ebad}{E^{\mathsf{bad}}}
\newcommand{\Egood}{E^{\mathsf{good}}}
\newcommand{\Aln}{A_{\mathsf{grade}}}
\newcommand{\Agn}{A_{\mathsf{slice}}}
\newcommand{\Eln}{\EE^{\mathsf{grade}}}
\newcommand{\Egn}{\EE^{\mathsf{slice}}}

\DeclareMathOperator{\Spec}{Spec}
\DeclareMathOperator{\Spanset}{Span}
\DeclareMathOperator{\Sepset}{Sep}
\newcommand{\dimb}{\overline{\mathrm{dim}}_B}

\usepackage{ulem}

\DeclareMathOperator{\Var}{Var}

\DeclareMathOperator{\diam}{diam}
\DeclareMathOperator{\CAT}{CAT}

\newcommand{\bigref}{B}
\newcommand{\Ap}{B}
\newcommand{\justA}{A}
\newcommand{\Azero}{A_0}
\newcommand{\Abiggest}{A^{+}}
\newcommand{\Acomplement}{\Abiggest\setminus \justA}
\newcommand{\AO}[1][1]{A_{\mathcal O}^{#1}}
\newcommand{\ret}[2]{#1^{(#2)}}
\newcommand{\bmu}{\bar\mu}
\newcommand{\bsig}{\bar\sigma}

\newcommand{\fraceps}{\eps/4}
\newcommand{\halffraceps}{\eps/8}
\newcommand{\fracfraceps}{\eps/20}
\newcommand{\halffracfraceps}{\eps/40}

\newcommand{\onezeta}{\zeta}

\newcommand{\epso}{{\eps_0}}
\newcommand{\halfepso}{\eps_0/2}
\newcommand{\twoepso}{{2\eps_0}}
\newcommand{\fourepso}{{4\eps_0}}

\newcommand{\tenepso}{{10\eps_0}}
\newcommand{\twentyepso}{{20\eps_0}}

\newcommand{\bX}{\partial_\infty X}

\newcommand{\WBow}{W_{\mathrm{Bow}}}
\newcommand{\wtA}{\widetilde{A}}

\numberwithin{equation}{section}
\numberwithin{figure}{section}

\theoremstyle{plain}
\newtheorem{theorem}{Theorem}[section]
\newtheorem{proposition}[theorem]{Proposition}
\newtheorem{lemma}[theorem]{Lemma}
\newtheorem{definition}[theorem]{Definition}
\newtheorem{remark}[theorem]{Remark}

\newtheorem{corollary}[theorem]{Corollary}
\newtheorem{notation}[theorem]{Notational Convention}

\newtheorem{condition}{Condition}

\newtheorem{thma}{Theorem}

\newtheorem{coralph}[thma]{Corollary}

\ifshowlongcomments

\else

\fi

\setlist[itemize]{itemsep=1ex}

\begin{document}

\title[Thermodynamic formalism for non-compact systems]{Thermodynamic formalism for non-compact systems with expansivity and specification}

\author{Vaughn Climenhaga}
\author{Daniel J. Thompson}
\author{Tianyu Wang}
\address{Department of Mathematics, University of Houston, Houston, Texas 77204, USA}
\address{Department of Mathematics, The Ohio State University, 100 Math Tower, 231 West 18th Avenue, Columbus, Ohio 43210, USA}
\address{School of Mathematical Sciences, Shanghai Jiao Tong University, No 800 Dongchuan Road, Shanghai 200240, P.R.China}

\email{vclimenh@charlotte.edu}
\email{thompson@math.osu.edu}
\email{tywangdynamics@126.com}

\begin{abstract}
We develop a new theory of thermodynamic formalism for non-compact flows, providing criteria for existence and uniqueness of equilibrium states, as well as ergodicity and a Gibbs property. These criteria include new versions of the expansivity and specification properties that hold for geodesic flows on non-compact negatively curved manifolds. Previous results on thermodynamic formalism for these geodesic flows were restricted by a pinched curvature condition, which we remove, allowing curvature to approach both $0$ and $-\infty$. Our results also apply to geodesic flows beyond the smooth setting, including locally $\CAT(-1)$ spaces.

\end{abstract}

\date{November 5, 2026}
%{\DTMnow}

\maketitle

\ifdraft
\thispagestyle{firstpage}
\pagestyle{fancy}
\fi

\section{Introduction and statement of results}\label{sec:intro}

\subsection{Overview}\label{sec:overview}

The expansivity and specification properties provide a powerful and flexible framework to study thermodynamic formalism for dynamical systems on compact phase spaces.
Rufus Bowen developed the original theory in the 1970's, identifying these as the key dynamical properties of uniformly hyperbolic systems guaranteeing that every sufficiently regular potential function has a unique equilibrium state, which is ergodic and has a Gibbs property \cite{rB74}.

In this paper, we extend Bowen's approach to non-compact systems, where existence of an equilibrium state also requires a recurrence property. 
Condition \ref{cond:basic} below describes our setting: a separable locally compact completely metrizable space $X$, equipped with a continuous flow $\FFF$ and a continuous potential function $\ph\colon X\to\RR$.
Conditions \ref{cond:E}--\ref{cond:UESB} give precise definitions of the hypotheses in the following summary of our main result, Theorem \ref{thm:mainES}:
\begin{quote}
\emph{If $(X,\FFF,\ph)$ satisfies suitable versions of expansivity, specification, the Bowen property, and strong positive recurrence, with $\sup\ph<\infty$ and finite topological entropy, then there exists a unique invariant measure maximizing $h_\mu(\FFF) + \int_X\ph\,d\mu$, and this equilibrium state is ergodic and satisfies a Gibbs property.}
\end{quote}
For geodesic flow in negative curvature, this has the following consequence.
\begin{quote}
\emph{Let $M$ be a non-compact negatively curved smooth Riemannian manifold,
which may have curvatures approaching $0$ and $-\infty$.
If $M$ satisfies strong positive recurrence, then its geodesic flow has a unique measure of maximal entropy.}
\end{quote}
Theorem \ref{thm:CAT} implies the statement above, and also applies to locally $\CAT(-1)$ spaces and to equilibrium states for a wide class of H\"older continuous potentials.

Our main results are in \S\ref{sec:GS-VP}--\S\ref{sec:Mis} below:
Theorem \ref{thm:welldefined} establishes a variational principle, and Theorems \ref{thm:uniform}--\ref{thm:existsES} give intermediate results used in the main existence and uniqueness result, Theorem \ref{thm:mainES}.
Before stating Conditions \ref{cond:basic}--\ref{cond:UESB} precisely and giving formal statements of these results, we discuss the context in which they live, including some of their new aspects, and the naturality of the conditions.

\subsubsection*{Relationship to previous work}
We survey related literature more fully in  \S\ref{sec:previous-results-in-the-literature}.
Here, we briefly describe the following novelties of our results.
\begin{itemize}
\item Applications to geodesic flows without a pinched curvature condition.
\item A new concept of `coherent metrics' 
to deal with the failure of expansivity.

\item A new argument for ergodicity and uniqueness.

\item Partition sum estimates using direct counting arguments.

\item A flexible notion of `adapted partitions' to accommodate induced maps.

\end{itemize}

\textsc{The pinched curvature condition.}
Equilibrium states for geodesic flows on non-compact manifolds with negative curvature have been studied using Patterson--Sullivan measures \cite{PPS15,GST23}. These results require the curvature to be `pinched'
away from $-\infty$ and $0$. Our geometric applications in Theorem \ref{thm:CAT}
%\S\ref{sec:app}, \S\ref{sec:only-neg}
remove this condition entirely, opening a broad range of new examples to analysis.
There is a well-developed theory for constructing an abundance of negatively curved surfaces using pants, cusps, funnels, and ends of various types \cite{dB16}. Our setting includes surfaces with infinite ends, as well as cusps whose curvature tends to $0$ or $-\infty$. In higher dimensions, there is a vast variety of manifolds with negative curvature \cite{mK07, pO20}. Beyond the pinched setting, explicit constructions exhibit several unexpected geometric and topological features: visibility and topological tameness may fail, and finite-volume ends need not be parabolic cusps \cite{iB16,nP11}. 

Our application to $\CAT(-1)$ spaces in Theorem \ref{thm:CAT} also goes well beyond the class of systems dealt with by other methods. Without smooth structure, the usual weighted Patterson--Sullivan construction is not available, since expressions of the form $\int_x^y \varphi$ are not generally well-defined,
and thermodynamic formalism for general $\CAT(-1)$ spaces has remained relatively undeveloped until recently \cite{DT25}. 

\smallskip
\textsc{The failure of expansivity}.
A fundamental obstruction to developing a theory of expansivity and specification that applies to geodesic flow on non-compact manifolds is the fact that in general, the most natural version of the expansivity property fails to hold, even in the simple case of a hyperbolic surface with a single cusp \cite{CD26}.
To deal with this, we introduce the idea of a family of \emph{coherent metrics} in Definition \ref{def:coherent-metrics}, which provides a version of expansivity that is weak enough to hold for non-compact negatively curved manifolds, but strong enough to use as a foundation for thermodynamic formalism.

\smallskip
\textsc{A new strategy for uniqueness}.
Another issue with expansivity is more subtle, but still crucial. Bowen's original results in \cite{rB74} were for discrete-time systems, and were extended to flows by Franco \cite{eF77}. However, as described in Remark \ref{rmk:changing-scales}, there is a `scale-changing' lemma that follows from expansivity in discrete time, but not in continuous time. Consequently, arguments for ergodicity and uniqueness from  \cite{eF77} must be mended, even in the compact case. One fix was described by the first two named authors in \cite{CT16}, but relies on compactness. Here, we resolve the issue by introducing a new strategy for proving both ergodicity and uniqueness in  \S\ref{sec:eu-other}, which may be of interest elsewhere. 

\smallskip
\textsc{Direct counting arguments}.
An important inspiration for our results is the extensive body of work by Sarig on thermodynamic formalism for countable-state Markov shifts, beginning with \cite{oS99}. The form of a number of our results is suggested by analogy with Sarig's work, but our techniques are very different. In particular, without transfer operator machinery to use, we give in \S\ref{sec:counting} an elementary counting argument for the essential partition sum bounds in Theorem \ref{thm:uniform}.

\smallskip
\textsc{Adapted partitions for induced maps}.
As this manuscript was being completed, we learned about the independent work of Florio, Schapira, and Vaugon \cite{FSV}, studying MMEs for a class of \emph{H-flows}, which are smooth non-compact systems satisfying various properties including transitivity, shadowing, a closing lemma, and a uniform expansivity condition. 
These properties bear similarities to the conditions that we work with,  although their more restrictive expansivity condition excludes geodesic flows on manifolds with cusps \cite{CD26}, which we allow, and they work only with the $0$-potential. They provide counting arguments similar to the ones we give in \S\ref{sec:counting}, and construct an MME via a sequence of measures on long orbit segments; however, they allow the possibility of some mass escaping to infinity, which we rule out in Theorem \ref{thm:mainGibbs}\ref{thm:tight}. 
Moreover, they do not address ergodicity and uniqueness, which we establish in Theorem \ref{thm:mainES}.
In the standard arguments for ergodicity and uniqueness \cite{rB74,eF77}, a central role is played by an \emph{adapted partition} for a collection of orbit segments of a fixed length.
The nature of our non-compact approach requires us to work with orbit segments of varying lengths, corresponding to certain induced maps; we accomplish this using \emph{sliced and graded adapted partitions}, which we introduce in \S\ref{sec:eq-st}, and which constitute the essential technical innovation allowing us to get stronger results than those in \cite{FSV}.
% \S\ref{sec:uniform-katok}--\S\ref{sec:eu-other}

\subsubsection*{Discussion of conditions}
We list again the conditions mentioned above, whose precise definitions we will give shortly:

\begin{minipage}{.38\textwidth}
\mbox{}
\begin{description}
\item[\ref{cond:basic}] basic topological setting,
\item[\ref{cond:E}] expansivity,
\item[\ref{cond:S}] specification, 
\end{description}
\mbox{}
\end{minipage}
\hfill
\begin{minipage}{.6\textwidth}
\mbox{}
\begin{description}
\item[\ref{cond:B}] the Bowen property (regularity of $\ph$);  \item[\ref{cond:SPR}] strong positive recurrence (SPR); 
\item[\ref{cond:UESB}] finite topological entropy and $\sup\ph<\infty$.
\end{description}
\mbox{}
\end{minipage}

Our Conditions \ref{cond:E}, \ref{cond:S}, and \ref{cond:B} are novel adaptations of the classical definitions, designed to be suitable for non-compact spaces; precise statements are in \S\ref{sec:SEB}.
Conditions \ref{cond:SPR} and \ref{cond:UESB} are formulated when they are needed, in \S\ref{sec:SPR-ES} and \S\ref{sec:Mis}, respectively. Condition \ref{cond:UESB} is quite mild, and guarantees that the definition of equilibrium state only involves finite values of $h_\mu(\FFF)$ and $\int_X\ph\,d\mu$. 
Here we say a few words regarding the SPR condition in \ref{cond:SPR}.

The SPR condition for countable-state Markov shifts was introduced by Sarig \cite{oS01}. In this setting, existence of an equilibrium state is equivalent to the weaker \emph{positive recurrence} property, and the uniform interior of this set of potentials is exactly the set of SPR potentials \cite{CS09}. Thus the SPR condition is the optimal condition for a robust existence and uniqueness result.

There is a worthwhile analogy with the theory of non-uniform specification and expansivity properties (for compact systems) developed by the first two authors \cite{CT16}, in which a crucial role is played by a \emph{pressure gap} condition corresponding to SPR. This theory has been successfully applied to geodesic flow in non-positive curvature and beyond \cite{BCFT18,CKK,CCESW}, the Lorenz attractor \cite{PYY}, and other non-uniformly hyperbolic systems \cite{kL26,HT26}, suggesting that the analogous SPR condition for non-compact flows will have a similarly broad reach.

Finally, in the setting of geodesic flows, our SPR condition agrees with previous versions defined via `entropy/pressure at infinity' \cite{GST23,DOP,DPPS}. The zero potential has SPR for non-elementary geometrically finite hyperbolic manifolds, and even for manifolds where this potential does not have SPR, such as a $\ZZ$-cover of a hyperbolic surface, any potential decaying quickly enough at infinity will satisfy SPR; see \cite{GST23} for the pinched curvature case.

\subsubsection*{Outline of results}

All of our results are formulated in the following setting.

\begin{condition}\label{cond:basic}
$X$ is a locally compact separable completely metrizable space,
on which we have a continuous flow 
$\FFF=(f_t)_{t\in \RR}$
and a continuous function $\ph\colon X\to \RR$,
called a \emph{potential}.
We equip $X$ with a one-parameter family of metrics $\DDD = (d_t)_{t\geq 0}$ such that each $d_t$ is compatible with the topology, and the family $\DDD$ satisfies the \emph{coherence} condition in Definition \ref{def:coherent-metrics} below.
Some results will be formulated in terms of a continuous map $F\colon X\to X$.%
\nome{$X$}{separable metric space with compact closed balls}%
\nome{$d$}{metric on $X$}%
\nome{$(f_t)$}{continuous flow on $X$}%
\nome{$\ph$}{continuous function $X\to\RR$}
\end{condition}

An \emph{equilibrium state} for $\ph$ is an $\FFF$-invariant Borel probability measure on $X$ that maximizes the quantity $h_\mu(\FFF) + \int_X \ph\,d\mu$. 
Our two main results can be summarized as follows.
\begin{itemize}
\item Theorem \ref{thm:welldefined}: Under a definition of the specification property designed for non-compact spaces, every sufficiently regular potential function satisfies a variational principle for the \emph{Gurevich--Sarig pressure}.
\item Theorem \ref{thm:mainES}: If in addition, the flow is expansive (using our definition which is satisfied by geodesic flows on negatively curved manifolds), has finite entropy and the potential is bounded above and satisfies the Bowen property and SPR, then there exists a unique equilibrium state.
\end{itemize}

Some of the intermediate steps in the proof of Theorem \ref{thm:mainES} are important enough to state as theorems in their own right:

\begin{itemize}
\item Theorem \ref{thm:uniform} provides uniform bounds on partition sums;
\item Theorem \ref{thm:mainGibbs} constructs an invariant measure with a Gibbs property;
\item Theorem \ref{thm:existsES} shows that this measure is an equilibrium state.
\end{itemize}

The uniqueness result in Theorem \ref{thm:mainES} readily yields an equidistribution result for weighted periodic orbits; see Corollary \ref{cor:equidist}. The application of these general results to geodesic flows in negative curvature is provided by Theorem \ref{thm:CAT}.

\subsection{Expansivity, specification and the Bowen property}\label{sec:SEB}

The precise conditions under which our results hold will be formulated in terms of \emph{orbit segments}, which correspond to pairs $(x,t) \in X \times [0,\infty)$, where $(x,t)$ represents the orbit segment $[0,t]\to X$ defined by $s \mapsto f_s(x)$.
We will work with orbit segments that start and end in a given  \emph{reference set} $A\in \RRR$, where $\RRR$ denotes the collection of all compact subsets of $X$ with non-empty interior.

We need a notion of distance between two orbit segments of the same length. The standard approach is to fix a  metric $d$ and then use the dynamically defined, or Bowen, metrics given by $\dbow_t(x,y) := \sup_{s\in [0,t]} d(f_s(x), f_s(y))$ to describe the distance between $(x,t)$ and $(y,t)$. To deal with our motivating examples, including negatively curved surfaces with cusps, we introduce a more flexible approach: 

\begin{definition}\label{def:coherent-metrics}
Let $\DDD = (d_t)_{t\geq 0}$ be a family of metrics compatible with the topology on $X$, and write $d:=d_0$.  We say that $\DDD$ is \emph{coherent} with respect to the flow $\FFF$ if:
\begin{itemize}
\item $(X, d)$ is a proper metric space (every closed ball is compact). If $(X, d)$ is a geodesic metric space, this is equivalent to the space being complete and locally compact, so it follows from \ref{cond:basic}.
\item $\lim_{t\to 0} d_t(x,y) = d(x,y)$.
\item For every $s,t\geq 0$ and $x,y\in X$, we have
\begin{equation}\label{eqn:coherent}
d_{s+t}(x,y) \geq \max \big( d_s(x,y), d_t(f_s x, f_s y) \big).
\end{equation}
\item
For every compact set $A \subset X$, there exists $\rmet = \rmet(A)>0$ such that given any $s,t\geq 0$, any $x\in A \cap f_{-s}(A) \cap f_{-(s+t)}(A)$, and any $y\in X$ such that $d_s(x,y) \leq \rmet$ and $d_t(f_s x, f_s y) \leq \rmet$, we have equality in \eqref{eqn:coherent}.
\end{itemize}
\end{definition}

One might also require that all the metrics in the family $\DDD$ give the same H\"older class, or at least the same uniform class. This holds true in our motivating examples but it is not required in our proofs.  

Throughout the paper, when we refer to the ``distance between points'' in $X$, we will always mean with respect to the metric $d = d_0$. Distances between orbit segments of length $t$ will always be taken with respect to $d_t$. 

The family of Bowen metrics $(\dbow_t)$ defined from a fixed metric $d$ is coherent. Indeed, if $K \subset X$ is compact and invariant, then at scales below $\rmet(K)$, this is the \emph{only} coherent family on $K$ which satisfies $d_0=d$. 

For the geodesic flow on manifolds with cusps, fixing a natural metric $d$, or equivalently using the family of Bowen metrics $(\dbow_t)$ derived from $d$, has a major drawback. It was demonstrated in \cite{CD26} that, using the obvious definition of expansivity, even the geodesic flow on the modular surface fails to be expansive in a spectacular way, see Remark \ref{expansivityremark}. To circumvent this major issue, we use a coherent system of metrics which is not a family of Bowen metrics. We define our family by lifting to the universal cover before taking the supremum (see \S\ref{sec:metricpropertiesgeodesicflow}), and we obtain a natural expansivity property with respect to this family. We denote this family by $\DDDH = (\dH_t)$, where H denotes `horosphere', to reflect the closer relationship this family of metrics has with the horosphere structure. Despite the failure of expansivity in the metric $d$, we can verify the following expansivity property.

\begin{condition}[Expansivity]\label{cond:E}
The flow $(X, \FFF)$ is expansive with respect to $\RRR$ and the family of metrics $\DDD$: For every $A\in \RRR$, there exists 
$\rexp(A) \in (0,\rmet(A)]$ and $\alpha=\alpha(A)>0$
such that every orbit that returns bi-infinitely often to $A$ has a trivial non-expansive set with `lag' $\alpha$: that is, if $x\in A$ and $s_k, t_k\to\infty$ are such that $f_{-s_k}(x) \in A$ and $f_{t_k}(x) \in A$ for all $k$, and $y\in X$ satisfies
\[
d_{s_k}(f_{-s_k}x,f_{-s_k}y) \leq \rexp(A)
\quad\text{and}\quad
d_{t_k}(x,y) \leq \rexp(A)
\quad\text{for all }k,
\]
then there exists $t\in [-\alpha,\alpha]$ such that $y = f_t(x)$.
\end{condition}

We give an informal description of the specification property, which is also considered with respect to the family of metrics $\DDD$; see \S\ref{s.defns} for a more detailed discussion.
\begin{condition}[Specification]\label{cond:S}
The flow $(X,\FFF)$ has specification at all scales with respect to $\RRR$ and the family of metrics $\DDD$ in the sense of Definition \ref{def:spec}: for every $\rho>0$\nome{$\rho$}{specification scale} 
and every $A\in \RRR$, there exists $\tau>0$\nome{$\tau$}{gap size in specification}
such that given any sequence of orbit segments, each of which starts and ends in $A$, and any sequence of transition times, each at least $\tau$, there exists a single orbit which has orbit segments which are within distance $\rho$ of each of the target orbit segments in turn, with distance computed  in the metric $d_{t_k}$ for each target orbit segment $(x_k, t_k)$, and taking exactly the prescribed transition time to switch from one to the next.
\end{condition}

Given $x\in X$ and $T\geq 0$, we denote the corresponding ergodic integral of a potential function $\ph\colon X\to \RR$ by
\begin{equation}\label{eqn:Phixt}
\Phi(x,T) = \int_0^T\ph(f_tx)\,dt.
\end{equation}
For each $\zeta>0$, we consider the \emph{Bowen ball} with respect to $d_t$, 
\begin{equation} \label{eqn:bowenball}
B_T(x,\zeta) = B_T(x,\zeta, \DDD) := \{ y\in X :d_T(x,y)<\zeta\}
\end{equation}
consisting of all points $y$ 
for which the orbit segment $(y,T)$ is within $\zeta$ of $(x,T)$.

\begin{condition}[Bowen property]\label{cond:B}
The system $X,\FFF, \DDD, \ph$ has the Bowen property w.r.t.\ $\RRR$ in the sense of Definition \ref{def:Bowen}: for every $A\in \RRR$, there exists a scale $\rbow(A)>0$ and a constant $Q = Q(A) > 0$ such that for all $x\in A$ and $T\geq 0$ such that $f_T(x)\in A$, and every 
$y\in B_T(x,\rbow(A))$, we have $|\Phi(x,T) - \Phi(y,T)| \leq Q$.
\end{condition}
\begin{notation}\label{note:scales}
Throughout the paper, these letters will be used consistently:
\begin{itemize}
\item for specification, scale will always be denoted by $\rho$, and gap threshold by $\tau$;
\item for expansivity, scale will always be denoted by $\rexp$, and ``lag'' by $\alpha$;
\item for the Bowen property, the scale $\rbow(A)$ will always be denoted by $\eps$, and distortion constant by $Q$.
\end{itemize}
Both expansivity and the Bowen property remain true if their scales are decreased. In particular, we can -- and will -- always assume that when Conditions \ref{cond:E} and \ref{cond:B} are both satisfied, the scales are chosen such that $\rmet \geq \rexp > 2\eps > 0$. This will guarantee that all pressure appears at scale $\eps$; see \eqref{eqn:P-eps}. We will sometimes use $\eps$ to denote a scale at which all pressure appears, as in \eqref{eqn:P-eps}, even if we do not assume the Bowen property.
\end{notation}

When $X$ is compact, the family of metrics $\DDD$ must be the usual family of dynamical metrics, and all our definitions reduce to the standard ones. For non-compact $X$, the transition time $\tau$ and the distortion bound $Q$ can be arbitrarily large.

\subsection{Gurevich--Sarig pressure and the variational principle}\label{sec:GS-VP}
 
Topological pressure is defined using a \emph{partition sum} over uniformly separated orbit segments.
In our non-compact setting, we must further restrict our attention to orbit segments that start and end in a given reference set $A$.

To make this precise, we say that a set $E \subset X$ is \emph{$(T,\zeta)$-separated} if for every $x,y \in E$ with $x\neq y$, 
we have $d_T(x,y) > \zeta$. Now given a reference set $A\in \RRR$ and $T>0$, write%
\nome{$A_T$}{points that return at time $T$, that is, $A\cap f_{-T} A$}
\begin{equation}\label{eqn:AT}
A_T := A \cap f_{-T} A,
\end{equation}
and define the associated \emph{partition sum} at scale $\zeta>0$ by
\begin{equation}\label{eqn:part-sum-1}
\Lambda(\ph,T,\zeta,A_T) := \sup \Big \{ \sum_{x\in E} e^{\Phi(x, T)} :
E \subset A_T \text{ is $(T,\zeta$)-separated} \Big\}.
\end{equation}
The \emph{Gurevich--Sarig pressure}\terms{Gurevich--Sarig pressure} at scale $\zeta$ with respect to $A$ is\nome{$P^A(\ph,\zeta)$}{Gurevich--Sarig pressure, also $P^A(\ph)$ and $P(\ph)$}
\begin{equation}\label{eqn:GSP-A}
P^A(\ph,\zeta) := \limsup_{T\to\infty} \frac 1T \log \Lambda(\ph,T,\zeta,A_T).
\end{equation}
Further, we define
\begin{equation}\label{eqn:GSP}
P^A(\ph) := \lim_{\zeta\to 0} P^A(\ph,\zeta),
\qquad
P(\ph) := \sup_{A\in \RRR} P^A(\ph).
\end{equation}
We define the (Gurevich--Sarig) entropy by setting $h^A= P^A(0)$ and $h_{GS}= P(0)$.

Let $\Mf$\nome{$\Mf$}{set of flow-invariant Borel probabilities}
denote the set of flow-invariant Borel
probability measures on $X$.
See \S\ref{sec:prob} for properties of this space in this non-compact setting.
To avoid expressions of the form $\infty - \infty$ in what follows, we write the positive and negative parts of the potential as $\ph^+ = \max(\ph,0)$ and $\ph^- = \ph - \ph^+$, and restrict our attention to\nome{$\Mf^\ph$}{set of $\mu \in \Mf$ avoiding $\infty - \infty$}
\begin{equation}\label{eqn:finiteinvariantmeasures}
\Mf^\varphi := \Big\{ \mu \in \Mf : h_\mu(\FFF) + \int \ph^+ \,d\mu < \infty  \text{ or } \int \varphi^- \,d \mu < \infty \Big\},
\end{equation}
where $h_\mu(\FFF)$ is the measure-theoretic (Kolmogorov--Sinai) entropy of the flow, or equivalently, of its time-$1$ map $f_1$, with respect to $\mu$.

If $\sup \{ h_\mu(\FFF) : \mu \in \Mf \} < \infty$\nome{$h$}{$\sup h(\mu)$, note that both $h(\mu)$ and $h_\mu(\FFF)$ appear, need to unify}
and $\sup \ph < \infty$,
or if $\inf \varphi > -\infty$,
then $\Mf^\varphi = \Mf$.  Consider the following collection of subsets:\nome{$\mathcal{K}$}{collection of compact flow-invariant subsets of $X$}\nome{$\mathcal{K}^*$}{collection of $K\in \mathcal{K}$ with specification}
\begin{align}\label{eqn:K}
\mathcal{K} &:= \{ K\subset X : K \text{ is compact and flow-invariant} \}. %\\
\end{align}
We write $P(K,\ph)$ for the pressure of $\ph$ on the restriction of the flow to $K$; this agrees with the classical topological pressure.\nome{$P(K,\ph)$}{pressure of $\ph$ on compact invariant set $K$}
We prove the following in \S\ref{sec:GS}.

\begin{thma}[Variational Principle]\label{thm:welldefined}
Let $X, \FFF, \DDD, \ph$ satisfy Conditions \ref{cond:basic}, \ref{cond:S} and \ref{cond:B}. Then for every $A\in \RRR$, we have
\begin{equation}\label{eqn:welldefined}
P(\ph) = P^A(\ph)
 = \sup_{\mu \in \Mf^\ph} \Big( h_\mu(\FFF) + \int \ph\,d\mu \Big) \\
= \sup_{K\in \mathcal{K}} P(K,\ph).
\end{equation}
Now let $X, \FFF, \DDD, \ph$ satisfy Conditions \ref{cond:basic}--\ref{cond:B}.  If $\eps < \frac 12 \rexp(A)$, we have
\begin{equation}\label{eqn:P-eps}
P(\ph) = P^A(\ph,\eps).
\end{equation}
Furthermore, the supremum over compact invariant sets is not attained:
\begin{equation}\label{eqn:K-gap}
\text{for every $K\in \mathcal{K}$, we have $P(K,\ph) < P(\ph)$.}
\end{equation}
\end{thma}

In the proof of \eqref{eqn:welldefined} and \eqref{eqn:P-eps}, Condition \ref{cond:B} is weakened to a  ``tempered distortion'' condition that holds for all uniformly continuous functions.

\begin{remark}\label{rmk:metric-choice}
A priori, the value of the Gurevich--Sarig pressure in \eqref{eqn:welldefined} depends on the family of metrics $\DDD$, while the supremum over $\Mf^\ph$ depends only on the topology. We conclude that within the class of coherent families of metrics $\DDD$ satisfying Conditions
\ref{cond:basic}, \ref{cond:S} and \ref{cond:B}, pressure is independent of the choice of $\DDD$. These conditions might be satisfied for some choices of families of metrics $\DDD$ and not for others, so one can imagine needing to find a `good' family of metrics $\DDD$ to verify our hypotheses and thus satisfy the variational principle. For our application to geodesic flows, the families of dynamical metrics $(\dbow_t)$ derived from the natural metric $d$ satisfy \ref{cond:basic}, \ref{cond:S} and \ref{cond:B}. However, we require a coherent family $(\dH_t)$ which is not a family of Bowen metrics for a single metric $d$ to additionally obtain the expansivity property of Condition \ref{cond:E}.
\end{remark}

\begin{remark}\label{rmk:metric-choice-2}
In the setting of Theorem \ref{thm:welldefined}, it follows from \eqref{eqn:welldefined} that the Gurevich--Sarig entropy $P(0)$ agrees with the Handel--Kitchens entropy, see \S\ref{sec:previous-results-in-the-literature}.
\end{remark}

For the intermediate results stated as Theorems \ref{thm:uniform}--\ref{thm:existsES}, 
we do not need the full strength of the expansivity property in Condition \ref{cond:E}; for these results,
it suffices to assume that there exists $\eps= \eps(A)$ such that $P(\ph) = P^A(\ph,\eps)$. In \S\S\ref{sec:counting}--\ref{sec:eq-st}, we state and prove versions of these results under this weaker condition.

\subsection{Strong positive recurrence and uniform bounds on partition sums}\label{sec:SPR-ES}

To study \emph{equilibrium states} achieving the supremum in \eqref{eqn:welldefined}, we will need stronger estimates on the partition sums from \eqref{eqn:part-sum-1}, for which we use a notion of strong positive recurrence
inspired by Sarig's criterion for countable-state Markov shifts; see \S\ref{sec:previous-results-in-the-literature} for the relationship between our condition and his.%

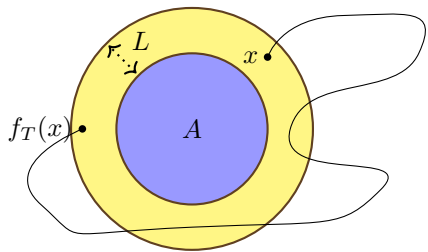
\begin{figure}[htbp]
\begin{tikzpicture}[circ/.style={circle,fill,inner sep=1pt}]
\draw[draw=brown!50!black,thick,fill=yellow!60!white] (0,0) circle(1.6);
\draw[draw=brown!50!black,thick,fill=blue!40!white] (0,0) node{$A$} circle (1);
%\draw (-.9,.9) node{$A^L$};
\draw[<->,dotted,thick] (135:1.04)--(135:1.56);
\draw (135:1.3) node[above right]{$L$};
\draw (1,.95) node[circ]{} node[left]{$x$};
\draw plot[smooth,tension=1.5] coordinates
{(1,.95)  (2,1.5)  (3,1)  (1.3,0) 
 (2.5,-0.8)  (0,-1.3)  (-2,-1)  (-1.45,0)};
\draw (-1.45,0) node[circ]{} node[left]{$f_T(x)$};
\end{tikzpicture}
\caption{An orbit segment contributing to $\Lambda^*(\ph,T,\zeta,L,A)$.}
\label{fig:SPR}
\end{figure}

To state our condition, fix $A\in \RRR$ and $\zeta,L,T>0$,
and consider the partition sum $\Lambda^*(\ph,T,\zeta,L,A)$\nome{$\Lambda^*(\ph,T,\zeta,L,A)$}{restricted partition sum}
obtained by modifying \eqref{eqn:part-sum-1} to use a sum over $\zeta$-separated orbit segments of the sort shown in Figure \ref{fig:SPR}, that\nome{$L$}{`padding' size in definition of SPR}\nome{$\eps$}{scale for SPR, might be overloaded}
\begin{enumerate}
\item begin and end within a distance $L$ of $A$, and
\item do not enter $A$ between time $0$ and time $T$.
\end{enumerate}

\begin{condition}[SPR]\label{cond:SPR}
We say that $\ph$ is \emph{strong positive recurrent (SPR)}  \terms{strongly positive recurrent}
for $A\in \RRR$ (in the family of metrics $\DDD$) if $P(\ph)<\infty$ and if%
\begin{equation}\label{eqn:SPR-intro}
\sup_{\zeta>0} \inf_{L>0} \limsup_{T\to\infty} \frac 1T \log \Lambda^*(\ph,T,\zeta,A, L) < P^A(\ph).
\end{equation}
\end{condition}
See \S\ref{s.SPR} for further details and discussion.
In particular, we show that 
\eqref{eqn:SPR-intro} is equivalent to the a priori stronger condition obtained by replacing $\inf_L$ with $\sup_L$.
In that section, we also show that if $\ph$ is SPR for $A\in \RRR$, then for every equilibrium state $\nu$ and every $\delta>0$, we have $\nu(A^\delta)>0$.

The SPR property implies uniform counting bounds on the partition sums, and also (with modified hypotheses) on weighted sums over periodic orbits.
We write $\Per(A,T,\alpha)$ for the set of all periodic orbits with period in $(T-\alpha,T]$ that intersect $A$, and given $c \in \Per(A,T, \alpha)$, we write $\Phi(c) = \Phi(x,t)$, where $t$ is the period of $c$ and $x = f_t(x)$ is any point on $c$.  We say that $\Per(A,T,\alpha)$ is $(T, \eps)$-separated if
given any $x,y \in A$ lying on different periodic orbits in $\Per(A,T,\alpha)$, we have $d_T(x,y) > \eps$.
 %$\{x, y\}$ is $(T, \eps)$-separated for any $x \in \gamma_1$, $y \in \gamma_2$ with $\gamma_1, \gamma_2 \in \Per(A,T,\alpha)$ with $\gamma_1 \neq \gamma_2$.

\begin{thma}[Uniform counting bounds]\label{thm:uniform}
Let $X,\FFF,\DDD,\ph$ satisfy Conditions \ref{cond:basic}--\ref{cond:SPR}. Fix $A\in \RRR$ and $\eps>0$ as in Convention \ref{note:scales}.  
Then for all $\zeta\in (0,\fraceps]$, there are constants %$C_U = C_U(A,\zeta)$, $C_L = C_L(A,\zeta)$ and $T_0=T_0(A,\zeta)$ 
$C_U,C_L,T_0>0$ depending on $(A,\zeta)$
such that for all $T>T_0$ we have
\begin{equation}\label{eqn:unif-counting-intro}
C_L e^{TP(\ph)} \leq \Lambda(\ph,T,\zeta,A_T) \leq C_U e^{TP(\ph)}.
\end{equation}
If the specification property in Condition \ref{cond:S} is strengthened to \emph{periodic} specification (see \S\ref{s.defns}), and for any $\alpha>0$ there exists $\zeta^\ast>0$ so that the set $\Per(A,T,\alpha)$ is $(T, \zeta^\ast)$-separated for all $T$, then there exist constants $C_L^*, C_U^*>0$ and $T_0>0$ such that for all $T>T_0$, we have
\begin{equation}\label{eqn:unif-count-per}
C_L^* \frac{1}{T} e^{TP(\ph)}
\leq \sum_{c\in \Per(A,T, \alpha)} e^{\Phi(c)}
\leq C_U^* e^{TP(\ph)}.
\end{equation}
\end{thma}
The lower bound in \eqref{eqn:unif-counting-intro} holds for partition sums over \emph{any} maximal $(T,\zeta)$-separated set in $A_T$; see \eqref{eqn:unif-counting-again-again}. 
%The bounds in \eqref{eqn:unif-counting-intro} play a crucial role in our next results concerning equilibrium states. 
The periodic specification property and the separation condition on $\Per(A,T,\alpha)$ are satisfied for our geodesic flows. The bounds in \eqref{eqn:unif-count-per} point in the direction of the Margulis asymptotics for periodic orbits, a long-standing motivating theme in this area.

\subsection{Misiurewicz measures, Gibbs measures, and equilibrium states}\label{sec:Mis}

For compact systems, Misiurewicz's proof of the variational principle \cite{mM76}  includes a construction that produces equilibrium states whenever all pressure has appeared at a certain scale. The first step is to
fix, for each $t>0$, a maximal $\zeta$-separated set $E_t$ of orbit segments, and then define measures $\mu_t$ by placing mass on the orbit segments according to the time averages of $\phi$. That is, we construct the measures
\begin{equation}
\sigma_t = \frac{\sum_{x\in E_t} e^{\Phi(x,t)} \delta_x} {\sum_{y\in E_t} e^{\Phi(y,t)}}
\quad\text{and}\quad
\mu_t = \frac 1t \int_0^t (f_s)_* \sigma_t \,ds.
\end{equation}
One then argues that given any $t_k\to \infty$ such that the weak*-limit $\mu := \lim_k \mu_{t_k}$ exists, this limiting measure is an equilibrium state; see \cite[\S9.3]{pW82}. 
Versions of this construction for periodic orbit segments appeared earlier in the work of Margulis \cite[\S6, Theorem 4]{gM04}, Bowen \cite{rB71, rB72geod}, and Ruelle \cite{dR73}. 
%For measures equidistributed along periodic orbits, this idea goes back at least to the 1970 thesis of Margulis \cite[\S6, Theorem 4]{gM04} for $\ph\equiv 0$ and the MME, and also to Bowen %\cite{rB71,rB72}. The case of more general potentials was studied by Ruelle \cite{dR73}.
Under the conditions of specification, expansivity, and the Bowen property, $\mu$ can be shown to satisfy the Gibbs property and to be the unique equilibrium state.

In the non-compact setting, we adapt this construction, using orbit segments that start and end in a fixed reference set $A$. We call $(\mu_t)_t$ a \emph{Misiurewicz family at scale $\zeta$ with respect to $A$}, and any weak*-accumulation point as $t\to\infty$ will be called a \emph{Misiurewicz measure}; see \S\ref{sec:Bow-Mis-exists} for precise definitions. When $X$ is not compact, there are families of probability measures with no weak*-accumulation points, so unlike the compact setting, we have no \emph{a priori} guarantee that Misiurewicz measures exist, and we must establish a \emph{tightness} condition; see \S\ref{sec:prob}.

The Gibbs property that we prove, which is stated precisely in \S\ref{sec:Gibbs}, will be in terms of orbit segments that start and end in $A$:  we will obtain lower and upper Gibbs bounds under the condition that $x\in A_t$.
The constants appearing in these bounds are allowed to depend on $\zeta$ and $A$, corresponding to what Buzzi calls (in the symbolic setting) the \emph{weak Gibbs property} in the appendix of \cite{BPP19}.

\begin{thma}[Misiurewicz measures]\label{thm:mainGibbs}
Let $X,\FFF,\DDD,\ph$ satisfy 
Conditions \ref{cond:basic}--\ref{cond:SPR}.  Fix $A\in \RRR$ and $\eps>0$ as in Convention \ref{note:scales}. Then the following are true.
\begin{enumerate}[label=\upshape{(\arabic{*})}]
\item Every Misiurewicz family for $\ph$ at scale $\zeta \in(0, \fraceps)$ w.r.t.\ $A$  is tight. 
In particular, there is a weak*-accumulation point, so the set of Misiurewicz  measures is non-empty;
\item\label{epso} 
If $A \in \RRR$ contains an $\epso$-ball for some $\epso \in (0,\fracfraceps)$, then
any Misiurewicz  measure
at scale $\epso$
 for $\ph$ w.r.t.\ $A$ is an $\FFF$-invariant probability measure that satisfies the lower Gibbs property at scale $\twoepso$ and the upper Gibbs property at scale $\halfepso$ on any $A' \in \RRR$.
\end{enumerate}
\end{thma}
By analogy with the results for compact systems,
Misiurewicz measures are natural candidates for equilibrium states. To proceed, though, we need to make some stronger assumptions on our system.

\begin{condition}\label{cond:UESB}
The system $(X,\FFF,\DDD,\ph)$ has the following properties.
\begin{itemize}
\item Finite entropy: the Gurevich--Sarig entropy $h_{GS}$ is finite.
\item Upper bound for potential: the potential $\ph$ is bounded from above on $X$.
\end{itemize}
\end{condition}
The finite entropy assumption can be checked in examples by finding a metric in which $X$ has finite box dimension and $\FFF$ is globally Lipschitz, see Lemma \ref{Lipschitzimpliesfiniteentropy}.

\begin{thma}[Existence]\label{thm:existsES}
Let $X,\FFF,\DDD,\ph$ satisfy 
%Conditions \ref{cond:basic}--\ref{cond:UESB}. Suppose that $\ph$ is SPR w.r.t.\ $A\in \RRR$
Conditions \ref{cond:basic}--\ref{cond:UESB}. Fix $A\in \RRR$ and $\eps>0$ as in Convention \ref{note:scales},
and let $\mu_\ph$ be a Misiurewicz measure
provided by Theorem \ref{thm:mainGibbs}.
Then $\mu_\ph$ is an equilibrium state for $\ph$.
\end{thma}

Our proofs of Theorems \ref{thm:uniform}, \ref{thm:mainGibbs}, and \ref{thm:existsES} do not use the full strength of the expansivity property in Condition \ref{cond:E},
and in \S\S\ref{sec:counting}--\ref{sec:eq-st} we replace expansivity by the weaker condition that all pressure has appeared by a certain scale. To establish uniqueness, we do use the  expansivity property itself.

\begin{thma}[Uniqueness]\label{thm:mainES}
Let $X,\FFF,\DDD,\ph$ satisfy Conditions \ref{cond:basic}--\ref{cond:UESB}.  Then
\begin{enumerate}[label=\upshape{(\arabic{*})}]
\item there exists a unique equilibrium state $\mu_\ph$ for $\ph$;
\item for all $A \in \RRR$, $\mu_\ph$ is a lower Gibbs measure at all scales, and an upper Gibbs measure at all sufficiently small scales;
\item $\mu_\ph$ is fully supported.
\end{enumerate}
\end{thma}

\begin{coralph}\label{cor:equidist}
Let $X,\FFF,\DDD,\ph$ satisfy Conditions \ref{cond:basic}--\ref{cond:UESB}, and suppose additionally that   \emph{periodic} specification is satisfied and the set $\Per(A,T,\alpha)$ is $(T, \eps)$-separated for all $T$. Then for any $\alpha>0$, the weighted periodic orbits equidistribute to $\mu_\ph$ in the following sense:
the Borel probability measures $\mu_{A,T,\alpha}$ defined by
\begin{equation}\label{eqn:mu-AT}
\mu_{A,T, \alpha}(E) = \frac{\sum_{c\in \Per(A,T, \alpha)} \int_0^{|c|} \ph(c(t)) \one_E(c(t)) \,dt}{\sum_{c\in \Per(A,T, \alpha)} e^{\Phi(c)}}
\end{equation}
have the property that for every $A\in \RRR$, we have $\mu_{A,T, \alpha} \xrightarrow{\mathrm{weak*}} \mu_\ph$ as $T\to\infty$.
\end{coralph}

\subsection{Application to geodesic flows}\label{sec:app}

We apply our results to the class of geodesic flows on negative curvature spaces. Following Ballmann \cite{wB95}, we define a Hadamard space to be a complete $\CAT(0)$ space. A Hadamard manifold is thus a complete simply connected manifold with non-positive sectional curvatures. We refer to \cite{DT25, BH99, wB95, BPP19, tR03} for more background. We discuss the context and previous results about equilibrium states for geodesic flow in \S \ref{sec:previous-results-in-the-literature}. 

\begin{thma}\label{thm:CAT}\nome{$\tilde X$}{universal cover}\nome{$\kappa$}{`curvature' bound for $\CAT(-\kappa^2)$}\nome{$\Gamma$}{non-elementary torsion-free Kleinian group}\nome{$GX$}{space of geodesics}\nome{$\Omega X$}{space of nonwandering geodesics}
Let $X_0 = X / \Gamma$ where:
\begin{itemize}
\item $X$ is a Hadamard manifold with all sectional curvatures taking values in $(-\infty, 0)$ OR $X$ is a proper geodesically complete Hadamard space with Alexandrov curvature bounded above by $-\kappa^2$ for some $\kappa>0$;
\item  $\Gamma$ is a non-elementary torsion-free Kleinian group.
\end{itemize}
Let $\Omega X_0$ be the set of non-wandering geodesics. Assume that the geodesic flow $(g_t)$ on $\Omega X_0$ is topologically mixing and has finite entropy. Let $\RRR$ denote the set of all bounded closed sets with non-empty interior in  $\Omega X_0$. Let $\varphi\colon \Omega X_0 \to \RR$ be continuous and bounded above. Suppose that $\varphi$ is Bowen w.r.t.\ $\RRR$ in the metric family $\DDDH$.  Suppose that $P(\varphi)<\infty$ and that $\varphi$ is SPR for the geodesic flow on $\Omega X_0$. The system $(\Omega X_0, (g_t),\varphi)$ admits a unique equilibrium state $\mu_\varphi$. The measure $\mu_\varphi$ is Gibbs as in Theorem \ref{thm:mainES} and is fully supported on $\Omega X_0$.
\end{thma}

When $X$ is a Hadamard manifold, the Hopf--Rinow theorem guarantees that the space is proper and geodesically complete. When $X$ is a Hadamard space, the properness condition gives local compactness. When $X$ is a manifold (resp.\ $\CAT(-a^2)$ space), the assumption that $\Gamma$ is torsion-free is needed for $X_0$ to be a manifold rather than an orbifold (resp.\ a locally $\CAT(-a^2)$ space rather than an orbispace).  All negative curvature manifolds $X_0$ (resp.\ all locally $\CAT(-a^2)$ spaces) arise this way.
 
The assumption that the geodesic flow is topologically mixing is mild. The geodesic flow on the non-wandering set is always topologically transitive. Failure of mixing is associated with having an arithmetic length spectrum (i.e.\ all closed geodesics have length in the set $c \ZZ$ for some $c>0$), see \cite{fD00, tR03}. This is a highly restrictive condition and conjectured to be impossible for negatively curved manifolds, see \cite[\S8.1]{PPS15}.  A version of our general results in which the specification property is weakened to allow transition times that are only bounded above by $\tau$ would allow us to immediately remove the topological mixing assumption. 
 
The entropy is finite if $(\Omega X_0, d_{GX})$ has finite box dimension. This holds for a large class of $\CAT(-a^2)$ spaces and negative curvature manifolds, see \S \ref{sec:finiteentropy}. We show in \S \ref{s.holderpotentials} that if $X$ is a pinched negative curvature manifold or a $\CAT(-a^2)$ space, every potential which is H\"older in a natural sense satisfies our Bowen property, so our result generalizes the corresponding result in \cite{PPS15}.

We emphasize that the Gibbs property provided by Theorem \ref{thm:CAT} applies at all sufficiently small scales, depending on the reference set. This contrasts with the Gibbs property obtained in \cite{PPS15, DT25} which is proved for Bowen balls in the universal cover only at sufficiently large scales.

We expect that the general results and techniques of this paper will apply to a wide range of examples;
further applications will be explored elsewhere.

\subsection{Previous results in the literature}\label{sec:previous-results-in-the-literature} 

Thermodynamic formalism for compact dynamical systems is a mature theory, and we make no attempt to survey the enormous range of results in this setting. Leading techniques include transfer operator methods, symbolic dynamics, and Bowen's specification approach. For the  geodesic flow on a closed manifold with negative curvature, the results are classical, and excellent references are \cite{BR75, FH19}. Tremendous recent progress using countable-state symbolic dynamics has been made for geodesic flows on closed surfaces without curvature assumptions \cite{LOP}. The survey article \cite{CT21} focuses on the specification approach in compact settings beyond uniform hyperbolicity. 

For non-compact systems, the theory has been well-developed in two settings: countable-state Markov shifts, and geodesic flows in pinched negative curvature. 
We discuss the general non-compact theory, then survey prior work in these settings.

\subsubsection*{General Theory} 
The theory in general non-compact settings is limited to various versions of the variational principle  and entropy formulas. We do not know of any previous results on the existence and uniqueness theory of equilibrium states in a general non-compact setting. Handel and Kitchens proved a variational principle for entropy for locally compact metric spaces in \cite{HK95}. Their definition of entropy uses $(t, \eps)$-separated sets which start in a compact set $K$, but in contrast to the Gurevich--Sarig definition, there is no restriction on where the orbit segment ends. In a fixed metric, this notion of entropy is only an upper bound for the measure-theoretic entropies. They show that taking an infimum of this quantity over all metrics agrees with the supremum of the measure-theoretic entropies. The disadvantage of this approach is that one does not know anything about the uniform class of the metric obtained by taking the infimum over all possible metrics, or even if the infimum is attained. This is a fundamental obstruction to using the Handel--Kitchens topological entropy in our analysis, and why we need to prove a variational principle using a definition that is in the spirit of Gurevich--Sarig.

There is a general variational principle of Pesin and Pitskel \cite{PP84} using ideas from dimension theory, but it relies on a space being a subset of a compact space, or a change of metric to compactify. The Pesin--Pitskel variational principle applies to an ergodic-theoretically defined subset of the space, and in general gives only an upper bound on the variational pressure. 

The paper \cite{HNP} clarifies that several notions of topological entropy in non-compact spaces crucially rely on the uniform class of the metric, and that these notions only give an upper bound on the variational pressure. Our Gurevich--Sarig definition of entropy allows us to work with a fixed family of metrics and to obtain equality in the variational principle in our setting. This is a major advantage of our approach which enables us to develop the thermodynamic formalism further. The variational principle for Gurevich--Sarig pressure was proved for countable state shifts of finite type in \cite{oS99} and for geodesic flow for pinched negative curvature manifolds in \cite{PPS15}.

A useful general result which is relevant for thermodynamic formalism is the Brin--Katok formula and Ruelle inequality obtained by Riquelme \cite{fR18}.  The recent preprint of Florio, Schapira, and Vaugon \cite{FSV} gives a variational principle for entropy and other results under hypotheses which we discussed in \S \ref{sec:overview}. 

\subsubsection*{Countable-state Markov Shifts}  

The theory of thermodynamic formalism for topological Markov shifts on a countable alphabet is highly developed, and serves as the inspiration for many aspects of this paper. A broad range of results can be found in the work of Mauldin and Urba\'nski \cite{MU96,MU01,MU03} and of Sarig \cite{oS99, oS01}; see \cite{oS15} for an overview of the area.
The introduction of ideas from recurrence theory has led to an extremely successful and influential body of work, which remains highly topical due to breakthroughs in the coding of smooth dynamical systems on compact phase spaces using countable state Markov shifts initiated in \cite{oS13}. This has led to further breakthroughs in thermodynamic formalism for smooth maps and flows on in compact manifolds \cite{BCS22, LOP}. 

Sarig's original definition \cite{oS01} of strong positive recurrence for potentials on countable-state Markov shifts was given in terms of positivity of a discriminant defined in terms of an induced pressure function. This definition is equivalent to a symbolic version of \eqref{eqn:SPR-intro}, see \cite[(2.13) and \S8.5]{vC18} for details. In \cite{CS09}, Cyr and Sarig proved that SPR is equivalent to existence of a suitable Banach space on which the transfer operator acts with a spectral gap.

\subsubsection*{Geodesic flows}

The theory of thermodynamic formalism for geodesic flows on pinched negative curvature manifolds is more recent, and has developed analogously to the theory of countable-state Markov shifts. The proof techniques are based on asymptotic geometry and the constructions pioneered by Patterson and Sullivan. The breakthrough result in this direction was Otal and Peign\'e's result \cite{OP04} that for a pinched negative curvature manifold, if an MME exists, then it is unique. This was extended to equilibrium states in the excellent manuscript \cite{PPS15}. In addition to pinched curvature, these papers require the curvature to have bounded first derivatives so that the foliations have H\"older regularity. The second author and Dilsavor in \cite{DT25} showed that this is not required and addressed an issue with the construction of partitions in \cite{OP04}.  Ledrappier's account of Otal and Peign\'e's proof in \cite{fL13} also sidesteps these issues. 

With these results in hand, the focus turns to criteria for uniqueness and further properties of equilibrium states. A notion of positive recurrence was developed by Pit and Schapira \cite{PS18} for pinched negative curvature manifolds, and they showed that this is equivalent to finiteness of the Gibbs measure and thus its uniqueness as an equilibrium state. The notion of an `entropy gap at infinity' was developed in papers including \cite{GST23, IRV}.  In particular, \cite{GST23} defined SPR as a `pressure gap at infinity' property. They observe that this is analogous to the pressure gap in \cite{BCFT}, which is indeed the basis of our approach here. In the pinched curvature setting with bounded derivatives of curvature, our results largely overlap with what is obtained by combining the uniqueness result of \cite{PPS15} with the SPR criterion from \cite{GST23}, although we give a different perspective and a weaker regularity condition on the potentials.

The place where our results break substantial new ground is in removing the curvature pinching condition, and weakening the manifold assumption. To the best of our knowledge, there are no previous results about MMEs or equilibrium states for manifolds with negative curvature approaching $0$ and $-\infty$, and we believe that pinching is needed for the current geometric arguments due to `bad' geometric phenomena which are possible, see \cite{iB16,nP11} (particularly, failure of visibility means that the Hopf coordinates are not always well-defined).

Our approach yields new results in the setting of geodesic flows over locally $\CAT(-1)$ spaces. The compact case was studied in \cite{CLT20a, CLT20} using first the specification property and then symbolic dynamics. The book \cite{BPP19} develops the Patterson--Sullivan construction in the non-compact setting, but only for potentials that depend on the space $X$ rather than on the space of geodesics $GX$, which is the natural domain for potentials. They construct a Bowen--Margulis measure for such potentials, but they only characterize it as an equilibrium state when the space is a tree. The paper by Dilsavor and the second named author in \cite{DT25} developed thermodynamic formalism for $\CAT(-1)$ spaces (also allowing the group $\Gamma$ to have torsion). This requires a coarse version of the Patterson--Sullivan construction to allow the potentials to depend on the natural phase space $GX$. The potentials in \cite{DT25} are bounded and satisfy a uniform Bowen property, and the result there is analogous to the main result of \cite{PPS15}:  without any recurrence assumption, a Radon measure is constructed. If the measure is conservative, it satisfies a Gibbs property; if the measure is finite, it is the unique equilibrium state. In contrast to the current work, the paper \cite{DT25} does not give any information on \emph{when} the measure is finite.  In the case that $\Gamma$ is torsion-free and SPR is satisfied, the current work is more general than \cite{DT25} by allowing the potential to be unbounded below and allowing the Bowen regularity to depend on the reference set.  We emphasize that the study of equilibrium states for $\CAT(-1)$ spaces poses major methodological challenges since the usual Patterson--Sullivan construction is not available. This is because expressions of the form $\int_x^y \varphi$ are not well-defined for a potential $\varphi\colon GX\to \RR$. New ideas such as the coarse geometry approach in \cite{DT25} or the specification approach here are needed to make progress.

\subsection{Organization} 

In \S \ref{s.defns}, we introduce our basic setup and definitions. In \S \ref{sec:GS}, we prove Theorem \ref{thm:welldefined}, establishing the well-definedness of our Gurevich--Sarig pressure and the Variational Principle. In \S \ref{sec:counting}, we prove Theorem \ref{thm:uniform} 
on uniform counting bounds. 
In \S\ref{sec:MandGibbs}, we prove Theorem \ref{thm:mainGibbs} by constructing Misiurewicz  measures and establishing Gibbs bounds for them.
In \S\ref{sec:eq-st}, we prove Theorem \ref{thm:existsES}, showing that any Misiurewicz measure is an equilibrium state.
In \S\ref{sec:uniform-katok}, we describe an inducing procedure that leads to a uniform Katok estimate. We rely on this in \S\ref{sec:eu-other} for the proof of uniqueness, which then leads to the other conclusions of Theorem \ref{thm:mainES}.
In \S\ref{s.application}, we develop our application to geodesic flows, proving Theorem \ref{thm:CAT}.

\begin{comment}
\draw[to] (all-spec) -- (r-spec);
\draw[to] (r-spec) to[out=30,in=150] node[midway,above]{Thm \ref{thm:supmult}} (super-mult);
\draw[to] (all-spec) to[out=30,in=150] node[midway,above]{Thm \ref{thm:other-vp}} (VP);
\draw[to] (super-mult) -- (count-leq);
\draw[to] (sub-mult) -- (count-geq);
\draw (Pexp) -- node[midway,and] (A) {} (VP);
\draw[to] (A) -- node[midway,below left] {Thm \ref{thm:all-P}} (Pr=P);
\draw (r-spec) to[out=-150,in=90] node[pos=.63,and] (B) {} (scales);
\draw[to] (B) -- node[midway,above]{Thm \ref{thm:scales}} (all-spec);
\draw (count-leq) -- node[midway,and] (C) {} (SPR);
\draw[to] (C) -- node[midway,right] {Thm \ref{thm:lower-count}} (sub-mult);
\draw (count-leq) to[out=-60,in=60] node[midway,and] (D) {} (Pr=P);
\draw[to] (D) -- (mu-leq);
\draw (count-geq) to[out=-45,in=45] node[pos=1/3,and] (E) {} (Pr=P);
\draw[to] (E) -- (mu-geq);
\draw (scales) to[out=120,in=210] node[pos=1/3,and] (F) {} (SPR);
\draw[to] (r-spec) to[out=180,in=60] (F) -- node[pos=.6] {Thm \ref{thm:get-gap}} (Pexp);
\end{comment}

\section{Definitions and basic properties} \label{s.defns}

%\subsection{Metrics and Bowen balls} \label{s:topassump} 

Throughout the paper, $X,\FFF,\DDD,\ph$ will be as in Condition \ref{cond:basic}: $X$ is a locally compact separable completely metrizable space, $\FFF = (f_t)_{t\in \RR}$ is a continuous flow on $X$, 
$\DDD = (d_t)_{t\geq 0}$ is a coherent family of compatible metrics as in Definition \ref{def:coherent-metrics},
and $\ph\colon X\to \RR$ is continuous.

\subsection{Basic notation and reference sets}\label{sec:ref-sets}
Given an interval $I \subset \mathbb R$, we often use the notation $f_I(x) = \{ f_sx : s \in I\}$. 
We extend the definition of Bowen ball from \eqref{eqn:bowenball} as follows: given $x\in X$, $\zeta>0$, and a bounded interval $I=[t_1, t_2]\subset \RR$ with length $|I|$, we write
\[
B_I(x,\rho) = \{y\in X : d_{|I|}(f_{t_1}x,f_{t_1}y)<\rho \}.
\]
Given $E\subset X$, let\nome{$B_n(E,\rho)$}{union of Bowen balls over $x\in E$}
$B_T(E,\rho) := \bigcup_{x\in E} B_T(x,\rho)$.

For our collection of reference sets, we let $\RRR$ be the collection of compact sets with non-empty interior in the space $X$.  We often need to adjust a reference set by slightly shrinking or enlarging it in the metric $d$, or by flowing the set out, as follows.  Given $S \subset X$ and $r>0$, consider the closed sets
\begin{equation}\label{eqn:shrinkenlarge}
S^r := \{x \in X : d(x, S) \leq r\} \quad \text{ and } \quad
S^{-r} := \{x \in S : B(x, r) \subset S\},
\end{equation}
so that $S^{-r} \subset S \subset S^r$.
For each $\tau>0$, we also consider the set
\begin{equation} \label{eqn:flowout}
S_{\mathcal O}^\tau:=f_{[-\tau,\tau]}(S) = \bigcup_{|t| \leq \tau} f_t(S)
\end{equation}
obtained from $S$ by flowing out along orbits for time $\tau$.
For every $A\in \RRR$, it follows from continuity of the flow and compactness of $A$ that $A^r$, $A^{-r}$, and $\AO[\tau]$ are compact for all $r,\tau>0$.
Furthermore, for $A\in \RRR$, the sets $A^r$ and $\AO[\tau]$ belong to $\RRR$ for all $r, \tau >0$ and all $r<0$ sufficiently small in modulus.

 \begin{lemma}\label{lem:bdd-excursions}
 The flow satisfies:\terms{bounded excursions}
 \begin{enumerate}[label=\upshape{(\alph{*})}]
\item \emph{bounded excursions} with respect to $\RRR$: For all $A \in \RRR$ and $t>0$, there exists $A' \in \RRR$ such that for all $|s|\leq t$, we have $f_s(A) \subset A'$. 
 \item\terms{uniform speed}\nome{$r$}{time bound in uniform speed definition, to stay within $\gamma$}\nome{$\gamma$}{distance bound in uniform speed}
 \emph{uniform speed} on $A\in \RRR$:
 for all $A \in \RRR$ and $\gamma>0$, there exists $r>0$ such that for all $x\in A$ and $t\in [-r,r]$, we have $d(f_tx, x)< \gamma$.
\end{enumerate}
 \end{lemma}
 \begin{proof}
For bounded excursions, 
it suffices to observe that $f_s(A) \subset \AO[t] \in \RRR$ for every $s\in [-t,t]$.
For uniform speed, given $r>0$,  define $\beta_r\colon A \to [0,\infty)$ by
\[
\beta_r(x) := \diam (f_{[-r,r]}(x)).
\]
For each $x\in A$, continuity gives $\beta_r(x) \searrow 0$ monotonically as $r\to 0$. Since $A$ is compact, Dini's theorem implies that $\beta_r \to 0$ uniformly, which proves the result.
\end{proof}

Using the bounded excursion property, continuity of $\ph$, and the fact that reference sets are compact, the following lemma is an easy exercise.

\begin{lemma}\label{lem:L}
For every $A\in \RRR$ and $\tau>0$, there exists $\Vt= \Vt(A, \tau)>0$\nome{$L$}{bound on ergodic integrals starting and ending in $A$, length $\tau$} 
such that if $x\in X$ has the property that $f_t x \in A$ for some $t\in [0,\tau]$, then $|\Phi(x,\tau)| \leq \Vt$.
\end{lemma}

\subsection{Specification}\label{s:spec}

As described in \S\ref{sec:overview},
given $x\in X$ and $t\in [0,\infty)$, we identify the pair $(x,t)$
with the orbit segment
$[0,t] \to X$ defined by $s \mapsto f_s(x)$.
%, when $t>0$, and the `length $0$' orbit segment $(x,0)$ is identified with the point $x$.
Given $\bx = (x_1,\dots, x_k) \in X^k$ and $\bt = (t_1,\dots, t_k) \in [0,\infty)^k$,
we identify $(\bx, \bt)$ with the sequence of orbit segments $((x_1, t_1), \ldots , (x_k, t_k))$.

For a compact dynamical system, the (classical) specification property says that we are able to 
approximate arbitrary sequences of orbit segments with a single orbit, provided the time allowed to transition from one orbit segment to the next exceeds a predetermined threshold. That is:

\begin{definition}\label{def:class-spec}
Let $K$ be a compact metric space and $\FFF\colon K \to K$ be a continuous flow. We say that $(K, \FFF)$ has the \emph{classical specification property at scale $\rho>0$} if there exists $\tau\geq 0$ such 
that for any $k\in \NN$,
and any sequence of orbit segments and transition times given as
\[
%k\in \NN,\quad
((x_1,t_1),\dots,(x_k,t_k)) \in (K \times [0, \infty))^k
\quad\text{and}\quad
\btau = (\tau_1,\dots,\tau_{k-1}) \in [\tau,\infty)^{k-1},
\]
%$k\in \NN$, any $((x_1,t_1),\dots,(x_k,t_k)) \in (K \times [0, \infty))^k$,\nome{$(\bx,\bt)$}{sequence of orbit segments for specification, $(x_1,t_1),\dots, (x_k,t_k)$} 
%and any $\btau = (\tau_1,\dots,\tau_{k-1}) \in [\tau,\infty)^{k-1}$,\nome{$\btau$}{sequence of gap sizes, $\tau_1,\dots,\tau_{k-1}$}
there is a \emph{shadowing point} $y \in K $ such that for all $j\in \{1,\dots, k\}$, we have
\begin{equation}\label{eqn:shadows}
f_{T_j}(y) \in B_{t_j}(x_j,\rho) %\text{ for all } 1\leq j\leq k,
\text{ where } 
T_j = t_1 + \tau_1 + t_2 + \tau_2 + \cdots + \tau_{j-1};
\end{equation}
%,  writing $T_j = t_1 + \tau_1 + t_2 + \tau_2 + \cdots + \tau_{j-1}$,\nome{$T_j$}{start time of $j$th orbit segment, $t_1 + \tau_1 + \cdots + \tau_{j-1}$} 
%we have  $f_{T_j}(y) \in B_{t_j}(x_j,\rho)$ for all $1\leq j\leq k$. 
see Figure \ref{fig:spec-book}.
We say the system $(K, \FFF)$ has the classical specification property if the flow has specification for all scales. (The value of $\tau$ is allowed to depend on $\rho$.)
\end{definition}

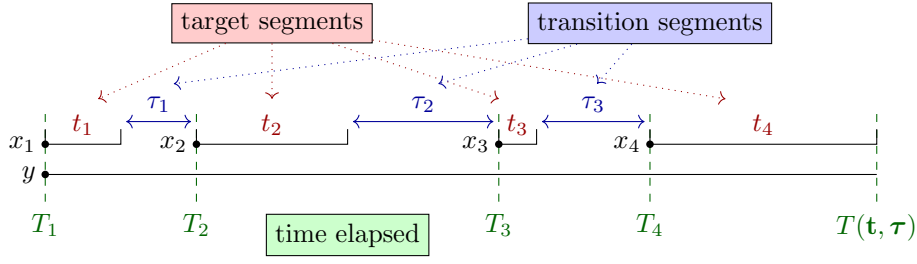
\begin{figure}[htbp]
\begin{tikzpicture}%
[
circ/.style={circle,fill,inner sep=1pt},
myarr/.style={{<[sep=2pt]}-{>[sep=2pt]},color=blue!60!black},
len/.style={pos=0.5,above,color=red!60!black},
totarget/.style={red!60!black,->,dotted},
totransition/.style={blue!60!black,->,dotted}
]
\def\h{0.2}
\draw[color=green!40!black,dashed] (0,2*\h)--(0,-4*\h) node[below]{$T_1$};
\draw[color=green!40!black,dashed] (2,2*\h)--(2,-4*\h) node[below]{$T_2$};
\draw[color=green!40!black,dashed] (6,2*\h)--(6,-4*\h) node[below]{$T_3$};
\draw[color=green!40!black,dashed] (8,2*\h)--(8,-4*\h) node[below]{$T_4$};
\draw[color=green!40!black,dashed] (11,2*\h)--(11,-4*\h) node[below]{$T(\bt,\btau)$};
\draw (0,\h) -- (0,0) node[left]{$x_1$} node[circ]{}
-- (1,0) node[len]{$t_1$}
-- (1,\h);
\draw (2,\h) -- (2,0) node[left]{$x_2$}
node[circ]{}
-- (4,0) node[len]{$t_2$}
-- (4,\h);
\draw (6,\h) -- (6,0) node[left]{$x_3$}
node[circ]{}
-- (6.5,0) node[len]{$t_3$}
-- (6.5,\h);
\draw (8,\h) -- (8,0) node[left]{$x_4$}
node[circ]{}
-- (11,0) node[len]{$t_4$}
-- (11,\h);
\draw[myarr] (1,1.5*\h)--(2,1.5*\h) node[pos=0.5,above]{$\tau_1$};
\draw[myarr] (4,1.5*\h)--(6,1.5*\h) node[pos=0.5,above]{$\tau_2$};
\draw[myarr] (6.5,1.5*\h)--(8,1.5*\h) node[pos=0.5,above]{$\tau_3$};
\draw (0,-2*\h) node[circ]{} node[left]{$y$} -- (11,-2*\h);
\node[draw,fill=red!20,rectangle] (target) at (3,8*\h) {target segments};
\node[draw,fill=blue!20,rectangle] (transition) at (8,8*\h) {transition segments};
\foreach \x in {0.7, 3, 6, 9}
	{\draw[totarget] (target) -- (\x,3*\h);}
\foreach \x in {1.7, 5.2, 7.3}
	{\draw[totransition] (transition) -- (\x,4*\h);}
\node[draw,fill=green!20,rectangle] at (4,-6*\h) {time elapsed};
\end{tikzpicture}
\caption{Bookkeeping in the specification property.}
\label{fig:spec-book}
\end{figure}

In a non-compact space, this is too strong a condition to ask for, since the distance between endpoints of these orbit segments may be unbounded. To state a realistic condition for non-compact spaces, we allow the transition time to depend on where the orbits start and end. To make this precise, given a set $A\subset X$, consider the collection\nome{$\Ar$}{all orbit segments that start and end in $A$} of orbit segments whose endpoints are \emph{constrained to lie in $A$}:
\begin{equation}\label{eqn:Ar}
\Ar := \{(x,t) \in A \times [0,\infty) : f_t x \in A \}.
\end{equation}
An element of $\Ar^k$ corresponds to a finite sequence of orbit segments, each of which starts and ends in $A$.

\begin{definition} \label{def:spec}\terms{specification}
Given a set $A\subset X$, we say that the flow has \emph{specification} with reference to $A$ (w.r.t.\ $A$) at scale $\rho>0$ with transition time $\tau\geq 0$ if given any $k\in \NN$, and any 
\[
%(\bx,\bt) = 
((x_1,t_1),\dots,(x_k,t_k)) \in \Ar^k
\quad\text{and}\quad
\btau = (\tau_1,\dots,\tau_{k-1}) \in [\tau,\infty)^{k-1},
\]
%$((x_1,t_1),\dots,(x_k,t_k)) \in \Ar^k$,\nome{$(\bx,\bt)$}{sequence of orbit segments for specification, $(x_1,t_1),\dots, (x_k,t_k)$} 
%and any $\btau = (\tau_1,\dots,\tau_{k-1}) \in [\tau,\infty)^{k-1}$,\nome{$\btau$}{sequence of gap sizes, $\tau_1,\dots,\tau_{k-1}$}
there is a \emph{shadowing point} $y \in X$ such that
\eqref{eqn:shadows} holds for all $j\in \{1,\dots, k\}$.
%\[
%f_{T_j}(y) \in B_{t_j}(x_j,\rho) \text{ for all } 1\leq j\leq k,
%\text{ where } 
%T_j = t_1 + \tau_1 + t_2 + \tau_2 + \cdots + \tau_{j-1}.
%\]
%,  writing $T_j = t_1 + \tau_1 + t_2 + \tau_2 + \cdots + \tau_{j-1}$,\nome{$T_j$}{start time of $j$th orbit segment, $t_1 + \tau_1 + \cdots + \tau_{j-1}$} 
%we have  $f_{T_j}(y) \in B_{t_j}(x_j,\rho)$ for all $1\leq j\leq k$. 
We say the system $(X, \FFF)$ has specification if for every $A\in \RRR$, the flow has specification w.r.t.\ $A$ at all positive scales.
\end{definition}

As above, for $\bx \in A^k$ and $\bt \in [0, \infty)^k$ such that $x_j \in A_{t_j}$ for all $j$, we identify $(\bx,\bt)$ with $((x_1,t_1),\dots,(x_k,t_k)) \in \Ar^k$. 
Given $A,\rho,\bx,\bt,\btau$ as above, it will often be convenient to write\nome{$\Spec_\rho^{\btau}(\bx,\bt)$}{set of $y$ that witness the shadowing in specification}
the set of shadowing points $y$ as
\begin{equation}\label{eqn:Spec}
\Spec_\rho^{\btau}(\bx,\bt) = \bigcap_{j=1}^k f_{-T_j}(B_{t_j}(x_j,\rho)).
\end{equation}
Given $\bx, \bt, \btau$ as in the definition of specification, it is useful to have notation for the total controlled lengths of orbit segments coming from points in  $\Spec_\rho^{\btau}(\bx,\bt)$. We define
\begin{equation}\label{eqn:Txtt}
T(\bt, \btau) = \sum_{j=1}^k t_j + \sum_{j=1}^{k-1} \tau_j.
\end{equation}

\begin{definition}\label{def:perspec}
We say that the flow has \emph{periodic specification} w.r.t.\ $A$ at scale $\rho>0$ with transition time $\tau\geq 0$ and lag time $\alpha>0$ if for every $(\bx,\bt) \in \Ar^k$,
every $\btau\in [\tau,\infty)^{k-1}$,
and every $r \geq T(\bt,\btau) + \tau$,
the set of shadowing points $\Spec_\rho^{\btau}(\bx,\bt)$ contains 
a fixed point of $f_t$ for some $t\in [r-\alpha,r+\alpha]$.
\end{definition}

Given any $t_1,\dots, t_k \geq 0$, and any family of sets $\bZ = (Z_1,\dots, Z_k)$\nome{$\bZ$}{family of sets $(Z_1,\dots,Z_k)$} 
with $Z_j \subset A_{t_j}$, we will write\nome{$\Spec_\rho^{\btau}(\bZ,\bt)$}{set of points that shadow each of the sets $Z_j$}
\begin{equation}\label{eqn:Spec-E}
\Spec_\rho^{\btau}(\bZ,\bt) = \bigcup_{\bx \in \bZ} \Spec_\rho^{\btau}(\bx,\bt)
\end{equation}
for the set of all $y$ whose orbit shadows some sequence of orbit segments corresponding to points in the sets $Z_j$. When all the transition times are the same, we use the shorthand
\begin{equation}\label{eqn:shorthand}
\Spec_\rho^r = \Spec_\rho^{(r,\dots,r)} \hspace{20pt} T(\bt, r) = \sum_{j=1}^k t_j + (k-1) r.
\end{equation}

\begin{remark}[Infinite sequence of orbit segments]
\label{rmk:inf-spec}
Given any $\{(x_i,t_i)\}_{i\in \NN} \in \Ar^\NN$, and any $\btau \in [\tau,\infty)^\NN$, one can define $\Spec_\rho^{\btau}(\bx,\bt)$ by replacing $k$ with $\infty$ in \eqref{eqn:Spec}. A similar notation can be used when $\NN$ is replaced by $\ZZ$ and we ask to shadow a bi-infinite sequence of orbit segments.  In both cases,
compactness of $A\in \RRR$ allows us to use the specification property in Definition \ref{def:spec} to deduce the a priori stronger fact that $\Spec_\rho^{\btau}(\bx,\bt) \neq \emptyset$ when either $\NN$ or $\ZZ$ is the indexing set.
For $\NN$, it suffices to take points $y_n$ that shadow $((x_1,t_1),\dots,(x_n,t_n))$, and then take any limit point $y$. One can then obtain the same result for $\ZZ$ by considering $\{(x_i, t_i)\}_{i\geq n}$ as $n\to-\infty$. This results in shadowing with error $\leq \rho$, instead of $<\rho$, but since $\rho$ can be arbitrary, it suffices to use the transition time $\tau$ from the finite-time property with $\rho/2$.
\end{remark}

\begin{remark}[Arbitrarily large transition time]
The transition time $\tau$ provided by the definition depends on $A$ and $\rho$, and that $\tau$ can be unbounded as $A$ and $\rho$ vary.  We write $\tau(A, \rho)$ for the time associated with this reference set and scale.  This is consistent with the behavior of mixing countable-state Markov shifts -- 
if $A$ is any finite collection of symbols, then by taking $\tau = \tau(A)$ sufficiently large that for every $a,b\in A$, the language contains a word beginning in $a$ and ending in $b$, we obtain the symbolic analogue of Definition \ref{def:spec} for all words beginning and ending with symbols from $A$. The constant $\tau(A)$ will usually be unbounded as $A$ increases. 
\end{remark}

\subsection{Expansivity}\label{s:exp} 
We defined expansivity for a flow $(X, \FFF)$ with respect to $\RRR$ and a family of metrics $\DDD$ in Condition \ref{cond:E}. See \cite[\S\S1.7--1.8]{FH19} for definitions of expansivity for flows on compact phase spaces with respect to a metric $d$. The main versions are kinematic expansivity, which our notion generalizes, and Bowen--Walters expansivity, which additionally asks for control on orbits after reparametrization. We discuss properties of our definition. First, we make the following definition.
\begin{definition} \label{def:generalexpansive}
For a flow $(X, \FFF)$ and a family of metrics $\DDD$, $x \in X$ and $\theta>0$, the \emph{bi-infinite Bowen ball}\nome{$\Gamma_\varepsilon(x)$}{bi-infinite Bowen ball for expansivity} is the set
\[
\Gamma_\theta (x; \DDD) =\bigcap_{t>0} B_{2t}(f_{-t} x, \theta).
\]
We say that $x$ is an \emph {expansive point} if there exists $\alpha>0$ with $\Gamma_\theta (x; \DDD) \subset f_{[-\alpha, \alpha]}(x)$, and otherwise $x$ is a \emph{non-expansive point} (at scale $\theta$ with respect to $\DDD$).
\end{definition}
Suppose that the family of metrics $\DDD$ is a coherent family. Let $A \in \RRR$. From \eqref{eqn:coherent} and the coherence property, we have that if $x \in A \cap f_{-s}A \cap f_{-(s+t)}A$, then for any $\theta \in (0, \rmet(A)/2]$ we have
\[
B_{s+t}(x, \theta) = B_s(x, \theta) \cap B_t(f_s x, \theta).
\]

Assume without loss of generality that $\rexp(A) \leq \rmet(A)/2$, and let $\theta \in (0, \rexp(A)]$. For $x \in A$ and any $s_k, t_k \to \infty$ such that $f_{-s_k} x, f_{t_k}x \in A$, we have $\Gamma_\theta (x; \DDD) \subseteq \bigcap_{k\geq 1} B_{[-s_k, t_k]}(x, \theta)= \bigcap_{k\geq 1} B_{[-s_k, 0]}(x, \theta)\cap B_{t_k}(x, \theta)$.   Thus,  if  $x\in A$ and $s_k, t_k\to\infty$ are such that $f_{-s_k}(x) \in A$ and $f_{t_k}(x) \in A$ for all $k$, and $y\in X$ satisfies
\[
d_{s_k}(f_{-s_k}x,f_{-s_k}y) \leq \theta
\quad\text{and}\quad
d_{t_k}(x,y) \leq \theta
\quad\text{for all }k,
\]
and Condition \ref{cond:E} holds, then there  exists $\alpha=\alpha(A)$ and $t\in [-\alpha,\alpha]$ such that $y = f_t(x)$. In other words,  $\bigcap_{k\geq 1} B_{[-s_k, t_k]}(x, \theta) \subset f_{[-\alpha,\alpha]}(x)$. We conclude that every such $x$ is an expansive point in the sense of Definition \ref{def:generalexpansive} with the lag time $\alpha$ depending only on $A$.

Our definition of expansivity implies that for all $ A \in \RRR$ and $\delta>0$, there exists $T=T(x, \delta)$ so that if $s_k, t_k \geq T$, and $x, f_{-s_k}x, f_{t_k} x \in A$ then
\begin{equation}\label{eqn:BT-close}
B_{[-s_k,t_k]}(x, \theta) \subset B(f_{[-\alpha,\alpha]}x, \delta).
\end{equation}
That is, every point in the two-sided Bowen ball is in a $\delta$-neighborhood of some $f_s\gamma$ with $|\gamma|\leq \alpha$. When $(f_s)$ is a geodesic flow, since it has unit speed, we can let $\alpha=\theta$.

We say that $(X, \FFF)$ is \emph{almost expansive} for a measure $\mu$ at scale $\theta$ with respect to the family $\DDD$ if for $\mu$-almost every $x \in X$, there exists $\alpha(x)>0$ so that
\[
\Gamma_\theta(x;\DDD) \subset f_{[-\alpha(x),\alpha(x)]}(x).
\]

\begin{lemma} \label{lem:almostexpansive}
Let $A \in \RRR$. Let $\theta >0$ be a scale so that $\theta \in (0, \rexp(A)]$. Let $\mu$ be a measure so that $\mu(A)>0$. The flow $(X, \FFF)$  is almost expansive for $\mu$ at scale $\theta$ with respect to the family $\DDD$.
\end{lemma}
\begin{proof}
By Poincar\'e recurrence, there exists a set $G \subset X$ such that  $\mu(G)=1$ and so that for every $x \in G$, there exists $t_0$ so that $f_{t_0} x \in A$. By the expansivity property we have, $\Gamma_\theta(f_{t_0}x;\DDD) \subset f_{[-\alpha,\alpha]}(f_{t_0}x)$.  It follows that $\Gamma_\theta(x;\DDD) \subset f_{-t_0}f_{[-\alpha,\alpha]}(f_{t_0}x)$, and we choose $\alpha(x)$ so that $f_{-t_0}f_{[-\alpha,\alpha]}(f_{t_0}x) \subset  f_{[-\alpha(x),\alpha(x)]}(x)$.
\end{proof}

%Let $A^{(\pm \infty)}$ be the set of $x \in A$ which returns to $A$ infinitely often in both forward and backwards time.   For a compact set $K \subset A^{(\pm \infty)}$, using Dini's theorem, we can ensure that for all $\delta>0$, there exists $T=T(K, \delta)$ so that \eqref{eqn:BT-close} holds for all $x \in K$.
\subsection{Marked specification} 
In the definition of specification, it is unavoidable that the ``transition'' orbit segments $(f_{T_j + t_j}y, \tau_j)$ are allowed to depend on all of the target orbit segments $(x_i,t_i)$. Definition \ref{def:spec} allows this dependence to be arbitrarily sensitive: a small change in $(x_k,t_k)$ could result in a large change in $f_{t_1 + t}y$ for some $t\in (0, \tau_1)$. In applications, one can often do better and guarantee that the transition segments are determined by the two adjacent orbit segments. We formalize this using the following definition: the idea is that the adjacent orbit segments are used to ``mark'' a transition route that the shadowing orbit must approximate.

\begin{definition}\label{def:markspec}
Given $A\subset X$, we say that the flow has \emph{marked specification} w.r.t.\ $A$ at scale $\rho>0$ with transition time $\tau\geq 0$ if
there exists a 
family of maps $\glu_s \colon \Ar^2 \to X$ (for $s \geq \tau$)
with the following property: given any $k\in \NN$, any $(\bx,\bt) \in \Ar^k$, and any $\btau\in [\tau,\infty)^{k-1}$, there is a point $y$ such that 
\begin{itemize}
\item writing $T_j = t_1 + \tau_1 + t_2 + \cdots + \tau_{j-1}$ as in Definition \ref{def:spec}, and 
\item writing $z_j := \glu_{\tau_j}((x_j,t_j),(x_{j+1},t_{j+1}))$ for $1\leq j < k$,
\end{itemize}
the orbit of $y$ shadows the orbit segments $(\bx,\bt)$ and the transition segments $(\bz,\btau)$:
\begin{equation}\label{eqn:shadows-both}
f_{T_j}(y) \in B_{t_j}(x_j,\rho)
\quad\text{and}\quad
f_{T_j + t_j}(y) \in B_{\tau_j}(z_j,\rho)
\quad\text{for all }j.
\end{equation}
\end{definition}

Together with the expansivity property, marked specification implies periodic specification.

\begin{lemma}\label{lem:mark-per}
If the flow has marked specification w.r.t.\ $A$ at scale $\rho>0$ with transition time $\tau$, 
and is expansive at scale $2\rho$ in the sense of Condition \ref{cond:E},
then it has periodic specification w.r.t.\ $A$ at scale $\rho$ with transition time $\tau$.
\end{lemma}
\begin{proof}
Fix $k\in \NN$, $(\bx,\bt) \in \Ar^k$, $\btau \in [\tau,\infty)^{k-1}$, and $r\geq T(\bt,\btau) + \tau$.
Let $\tau_k = r - T(\bt,\btau)$ and write $\tilde{\btau} = (\tau_1,\dots, \tau_k)$.
As in Remark \ref{rmk:inf-spec}, extend $\bx,\bt,\tilde{\btau}$ periodically, and let $y$ be a shadowing point for this bi-infinite sequence of orbit segments, in the sense of \eqref{eqn:shadows-both}. 
Since each of $\bx,\bt,\tilde{\btau}$ is unchanged when we shift them by $k$ indices, we see that $f_r(y)$ is also a shadowing point for $\bx,\bt,\tilde{\btau}$ in the sense of \eqref{eqn:shadows-both}. 
It follows that $f_r(y) \in \Gamma_{2\rho}(y; \DDD)$. 
By Condition \ref{cond:E}, we conclude that $f_r(y) = f_s(y)$ for some $s\in [-\alpha,\alpha]$, which completes the proof.
\end{proof}

\subsection{Regularity conditions}\label{s:bowen}

Let $\ph\colon X\to \RR$ be a continuous potential function,
and recall that we allow $\ph$ to be unbounded.
We write 
\[
\sup \ph = \sup_{x \in X} \ph(x)
\quad\text{and}\quad
\| \ph \| = \sup_{x \in X} |\ph(x)|.
\]
For $E \subset X$, we write $\|\varphi\|_{E}=\sup_{x\in E}\{|\varphi(x)|\}$. If $E$ is compact, then $\|\varphi\|_{E} < \infty$.
For $x \in X$, $t\geq 0$, recall from \eqref{eqn:Phixt} that we write $\Phi(x,t) = \int_0^t\ph(f_sx)\,ds$.
Similarly, we use the convention that for a potential $g$, or $h$, then $G(x,t)$, $H(x, t)$ is the corresponding integral along the orbit segment $(x, t)$. For any $T,\zeta>0$, let\nome{$\Var(\ph,T,\eps,x)$}{variation of ergodic integral along $\eps$-Bowen ball}
\begin{equation}\label{eqn:Var-x}
\Var(\ph, T, \zeta, x) = \sup\{ |\Phi(x, T)- \Phi(y,T)| : y \in B_T(x, \zeta)\},
\end{equation}
and given $E\subset X$, we write\nome{$\Var(\ph,T,\eps,E)$}{sup of variations over $x\in E$}
\begin{equation}\label{eqn:Var-E}
\Var(\ph,T,\zeta,E) = \sup_{x\in E} \Var(\ph,T,\zeta,x).
\end{equation}

\begin{definition}\label{def:Bowen}\terms{Bowen property}\nome{$Q$}{constant in Bowen property, $Q(A,A')$}
A potential function $\ph\colon X\to \RR$ has the \emph{Bowen property on $\RRR$ at scale $\zeta>0$} if for all $A \in \RRR$, there exists $Q = Q(A) > 0$ such that
\begin{equation}\label{eqn:Bowen}
\text{for every $T>0$, we have $\Var(\ph,T,\zeta,A_T) \leq Q$.}
\end{equation}
Equivalently, for all $T>0$ and $x\in A$ with $f_T(x) \in A$, we have $\Var(\ph, T, \zeta, x) \leq Q$. We say that $\ph$ has the Bowen property on $\RRR$ if there exists $\zeta>0$ so that it has the Bowen property on $\RRR$ at scale $\zeta>0$.
\end{definition}

%If $\ph\colon X\to \RR$ has the Bowen property on $\RRR$ at scale $\zeta>0$, it is immediate that it has the Bowen property at smaller scales. We observe that if $(X,\FFF, \ph)$  is expansive and has the Bowen property on $\RRR$, we can can always find a scale $\eps>0$ which satisfies the properties stated in Condition \ref{cond:EB}.  
The Bowen property plays a crucial role in our results. We emphasize that the Bowen property depends on the family of metrics $\DDD$ since the Bowen balls are defined in terms of $\DDD$. For Theorem \ref{thm:welldefined}, we work with a weaker regularity property.

\begin{definition}\label{def:slow-dist}\terms{distortion rate}
Given a continuous function $\ph \colon X\to \RR$ and a reference set $A\in \RRR$, the \emph{distortion rate} of $\ph$ with respect to $A$ at scale $\zeta>0$ 
is\nome{$\delta^A(\zeta)$}{distortion rate w.r.t.\ $A$ at scale $\eps$}
\begin{equation}\label{eqn:delta-A}
\delta^A(\zeta) := \limsup_{T\to\infty} \frac 1T \Var(\ph,T,\zeta,A_T).
\end{equation}
We say that $\ph$ has \emph{tempered distortion} w.r.t.\ $A$ if $\delta^A(\zeta)\to 0$ as $\zeta\to 0$.\terms{tempered distortion}
\end{definition}

Figure \ref{fig:regularity} shows the relationship between the various regularity properties.
If $\ph$ has the Bowen property at scale $\zeta$, then $\delta^A(\zeta) = 0$, so $\ph$ has tempered distortion.\footnote{%
The Bowen property can be interpreted as \emph{bounded distortion} for the ergodic integrals $\Phi(x,t)$, and Definition \ref{def:slow-dist} weakens ``bounded'' to ``tempered''. See \cite{kG19,QX24} for related ``tempered distortion'' properties.}  
Uniformly continuous functions also have tempered distortion since $\delta^A$ is bounded above by the modulus of continuity. If $K \in \mathcal K$, i.e., $K$ is compact and invariant, then any continuous function has tempered distortion on $K$. The nearby orbits in the definition of $\Var(\ph,T,\zeta,K)$ need not remain in $K$ themselves, but they stay in the compact set $K^\zeta$, on which we have uniform continuity. A potential can have the Bowen property on each $A\in \RRR$ without being uniformly continuous.

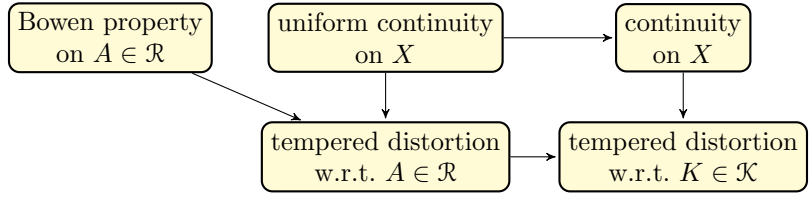
\begin{figure}[htbp]
\begin{tikzpicture}[
every matrix/.style={ampersand replacement=\&,column sep=0.25in, row sep=0.25in},
property/.style={draw,thick,rounded corners,fill=yellow!20},
every node/.style={align=center},
to/.style={->,>=stealth',shorten >=1pt,font=\sffamily\footnotesize},
and/.style={circle,fill=black,inner sep=1pt}]

\matrix{
\node[property] (Bow) {Bowen property \\ on $A\in \RRR$}; \&
\node[property] (unif-cts) {uniform continuity \\ on $X$}; \&
\node[property] (cts) {continuity \\ on $X$}; \\
\&
\node[property] (slow) {tempered distortion \\ w.r.t.\ $A\in \RRR$}; \&
\node[property] (slow-K) {tempered distortion \\ w.r.t.\ $K\in \mathcal{K}$}; \\
};

\draw[to] (Bow) -- (slow);
\draw[to] (unif-cts) -- (cts);
\draw[to] (unif-cts) -- (slow);
\draw[to] (slow) -- (slow-K);
\draw[to] (cts) -- (slow-K);
\end{tikzpicture}
\caption{Regularity properties for $\ph$.}
\label{fig:regularity}
\end{figure}

Given a finite sequence of orbit segments $(\bx,\bt) = ((x_1,t_1),\dots, (x_k,t_k)) \in (X\times [0,\infty))^k$, we can write \nome{$\Phi(\bx,\bt)$}{sum of $\Phi(x_j,t_j)$}
\begin{equation}\label{eqn:Phi-k}
\Phi(\bx,\bt) := \sum_{j=1}^k \Phi(x_j,t_j) = \sum_{j=1}^k \int_0^{t_j} \ph(f_s x_j) \,ds.
\end{equation}
Suppose that  the flow has specification on $A$ at scale $\rho$, with transition time $\tau$.
By Lemma \ref{lem:L}, there is $\Vt = \Vt(A,\tau)>0$ such that for every $x\in A$, we have $|\Phi(x,\tau)| \leq \Vt$.

\begin{lemma}\label{lem:distort}
Recalling the notation in Definition \ref{def:spec} and \eqref{eqn:Spec}--\eqref{eqn:shorthand}, given any $(\bx,\bt) \in \Ar^k$ and any $y\in \Spec_\rho^\tau(\bx,\bt)$, 
and writing $T=T(\bt, \tau)$,
we have
\begin{equation}\label{eqn:distort}
|\Phi(y,T) - \Phi(\bx,\bt)| \leq (k-1)\Vt + \sum_{j=1}^k\Var(\ph,t_j,\rho,A_{t_j}).
\end{equation}
In particular, if $\ph$ has the Bowen property, then
\begin{equation}\label{eqn:distort-VQ}
|\Phi(y,T) - \Phi(\bx,\bt)| \leq k(\Vt+Q).
\end{equation}
\end{lemma}
\begin{proof} Observe that
\[
\Phi(y,T) = \Phi(y,t_1) + \Phi(f_{t_1} y, \tau) + \Phi(f_{t_1 + \tau} y, t_2) + \cdots + \Phi(f_{t_1 + \tau + \cdots + t_{k-1} + \tau} y, t_k).
\]
The even-numbered terms each have absolute value at most $\Vt$ by Lemma \ref{lem:L}. The odd-numbered terms are each within $\Var(\ph,t_j,\rho,A_{t_j})$ of $\Phi(x_j,t_j)$.
\end{proof}

\subsection{Partition sums}\label{sec:part-sums}

Given a finite set $E \subset X$ and a time $T>0$, the \emph{partition sum} associated to the pair $(E,T)$ is
\begin{equation}\label{eqn:LET}
\Lambda(\ph,T,E) := \sum_{x\in E} e^{\Phi(x,T)}.
\end{equation}
To define partition sums for an infinite set $Z\subset X$, we fix a scale $\zeta>0$ and recall from \S\ref{sec:GS-VP} that a set $E \subset Z$ is \emph{$(T,\zeta)$-separated} if any distinct $x,y\in E$ have the property that $d_T(x,y) > \zeta$.
%if for every $x,y\in E$ with $x\neq y$, there is $t\in [0,T]$ such that $d(f_tx,f_ty)> \zeta$.\terms{$(T,\varepsilon)$-separated}
The \emph{partition sum} associated to an infinite set $Z\subset X$ at time $T$ and scale $\zeta$ is defined to be\terms{partition sum}\nome{$\Lambda(\varphi,T,\zeta,Z)$}{partition sum}%
\begin{equation}\label{eqn:partition-sums}
\Lambda(\ph,T,\zeta,Z) = \sup \big\{ \Lambda(\ph,T,E) : E\subset Z \text{ is $(T,\zeta)$-separated} \big\}.
\end{equation}

\begin{definition}\label{def:maximal}
A $(T,\zeta)$-separated set $E\subset Z$ is \emph{maximal (with respect to inclusion)} if there does not exist $z\in Z \setminus E$ such that $E \cup \{z\}$ is $(T,\zeta)$-separated.
\end{definition}

Observe that the supremum in \eqref{eqn:partition-sums} 
does not change if we restrict it to maximal $(T,\zeta)$-separated sets $E\subset Z$. Moreover, any such set $E$ is \emph{$(T,\zeta)$-spanning for $Z$} in the sense that
\begin{equation}\label{eqn:spanning}
Z \subset \overline B_T(E, \zeta).  
\end{equation}

The Bowen property has the following consequence.

\begin{lemma}\label{lem:all-sep}
Suppose that $\ph$ has the Bowen property on $A$ at scale $\zeta > 0$ with constant $Q = Q(A,\zeta)$. 
Then for every $T>0$, every $Z \subset A_T$, and every 
$(T,\zeta)$-separated set $E \subset A_T$ that is $(T,\zeta)$-spanning for $Z$, we have
\begin{equation}\label{eqn:sum-geq}
%\sum_{x\in E} e^{\Phi(x,T)} 
\Lambda(\ph,T,E)
\geq e^{-Q} \Lambda(\ph,T,2\zeta,Z).
\end{equation}
In particular, \eqref{eqn:sum-geq} holds for every maximal $(T,\zeta)$-separated set $E_T \subset Z$.
\end{lemma}
\begin{proof}
Since $E$ is %maximal $(T,\zeta)$-separated in $Z$, 
$(T,\zeta)$-spanning for $Z$, we have $\bigcup_{x\in E} \overline{B}_T(x,\zeta) \supset Z$. Given any $(T,2\zeta)$-separated set $E' \subset Z$, for every $y\in E'$ there exists $x = x(y) \in E$ such that $y\in \overline{B}_T(x,\zeta)$. We claim that the map $y\mapsto x$ is injective; indeed, given any $z\in \overline{B}_T(x,\zeta)$, we have 
\[
%d(f_t y, f_t z) \leq d(f_t y, f_t x) + d(f_t x, f_t z) \leq \zeta + \zeta = 2\zeta
%\quad\text{for all } t\in [0,T],
d_T(y,z) \leq d_T(y,x) + d_T(x,z) \leq \zeta + \zeta = 2\zeta,
\]
and since $E'$ is $(T,2\zeta)$-separated, this implies that $z=y$ or $z\notin E'$. We conclude that
\[
\Lambda(\ph,T,E) = \sum_{x\in E} e^{\Phi(x,T)}
\geq \sum_{y\in E'} e^{\Phi(x(y),T)}
\geq \sum_{y\in E'} e^{-Q} e^{\Phi(y,T)},
\]
there the first inequality uses injectivity and the second uses the Bowen property. Taking a supremum over all $(T,2\zeta)$-separated sets $E' \subset Z$ proves \eqref{eqn:sum-geq}.
\end{proof}

To prove the uniform counting bounds stated in Theorem \ref{thm:uniform}, as well as analogues for other partition sums, we will need the following estimates showing that the partition sums are ``approximately'' supermultiplicative and submultiplicative. Recall that a sequence $(\Lambda_n)_{n\in\NN}$ is \emph{supermultiplicative} if $\Lambda_{n+k} \geq \Lambda_n \Lambda_k$ for all $n,k$, and \emph{submultiplicative} if $\Lambda_{n+k} \leq \Lambda_n \Lambda_k$ for all $n,k$. We will obtain inequalities for partition sums that are in a similar spirit, but involve constants related to the Bowen and specification properties, as well as changes in scale. We state some of these inequalities now, and then return to a more complete discussion of this issue in \S\ref{sec:counting}.

\begin{lemma}\label{lem:spec-sum-Bow}
Suppose the flow has specification at scale $\rho$ w.r.t.\ $A\in \RRR$ with transition time $\tau$, and that $\ph$ has the Bowen property on $A$ at scale $\rho$ with constant $Q=Q(A,\rho)$.
Fix $k\in \NN$, $\bt = (t_1,\dots, t_k) \in [0,\infty)^k$, and $\bZ = (Z_1,\dots, Z_k)$, where $Z_j \subset A_{t_j}$ for all $1\leq j\leq k$.
Let $\Vt=\Vt(A,\tau)$ be given by Lemma \ref{lem:L}, and let $T(\bt, \tau)$ be given by \eqref{eqn:shorthand}.
Then for every $\zeta>0$, we have
\begin{equation}\label{eqn:spec-sum-Bow}
\Lambda(\ph,T,\zeta,\Spec_\rho^{\tau}(\bZ,\bt))
\geq e^{-k(\Vt+Q)} \prod_{j=1}^k \Lambda(\ph,t_j,\zeta+2\rho,Z_j).
\end{equation}
If in addition $A^{-\rho}$ is nonempty, then we have
\begin{equation}\label{eqn:supermult}
\Lambda(\ph,T+2\tau,\zeta,A_{T+2\tau})
\geq e^{-(k+2)(\Vt+Q)} \prod_{i=1}^k \Lambda(\ph,t_i,\zeta+2\rho, A_{t_i}).
\end{equation}
\end{lemma}

Lemma \ref{lem:spec-sum-Bow} is a consequence of the following result, which has a priori weaker hypotheses than the specification or Bowen properties, and implies \eqref{eqn:spec-sum-Bow}. Note that \eqref{eqn:supermult} then follows immediately by fixing $x_0 \in A^{-\rho}$ and applying \eqref{eqn:spec-sum-Bow} to $\mathbf{Z} = (\{x_0\},A_{t_1},\dots,A_{t_k},\{x_0\})$ and $\bt' = (0,t_1,\ldots,t_k,0)$.

\begin{lemma}\label{lem:spec-sum}
Fix $A\in \RRR$ and let $k,\bt,\bZ,T,\Vt$ be as in Lemma \ref{lem:spec-sum-Bow}.
Suppose that $Z_0 \subset X$, $\rho>0$, and $\tau\geq 0$ are such that 
\begin{equation}\label{eqn:Z0-spec}
\text{for every $\bx \in \bZ$, we have $Z_0 \cap \Spec_\rho^\tau(\bx,\bt) \neq \emptyset$}. 
\end{equation}
Then for every $\zeta>0$, we have\nome{$\zeta$}{auxiliary scale in partition sum arguments}
\begin{equation}\label{eqn:spec-sum}
\Lambda(\ph,T,\zeta,Z_0)
\geq e^{-k\Vt} e^{-\sum_{j=1}^k \Var(\ph,t_j,\rho,A_{t_j})} \prod_{j=1}^k \Lambda(\ph,t_j,\zeta+2\rho,Z_j).
\end{equation}
\end{lemma}
\begin{proof}
Fix $(t_j,\zeta+2\rho)$-separated sets $E_j \subset Z_j$. Given $x_j \in E_j$ and $\bx = (x_1,\dots, x_k)$, choose $y=y(\bx) \in \Spec_\rho^\tau(\bx,\bt) \cap Z_0$. By %\eqref{eqn:distort}, 
Lemma \ref{lem:distort},
we have
\begin{equation}\label{eqn:Phi-yT}
\Phi(y,T) \geq \Phi(\bx,\bt) - k\Vt - \sum_{j=1}^k \Var(\ph,t_j,\rho,A_{t_j}).
\end{equation}
Write $\prod\bE := \prod_{j=1}^k E_j$.\nome{$\prod\bE$}{product of sets $E_j$}
We claim that the set $Y = \{y(\bx) : \bx \in \prod\bE\} \subset Z_0$ is $(T,\zeta)$-separated, and the map $\bx \mapsto y(\bx)$ is injective on $\prod\bE$.
To this end, it suffices to show that given any $\bx,\bx'\in \prod\bE$ with $\bx\neq \bx'$, we have $d_T(y(\bx),y(\bx')) > \zeta$.

Fixing such a $\bx,\bx'$, there exists $j$ such that $x_j \neq x_j'$, and since $E_j$ is $(t_j,\zeta+2\rho)$-separated, 
we have $d_{t_j}(x_j, x_j') > \zeta + 2\rho$.
Writing $T_j := t_1 + \tau + t_2 + \tau + \cdots + \tau + t_{j-1}$, we can iterate \eqref{eqn:coherent} to get
\[
d_{t_j}\big(f_{T_j}((y(\bx)), f_{T_j}(y(\bx'))\big)
\leq d_T\big(y(\bx),y(\bx')\big),
\]
so by the triangle inequality,
\begin{align*}
\zeta + 2\rho &< d_{t_j}(x_j,x_j') \\
&\leq d_{t_j}\big(x_j, f_{T_j}(y(\bx))\big)
+ d_{t_j}\big(f_{T_j}((y(\bx)), f_{T_j}(y(\bx'))\big)
+ d_{t_j}\big(f_{T_j}(y(\bx')), x_j'\big) \\
&\leq \rho + d_T\big(y(\bx),y(\bx')\big) + \rho,
\end{align*}
where the last inequality also uses the fact that $d_{t_j}\big(x_j, f_{T_j}(y(\bx))\big) < \rho$ for all $\bx$ and $j$. This implies that $d_T(y(\bx),y(\bx')) > \zeta$, so $Y$ is $(T,\zeta)$-separated and $y\colon \prod\bE\to Y$ is a bijection.
Together with \eqref{eqn:Phi-yT}, we conclude that
\[
\Lambda(\ph,T,\zeta,Z_0) \geq \sum_{\bx \in \prod\bE} e^{\Phi(y(\bx),T)} 
\geq e^{-k\Vt - \sum_{j=1}^k \Var(\ph,t_j,\rho,A_{t_j})} \sum_{\bx \in \prod\bE}
e^{\sum_{j=1}^k \Phi(x_j,t_j)}.
\]
The last sum is equal to $\prod_{j=1}^k \sum_{x_j \in E_j} e^{\Phi(x_j,t_j)}$, and taking a supremum over all $(t_j,\zeta+2\rho)$-separated sets $E_j \subset Z_j$ establishes \eqref{eqn:spec-sum}.
\end{proof}

\begin{lemma}\label{lem:submult}
Suppose that $\ph$ has the Bowen property on $A\in \RRR$ at scale $\eps>0$ with constant $Q=Q(A,\eps)$.
Fix $k\in \NN$ and $\bt = (t_1,\dots, t_k) \in [0,\infty)^k$.
For each $j \in \{0,1,\dots, k\}$, let $T_j := \sum_{i=1}^j t_i$.
Suppose that $Z\subset A$ has $f_{T_j}(Z) \subset A$ for all $0\leq j\leq k$.
Then for
every $\zeta \in [2\eps, \rmet(A)]$, we have
\begin{equation}\label{eqn:submult}
\Lambda(\ph,T_k,\zeta,Z)
\leq e^{kQ} \prod_{j=1}^k \Lambda(\ph,t_j,\eps,f_{T_{j-1}}(Z)).
\end{equation}
\end{lemma}
\begin{proof}
For each $0\leq j < k$, fix a $(t_{j+1},\eps)$-separated set $E_j \subset f_{T_j}(Z)$ that is maximal with respect to inclusion, so that $f_{T_j}(Z) \subset \overline B_{t_{j+1}}(E_j,\eps)$.
Then for any $(T_k,\zeta)$-separated set $E \subset Z$, there exists a map
$\pi \colon E \to \prod_{j=0}^{k-1} E_j$ such that
\begin{equation}\label{eqn:fTj}
f_{T_j}(x) \in B_{t_{j+1}}(\pi(x)_j, \eps) \quad\text{for all } j.
\end{equation}
Using the Bowen property, this implies that
\begin{equation}\label{eqn:Bow-split}
\Big|\Phi(x,T_k) - \sum_{j=0}^{k-1} \Phi(\pi(x)_j, t_{j+1}) \Big| \leq kQ.
\end{equation}
We claim that $\pi$ is injective. Suppose $x,y\in E$ have $\pi(x) = \pi(y)$. Then for every $j \in \{0,\dots, k-1\}$, \eqref{eqn:fTj} gives $d_{t_{j+1}}(f_{T_j} x, f_{T_j} y) \leq 2\eps \leq \zeta$, and 
using the fact that the family $\DDD$ is coherent and $\zeta \leq \rmet(A)$,
we conclude that 
\[
d_{T_k}(x,y) = \max_{j\in \{0,\dots, k-1\}} d_{t_{j+1}} ( f_{T_j} x, f_{T_j} y) \leq 2\eps \leq \zeta.
\]
Since $E$ is $(T_k,\zeta)$-separated, this gives $x=y$.
 Using \eqref{eqn:Bow-split} and injectivity, we have
\[
\sum_{x\in E} e^{\Phi(x,T_k)}
\leq \sum_{x\in E} e^{kQ + \sum_{j=0}^{k-1} \Phi(\pi(x)_j, t_{j+1})}
\leq e^{kQ} \prod_{j=0}^{k-1} \sum_{y\in E_j} e^{\Phi(y,t_{j+1})}.
\]
Taking a supremum over all such $E$ gives \eqref{eqn:submult}.
\end{proof}

\begin{remark} \label{rmk:changing-scales}
For a continuous flow on a compact space, expansivity for flows does not allow us to uniformly change scales in partition sums.  Such a ``scale-changing'' property would be true if one could show that the number of balls $B_t( \cdot, \delta )$ required to cover a ball $B_t(x, \zeta)$  is uniformly bounded. This is claimed in Franco \cite[Lemma 2.1]{eF77}, and restated in \cite{FH19} as Proposition 4.2.18, but the proof contains an error at equation (4.3.9). Without additional hypotheses on the time-parametrization of the orbit segments, we do not expect the conclusion of these lemmas to be true for flows in the compact setting, nor do we expect an analogous result in our setting.
\end{remark}

\subsection{Discrete-time systems} \label{Discrete-timesetup}
Some arguments are carried out in the discrete-time setting by fixing $\tau>0$ and considering the time-$\tau$ map $F = f_{\tau}$\nome{$F$}{time-$t_0$ map $f_{t_0}$}. We can study the potential function $\psi= \ph_{\tau}(x) = \Phi(x,\tau)$. For discrete-time arguments, we can equip $X$ with the family of metrics $(D_n)_{n \geq1}$ given by\nome{$(D_n)$}{dynamical metric $d_{\tau}$}
\begin{equation}\label{eqn:dyn-met}
D_n(x,y) = d_{n \tau}(x,y). %=  \sup_{t\in [0,\tau]} d(f_t x, f_t y),
\end{equation}
It follows from Condition \ref{cond:basic} that 
$(X,D_1) = (X, d_\tau)$  is a %separable metric space in which every closed ball is compact, 
locally compact separable complete metric space,
$F\colon X\to X$ is continuous, $\psi\colon X\to \RR$ is continuous,\nome{$\psi$}{potential for discrete-time system} and the family of metrics $(D_n)$ satisfy the following analogue of Definition \ref{def:coherent-metrics}. 

\begin{definition}\label{def:coherent-metrics-discrete}
Let $\DDD=(D_n)_{n\in \NN}$ be a family of metrics compatible with the topology on $X$, and write $D:=d_1$.  We say that $(D_n)_{n\in \NN}$ is \emph{coherent} with respect to the map $F$ if: 
\begin{itemize}
\item $(X, D)$ is a proper metric space.
\item for every $n,m\geq 0$ and $x,y\in X$, we have
\begin{equation}\label{eqn:coherentD}
D_{n+m}(x,y) \geq \max \big( D_n(x,y), D_m (f_s x, f_s y) \big);
\end{equation}
\item
for every compact set $A \subset X$, there exists $\rmet = \rmet(A)>0$ such that given any $n,m\geq 0$, any $x\in A \cap f_{-n}(A) \cap f_{-(n+m)}(A)$, and any $y\in X$ such that $D_n(x,y) \leq \rmet$ and $D_m(f_n x, f_n y) \leq \rmet$, we have equality in \eqref{eqn:coherent}.
\end{itemize}
\end{definition}

We let\nome{$B_n^F(x,\rho)$}{discrete-time Bowen ball}
\[
B_n^F(x,\rho) := \{ y\in X : D_{n}(x, y) < \rho  \},
\]
and for $I = [t_1, t_2]\subset \RR$, let $n(I) = (\lceil t_2 \rceil -1) - \lceil t_1 \rceil$, and let  
\[
B_I^F(x,\rho)  :=  \{ y\in X : D_{n(I)}(F^{\lceil t_1 \rceil} x, F^{\lceil t_1 \rceil}y) < \rho \}.
\] 
When $F=f_\tau$ and $D_n=d_{n\tau}$, we have
\[
B_n^{f_\tau}(x,\rho) = \{ y\in X : D_{n}(x, y) < \rho  \} = \{ y\in X : d_{n\tau}(x, y) < \rho  \} = B_{n\tau}(x, \rho).
\]
For a finite set $E \subset X$, we write $B_n^F(E,\rho) := \bigcup_{x\in E} B_n^F(x,\rho)$. We write 
\[
\Psi (x,n)= \Psi^F(x, n) = \sum_{i=0}^{n-1} \psi(F^i x)
\]
for the Birkhoff sum, and we define\nome{$\Psi(x,n)$}{Birkhoff sum of $\psi$}\nome{$\Var(g,\psi,n,\zeta,x)$}{discrete-time variation}
\begin{equation}\label{eqn:Var-xF}
\Var(F, \psi, n, \zeta, x) = \sup\{ |\Psi(x, n)- \Psi(y,n)| : y \in B_n(x, \zeta)\},
\end{equation}
and given $E\subset X$, we write
\begin{equation}\label{eqn:Var-EF}
\Var(F, \psi,n,\zeta,E) = \sup_{x\in E} \Var(F, \psi,n,\zeta,x).
\end{equation}
We define the Bowen property and tempered distortion properties in discrete-time analogously to Definitions \ref{def:Bowen} and \ref{def:slow-dist}. In particular, writing $A^F_n= A \cap F^{-n}A$,
then the potential $\psi$ has \emph{tempered distortion w.r.t.\ $A$}
 if the quantity
\begin{equation}\label{eqn:delta-AF}
\delta^A_F(\zeta):=  \limsup_{n \to \infty} \frac 1n \Var(F, \psi,n,\zeta,A^F_n)
\end{equation}
satisfies $\delta_F^A(\zeta) \to 0$ as $\zeta\to 0$.

We define partition sums in the discrete-time setting. That is, given a 
map $F\colon X\to X$ and a coherent family of metrics $(D_n)_{n\in \NN}$, we say that a set $E \subset Z$ is \emph{$(n,\zeta)$-separated} for a time $n\in \NN$ and a scale $\zeta>0$
if for every $x,y\in E$ with $x\neq y$, 
we have $D_n(x,y) > \zeta$.
%there is $j\in \{0, \ldots, n-1\}$ such that $D(F^jx,F^jy)> \zeta$. 
We define
\begin{equation}\label{eqn:partition-sumsdiscrete}
\Lambda(\ph, F, n,\zeta,Z) = \sup \Big\{ \sum_{x\in E} e^{\Phi(x, n)} : E\subset Z \text{ is $(n,\zeta)$-separated} \Big\}.
\end{equation}

We have the following relationships between a coherent family and a family of Bowen metrics.

\begin{lemma} \label{lem:coherencediscreteD}
Let $(D_n)_{n\in \NN}$ be a coherent  family of metrics and write $D=d_1$. Let $\Dbow_n(x, y) := \max \{D(F^ix, F^iy) : i \in \{0, \ldots, k-1\}$ be the family of Bowen metrics defined from the metric $D$. 
\begin{enumerate}
\item For all $x, y \in X$, $n \in \NN$, we have $\Dbow_n(x, y) \leq D_n(x, y)$.
\item Let $K \subset X$ be compact and $F$-invariant. Let $r= \rmet(K)$. For all $x, y \in K$ with $\Dbow_n(x, y) \leq r$, then $\Dbow_n(x, y) \geq D_n(x, y)$. 
\item   If $(X, F)$ has specification with respect to $\RRR$ and the family $(D_n)$, then it has specification with respect to $\RRR$ and the family $(\Dbow_n)$.
\item If $\psi$ has tempered distortion with respect to $\RRR$ and the family $(\Dbow_n)$, then it has tempered distortion with respect to $\RRR$ and the family $(D_n)$.
\end{enumerate}
\begin{proof}
It follows from \eqref{eqn:coherentD} that $\Dbow_2(x, y) \leq D_2(x, y)$, and (1) follows by an inductive argument. For (2), we now assume that  $x, y \in A$ with $\Dbow_{n}(x, y) \leq r$. Since equality holds in \eqref{eqn:coherentD}, we see that $D_2(x, y) = \max (D(x, y), D(Fx, Fy))= \Dbow_2(x, y)$. By an inductive argument, we have $D_n(x,y)\leq r$. Property (3) and (4) are immediate from property (1).
%It follows from Property (1) that if $\psi$ has tempered distortion with respect to the family $(\Dbow_n)$ then $\psi$ has tempered distortion with respect to the family $(\Dbow_n)$. 
\end{proof}
\end{lemma}

\subsection{The space of probability measures}\label{sec:prob}

Let $\mathcal{M}$ denote the space of Borel probability measures on $X$.
By Condition \ref{cond:basic}, the metric space $(X,d)$ is assumed to be complete, separable, and locally compact. We now recall several important consequences of this assumption for $\mathcal{M}$.

\begin{proposition}\label{prop:regular}
For every Borel set $E \subset X$ and every $\mu \in \mathcal{M}$, we have
\begin{equation}\label{eqn:regular}
\mu(E) = \sup \{ \mu(K) : K\subset E \text{ is compact} \}
= \inf \{ \mu(U) : U\supset E  \text{ is open} \}.
\end{equation}
\end{proposition}
\begin{proof}
See \cite[\S7.2]{Fol}; the key point is that separability and local compactness imply that every open set in $X$ is $\sigma$-compact, so \cite[Theorem 7.8]{Fol} applies.
\end{proof}

\begin{definition}
A Borel probability measure $\mu \in \mathcal{M}$ is the \emph{weak* limit} of a sequence $(\mu_n)_{n\in \NN} \subset \mathcal{M}$ if $\int \psi \,d\mu_n \to \int \psi \,d\mu$ for all bounded continuous $\psi \colon X\to \RR$. 
\end{definition}

\begin{remark}
Weak*-convergence is sometimes referred to as \emph{weak convergence} or \emph{narrow convergence} \cite[\S8.1]{vB07}.
On the space of all finite Borel measures, it is a more restrictive condition than \emph{vague convergence}, which only requires convergence of the integrals for compactly supported test functions \cite[\S30]{hB01}. The two notions coincide on the space of Borel probability measures, in the sense that given a sequence $(\mu_n)_{n\in \NN}$ of Borel probability measures, and a candidate limiting measure $\mu$ that is known in advance to be a probability measure, the sequence $\mu_n$ converges to $\mu$ in the weak* sense if and only if it converges vaguely \cite[Corollary 30.9]{hB01}.
\end{remark}

%Dudley \cite{rD89}

\begin{definition}\label{def:tight}
A collection of Borel probability measures $\mathcal{M}' \subset \mathcal{M}$ is \emph{tight} if for every $\eta>0$, there exists a compact set $K\subset X$ such that $\nu(K^c) < \eta$ for all $\nu \in \mathcal{M}'$.
\end{definition}

Since we work in a complete separable metric space and limit our attention to probability measures, Prohorov's Theorem \cite[Theorem 8.6.2]{vB07} takes the following form.

\begin{theorem}\label{thm:Prohorov}
A collection of Borel probability measures $\mathcal{M}' \subset \mathcal{M}$ is tight if and only if every sequence in $\mathcal{M}'$ contains a subsequence that converges in the weak* topology.
\end{theorem}

%$X$ is complete and separable, so we are in the setting of \cite[\S8.3]{vB07}. The weak* topology is induced by the Kantorovich--Rubinshtein metric.

%\newpage

\section{The variational principle for Gurevich--Sarig pressure}\label{sec:GS}

\subsection{Overview}
This section is devoted to the proof of the following variational principle, which proves \eqref{eqn:welldefined} and \eqref{eqn:P-eps} of Theorem \ref{thm:welldefined} under weakened hypotheses. Recall that the Gurevich--Sarig pressure $P^A(\ph, \zeta), P^A(\ph), P(\ph)$ and entropy $h_{GS}$, were defined in \S\ref{sec:GS-VP}.

\begin{theorem}\label{thm:vp}
Let $X,\FFF,\DDD,\ph$ satisfy Condition \ref{cond:basic}. Suppose that $\ph$ has tempered distortion on $\RRR$ and that $\FFF$ has specification w.r.t.\ $\RRR$ at all scales. Then for every $A\in \RRR$, we have
\begin{equation}\label{eqn:vp}
P(\ph) = P^A(\ph)
 = \sup_{\mu \in \Mf^\ph} \Big( h_\mu(\FFF) + \int \ph\,d\mu \Big) 
= \sup_{K\in \mathcal{K}} P(K,\ph).
\end{equation}
If in addition $A\in \RRR$ is such that $\FFF$ is expansive w.r.t.\ $A$ and the family of metrics $\DDD$ at scale $\rexp \in (0,\rmet(A)]$, then for every $\eps \in (0,\frac12 \rexp)$, we have $P^A(\ph,\eps) = P(\ph)$.
\end{theorem}

Once Theorem \ref{thm:vp} is proved, in order to complete the proof of Theorem \ref{thm:welldefined}, we additionally need to prove \eqref{eqn:K-gap}, which requires the Bowen property (not just tempered distortion) and is proved in Theorem \ref{thm:K-gap} in \S\ref{s.uppercounting}. The following subsections contain more detailed results that imply Theorem \ref{thm:vp}.

\begin{remark}
In general, $P(\ph)$ could be $\pm\infty$. However, we have $P(\ph)>-\infty$ under very mild assumptions: all we need is a single point $x$ with bounded forward orbit, since if $\{f_t(x) \}_{t\geq 0} \subset B \in \RRR$, then
\[
P(\ph) = P^B(\ph) \geq \liminf_{t\to\infty} \frac 1t \int_0^t \ph(f_s x) \,ds
\geq \inf \ph|_B > -\infty.
\]
Such an $x$ and $B$ exist whenever the flow has specification w.r.t.\ some $A\in \RRR$. Moreover, if Condition \ref{cond:UESB} holds, so that $h_{GS} < \infty$ and $\sup \ph < \infty$, then we see immediately from \eqref{eqn:part-sum-1}--\eqref{eqn:GSP} that $P(\ph) \leq h_{GS} + \sup\ph < \infty$. We conclude that whenever Conditions \ref{cond:S} and \ref{cond:UESB} are satisfied, we have $P(\ph) \in (-\infty,\infty)$.
\end{remark}

\subsection{Katok pressure inequality}\label{sec:half-var}

In this section, we prove half of the variational principle:

\begin{theorem}\label{thm:half-var}
Let $X,\FFF,\DDD,\ph$ satisfy Condition \ref{cond:basic}, and suppose that $\ph$ has tempered distortion on $\RRR$. Then we have
\begin{equation}\label{eqn:half-var}
\sup_{K\in \mathcal{K}} P(K,\ph)
\leq \sup_{\mu \in \Mf^\ph} \Big( h_\mu(\FFF) + \int \ph \,d\mu \Big)
\leq P(\ph).
\end{equation}
\end{theorem}

The first inequality in \eqref{eqn:half-var} follows from the classical variational principle. In order to prove Theorem \ref{thm:half-var}, therefore, it suffices to prove that for every ergodic $\mu \in \Mf^\ph$, we have $h_\mu(\FFF) + \int\ph\,d\mu \leq P(\ph)$. We do this in Proposition \ref{prop:half-var} below.

The main ingredient in the proof is a Katok pressure inequality. We prove this using a Katok entropy inequality proved by Riquelme \cite{fR18}. Since Riquelme works in discrete-time and uses Bowen metrics, our arguments go via this setting.

Consider a locally compact separable complete metric space, $F\colon X\to X$ a continuous map, $\psi\colon X\to \RR$ is continuous, and a coherent family of metrics $\DDD$. Let $\DDDBow= (\Dbow_n)$ be the family of Bowen metrics defined from the metric $D=d_1$. Recall that for $A \subset X$, we write $A^F_n = A \cap F^{-n}A$,\nome{$A_n^F$}{returning trajectories for discrete-time map}
and that $\MF^\psi$ denotes the set of $F$-invariant Borel probability measures $\mu$ on $X$ such that at least one of $h_\mu(F) + \int \psi^+ \,d\mu$ and $\int\psi^-\,d\mu$ is finite.

\begin{theorem} \label{thm:Katok-P}
Let $X,F, \DDD,\psi$ be as above. Fix $A \in \RRR$ and an ergodic measure $\mu \in \MF^\psi$.
For each $\gamma \in (0,1)$, $\zeta>0$, $Z \subset X$ and $n\in \NN$, let \nome{$\Lambda_F^A(\psi,\mu,n,\zeta,\gamma)$}{partition sum required to reach $\mu$-weight $\gamma$}\nome{$\gamma$}{weight threshold in Katok formula}
\begin{equation}\label{eqn:L-Katok}
\Lambda^\mu_\gamma(\psi, n, \zeta, Z, \DDD) := \inf \Big\{ \sum_{x\in E} e^{\Psi(x,n)} : E\subset Z,\ \mu(B_n^F(E,\zeta; \DDD)) \geq \gamma \Big\}.
\end{equation}
Suppose that $\psi$ has tempered distortion on $A$. If $n_k \to \infty$ is a sequence such that $\mu(A^F_{n_k}) \geq \gamma$ for all $k$, then
\begin{equation}\label{eqn:Katok-0}
h_\mu(F) + \int \psi \,d\mu \leq \lim_{\zeta\to 0} \liminf_{k\to\infty} \frac 1{n_k} \log \Lambda^\mu_\gamma(\psi, n_k, \zeta, A_{n_k}^F, \DDD).
\end{equation}
\end{theorem}
\begin{proof}
Our proof of Theorem \ref{thm:Katok-P} is based on a Katok entropy inequality proved by Riquelme \cite{fR18}, which we use in the following form:\nome{$\eta$}{weight threshold in Katok entropy inequality cited from Riquelme}\nome{$h_\mu^\zeta(F,\eta)$}{growth of cardinality of $(n,\zeta)$-sep set reaching weight $\eta$}
for every $0<\eta<1$, the quantity
\begin{equation}\label{eqn:hmu-eps-eta}
h^\mu_{\eta}(F,\zeta) := \liminf_{n\to\infty} \frac 1n \log \Lambda^\mu_\eta (0, n, \zeta, X, \DDDBow)
\end{equation}
has the property that
\begin{equation}\label{eqn:eps-to-0}
\lim_{\zeta\to 0} h^\mu_{\eta}(F,\zeta) \geq h_\mu(F).
\end{equation}
By Lemma \ref{lem:coherencediscreteD}, we know that for all $x\in X$ and $\zeta>0$, we have $B_n(x, \zeta, \DDD) \subset B_n(x, \zeta, \DDDBow)$. It follows that
\begin{equation} \label{eqn:hmu-eps-etacoherent}
 \liminf_{n\to\infty} \frac 1n \log \Lambda^\mu_\eta (0, n, \zeta, X, \DDD) \geq h^\mu_\zeta(F, \zeta),
\end{equation}
and that
\begin{equation}\label{eqn:eps-to-0coherent}
\lim_{\zeta\to 0}  \liminf_{n\to\infty} \frac 1n \log \Lambda^\mu_\eta (0, n, \zeta, X, \DDD) \geq h_\mu(F).
\end{equation}
For the rest of the proof, all balls and partitions sums are considered in the metric family $\DDD$, so we suppress this from the notation. To deduce \eqref{eqn:Katok-0} from \eqref{eqn:eps-to-0coherent}, we need to produce a large set in $A_n^F$ on which the Birkhoff averages are close to $\int\psi\,d\mu$. 

Let $\gamma$ be as in the hypothesis of Theorem \ref{thm:Katok-P}. We first proceed under the additional assumption that $\int |\psi|\, d \mu <\infty$, and deal with the nonintegrable case later.
The Birkhoff ergodic theorem
gives $\frac 1n \Psi(x, n) \to \int\psi\,d\mu$ for $\mu$-a.e.\ $x\in X$, and Egorov's theorem guarantees that this convergence can be made uniform on a set $G$ of measure at least $1-\gamma/2$. Thus $\mu(G^c) < \gamma/2$, and writing\nome{$\omega_n$}{max deviation of Birkhoff averages from integral on set of large measure}
\begin{equation}\label{eqn:omega-n}
\omega_n := \sup_{x\in G} \Big| \frac 1n \Psi(x, n) - \int\psi\,d\mu \Big|,
\end{equation}
we have $\omega_n\to 0$ as $n\to\infty$.
Suppose $n\in \NN$ and $E \subset A_n^F$ satisfy $\mu(B_n^F(E,\zeta)) \geq \gamma$.
Consider the subset\nome{$\Ebad$}{points where Birkhoff averages robustly stays away from integral}
\begin{equation}\label{eqn:Ebad}
\Ebad := \{x\in E : B^F_{n}(x,\zeta) \subset G^c \};
\end{equation}
then we have
\[
\mu(B^F_{n}(\Ebad,\zeta)) \leq \mu(G^c)
< \gamma/2.
\]
Let $\Egood := E \setminus \Ebad$, then\nome{$\Egood$}{points with good Birkhoff averages} 
$B^F_{n}(E,\zeta) \subset B^F_{n}(\Egood,\zeta) \cup B^F_{n}(\Ebad,\zeta)$, so
\[
\gamma \leq \mu(B^F_{n}(E,\zeta))
\leq \mu(B^F_{n}(\Egood,\zeta)) + \gamma/2,
\]
from which we conclude that $\mu(B^F_{n}(\Egood,\zeta)) \geq \gamma/2$.
For every $x\in \Egood$, there exists $y\in B^F_{n}(x,\zeta) \cap G$.
Since $E \subset A^F_{n}$, we have
\[
|\Psi(x, n)-\Psi(y, n))| \leq  \Var(F, \psi,n,\zeta,A^F_{n}) =: V(n) \
\]
Moreover, by \eqref{eqn:omega-n} we have $|\frac 1{n} \Psi(y, n) - \int\psi\,d\mu| < \omega_{n}$, and we conclude that
\[
\frac 1{n} \Psi(x, n) \geq \int \psi \,d\mu - \omega_{n} - \frac{1}{n} V(n).\]
Summing gives
\[
\sum_{x\in E} e^{\Psi(x, n)}
\geq \sum_{x\in \Egood} e^{\Psi(x, n)}
\geq (\#\Egood) e^{n (\int \psi\,d\mu - \omega_{n}) -  V(n)}.
\]
This bound applies to any $E \subset A_n^F$ with $\mu(B_n^F(E,\zeta)) \geq \gamma$. If $n\in \NN$ is such that $\mu(A_n^F) \geq \gamma$, then taking an infimum over such $E$, we obtain
\[
\frac 1{n}
\log \Lambda^\mu_\gamma(\psi, n, \zeta, A_{n}^F) \geq
\frac 1{n}\log \Lambda^\mu_\gamma (0, n, \zeta, X) + \int\psi\,d\mu - \omega_{n} - \frac{1}{n} V(n).
\]

Under the hypotheses of Theorem \ref{thm:Katok-P}, this applies along a subsequence $n_k\to\infty$, and using \eqref{eqn:hmu-eps-etacoherent}, we deduce that
\begin{align*}
\liminf_{k\to\infty} \frac 1{n_k} \log \Lambda_\gamma^\mu(\psi,n_k,\zeta,A_n^F)
&\geq h^\mu_{\gamma}(F,\zeta) + \int \psi\,d\mu - \limsup_{k\to\infty}\Big(\omega_{n_k} + \frac{1}{n_k} V(n_k)\Big) \\
&= h^\mu_{\gamma}(F,\zeta) + \int \psi\,d\mu - \delta^A(\zeta).
\end{align*}
By \eqref{eqn:eps-to-0coherent} and the tempered distortion property, this last bound satisfies
\[
\lim_{\zeta\to 0} \Big( h^\mu_{\gamma}(F,\zeta) + \int \psi\,d\mu - \delta^A(\zeta) \Big)
\geq h_\mu(F) + \int \psi \,d\mu,
\]
which proves \eqref{eqn:Katok-0} when $\int |\psi| d \mu <\infty$.

Now we deal with the nonintegrable case, when $\int \psi^+ \,d\mu = \infty$ or $\int \psi^- \,d\mu = \infty$.
In the latter case, since $\mu\in \mathcal{M}_{F}^{\psi}$, we have $h_{\mu}(F)<\infty$, so the LHS of \eqref{eqn:Katok-0} must be $-\infty$, and we are done. Therefore, we are left with the case that 
$\int \psi^- \,d\mu < \infty$ and $\int \psi^+ \,d\mu = \infty$.

Fix $M\in \mathbb{R}$. For any $\beta\in \mathbb{R}$ and each $x\in X$, write $\psi^{\beta}(x):=\min\{\psi(x),\beta\}$. Since $\lim_{\beta\to \infty}\int \psi^{\beta}\,d\mu=\int \psi \,d\mu=\infty$, there exists $\beta_M\in \mathbb{R}$ such that $\int \psi^{\beta_M}\,d\mu>M$. An application of the Birkhoff ergodic theorem to $(F,\mu,\psi^{\beta_M})$ implies that for $\mu$-a.e.\ $x\in X$, there exists $N(x)\in \mathbb{N}$ such that
\[
\frac 1n \Psi(x, n)>M, \quad \forall n\geq N(x).
\]
For each $N\in \mathbb{N}$, let $X(N):=\{x\in X:N(x)\leq N\}$. Since $\mu(\bigcup_{N\in \mathbb{N}}X(N))=1$ and $X(N)$ is an increasing sequence of sets, there exists some $N=N(M,\gamma)\in \mathbb{N}$ such that $\mu(X(N))>1-\gamma/2$. Let $G:=X(N)$. For any $\zeta>0$, $n\in \mathbb{N}$, and $E\subset A_n^F$ with $\mu(B_n^F) > \gamma$, define $E^{\text{bad}}$ and $E^{\text{good}}$ as in \eqref{eqn:Ebad} for this choice of $G$. For every $x\in E^{\text{good}}$, there exists $y\in B_n^F(x,\zeta)\cap G$, which satisfies 
\[
|\Psi(x,n)-\Psi(y,n)|\leq V(n) \quad \text{ and } \quad \frac 1n \Psi(y, n)>M.
\]
Consequently, we have
\begin{equation} \label{eq:Psiaverage}
    \frac 1n \Psi(x, n)>M-\frac{1}{n}V(n).
\end{equation}
Using \eqref{eq:Psiaverage}, we repeat the rest of the argument we used for the integrable case, with $M$ in the place of $\int \psi \,d\mu-\omega_n$. This gives
\[
\liminf_{k\to\infty} \frac 1{n_k} \log \Lambda_\gamma^\mu(\psi,n_k,\zeta,A_n^F)>h^\mu_{\gamma}(F,\eps) + M - \delta^A(\eps),
\]
which by \eqref{eqn:eps-to-0coherent} and the tempered distortion property on $\psi$ implies that
\[
\lim_{\eps\to 0} \Big( h^\mu_{\gamma}(F,\eps) + \int \psi\,d\mu - \delta^A(\eps) \Big)
> h_\mu(F) + M\geq M.
\]
Since $M$ was arbitrary, we conclude that \eqref{eqn:Katok-0}. 
\end{proof} 

Returning to the continuous-time setting, we  prove the following proposition.

\begin{proposition}\label{prop:half-var}
Let $X,\FFF,\DDD,\ph$ satisfy Condition \ref{cond:basic}.
Suppose that $\mu \in \Mf^\ph$ is ergodic and $A\in \RRR$ is such that $\mu(A)>0$ and $\ph$ has tempered distortion on $A$. Then
\[
h_\mu(\FFF) + \int \ph \,d\mu \leq P^A(\ph).
\]
\end{proposition}

\begin{proof}
If $\mu$ is a flow-invariant ergodic Borel probability measure on $X$, then there exists $\tau>0$ such that $\mu$ is ergodic for $f_{\tau}$. Indeed, this is true for all but at most countably many $\tau$; see \cite[Lemma 7]{OP04} or \cite[Theorem 3.3.13]{FH19}.
We set $F = f_{\tau}$, 
$D_n(x,y) = d_{n \tau}(x,y)$, 
and $\psi(x) = \Phi(x,\tau)$. 
Given $A\in \RRR$ such that $\mu(A)>0$ and $\ph$ has tempered distortion w.r.t.\ $A$ (for the flow), we see that $\psi$ has tempered distortion w.r.t.\ $A$ (for $F$).

Now fix $0 < \gamma < \mu(A)^2$ and apply Khintchine's recurrence theorem (\cite[Theorem 2.3.3]{kP83}, \cite{aK35}) to deduce that there exists a sequence $n_k\to\infty$ such that $\mu(A^F_{n_k}) \geq \gamma$ for all $k$.
Given $k\in \NN$, let $E\subset A_{n_k\tau} = A^F_{n_k}$
 be a $(n_k \tau, \zeta)$-separated set that is maximal with respect to inclusion, and thus satisfies $A^F_{n_k} \subset \overline{B}_{n_k}^F(E,\zeta)$.
This gives $\mu(B_{n_k}^F(E,2\zeta)) \geq \mu(A^F_{n_k}) \geq \gamma$, and by \eqref{eqn:L-Katok} we see that
\[
\Lambda(\ph,n_k \tau,\zeta, A_{n_k \tau}) \geq
\sum_{x\in E} e^{\Phi(x,n_k \tau)} = \sum_{x\in E} e^{\Psi(x, n_k)}
\geq \Lambda^\mu_\gamma(\psi, n_k, 2\zeta, A_{n_k}^F, (D_n)).
\]
Thus from \eqref{eqn:Katok-0} we have
\begin{align*}
P^A(\ph) &\geq \lim_{\zeta\to 0} \limsup_{k\to\infty} \frac 1{n_k \tau}
\log \Lambda(\ph,n_k \tau,\zeta,A_{n_k \tau}) \\
&\geq \frac 1{\tau} \lim_{\zeta\to 0} \limsup_{k\to\infty} \frac 1{n_k} \log\Lambda^\mu_\gamma(\psi, n_k, 2\zeta, A_{n_k}^F, (D_n))\\
&\geq \frac 1{\tau} \Big( h_\mu(f_{\tau}) + \int_X \int_0^{\tau} \ph(f_t x) \,dt \,d\mu(x) \Big)
= h_\mu(\FFF) + \int_X \ph \,d\mu,
\end{align*}
where the last equality uses Abramov's formula $h_\mu(f_{\tau}) = \tau h_\mu(\FFF)$ \cite{lA59}
together with Fubini's theorem and the fact that $\mu$ is flow-invariant.
\end{proof}

\subsection{Well-definedness and approximation by compacta}
We now add the assumption of specification and we complete the proof of \eqref{eqn:vp} in Theorem \ref{thm:vp} by showing that Gurevich--Sarig pressure is well-defined (independent of the choice of reference set) and that it can be approximated by the classical pressure on compact subsystems:

\begin{theorem}\label{thm:PAKT}
Let $X,\FFF,\DDD,\ph$ satisfy Condition \ref{cond:basic}. Suppose that $\ph$ has tempered distortion on $\RRR$ and that $\FFF$ has specification w.r.t.\ $\RRR$ at all scales in the metric family $\DDD$. Then for every $A\in \RRR$, we have $P^A(\ph) = P(\ph)$. Moreover, for each $\Delta>0$, the set
\begin{equation}\label{eqn:KTA}
K_\Delta^A := \bigcap_{t\in \RR} f_{-t}(f_{[0,\Delta]} A)
\end{equation}
is compact and $\FFF$-invariant, and these sets satisfy
\begin{equation}\label{eqn:KT-P}
\limsup_{\Delta\to\infty} P(K_\Delta^A,\ph) = P(\ph).
\end{equation}
\end{theorem}

\begin{remark}\label{rmk:syndetic}
Observe that $K_\Delta^A$ comprises all points $x\in X$ with \emph{$\Delta$-syndetic returns} to $A$, in the sense that for every $t\in \RR$, there exists $s\in [0,\Delta]$ such that $f_{t+s}(x) \in A$. The sets $K_\Delta^A$ are analogous to the sets used in \cite[Theorem 2]{oS99} to prove that Gurevich pressure of countable-state Markov shifts can be approximated by classical pressure of compact subshifts. If we additionally assume that we have the marked specification property of Definition \ref{def:markspec}, then the flow restricted to the sets $K_\Delta^A$ has the classical specification property of Definition \ref{def:class-spec}.
\end{remark}

For the first claim in Theorem \ref{thm:PAKT}, that pressure is independent of $A$, we prove:

\begin{proposition}\label{prop:A'A}\nome{$A'$}{another reference set when we need to discuss independence of something from choice of $A$}
Let $X,\FFF,\DDD,\ph$ satisfy Condition \ref{cond:basic}. If $\rho>0$ and $A,A' \in \RRR$ are such that the flow has specification at scale $\rho$ w.r.t $A\cup A'$, and if $A$ contains a ball of radius $\rho$, then for all $\zeta>0$, we have 
\begin{equation}\label{eqn:PA'A}
P^A(\ph,\zeta) \geq P^{A'}(\ph,\zeta+2\rho) - \delta^{A'}(\rho).
\end{equation}
\end{proposition}
\begin{proof}
Suppose $A,A'\in \RRR$ are such that the flow has specification at scale $\rho$ w.r.t.\ $A\cup A'$ with transition time $\tau>0$, and fix $x_0\in A$  such that $B(x_0,\rho) \subset A$. Given $T>0$, observe that
\[
\Spec_\rho^\tau((x_0,0),(A'_T,T),(x_0,0)) \subset A_{T+2\tau},
\]
so by Lemma \ref{lem:spec-sum}, for every $\zeta>0$, we have
\[
\Lambda(\ph,T+2\tau,\zeta,A_{T+2\tau})
\geq e^{-3\Vt - \Var(\ph,T,\rho,A'_T)} \Lambda(\ph,T,\zeta+2\rho,A'_T).
\]
Taking logs, dividing by $T$, and sending $T\to \infty$ gives \eqref{eqn:PA'A}.
\end{proof}

To deduce the first claim in Theorem \ref{thm:PAKT}, observe that when $\FFF$ has specification w.r.t.\ $\RRR$ and $\ph$ has tempered distortion,
we can put $\zeta=\rho$ in Proposition \ref{prop:A'A} and send $\rho \to 0$ to conclude that $P^A(\ph) \geq P^{A'}(\ph)$ for every $A,A' \in \RRR$. Since $P(\ph) = \sup_{A\in \RRR} P^A(\ph)$, this shows that $P^A(\ph) = P(\ph)$ for all $A\in \RRR$.

Now we turn our attention to the sets $K_\Delta^A$ defined in \eqref{eqn:KTA}.
Each $K_\Delta^A$ is invariant by definition, and compact as a result of compactness of $A$ and continuity of the flow.
So it remains to prove \eqref{eqn:KT-P}, which we do via the next result:

\begin{proposition}\label{prop:K}
Let $X,\FFF,\DDD,\ph$ satisfy Condition \ref{cond:basic}.
If $\FFF$ has specification at scale $\rho>0$ w.r.t.\ $A\in \RRR$, then for every $\zeta>0$ we have
\begin{equation}\label{eqn:limsup-K}
\limsup_{\Delta\to\infty} P(K_\Delta^A,\ph,\zeta) \geq P^{A^{-\rho}}(\ph,\zeta + 2\rho) - \delta^A(\rho).
\end{equation}
\end{proposition}
\begin{proof}
Let $\tau=\tau(A,\rho)$ be the transition time in the specification property at scale $\rho$ w.r.t.\ $A$, and let $t = \Delta - \tau$.
Given $n\in \NN$, write $\bt := (t,\dots, t) \in (0,\infty)^n$, and observe that for every
$\bx \in ((A^{-\rho})_t)^n$, we have
\[
\emptyset \neq \Spec_\rho^\tau(\bx,\bt) \subset K_\Delta^A.
\]
Applying Lemma \ref{lem:spec-sum}, we deduce that with $\Vt=\Vt(A,\tau)$ as in Lemma \ref{lem:L}, and $T = n\Delta - \tau$, we have
\[
\Lambda(\ph,T,\zeta,K_\Delta^A) \geq e^{-n\Vt} e^{-n \Var(\ph,t,\rho,(A^{-\rho})_{t}) } \Lambda(\ph,t,\zeta+2\rho,(A^{-\rho})_{t})^n.
\]
Observing that
$\Var(\ph,t,\rho,(A^{-\rho})_{t}) \leq \Var(\ph,t,\rho,A_t)$, we can take logs and divide by $T$ to obtain
\[
\frac 1T \log \Lambda(\ph,T,\zeta,K_\Delta^A) \geq \frac{-n(\Vt+\Var(\ph,t,\rho,A_t))+ n\log\Lambda(\ph,t,\zeta+2\rho,(A^{-\rho})_t)}{n\Delta - \tau}.
\]
Sending $n\to\infty$ (and thus $T\to\infty$), we get
\[
P(K_\Delta^A,\ph,\zeta) \geq -\frac{\Vt+\Var(\ph,t,\rho,A_{t})}{\Delta} + \frac 1{\Delta} \log\Lambda(\ph,t,\zeta+2\rho,(A^{-\rho})_t).
\]
By the definition of $\delta^A(\rho)$ and the fact that $\Delta = t+\tau$,
sending $\Delta\to\infty$ (and thus $t\to\infty$) proves \eqref{eqn:limsup-K}.
\end{proof}

To complete the proof of Theorem \ref{thm:PAKT}, we take $\zeta = \rho$ in Proposition \ref{prop:K} and combine \eqref{eqn:limsup-K} with Proposition \ref{prop:A'A} to deduce that for every $\rho>0$, we have
\[
\limsup_{\Delta\to\infty} P(K_\Delta^A,\ph)
\geq \limsup_{\Delta\to\infty} P(K_\Delta^A,\ph,\rho)
\geq P^A(\ph,5\rho) - 2\delta^A(\rho).
\]
Sending $\rho\to 0$ and using tempered distortion gives $\limsup_{T\to\infty} P(K_T^A,\ph) \geq P(\ph)$. Since Theorem \ref{thm:half-var} gives $P(K_T^A,\ph) \leq P(\ph)$, this proves Theorem \ref{thm:PAKT}.

\subsection{Pressure appears at a scale}
Now we add an expansivity condition that allows us to deduce the final claim in Theorem \ref{thm:vp}.

\begin{theorem}\label{thm:exp}
Let $X,\FFF,\DDD,\ph$ satisfy Condition \ref{cond:basic}. Suppose that $\ph$ has tempered distortion on $\RRR$ and that $\FFF$ has specification w.r.t.\ $\RRR$ at all scales.
Suppose that $A\in \RRR$ and $\rexp \in (0,\rmet(A)]$ are such that $\FFF$ is expansive on $K_\Delta^A$ at scale $\rexp > 0$ for every $\Delta>0$, where $K_\Delta^A$ is defined as in \eqref{eqn:KTA}.
Then $P^A(\ph,\eps) = P(\ph)$ for every $\eps \in (0,\frac 12 \rexp)$.
\end{theorem}
\begin{proof}
Let $\rho_0>0$ be sufficiently small that  $\eps + 2\rho_0 < \frac12\rexp$(A) and $A$ contains a $\rho_0$-ball. Using Lemma \ref{lem:coherencediscreteD}, the expansivity condition with respect to $\RRR$ and the family of metrics $\DDD$ ensures that on any compact invariant set that the time-$1$ map is expansive in the classical case in the Bowen family of metrics $\DDDBow$.
It follows that for every $\rho \in (0,\rho_0]$ and $\Delta>0$, the expansivity condition together with \cite[Proposition 3.7]{CT16} gives
\begin{equation}\label{eqn:all-P}
P(K_\Delta^A, \ph, \eps + 2\rho) = P(K_\Delta^A, \ph).
\end{equation}
Fixing $\Delta>0$, let $A' \in \RRR$ contain $K_\Delta^A$.
Then for every $\rho \in (0,\rho_0]$, we have
\begin{align*}
P^A(\ph,\eps) & 
\geq P^{A'}(\ph,\eps + 2\rho) - \delta^{A'}(\rho)
&& \text{by Proposition \ref{prop:A'A}} \\
&\geq P(K_\Delta^A,\ph,\eps + 2\rho) - \delta^{A'}(\rho) &&
\text{since $K_\Delta^A \subset A'$} \\
&= P(K_\Delta^A,\ph) - \delta^{A'}(\rho)
&& \text{by \eqref{eqn:all-P}}.
\end{align*}
Sending $\rho\to 0$ and using tempered distortion, we get $P^A(\ph,\eps) \geq P(K_\Delta^A,\ph)$. As $\Delta\to\infty$, the RHS goes to $P(\ph)$ by Theorem \ref{thm:PAKT}, which completes the proof.
\end{proof}

There is one statement in Theorem \ref{thm:welldefined} that is not covered by Theorem \ref{thm:vp}: namely, the result that no compact invariant set can carry full pressure. This is provided by Theorem \ref{thm:K-gap} in \S\ref{s.uppercounting}.

\section{Uniform counting bounds and strong positive recurrence}\label{sec:counting}

In this section, we prove the following result, which suffices to prove Theorem \ref{thm:uniform}.

\begin{theorem}\label{thm:uniformgeneral}
Let $X,\FFF,\DDD,\ph$ satisfy Conditions \ref{cond:basic}, \ref{cond:S} and \ref{cond:B}. Let $\ph$ be SPR w.r.t.\ $A\in \RRR$,
and suppose that the scale $\eps=\eps(\AO[1])>0$ from Condition \ref{cond:B} satisfies $P^A(\ph,\eps) = P(\ph)$.
% Let $\eps>0$ be the scale $\eps(\AO[1])$ from Condition \ref{cond:VP}. 
Then for all $\zeta\in (0,\fraceps)$, there are constants $C_U = C_U(A,\zeta)$, $C_L = C_L(A,\zeta)$ and $T_0=T_0(A,\zeta)$ such that for all $T>T_0$ we have
\begin{equation}\label{eqn:unif-counting-intro-again}
C_L e^{TP(\ph)} \leq \Lambda(\ph,T,\zeta,A_T) \leq C_U e^{TP(\ph)}.
\end{equation}
If $\zeta \in (0,\halffraceps)$ and $E \subset A_T$ is \emph{any} $(T,\zeta)$-separated set that is maximal with respect to inclusion, then
\begin{equation}\label{eqn:unif-counting-again-again}
C_L(A,2\zeta) e^{-Q} e^{TP(\ph)} \leq \sum_{x\in E} e^{\Phi(x,T)} \leq C_U(A,\zeta) e^{TP(\ph)}.
\end{equation}
\end{theorem}

Observe that \eqref{eqn:unif-counting-again-again} is immediate from \eqref{eqn:unif-counting-intro-again} and Lemma \ref{lem:all-sep}, so it suffices to prove \eqref{eqn:unif-counting-intro-again}. Theorem \ref{thm:uniform} also states counting bounds on periodic orbits under additional assumptions and we establish this corollary in \S \ref{sec:periodicorbits}.

\subsection{Upper counting bounds via supermultiplicativity}\label{s.uppercounting}

We prove some consequences of the supermultiplicativity property in Lemma \ref{lem:spec-sum-Bow}. In this section we do not yet need strong positive recurrence; that will enter in \S\ref{s.SPR}.
We start with a uniform upper counting bound.

\begin{lemma} \label{upper}
Let $X,\FFF,\DDD,\ph$ satisfy Condition \ref{cond:basic}.
Suppose that $\FFF$ has specification w.r.t.\ $A\in\RRR$ and that $\ph$ satisfies the Bowen property for $A$ at some positive scale.
Then for every $\zeta > 0$, there exists $C_U = C_U(A,\zeta)$ such that
\begin{equation}\label{eqn:CU}
\Lambda(\ph,t,\zeta, A_t) \leq C_Ue^{tP^A(\varphi)} \text{ for all }t\geq 0.
\end{equation}
\end{lemma}
\begin{proof} 
When $t=0$, the result is a consequence of compactness of $A$. 
For $t>0$, fix $\eps\in (0,\zeta/2)$ such that $A^{-\eps}\neq \emptyset$ and the Bowen property holds on $A$ at scale $\eps$,
with constant $Q$. Then for each $k \in \NN$, we can apply \eqref{eqn:supermult} from Lemma \ref{lem:spec-sum-Bow} with $\bt = (t,\dots, t)$, where $t$ appears $k$ times and thus $T = T(A,\bt) = (k-1)\tau + kt$, to obtain
\[
\Lambda(\ph,t,\zeta,A_t)^k \leq e^{(k+2)(\Vt+Q)} \Lambda(\ph,S,\zeta',A_S),
\]
where $S= S(k) := T+2\tau = (k+1)\tau + kt$ and $\zeta' = \zeta-2\eps$. Taking logs and dividing by $k$ gives
\[
\log \Lambda(\ph,t,\zeta,A_t)
\leq \frac{(k+2)(\Vt+Q)}{k} + \frac{S}{k} \cdot \frac 1S \log \Lambda(\ph,S,\zeta',A_S).
\]
Sending $k\to\infty$ and observing that $\frac 1k S(k) \to \tau + t$, we obtain
\[
\log \Lambda(\ph,t,\zeta,A_t)
\leq \Vt+Q+(\tau+t) \limsup_{S\to\infty} \frac 1S \log \Lambda(\ph,S,\zeta',A_S).
\]
The lim sup is bounded above by $P^A(\ph,\zeta') \leq P^A(\ph)$, so we can take the exponential of both sides and obtain \eqref{eqn:CU}
with $C_U = e^{\Vt+Q+\tau P^A(\ph)}$.
\end{proof}

With this uniform upper counting bound in hand, we can 
%; observe that for this argument, we need the Bowen property itself, not just tempered distortion.
use an argument similar to \cite[Proposition 3.1]{CT2}, \cite[Theorem B]{BCFT}, and \cite[Theorem 6.1]{CFT19} to
complete the proof of Theorem \ref{thm:welldefined} via the following.

\begin{theorem}\label{thm:K-gap}
Let $X,\FFF,\DDD,\ph$ satisfy Condition \ref{cond:basic}. Suppose that $\FFF$ has specification w.r.t.\ $\RRR$ at all scales, and that $\ph$ has the Bowen property w.r.t.\ $\RRR$.
Then for every $K\in \mathcal{K}$ and $\delta>0$, we have $P(K,\ph,\delta) < P(\ph)$.
\end{theorem}

\begin{proof}%[Proof of Theorem \ref{thm:K-gap}]
To obtain a contradiction, suppose that $K\in \mathcal{K}$ and $\delta>0$ are such that $P(K,\ph,\delta) = P(\ph)$. Fix $x_0 \in X\setminus K$, $\rho \in (0,\delta/8)$ and $A \in \RRR$ such that $B(x_0,2\rho) \subset A\setminus K$, and in particular
\begin{equation}\label{eqn:xy-disj}
B(x_0,\rho) \cap B(y,\rho) = \emptyset \quad\text{for all } y\in K.
\end{equation}
Since $\delta \mapsto P(K,\ph,\delta)$ is monotonic, we can assume without loss of generality that $\delta \in (0,\rmet(K)]$ and that $\delta$ is small enough so that $\varphi$ has the Bowen property on $A$ at scale $\delta$. Finally, %fix $A\in \RRR$ such that $K \cup B(x_0,2\rho) \subset A$, and
fix $\zeta>0$ such that $\zeta + 2\rho \leq \delta/4$. 

We will prove a lower counting bound on $\Lambda(\ph,t,\zeta+2\rho,K)$, given in \eqref{eqn:K-lower}, and then use this together with the specification property and the separation between $x_0$ and $K$ to prove that $P^A(\ph) > P(K,\ph,\delta) = P(\ph)$, a contradiction.

Towards the lower counting bound, observe that Lemma \ref{upper} gives
\[
\Lambda(\ph,T,\delta/2,K) \leq \Lambda(\ph,T,\delta/2,A_T) \leq C_U(A,\delta/2) e^{TP(\ph)},
\]
so that for every $t>0$ and every $s\in [0,t]$, we have
\begin{equation}\label{eqn:K-leq}
\Lambda(\ph,s,\delta/2,K) \leq C_U e^{tP(\ph)}.
\end{equation}
Continuing to fix $t>0$, 
let $T>0$ be arbitrary, and write $T=nt+s$, where $n\in \NN$ and $s\in [0,t]$. 
Since $\delta \leq \rmet(K)$, we can apply Lemma \ref{lem:submult} twice and use \eqref{eqn:K-leq} to obtain
\begin{align*}
\Lambda(\ph,T,\delta,K) &\leq e^{2Q} \Lambda(\ph,n t,\delta/2,K) \Lambda(\ph,s,\delta/2,K) \\
&\leq e^{(n+2)Q} \Lambda(\ph,t,\zeta+2\rho,K)^{n} C_U e^{tP(\ph)}.
\end{align*}
Taking logs and dividing by $n$, we have
\[
\frac 1{T} \cdot \frac{T}{n} \log \Lambda(\ph,T,\delta,K)
\leq \frac{n +2}{n} \cdot Q + \log \Lambda(\ph,t,\zeta+2\rho,K) + \frac 1{n} (tP(\ph) + \log C_U).
\]
Observe that $\frac Tn\to t$ as $T\to\infty$, so we can take a $\limsup$ of both sides of this inequality and deduce that
\[
tP(K,\ph,\delta) \leq Q + \log \Lambda(\ph,t,\zeta+2\rho,K).
\]
Since $P(K,\ph,\delta) = P(\ph)$ by assumption, this yields the desired lower counting bound:
\begin{equation}\label{eqn:K-lower}
\Lambda(\ph,t,\zeta+2\rho,K) \geq e^{-Q} e^{t P(\ph)}.
\end{equation}
Let $\tau>0$ be the transition time in the specification property at scale $\rho$ on $A$. Given $N\in \NN$, let $T_N := 2\tau(N+1)$, and for each $x\in A_{T_N}$ let
\[
J(x) := \{ j\in \{1,\dots, N\} : f_{2\tau j}(x) \in B(x_0,\rho) \}.
\]
For each $J \subset \{1,\dots, N\}$, let
\[
A_{T_N}(J) := \{ x\in A_{T_N} : J(x) = J \},
\]
and observe that
\begin{equation}\label{eqn:LJ}
\Lambda(\ph,T_N,\zeta,A_{T_N})
= \sum_{J \subset \{1,\dots, N\}} \Lambda(\ph,T_N,\zeta,A_{T_N}(J)).
\end{equation}
Now fix $J \subset \{1,\dots, N\}$, let $\ell = \# J$, and enumerate the elements of $J$ in increasing order as $j_1 < j_2 < \cdots < j_\ell$. 
Write $j_0 = 0$ and $j_{\ell+1} = N+1$, and
for each $i \in \{0,\dots, \ell\}$, let $t_i = (j_{i+1} - j_i - 2)\tau$. Putting
\[
\mathbf{A}^J := (x_0, A_{t_1}, x_0, A_{t_2}, \dots, x_0, A_{t_\ell}, x_0)
\quad\text{and}\quad
\bt^J := (0,t_0,0,t_1,0,t_2, \dots, 0, t_\ell,0),
\]
and using %the fact that  $B(x_0,\rho) \cap B(y,\rho) = \emptyset$ for all $y\in K$ (a consequence of the inclusion $B(x_0,2\rho) \subset A\setminus K$), we
\eqref{eqn:xy-disj},
we see that
\[
\Spec_\rho^\tau(\mathbf{A}^J,\bt^J) \subset A_{T_N}(J).
\]
By Lemma \ref{lem:spec-sum-Bow} and \eqref{eqn:K-lower}, we have
\begin{align*}
\Lambda(\ph,{T_N},\zeta,A_{T_N}(J))
&\geq e^{-(2\ell+3)(\Vt+Q)} \prod_{i=0}^\ell \Lambda(\ph,t_i,\zeta+2\rho,K) \\
&\geq e^{-(2\ell+3)(\Vt+Q)} e^{-(\ell+1)Q} e^{(\sum_{i=0}^\ell t_i) tP(\ph)} \\
&\geq e^{-(3\ell+4)(\Vt+Q)} e^{(T_N - 2\tau \ell)P(\ph)}.
\end{align*}
Recalling \eqref{eqn:LJ} and summing over all $J \subset \{1,\dots, N\}$, we obtain
\[
\Lambda(\ph,{T_N},\zeta,A_{T_N})
\geq e^{-4(\Vt+Q)} e^{T_N P(\ph)} \sum_{\ell=0}^N \binom{N}{\ell} \big(e^{-3(\Vt+Q) - 2\tau P(\ph)}\big)^\ell.
\]
Let $C := e^{-4(\Vt+Q)}$ and $C' := e^{-3(\Vt+Q) - 2\tau P(\ph)}$; then this becomes
\[
\Lambda(\ph,{T_N},\zeta,A_{T_N}) \geq C e^{T_N P(\ph)} (1+C')^N.
\]
Taking logarithms and dividing by $T_N = 2\tau(N+1)$, we get
\[
\frac 1{T_N} \log\Lambda(\ph,{T_N},\zeta,A_{T_N}) 
\geq \frac{\log C}{T_N} + P(\ph) + \frac{N}{2\tau(N+1)} \log(1+C').
\]
Sending $N\to \infty$ gives
\[
P^A(\ph,\zeta) \geq P(\ph) + (2\tau)^{-1} \log(1+C') > P(\ph),
\]
a contradiction.
\end{proof}
In the setting of Theorem \ref{thm:welldefined}, 
for every $K\in \mathcal{K}$, we can fix any $\delta \in (0,\frac 12\rexp(K))$ and obtain $P(K,\ph) = P(K,\ph,\delta) < P(\ph)$, establishing \eqref{eqn:K-gap}.
\begin{remark}
Despite the appearance of \eqref{eqn:K-leq} and \eqref{eqn:K-lower} in the proof of Theorem \ref{thm:K-gap}, it does not actually establish uniform counting bounds for $K$, since it proceeds under a counterfactual. If $K$ has specification, then such bounds hold, but there is no reason to expect that every $K\in \mathcal{K}$ has specification.
\end{remark}
One further consequence of supermultiplicativity, which will be used in the proof of the lower bound in \S\ref{sec:lower-counting}, is that the $\limsup$ in the definition of pressure can be replaced by a limit. 

\begin{lemma}\label{lem:limsupislim}
Let $X,\FFF,\DDD,\ph$ satisfy Condition \ref{cond:basic}.
Suppose that $\FFF$ has specification w.r.t.\ some $A \in \RRR$,
and that there exists $\eps= \eps(A)>0$ such that $P^A(\ph) = P^A(\ph,\eps)$.  Then for every $\zeta \in (0,\eps)$, we have
\begin{equation}\label{eqn:lim-exists}
P^A(\varphi)=\lim_{T\rightarrow \infty}\frac{1}{T}\log \Lambda(\varphi,T,\zeta,A_T).
\end{equation}
\end{lemma}
\begin{proof}
Given $A,\eps,\zeta$ as in the statement, choose $\delta>0$ small enough that $A^{-\delta}\neq \emptyset$ and $\zeta':= \zeta + 2\delta < \eps$.
Let $\tau=\tau(A,\delta)$ be the transition time from the specification property. Using \eqref{eqn:supermult} from Lemma \ref{lem:spec-sum-Bow},, given $k\geq2$, $\bt=(t_1,\cdots,t_k)$, and $T=T(\bt, \tau)$, we have 
\begin{equation} \label{subadd0}
    \begin{aligned}
    \Lambda(\varphi,T+2\tau,\zeta,A_{T+2\tau})\geq  R^k\prod_{j=1}^k \Lambda(\varphi,t_j,\zeta',A_{t_j})
    \end{aligned}
\end{equation}
where $R= e^{-2(Q+\Vt)}$. Writing 
\[
B(t):=\log \Lambda(\varphi,t,\zeta,A_t)
 \quad\text{and}\quad
C(t):=\log (R\Lambda(\varphi,t,\zeta',A_t)),
\]
the inequality in \eqref{subadd0} becomes, after taking logarithms,
\begin{equation} \label{subadd}
    \begin{aligned}
    B(T+2\tau)\geq \sum_{j=1}^k C(t_j).
    \end{aligned}
\end{equation}
Our hypotheses imply that $P^A(\ph)= P^A(\ph, \zeta')=P^A(\ph, \zeta)$ and so we have
\begin{equation} \label{eq:limsupequal}
    \limsup_{S\rightarrow \infty}\frac{B(S)}{S}=\limsup_{S\rightarrow \infty}\frac{C(S)}{S}=P^A(\varphi)
\end{equation}
To prove \eqref{eqn:lim-exists}, it suffices to show that $\liminf \frac 1S B(S) \geq P^A(\ph)$, which we will do using \eqref{subadd}. Fix $T_0 > \tau$, and observe that every $S > T_0 + 3\tau$ can be written as
\begin{equation}\label{eqn:S}
S = 2\tau + (k-1)(T_0 + \tau) + s
\end{equation}
for some $k = k(S) \in \NN$ and $s \in [T_0 + \tau, 2(T_0+\tau)]$.
(It suffices to take $k = \lfloor\frac{S-2\tau}{T_0+\tau}\rfloor$.)
Writing $\bt = (T_0,\dots,T_0,s) \in [\tau,\infty)^k$,
observe that
\[
T(\bt,\tau) =  (k-1)(T_0 + \tau) + s = S-2\tau,
\]
and apply \eqref{subadd} to obtain
\begin{equation}\label{eqn:BS-geq}
B(S) \geq (k-1)C(T_0) + C(s).
\end{equation}
We claim that there is $C_0$ such that
\begin{equation}\label{eqn:Cs-geq}
C(s) \geq -C_0 \text{ for all } s\in [T_0+\tau, 2(T_0+\tau)].
\end{equation}
Indeed, let $A' \in \RRR$ be such that $f_t(A^\delta) \subset A'$ for all $t\in [-\tau,\tau]$ (such an $A'$ exists by Lemma \ref{lem:bdd-excursions}), and observe that given any $s\geq \tau$, we can use the specification property w.r.t.\ $A$ to produce $x\in A_s$ such that $f_t(x) \in A'$ for all $t\in [0,s]$, and thus 
\begin{equation}\label{eqn:Cs-geq-0}
C(s) \geq \Phi(x,s) + \log R \geq -s \|\ph|_{A'}\| + \log R
\quad\text{for all } s \geq \tau,
\end{equation}
which proves \eqref{eqn:Cs-geq}.
Combining \eqref{eqn:BS-geq} and \eqref{eqn:Cs-geq}, we conclude that
\[
\liminf_{S\to\infty} \frac{B(S)}S
\geq \liminf_{S\to\infty} \frac{(k(S)-1) C(T_0) - C_0}S = \frac{C(T_0)}{T_0+\tau}.
\]
This holds for every $T_0>\tau$, so we have
\[
\liminf_{S\to\infty} \frac{B(S)}S
\geq \limsup_{T_0\to\infty} \frac{C(T_0)}{T_0+\tau} = P^A(\ph),
\]
which completes the proof of \eqref{eqn:lim-exists}.
\end{proof}

\subsection{Strong positive recurrence}\label{s.SPR}

To prove the lower counting bound in \eqref{eqn:unif-counting-intro-again}, we return to the notion of strong positive recurrence introduced in \S\ref{sec:SPR-ES}, which involves partition sums for orbit segments that start and end near $A$ but never enter $A$, as was illustrated in Figure \ref{fig:SPR}. To make this precise, we fix a parameter $L>0$,
and recall that $A^L = \{x : d(x, A) \leq L\}$. 
Given $T>0$, consider the set\nome{$A^*_{T,L}$}{orbit segments start/ending in $A_L^+$, not entering $A$ in between}
\begin{equation}\label{eqn:ATL}
A^*_{T, L} := 
\{ x \in (A^L)_T : \text{for all $s\in [0,T]$, we have } f_s(x) \notin A \}.
\end{equation}
%of points whose orbit segments of length $T$ are as shown in Figure \ref{fig:SPR}.
We denote the partition sums for the sets $A^*_{T, L}$ by\nome{$\Lambda_A^*(T,L,\zeta)$}{restricted partition sum over $A^*_{T,L}$}
\begin{equation}\label{eqn:partition-sums*}
\Lambda^*(\ph,T,\zeta, A, L):=\Lambda(\ph,T,\zeta, A^\ast_{T, L} ),
\end{equation}
and define
\[
P^\ast( \ph, \zeta, A, L):= \limsup_{T \to \infty}  \frac{1}{T} \log \Lambda^*(\ph,T,\zeta, A, L).
\]

\begin{definition} \label{SPR}\terms{strong positive recurrent}
Given $P^* \in (0, P^A(\ph))$, we say that $\ph$ is \emph{$P^*$-strong positive recurrent ($P^*$-SPR)} for $A\in \RRR$ at scale $\zeta>0$ and distance $L>0$ if we have 
\begin{equation}\label{eqn:spr0}
P^*( \ph, \zeta, A, L) < P^*.
\end{equation}
We say that $\ph$ is \emph{$P^*$-SPR} for $A$ if for every $\zeta>0$, there exists $L = L(A,\zeta)$ such that \eqref{eqn:spr0} holds; equivalently, if $\inf_L P^*(\ph,\zeta,A,L) < P^*$ for every $\zeta>0$.

We say that $\ph$ is \emph{strong positive recurrent (SPR)} for $A$ if it is $P^*$-SPR for $A$ for some $P^* < P^A(\ph)$; equivalently, if $\sup_\zeta \inf_L P^*(\ph,\zeta,A,L) < P(\ph)$, as in \eqref{eqn:SPR-intro}.
\end{definition}

With $\ph,A$ fixed, observe that the function $(\zeta,L) \mapsto P^*(\ph,\zeta,A,L)$ is nonincreasing in $\zeta$ and nondecreasing in $L$. This follows from the corresponding property of the partition sums for each fixed $T$.  

The next result lets us use the SPR property to deduce upper bounds on more general collections of orbit segments that do not intersect $A$.

\begin{proposition}\label{prop:transferSPR}
Suppose that $X,\FFF,\DDD,\ph$ satisfy Conditions \ref{cond:basic} and \ref{cond:S}.
Let $A,A' \in \RRR$ be such that $A^\delta \subset A'$ for some $\delta>0$, and suppose that $\ph$ has the Bowen property on $A'$ at scale $\delta$, and is $P^*$-SPR on $A$ for some $P^* < P(\ph)$. Then for every $B\in \RRR$ and $\zeta>0$, there exist $C_1>0$ and $P' \in (0,P^*)$ such that given any $t\geq 0$ and any $Y \subset B_t$ satisfying $f_{[0,t]}(Y) \cap A' = \emptyset$, we have
\begin{equation}\label{eqn:Lambda-Y}
\Lambda(\ph,t,\zeta,Y) \leq C_1 e^{tP'}.
\end{equation}
\end{proposition}

\begin{proof}
Let $A, A' \in \RR$, $\delta \in (0,1]$, $P^*$ and $\ph$ be as in the statement of the lemma.  Given $B\in \RRR$, let $\rmet := \rmet(A^1 \cup B)$. We will prove the result for all $\zeta \in (0,\rmet]$, which is sufficient by monotonicity of $\zeta \mapsto \Lambda(\ph,t,\zeta,Y)$.

Since $A$ has non-empty interior, there exist $x\in A$ and $\delta_1 >0$ such that $B(x,\delta_1)\subset A$. 
Fix $\eta \in (0,\min\{\zeta/4,\delta,\delta_1\})$ and consider the set
\[
Z := \Spec_\eta^\tau\big( (\{x\},0),\, (Y,t),\, (\{x\},0) \big),
\]
which consists of points in $B(x,\eta) \subset A$ whose orbits lie in $B$ at times $\tau$ and $\tau+t$, avoid $A'$ for times in $[\tau,\tau+t]$, and lie in $B(x,\eta)$ at time $2\tau+t$. By Lemma \ref{lem:spec-sum-Bow}, we have
\begin{equation}\label{eqn:L*-leq}
\Lambda(\ph,t+2\tau,\zeta-2\eta,Z) \geq e^{-3(\Vt+Q)} \Lambda(\ph,t,\zeta,Y).
\end{equation}
We will obtain an upper bound for the left-hand side using the fact that $\ph$ is $P^*$-SPR on $A$. To this end, given $z\in Z$, let
\begin{align*}
a(z) &:= \sup \{s\in [0,\tau] : f_s(z) \in A \}, \\
b(z) &:= \sup \{s\in [0,\tau] : f_{t+2\tau-s}(z) \in A \}.
\end{align*}
Since $\ph$ is $P^*$-SPR on $A$, there exists $L_0 = L_0(A,\eta) \in (0,1]$ such that $\ph$ is $P^*$-SPR on $A$ at scale $\eta$ and distance $L_0$. By the uniform speed property in Lemma \ref{lem:bdd-excursions}, there exists $m\in \NN$ sufficiently large that writing $r := \tau/m$, we have $\AO[r] \subset A^{L_0}$. In particular, for every $z\in Z$, we have
\begin{equation}\label{eqn:fsz-in}
\begin{gathered}
f_s(z) \in A^{L_0} \text{ for all } s\in [a(z),a(z) + r), \text{ and} \\
f_{t+2\tau-s}(z) \in A^{L_0} \text{ for all } s\in [b(z),b(z) + r).
\end{gathered}
\end{equation}
Given $i,j \in \{1,\dots m\}$, consider the set
\begin{equation}\label{eqn:Zij}
Z_{i,j} := \{ z\in Z : a(z) \in [(i-1)r, ir)
\text{ and } b(z) \in [(j-1)r, jr) \}.
\end{equation}
Given $z\in Z_{i,j}$, we deduce from \eqref{eqn:fsz-in} that
\begin{equation}\label{eqn:fir-in}
f_{ir}(z) \in A^{L_0} \subset A^1
\quad\text{and}\quad
f_{t+2\tau - jr}(z) \in A^{L_0} \subset A^1.
\end{equation}
Observe that $Z = \bigcup_{i=1}^m \bigcup_{j=1}^m Z_{i,j}$, so
\begin{equation}\label{eqn:sum-Zij}
\Lambda(\ph,t+2\tau,\zeta-2\eta,Z) 
\leq \sum_{i=1}^m \sum_{j=1}^m \Lambda(\ph,t+2\tau,\zeta-2\eta,Z_{i,j}). 
\end{equation}
Writing $t_{i,j} := t+2\tau-(i+j)r$
and using the fact that $2\eta < \zeta - 2\eta \leq \rmet(A^1 \cup B)$
together with the inclusions in \eqref{eqn:fir-in},
we can apply Lemma \ref{lem:submult} to each 
$Z_{i,j}$ with $\bt = (ir, t_{i,j}, jr)$, obtaining
\begin{multline}\label{eqn:L-L3}
\Lambda(\ph,t+2\tau,\zeta-2\eta,Z_{i,j})
\leq e^{3Q} \Lambda(\ph,ir,\eta,Z_{i,j})
\cdot \Lambda(\ph,t_{ij},\eta,f_{ir}(Z_{i,j})) \\
\cdot \Lambda(\ph,jr,\eta,f_{t+2\tau-jr}(Z_{i,j})).
\end{multline}
 The inclusions in \eqref{eqn:fir-in} give
\begin{equation}\label{eqn:Zij-in}
f_{ir}(Z_{i,j}) \subset A^*_{t_{i,j},L_0}
\quad\Rightarrow\quad
\Lambda(\ph,t_{ij},\eta,f_{ir}(Z_{i,j})) \leq \Lambda^*(\ph,t_{i,j},\eta,A,L_0).
\end{equation}
We also have $Z_{i,j} \subset (A^{L_0})_{ir}$ and $f_{t+2\tau-jr}(Z_{i,j}) \subset (A^{L_0})_{jr}$, so writing
\[
C_2 := \max_{1\leq i\leq m} \Lambda(\ph,ir,\eta,A^{L_0}) < \infty,
\]
we deduce from \eqref{eqn:L*-leq}, \eqref{eqn:sum-Zij}, \eqref{eqn:L-L3}, and \eqref{eqn:Zij-in} that
\[
%\Lambda^*(\ph,t,\zeta,A',L) 
\Lambda(\ph,t,\zeta,Y)
\leq e^{3(\Vt+Q)} e^{3Q} C_2^2 \sum_{i=1}^m \sum_{j=1}^m \Lambda^*(\ph,t_{i,j},\eta,A,L_0).
\]
Since $\ph$ is $P^*$-SPR on $A$, there exists $P' \in (P^*(\ph,\eta,A,L_0), P^*)$. Then there exists $C_3>0$ such that $\Lambda^*(\ph,T,\eta,A,L_0) \leq C_3 e^{TP'}$ for all $T$, and we conclude that
\[
%\Lambda^*(\ph,t,\zeta,A',L) 
\Lambda(\ph,t,\zeta,Y) \leq e^{3\Vt + 6Q} C_2^2 m^2 C_3 e^{(t+2 \tau)P'}
\text{ for all } t\geq 0,
\]
%so $P^*(\ph,\zeta,A',L) \leq P' < P^*$, 
which proves Proposition \ref{prop:transferSPR}.
\end{proof}

\begin{corollary}\label{cor:ES-A}
Suppose that $X,\FFF,\DDD,\ph$ satisfy Conditions \ref{cond:basic}, \ref{cond:S}, and \ref{cond:B}, and that $\ph$ is SPR for $A \in \RRR$. Then for every equilibrium state $\nu \in \MF^\ph$ and every $\delta>0$, we have $\nu(A^\delta) > 0$.
\end{corollary}
\begin{proof}
Let $Y := \{ x\in X : f_t(x) \notin A^\delta$ for all $t\in \RR\}$. Then $Y$ is flow-invariant, so 
given any ergodic $\nu \in \MF^\ph$ such that $\nu(A^\delta)=0$, we must have $\nu(Y)=1$. 

Fix $B \in \RRR$ such that $\nu(B)>0$, and observe that we have $\nu(B\cap Y) = \nu(B) > 0$.
Now $Y,\FFF|_Y,\DDD|_Y,\ph|_Y$ satisfy Condition \ref{cond:basic}, so we can apply Proposition \ref{prop:half-var} to this restricted flow, obtaining
\[
h_\nu(\FFF) + \int \ph \,d\nu
\leq \lim_{\zeta\to 0} \lim_{t\to\infty} \frac 1t \log \Lambda(\ph,t,\zeta,B_t \cap Y) \leq P^*,
\]
where the last inequality uses Proposition \ref{prop:transferSPR}. Using the ergodic decomposition, it follows that every $\nu\in \MF^\ph$ satisfying $\nu(A^\delta)=0$ must have $h_\nu(\FFF) + \int\ph\,d\nu \leq P^*$, which proves the corollary.
\end{proof}

\begin{corollary}\label{cor:transferSPR}
Suppose that $X,\FFF,\DDD,\ph$ satisfy Conditions \ref{cond:basic} and \ref{cond:S}. 
Let $A,A' \in \RRR$ be such that $A^\delta \subset A'$ for some $\delta>0$.
Suppose that $\ph$ has the Bowen property on $A'$ at scale $\delta$ and is $P^*$-SPR on $A$.
Then 
$\ph$ is $P^*$-SPR for $A'$ at every scale $\zeta>0$ and every distance $L>0$.
\end{corollary}
\begin{proof}
Fix $\zeta,L>0$. Writing $B := (A')^L$ and $Y := (A')^*_{t,L}$, we see that $Y \subset B_t$ and that
\[
\Lambda^*(\ph,t,\zeta,A',L)
= \Lambda(\ph,t,\zeta,Y) \leq C_1 e^{tP'}
\]
by Proposition \ref{prop:transferSPR}, 
so $\ph$ is $P^*$-SPR for $A'$ at scale $\zeta$ and distance $L$.
\end{proof}

Now we can use Corollary \ref{cor:transferSPR} to guarantee that the SPR property does not depend on our particular choice of $L$.

\begin{lemma} \label{lem:SPRnoL}
Suppose that $X,\FFF,\DDD,\ph$ satisfy Conditions \ref{cond:basic} and \ref{cond:S}, and that for every $A\in \RRR$, there exists a scale $\eps = \eps(A)>0$ at which $\ph$ has the Bowen property.
If $\varphi$ is $P^*$-SPR for $A\in \RRR$, then $\ph$ is $P^*$-SPR for $A$ at any scale $\zeta>0$ and any distance $L>0$.
\end{lemma}

\begin{proof}
By monotonicity of $P^*(\ph,\zeta,A,L)$ in $\zeta$ and $L$, it suffices to consider the case when $L_0 \leq 1$ and $\zeta \leq \rmet(A^1)$.

Let $\eta = \zeta/3$.
Since $\ph$ is $P^*$-SPR for $A$, there is $L_0 > 0$ such that $P^*(\ph,\eta,A,L_0) < P^*$. Write $\delta := L_0/2$, and observe that by Corollary \ref{cor:transferSPR}, we also have $P^*(\ph,\eta,A^\delta,L) < P^*$ for every $L>0$. Fixing $P' < P^*$ sufficiently close to $P^*$, we see that for every $L>0$, there is $C>0$ such that for all $T\geq 0$, we have
\begin{equation}\label{eqn:L*leq}
\Lambda^*(\ph,T,\eta,A,L_0) \leq C e^{TP'}
\quad\text{and}\quad
\Lambda^*(\ph,T,\eta,A^\delta,L) \leq C e^{TP'}.
\end{equation}
We will prove that $P^*(\ph,\eta,A,L) \leq P'$. Since the left-hand side  is nondecreasing in $L$, it suffices to consider $L>L_0$. Fixing $L>L_0$ and $t>0$, we work towards an upper bound on $\Lambda^*(\ph,t,\eta,A,L)$.
To this end, consider the sets
\[
Z := \{ x \in A_{t,L}^* : f_{[0,t]}(x) \cap A^{L_0} \neq \emptyset\}
\quad\text{and}\quad
Z' := A_{t,L}^* \setminus Z \subset (A^\delta)_{t,(L-\delta)}^*.
\]
We will decompose $Z$ in a manner similar to \eqref{eqn:Zij}.
By the uniform speed property in Lemma \ref{lem:bdd-excursions}, there is $N\in \NN$ such that $r := t/N > 0$ has the following property:
\begin{equation}\label{eqn:slow-L0}
d(f_s(x),x)<\delta \quad \text{for all }s\in [-r,r] \text{ and }x\in A^{L_0}.
\end{equation}
Given $x\in Z$, let $I(x) := \{ k \in \{1,\dots, N\} : f_{kr}(x) \in A^{L_0} \}$; this is nonempty by the definition of $Z$.
For each $i,j \in \{1,\dots, N\}$, let
\[
Z_{i,j} := \{ x\in A_{t,L}^* :
\min I(x) = i \text{ and } \max I(x) = j \}.
\]
Given any $x\in Z_{i,j}$, we see from \eqref{eqn:slow-L0} that $f_{[0,ir]}(x)$ and $f_{[jr,t]}(x)$ are both disjoint from $A^{\delta}$, and we conclude that
\[
Z_{i,j} \subset (A^\delta)_{ir,(L-\delta)}^*,
\qquad
f_{ir}(Z_{i,j}) \subset A_{(j-i)r,L_0}^*,
\qquad
f_{jr}(Z_{i,j}) \subset (A^\delta)_{(t-jr),(L-\delta)}^*.
\]
By \eqref{eqn:L*leq}, these inclusions imply that
\begin{gather*}
\Lambda(\ph,ir,\eta,Z_{i,j}) \leq
\Lambda^*(\ph,ir,\eta,A^\delta,L-\delta) \leq C e^{ir P'}, \\
\Lambda(\ph,(j-i)r,\eta,f_{ir}(Z_{i,j}))
\leq \Lambda^*(\ph,(j-i)r,\eta,A,L_0) \leq C e^{(j-i)rP'}, \\
\Lambda(\ph,t-jr,\eta,f_{jr}(Z_{i,j})) \leq
\Lambda^*(\ph,t-jr,\eta,A^\delta,L-\delta) \leq C e^{(t-jr)P'}.
\end{gather*}
Writing $Q$ for the bound for the Bowen property on $A^L$, these estimates together with Lemma \ref{lem:submult} (which applies since $\zeta \leq \rmet(A^1)$) give
\begin{equation}\label{eqn:Lij-leq}
\Lambda(\ph,t,\zeta,Z_{i,j}) \leq
e^{3Q} (C e^{irP'}) (C e^{(j-i)rP'}) (C e^{(t-jr)P'}) = e^{3Q} C^3 e^{tP'}.
\end{equation}
Recalling that $A_{t,L}^* = Z' \cup Z \subset (A^\delta)_{t,L-\delta}^* \cup \bigcup_{i=1}^N \bigcup_{j=1}^N Z_{i,j}$, we conclude that
\[
\Lambda^*(\ph,t,\zeta,A,L) \leq (C + N^2 e^{3Q} C^3) e^{tP'}.
\]
Since $N$ grows linearly with $t$, this completes the proof.
\end{proof}

Fix $L^* = L^*(A) := \inf \{ L>0 : f_{[-1,1]}(A) \subset A^L \}$, observing that this is finite by the bounded excursion property in Lemma \ref{lem:bdd-excursions}. By Lemma \ref{lem:SPRnoL}, if $\ph$ is $P^*$-SPR for $A\in \RRR$, then $\ph$ is $P^*$-SPR for $A$ at any scale $\zeta>0$ and at distance $L = L^*$, so from now on we simply write
\begin{equation} \label{eqfirstreturnsymbols}
    A^\ast_{T} = A^\ast_{T, {L^*}}
\quad\text{and}\quad
\Lambda^\ast(\ph,T,\zeta, A_T):=\Lambda^\ast(\ph,T,\zeta, L^*, A_T).
\end{equation}

\subsection{Long returns are exponentially rare}\label{sec:exp-rare}

Roughly speaking, the SPR property says that orbit segments making a first return to $A$ are exponentially rare in the sense that the corresponding partition sum has an exponential growth rate smaller than $P$. In this section, we prove that the same conclusion is true for orbit segments with a ``uniformly positive frequency of long returns''; see Theorem \ref{thm:exp-rare}. This fact, which can be thought of as a sort of large deviations result, plays a crucial role in obtaining both the submultiplicativity required for lower counting bounds (\S\ref{sec:lower-counting}) and the tightness required to prove existence of Misiurewicz measures (\S\ref{sec:Bow-Mis-exists}).

Given $A\in \RRR$ and $\Delta>0$, define for each $T>0$ and $x\in A_T$ the quantity
\[
W_\Delta^T(x) := \Leb \{ t\in [0,T] : f_t(x)\notin \AO[\Delta] \},
\]
which records the amount of time the orbit segment corresponding to $(x,T)$ spends in time intervals longer than the threshold $\Delta$ on which it ``wanders'' away from $A$.
Given $\gamma \in (0,1)$, consider the sets
\begin{equation}\label{eqn:ATG}
A_T^\GGG := \{ x\in A_T : W_\Delta^T(x) < \gamma T \}
\quad\text{and}\quad
A_T^\BBB := \{ x\in A_T : W_\Delta^T(x) \geq \gamma T \},
\end{equation}
which we think of as consisting of good and bad points, respectively, according to how much time their orbits spend in long wandering intervals.

\begin{theorem}\label{thm:exp-rare}
Let $X,\FFF,\DDD,\ph$ satisfy Conditions \ref{cond:basic}, \ref{cond:S}, and \ref{cond:B}. 
Suppose that $A\in \RRR$ and $P^* < P$ are such that $\ph$ is $P^*$-SPR on $A$, and that $\eps \in (0,\rmet(A^{L^*})]$ is such that $P^A(\ph, \eps) = P$.
For every $\zeta \in (0,\eps]$, every $\gamma \in (0,1)$, and every $\beta \in (0,\gamma(P-P^*))$, there exists $\Delta>0$ and $C_B > 0$ such that
the set $A_T^\BBB = A_T^\BBB(\Delta,\gamma)$ defined in \eqref{eqn:ATG} satisfies
\begin{equation}\label{eqn:exp-rare}
\Lambda(\ph,T,\zeta,A_T^\BBB) \leq C_B e^{-\beta T} e^{PT}
\quad\text{for every } T>0.
\end{equation}
\end{theorem}

In \S\ref{sec:wandering-sched}, we describe how to choose $\Delta$, and set up the bookkeeping notation that we will use in \S\ref{sec:long-returns} to prove \eqref{eqn:exp-rare}.

\subsubsection{Wandering schedules}\label{sec:wandering-sched}

We will prove Theorem \ref{thm:exp-rare} via the time-$1$ map $F = f_1$. First we describe how to choose $\Delta$.
Let $Q$ be given by the Bowen property on $A^{L^*}$, and let $C_U = C_U(A,\zeta/2)>0$ be given by \eqref{eqn:CU}.
By Lemma \ref{lem:SPRnoL}, there exists $C_F = C_F(A,\zeta/2)>0$ such that for every $T>0$, we have
\begin{equation}\label{eqn:CF}
\Lambda^*(\ph,T,\zeta,A_T) \leq C_F e^{TP^*},
\end{equation}
where we recall from \eqref{eqfirstreturnsymbols} that $A_T^* = A_{T,L^*}^*$ and $\AO \subset A^{L^*}$.

Let $H(p) = -p\log p - (1-p)\log (1-p)$ be the bipartite entropy function.
Fix $\gamma,\beta$ as in the statement of Theorem \ref{thm:exp-rare}, and 
choose $\Delta\in \NN$ sufficiently large that 
\begin{equation}\label{eqn:choose-Delta}
H(\Delta^{-1}) + \Delta^{-1}(2Q + \log C_U C_F) < \gamma (P-P^*) - \beta.
\end{equation}
Given $x\in A$, let
\begin{equation}\label{eqn:Rx}
R(x) := \{ n \in \NN_0 : F^n(x) \in \AO \}
\end{equation}
denote the set of times at which the $F$-orbit of $x$ \emph{remains} in or \emph{returns} to $\AO$. 
Let
\[
D(x) := \{ n \in R(x) : (n,n+\Delta) \cap R(x) = \emptyset \}
\]
be the set of times at which the $F$-orbit of $x$ is \emph{departing} $\AO$ for a ``long wandering interval'', meaning that it will spend a duration at least $\Delta$ outside of $\AO$. Observe that the flow-orbit remains outside $A$ during this time.
Similarly, let
\[
E(x) := \{ n \in R(x) : (n-\Delta,n) \cap R(x) = \emptyset \}
\]
be the set of times at which the $F$-orbit of $x$ is \emph{entering} $\AO$ after a long wandering interval; see Figure \ref{fig:RDE}.

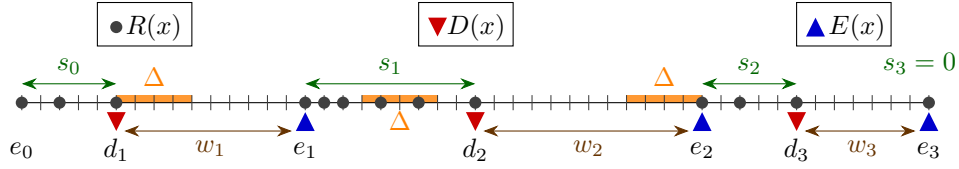
\begin{figure}[htbp]
\begin{tikzpicture}[scale=0.25,
slen/.style={{Stealth[length=2mm]}-{Stealth[length=2mm]},color=green!40!black},
wlen/.style={{Stealth[length=2mm,sep=3pt]}-{Stealth[length=2mm,sep=5pt]},color=orange!40!black}]
\def\h{.4};
\def\r{.5};
\fill[orange!80] (5,0) rectangle (9,\h);
\fill[orange!80] (18,0) rectangle (22,\h);
\fill[orange!80] (32,0) rectangle (36,\h);
\node[orange,above] at (7,\h){$\Delta$};
\node[orange,below] at (20,0){$\Delta$};
\node[orange,above] at (34,\h){$\Delta$};
\draw (0,0)--(48,0);
\foreach \x in {0,1,...,48}
	{\draw[darkgray] (\x,-\h)--(\x,\h);}
\foreach \x in {0,2,5,15,16,17,19,21,24,36,38,41,48}
	{\fill[darkgray] (\x,0) circle (.8*\h);}
\foreach \x in {5,24,41}
	{\fill[red!80!black] (\x-\r,-\r) -- (\x+\r,-\r) -- (\x,-3*\r) -- cycle;}
\foreach \x in {15,36,48}
	{\fill[blue!80!black] (\x-\r,-3*\r) -- (\x+\r,-3*\r) -- (\x,-\r) -- cycle;}
\node at (5,-5*\r) {$d_1$};
\node at (24,-5*\r) {$d_2$};
\node at (41,-5*\r) {$d_3$};
\node at (0,-5*\r) {$e_0$};
\node at (15,-5*\r) {$e_1$};
\node at (36,-5*\r) {$e_2$};
\node at (48,-5*\r) {$e_3$};
\draw[slen] (0,2*\r)--(5,2*\r) node[pos=0.5,above]{$s_0$};
\draw[slen] (15,2*\r)--(24,2*\r) node[pos=0.5,above]{$s_1$};
\draw[slen] (36,2*\r)--(41,2*\r) node[pos=0.5,above]{$s_2$};
\draw[color=green!40!black] (47.5,2*\r) node[above]{$s_3 = 0$};
\draw[wlen] (5,-3*\r)--(15,-3*\r) node[pos=0.5,below]{$w_1$};
\draw[wlen] (24,-3*\r)--(36,-3*\r) node[pos=0.5,below]{$w_2$};
\draw[wlen] (41,-3*\r)--(48,-3*\r) node[pos=0.5,below]{$w_3$};
\fill[darkgray] (5,4) circle (.8*\h) node[right,black] {$R(x)$};
\draw (4,3) rectangle (9,5.3);
\fill[red!80!black] (22-\r,4+\r) -- (22+\r,4+\r) -- (22,4-\r) -- cycle;
\draw (22,4) node[right] {$D(x)$};
\draw (21,3) rectangle (26,5.3);
\fill[blue!80!black] (42-\r,4-\r) -- (42+\r,4-\r) -- (42,4+\r) -- cycle;
\draw (42.2,4) node[right] {$E(x)$};
\draw (41,3) rectangle (46,5.3);
\end{tikzpicture}
\caption{Defining the wandering schedule. Three $\Delta$-intervals are shaded to illustrate the difference between \emph{wandering} intervals (lengths $w_i$) and intervals that \emph{stay} near $\AO$ (lengths $s_i$).}
\label{fig:RDE}
\end{figure}

Fix an integer $T \gg \Delta$. Like $A$ and $\Delta$, we will suppress $T$ from much of the notation in this part of the arguments.

The \emph{wandering schedule} of a point $x\in A_T$ through time $T$ will be a pair 
$(\bs,\bw)$, where $\bs = (s_0,\dots, s_\ell)$ and $\bw = (w_1,\dots, w_\ell)$
should be interpreted as the lengths of intervals that \emph{stay} near $\AO$ and \emph{wander} away from $\AO$, respectively.

More precisely,
%let $\ell = \ell(x) := \# D(x) \cap [0,T]$, and
enumerate the elements of $D(x) \cap [0,T)$ in increasing order as $d_1,\dots, d_\ell$, and do the same for $E(x) \cap (0,T]$, so that we have
\begin{equation}\label{eqn:ed}
e_0 := 0 \leq d_1 < e_1 \leq d_2 < e_2 \leq \cdots \leq d_\ell < e_\ell \leq T =: d_{\ell+1},
\end{equation}
and then let
\begin{equation}\label{eqn:sw}
s_i = d_{i+1} - e_i \text{ for } 0\leq i\leq \ell
\quad\text{and}\quad
w_i = e_i - d_i \text{ for } 1\leq i \leq \ell.
\end{equation}
Observe that $s_i$ can take any value in $\NN_0$, but we must have $w_i \geq \Delta$. 

Given any $x\in A_T$ and any $t\in [0,T]$ such that $f_t(x) \notin \AO[\Delta]$, there exists $i\in \{1,\dots, \ell\}$ such that $d_i + \Delta - 1 < t < e_i - (\Delta - 1)$. Thus
\[
W_\Delta^T(x) \leq \sum_{i=1}^\ell \big(w_i - 2(\Delta-1)\big) \leq \sum_{i=1}^\ell w_i =: W(x),
\]
and we see that for every $x\in A_T^\BBB$, the \emph{wandering duration} $W(x)$ satisfies
\[
W(x) \geq W_\Delta^T(x) \geq \gamma T.
\]
We observe that if $w_i\geq \Delta$ for all $i$, then $\ell \leq W(x)/\Delta \leq T/\Delta$.
\subsubsection{Controlling long returns}\label{sec:long-returns}

This section is devoted to the proof of \eqref{eqn:exp-rare}.
With $A,\Delta,T$ fixed, consider for each $W \in \{0,\dots,T\}$ and $\ell \in \{1,\dots\lfloor W/\Delta \rfloor\}$ the set of points
\[
A_T^{W,\ell} := \{ x\in (A^+)_T : W(x) = W \text{ and } \ell(x) = \ell \}.
\]
Given such a pair $(W,\ell)$, let $\Sched(W,\ell)$ denote the set of all possible wandering schedules that could produce it: that is,
\[
\Sched(W,\ell) = \bigg\{ (\bs,\bw) \in \NN_0^{\ell+1}\times \NN^\ell:
\sum_{i=1}^\ell w_i = W,\
\sum_{i=0}^\ell s_i = T-W, \text{ and }
w_i\geq \Delta\text{ for all } i\bigg\}.
\]
Given $(\bs,\bw) \in \Sched(W,\ell)$, let $A_T^{\bs,\bw}$ be the set of points in $(A^+)_T$ that have this wandering schedule. Thus we have
\begin{equation}\label{eqn:ATB}
A_T^\BBB \subset \bigcup_{W=\lceil \gamma T\rceil}^T 
\bigcup_{\ell=1}^{\lfloor W/\Delta\rfloor} A_T^{W,\ell}
\quad\text{and}\quad
A_T^{W,\ell} = \bigcup_{(\bs,\bw) \in \Sched(W,\ell)} A_T^{\bs,\bw}.
\end{equation}
Now we can begin estimating partition sums.
Fix $W \geq \gamma T$ and $\ell \leq W/\Delta$.
Given $(\bs,\bw) \in \Sched(W,\ell)$, let $d_i,e_i$ denote the associated departure and entry times as in \eqref{eqn:ed} and \eqref{eqn:sw}. Then we have
\begin{align*}
f_{e_i}(A_T^{\bs,\bw}) &\subset A_{s_i} \text{ for every } 0\leq i\leq \ell, \\
f_{d_i}(A_T^{\bs,\bw}) &\subset (\AO)^*_{F,w_i} \subset A^*_{w_i} \text{ for every } 1\leq i\leq \ell.
\end{align*}
With $\bt = (e_0,d_1,e_1,d_2,\dots,e_\ell)$ and $\rho = 2\zeta$, this lets us apply the submultiplicativity Lemma \ref{lem:submult} (since $\zeta \leq \rmet(A^{L^*})$, obtaining
\[
\Lambda(\ph,T,\zeta,A_T^{\bs,\bw})
\leq e^{(2\ell+1)Q} \prod_{i=0}^\ell \Lambda(\ph,s_i,\zeta/2,A_{s_i})
\prod_{i=1}^\ell \Lambda(\ph,w_i,\zeta/2,A^*_{w_i}).
\]
By the upper bounds in \eqref{eqn:CU} and \eqref{eqn:CF} and the assumption that $P^A(\ph) = P$,
we obtain
\begin{equation}\label{eqn:sw-leq}
\begin{aligned}
\Lambda(\ph,T,\zeta,A_T^{\bs,\bw})
&\leq e^{(2\ell+1)Q} C_U^{\ell+1} e^{\sum_{i=0}^\ell s_i P}
C_F^\ell e^{\sum_{i=1}^\ell w_i P^*} \\
&= e^Q C_U e^{\ell(2Q+\log C_U C_F)} e^{(T-W) P} e^{WP^*}.
\end{aligned}
\end{equation}
To simplify the remaining computations, we write
\[
C_4 := e^Q C_U,
\quad
\theta := 2Q + \log C_U C_F,
\quad\text{and}\quad
\alpha := P-P^*.
\]
Then \eqref{eqn:sw-leq} and the estimate $\ell \leq T\Delta^{-1}$ give
\begin{equation}\label{eqn:sw-leq-2}
\Lambda(\ph,T,\zeta,A_T^{\bs,\bw})
\leq C_4 e^{\Delta^{-1}\theta T} e^{TP} e^{-W\alpha}.
\end{equation}
Summing over the union in the second half of \eqref{eqn:ATB} and using \eqref{eqn:sw-leq-2} together with the bound $W\alpha \geq \alpha\gamma T$, we get
\begin{equation}\label{eqn:Wl-leq}
\Lambda(\ph,T,\zeta,A_T^{W,\ell}) \leq (\#\Sched(W,\ell)) C_4 e^{(\Delta^{-1}\theta - \alpha\gamma)T} e^{TP}.
\end{equation}
To bound $\#\Sched(W,\ell)$, we use the following.

\begin{lemma}\label{lem:c&o}
Given $N,\ell \in \NN$ and $m\geq 0$, let $Z(N,\ell,m)$ denote the set of $\ell$-tuples $\mathbf{n} = (n_1,\dots, n_\ell) \in \NN_0^\ell$ with the property that $\sum_{i=1}^\ell n_i = N$ and $n_i \geq m$ for each $i$. Then $\#Z(N,\ell,m) = \binom{N-\ell (m -1) -1}{\ell-1}$.
\end{lemma}
\begin{proof}
When $m=1$, there is a bijection between $Z(N,\ell,1)$ and the collection of $(\ell-1)$-element subsets of $\{1,\dots, N-1\}$, given by associating $\mathbf{n} = (n_1,\dots, n_\ell)$ to the set $\{ n_1, n_1 + n_2, + \cdots, n_1 + \cdots + n_{\ell-1} \}$. Thus $\# Z(N,\ell,1) = \binom{N-1}{\ell-1}$.
When $m\neq 1$ (this includes the case $m=0$), there is a bijection between $Z(N,\ell,m)$ and $Z(N-\ell(m-1), \ell, 1)$, by associating $\mathbf{n}$ to the $\ell$-tuple $(n_1 - (m-1), \dots, n_\ell - (m-1))$. This proves the lemma.
\end{proof}

Since $\Sched(W,\ell) = Z(T-W,\ell+1,0) \times Z(W,\ell,\Delta)$, we see from Lemma \ref{lem:c&o} that
\begin{equation}\label{eqn:card-S}
\#\Sched(W,\ell)
= \binom{T-W + \ell}{\ell} \binom{W-\ell(\Delta-1)-1}{\ell-1}
\leq \binom{T}{\ell}^2,
\end{equation}
where the inequality uses the fact that $\ell \leq W/\Delta \leq W$ so $T-W+\ell\leq T$, together with monotonicity properties of binomial coefficients.

Standard estimates (for example, see \cite[Lemma 5.11]{vC18}) show that
\begin{equation}\label{eqn:binom-leq}
\binom{T}{\ell} \leq T e^{H(\ell/T) T} \leq T e^{H(\Delta^{-1}) T},
\end{equation}
where the second inequality uses monotonicity of the entropy function $H$ on $(0,\frac 12]$ and the fact that $\ell/T \in (0,\Delta^{-1}]$.
Combining this with \eqref{eqn:Wl-leq} and \eqref{eqn:card-S}, 
we get
\[
\Lambda(\ph,T,\zeta,A_T^{W,\ell}) \leq
T^2 e^{2H(\Delta^{-1})T} C_4 e^{(\Delta^{-1}\theta - \alpha\gamma)T} e^{TP}.
\]
With this in hand, we can use the first part of \eqref{eqn:ATB} to deduce that
\begin{equation}\label{eqn:ATB-leq-0}
\Lambda(\ph,T,\zeta,A_T^\BBB) \leq  C_4 T^4 e^{(2H(\Delta^{-1}) + \Delta^{-1}\theta - \alpha\gamma)T} e^{TP}.
\end{equation}
By the choice of $\Delta$ in \eqref{eqn:choose-Delta}, we have $2H(\Delta^{-1}) + \Delta^{-1}\theta - \alpha\gamma < -\beta$, so \eqref{eqn:ATB-leq-0} implies \eqref{eqn:exp-rare}
and completes the proof of Theorem \ref{thm:exp-rare}.

\subsection{Frequent returns give lower counting bounds}\label{sec:lower-counting}

In this section, we prove the following:

\begin{theorem}\label{thm:counting2}
Let $X,\FFF,\DDD,\ph$ satisfy Conditions \ref{cond:basic}, \ref{cond:S}, and \ref{cond:B}. Suppose that $\ph$ is SPR w.r.t.\ $A\in \RRR$, and that
$\eps \in (0,\rmet(A^{L^*})]$ is such that  $P^{\AO}(\ph, \eps) = P$.
Then for every $\onezeta \in (0,\fraceps]$,  there are constants $C_L = C_L(A,\onezeta)$ and $T_0=T_0(A,\onezeta)$ such that for all $T>T_0$ we have
\begin{equation}\label{eqn:unif-counting}
\Lambda(\ph,T,\onezeta,A_T) \geq C_L e^{TP}.
\end{equation}
Consequently, given $\zeta \in (0,\halffraceps]$ and any maximal $(T,\zeta)$-separated set $E \subset A_T$, Lemma \ref{lem:all-sep} gives
\begin{equation}\label{eqn:unif-counting-E}
\sum_{x\in E} e^{\Phi(x,T)} \geq C_L(A,2\zeta) e^{-Q} e^{TP(\ph)}.
\end{equation}
\end{theorem}

In the compact setting, a lower bound along the lines of \eqref{eqn:unif-counting} is obtained from submultiplicativity; recall the use of Lemma \ref{lem:submult} in the proof of Theorem \ref{thm:K-gap}. Since we work now in the non-compact setting and deal with partition sums constrained by $A$, the submultiplicativity provided by Lemma \ref{lem:submult} will only be helpful when we consider times at which all of the orbits being counted return to a fixed reference set. It is here that Theorem \ref{thm:exp-rare} plays a crucial role.

To prove Theorem \ref{thm:counting2}, fix $A\in \RRR$ and $\eps \in (0,\rmet(A^{L^*})]$ such that $\ph$ is SPR w.r.t.\ $A$, and $P^{\AO}(\ph,\eps) = P$.
Fix $\zeta \in (0,\eps/4]$, and apply Theorem \ref{thm:exp-rare} with $\gamma = \frac 13$ and any sufficiently small $\beta>0$ to conclude that there exists $\Delta,C_B>0$ such that
the sets $A_T^\BBB = A_T^\BBB(\Delta,\frac13)$ defined in \eqref{eqn:ATG} satisfy 
\begin{equation}\label{eqn:ATB-small}
\ulim_{T\to\infty} \frac 1T \log \Lambda(\ph,T,4\zeta,A_T^\BBB) < P.
\end{equation}
At the same time, Lemma \ref{lem:limsupislim} gives
\begin{equation}\label{eqn:AT-big}
\lim_{T\to\infty} \frac 1T \log \Lambda(\ph,T,4\zeta,A_T) = P.
\end{equation}
We have $A_T = A_T^\GGG \cup A_T^\BBB$, where as in \eqref{eqn:ATG},
\begin{equation}\label{eqn:ATG3}
A_T^\GGG = A_T^\GGG\Big(\Delta,\frac 13 \Big) := \Big\{ x\in A_T :  \Leb \{ t\in [0,T] : f_t(x)\notin \AO[\Delta] \} < \frac 13 T \Big\}.
\end{equation}
Thus we can deduce from \eqref{eqn:ATB-small} and \eqref{eqn:AT-big}  that
\begin{equation}\label{eqn:limG}
\lim_{T\to\infty} \frac 1T \log \Lambda(\ph,T,4\zeta,A_T^\GGG) = P.
\end{equation}
The bulk of the work in the proof will go into using \eqref{eqn:ATG3} and \eqref{eqn:limG} to show that there exists $C_\Delta>0$ such that 
\begin{equation}\label{eqn:C-Delta}
Z_k
:= \sup_{|t-k\Delta| \leq 4\Delta} \Lambda(\ph,t,2\zeta,(\AO)_t) \geq C_\Delta e^{k\Delta P}
\quad\text{for all } k\in \NN.
\end{equation}
Our strategy to do this is to cover $A_{nk\Delta}^\GGG$ by sets $A_{nk\Delta}^{\GGG,J,J'}$ consisting of points whose orbits lie in $\AO$ at a fixed set of times, giving a decomposition of the interval $[0,nk\Delta]$ with respect to which we can apply the submultiplicativity Lemma \ref{lem:submult}. 

To carry this out, we need some notation, which is illustrated in Figure \ref{fig:IJJ}. With $k\geq 5$ and $b\in \{1,\dots, k\}$ fixed, consider for each $a \in \{1,2,\dots,n-2\}$ the following interval of integers centered at $(b+ak)\Delta$:
\begin{equation}\label{eqn:Ia}
I_a = I_a^{k,b} := \NN \cap [(b+ak-2)\Delta, (b+ak+2)\Delta]
\subset \NN \cap (0,nk).
\end{equation}
The last inclusion in \eqref{eqn:Ia} relies on the fact that $1\leq a\leq n-1$; in particular, we do not include the indices $a=0$ and $a=n-1$. Let
\[
\mathbb{J}_b := \Big\{ J \subset \bigcup_{a=1}^{n-2} I_a : \#J = \frac n4 \text{ and }
\#J\cap I_a \leq 1 \text{ for all } a\in \NN \Big\};
\]
that is, $J\subset \NN \cap (0,nk\Delta)$ is an element of $\mathbb{J}_b$ if and only if it can be indexed as $J = \{ j_1 < j_2 < \cdots < j_{n/4} \}$ such that $j_i \in I_{a_i}^{k,b}$ for each $i$, where $a_1 < \cdots < a_{n/4}$. Given $J\in \mathbb{J}_b$, denote the corresponding set $\{a_i : 1\leq i\leq n/4\}$ by
\[
\mathbf{a}(J) := \{ a\in \{1,\dots, n-2\} : J \cap I_a \neq \emptyset \}.
\]
Now we consider the following set of pairs $(J,J')$:
\begin{align*}
\mathbb{J}_b^* := \{(J,J') \in \mathbb{J}_b^2 : 
\mathbf{a}(J') &= \{ a + 1 : a\in \mathbf{a}(J) \} 
\text{ and}\\
&
J \cap I_a = J' \cap I_a
\text{ for all } 
a\in \mathbf{a}(J) \cap \mathbf{a}(J') \}.
\end{align*}
We will use these pairs to index sets of points in $A_{nk\Delta}^\GGG$, by consider points whose orbits return to $\AO$ at all times in $J\cup J'$. That is, to each $(J,J') \in \mathbb{J}_b^*$, we associate the set
\[
A_{nk\Delta}^{\GGG,J,J'}
:= \{ x\in A_{nk\Delta} : f_j(x) \in \AO \text{ for all } j\in J \cup J' \}.
\]
Given $(J,J') \in \mathbb{J}_b^*$, we will write $j_i$ and $j'_i$ for the elements of $J$ and $J'$, respectively, enumerated in increasing order, so that
\[
0 < j_1
< j'_1 \leq j_2 < j_2' \leq j_3 < \cdots \leq j_{n/4} < j'_{n/4} 
< nk\Delta.
\]
For notational convenience, we write $j'_0 = 0$ and $j_{(n/4)+1} = nk\Delta$.
Then writing
\begin{equation}\label{eqn:lili}
\ell_i := j'_{i} - j_i
\quad\text{and}\quad
\ell_i' := j_{i+1} - j'_{i},
\end{equation}
we see that 
$\ell'_0 + \ell_1 + \ell'_2 + \cdots + \ell_{n/4} + \ell'_{n/4} = nk\Delta$,
and that for every $i$, we have
$\ell_i \in [(k-4)\Delta,(k+4)\Delta]$.
This in turn implies that
\begin{equation}\label{eqn:sum-ell'}
\sum_{i=1}^{n/4} \ell'_{i-1}
= nk\Delta - \sum_{i=1}^{n/4} \ell_i
\leq nk\Delta - \frac n4(k-4)\Delta
\leq \frac 34 nk\Delta + 4n\Delta
\end{equation}

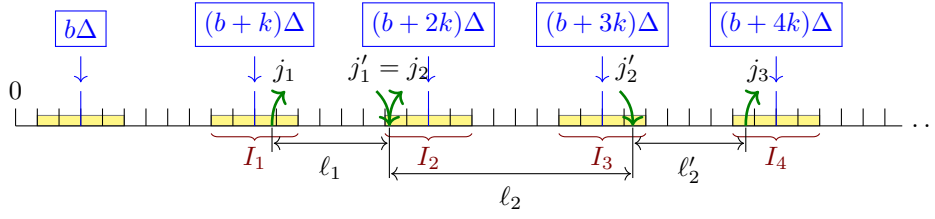
\begin{figure}[htbp]
\begin{tikzpicture}[scale=1.15]
\def\h{0.2}
\def\b{0.75}
\foreach \x in {0,2,4,6,8}
	{ \draw[black!50,fill=yellow!60] (\x+.25,0) rectangle (\x+1.25,.12);}
\draw (0,0) -- (10.2,0) node[right] {$\dots$};
\foreach \x in {0, 0.25, ..., 10}
	{ \draw (\x,0)--(\x,\h); }
\draw (0,\h) node[above] {$0$};
\draw[blue,->] (0.75,0)-- ++(0,2*\h)
++(0,2*\h) node[above]{\fbox{$b\Delta$}} --+(0,-1.5*\h);
\draw[blue,->] (2.75,0)-- ++(0,2*\h) 
++(0,2*\h) node[above]{\fbox{$(b+k)\Delta$}} --+(0,-1.5*\h);
\draw[blue,->] (4.75,0)-- ++(0,2*\h) 
++(0,2*\h) node[above]{\fbox{$(b+2k)\Delta$}} --+(0,-1.5*\h);
\draw[blue,->] (6.75,0)-- ++(0,2*\h) 
++(0,2*\h) node[above]{\fbox{$(b+3k)\Delta$}} --+(0,-1.5*\h);
\draw[blue,->] (8.75,0)-- ++(0,2*\h) 
++(0,2*\h) node[above]{\fbox{$(b+4k)\Delta$}} --+(0,-1.5*\h);
\draw[red!50!black,decoration={brace, mirror, raise=0.1cm},decorate]
(2.25,0)--(3.25,0) node[pos=0.5,anchor=north,yshift=-0.15cm] {$I_1$};
\draw[red!50!black,decoration={brace, mirror, raise=0.1cm},decorate]
(4.25,0)--(5.25,0) node[pos=0.5,anchor=north,yshift=-0.15cm] {$I_2$};
\draw[red!50!black,decoration={brace, mirror, raise=0.1cm},decorate]
(6.25,0)--(7.25,0) node[pos=0.5,anchor=north,yshift=-0.15cm] {$I_3$};
\draw[red!50!black,decoration={brace, mirror, raise=0.1cm},decorate]
(8.25,0)--(9.25,0) node[pos=0.5,anchor=north,yshift=-0.15cm] {$I_4$};
\foreach \x in {2.95, 4.3, 8.4}
	{\draw[->,green!50!black,line width=1] (\x,0) 
	to[out=90,in=225] ++(.15,2*\h);}
\draw
(3.1,2*\h) node[above] {$j_1$}
%(4.45,2*\h) node[above] {$j_2$}
(8.55,2*\h) node[above] {$j_3$};
\foreach \x in {4.3,7.1}
	{\draw[<-,green!50!black,line width=1] (\x,0) 
	to[out=90,in=315] ++(-.15,2*\h);}
\draw %(4.15,2*\h) node[above,xshift=-.1cm] {$j_1'$}
(6.95,2*\h) node[above,xshift=.1cm] {$j_2'$};
\draw (4.3,2*\h) node[above]{$j_1' = j_2$};
\draw %[orange!60!black]
(2.95,-.02) -- (2.95,-.3)
(4.3,-.02) -- (4.3,-.6);
\draw%[orange!60!black,<->] 
[<->] (2.97,-.2) -- (4.28,-.2)
node[pos=0.5,anchor=north,yshift=-0.05cm] {$\ell_1$};
\draw (7.1,-.02)--(7.1,-.6);
\draw[<->] (4.32,-.55) -- (7.08, -.55)
node[pos=0.5,anchor=north,yshift=-0.05cm] {$\ell_2$};
\draw (8.4,-.02) -- (8.4,-.3);
\draw[<->] (7.12,-.2)--(8.38,-.2)
node[pos=0.5,anchor=north,yshift=-0.05cm] {$\ell_2'$};
\end{tikzpicture}
\caption{Bookkeeping in the proof of Theorem \ref{thm:counting2}. Each small interval has length $\Delta$, and here we choose $k=8$ and $b=3$. Curved arrows represent $J$ and $J'$. 
We have $\mathbf{a}(J) = \{1,2,4,\dots\}$ and $\mathbf{a}(J') = \{2,3,5,\dots\}$. Since $2 \in \mathbf{a}(J) \cap \mathbf{a}(J')$, we have $j_1' = j_2$, so $\ell_2' = 0$.}
\label{fig:IJJ}
\end{figure}

\begin{lemma}\label{lem:all-in-JJ'}
For every $n\geq 24$ and $k\geq 4$, we have
\begin{equation}\label{eqn:all-in-JJ'}
A_{nk\Delta}^\GGG \subset \bigcup_{b=1}^k \bigcup_{(J,J') \in \mathbb{J}_b^*} A_{nk\Delta}^{\GGG,J,J'}.
\end{equation}
\end{lemma}
\begin{proof}
Given $x\in A_{nk\Delta}^\GGG$, it follows from \eqref{eqn:ATG3} that
\begin{equation}\label{eqn:23nk}
\Leb \{ t\in [0,nk\Delta] : f_t(x) \in \AO[\Delta]\} \geq \frac 23 nk\Delta.
\end{equation}
Consider the collection of intervals of the form $((i-1)\Delta, i\Delta]$ for $i \in \{1,\dots, nk\}$; by \eqref{eqn:23nk}, at least $\frac 23nk$ of these intervals contain a time $t$ with $f_t(x) \in \AO[\Delta]$. For each such interval, we have $f_{i\Delta} \in \AO[2\Delta]$ by definition, and thus the set
\[
I(x) := \{ i\in \NN \cap [1,nk] : f_{i\Delta}(x) \in \AO[2\Delta] \}
\]
satisfies $\#I(x) \geq \frac 23 nk$. 
We want a lower bound on the cardinality of the set
\[
I_k(x) := \{ i\in I(x) : i+k \in I(x) \}.
\]
To this end, write $I^c(x) := (\NN\cap [1,nk]) \setminus I(x)$ and $I_k^c(x) := (\NN \cap [1,nk]) \setminus I_k(x)$. Now observe that
\begin{align*}
\#I_k^c(x) &= \#\{ i \in \NN \cap [1,nk]
: i\notin I(x) \text{ or } i+k\notin I(x) \} \\
&\leq 2\# I^c(x) + k \leq \frac 23 nk + k.
\end{align*}
Provided $n\geq 36$, it follows that
\[
\#I_k(x) \geq \frac 13 nk - k \geq \Big(\frac14 n + 2 \Big)k.
\]
Working modulo $k$, we conclude that there exists $b\in \{1,\dots, k\}$ such that
\[
\#\big( I_k(x) \cap (b+k\ZZ) \big) \geq \frac 14n + 2.
\]
Since each $i\in I_k(x) \cap (b+k\ZZ)$ can be written as $b + ka$ for some $a\in \{0,\dots, n-1\}$, this implies that the set
\[
\mathbf{a}(x) := \big\{
a \in \{1,\dots, n-2\} : f_{(b+ka)\Delta}(x) \in \AO[2\Delta] \big\}
\]
satisfies $\#\mathbf{a}(x) \geq \frac 14 n$.
Enumerate the first $\frac 14n$ elements of $\mathbf{a}(x)$ as $a_1 < a_2 < \cdots < a_{n/4}$.
For each $i \in \{1,\dots, \frac 14 n\}$, 
there exists at least one $j\in I_{a_i}^{k,b}$ such that $f_j(x) \in \AO$, and at least one $j' \in I_{a+1}^{k,b}$ such that $f_{j'}(x) \in \AO$. Let $j_i$ be the smallest such $j$, and $j'_i$ the smallest such $j'$. Then writing $J := \{j_i : 1\leq i\leq \frac 14n\}$ and $J' := \{j'_i : 1\leq i\leq \frac 14n\}$, we have $(J,J') \in \mathbb{J}^*$, and $x \in A_{nk\Delta}^{\GGG,J,J'}$, which completes the proof.
\end{proof}

Given $b\in \{1,\dots, k\}$ and $(J,J') \in \mathbb{J}_b^*$, since $4\zeta \leq \rmet(\AO)$, we can apply Lemma \ref{lem:submult} to $A_{nk\Delta}^{\GGG,J,J'}$ with the corresponding decomposition
$nk\Delta = \ell'_1 + \ell_1 + \cdots + \ell'_{n/4} + \ell_{n/4}$ to get
\begin{equation}\label{eqn:LGJJ}
\begin{aligned}
\Lambda(\ph,nk\Delta,4\zeta,A_{nk\Delta}^{\GGG,J,J'})
&\leq e^{nQ/2} \prod_{i=1}^{n/4}
\Lambda(\ph,\ell'_i,2\zeta,(\AO)_{\ell'_i})
\Lambda(\ph,\ell_i,2\zeta,(\AO)_{\ell_i}) \\
&\leq e^{nQ/2} 
\prod_{i=1}^{n/4}
C_U e^{\ell'_i P} Z_k 
\leq (e^Q C_U)^n 
e^{\frac 34 nk\Delta P}
e^{n\Delta}
Z_k^{\frac 14 n},
\end{aligned}
\end{equation}
where the second inequality uses the upper bound in \eqref{eqn:CU}, and the last inequality uses
\eqref{eqn:sum-ell'}.

Given $b\in \{1,\dots, k\}$
and $\mathbf{a} \subset \{1,\dots, n-1\}$ with $\#\mathbf{a} = n/4$, 
we have
\[
\{ J\in \mathbb{J}_b : \mathbf{a}(J) = \mathbf{a}\}
= \prod_{a\in \mathbf{a}} I_a.
\]
Since $\#I_a = 4\Delta+1 \leq 5\Delta$ and the number of choices of $\mathbf{a}$ is bounded above by $2^n$, we have
\[
\#\mathbb{J}_b
\leq 2^n \cdot (5\Delta)^{n/4}
\quad\Rightarrow\quad
\#\mathbb{J}_b^* \leq (20\Delta)^n.
\]
By Lemma \ref{lem:all-in-JJ'}, 
we can sum \eqref{eqn:LGJJ} over all choices of $b \in \{1,\dots, k\}$ and $(J,J') \in \mathbb{J}_b^*$ to obtain
\begin{equation}\label{eqn:ATG-leq}
\Lambda(\ph,nk\Delta,4\zeta,A_{nk\Delta}^\GGG)
\leq k (20\Delta)^n
(e^{Q\Delta} C_U)^n e^{\frac 34 nk\Delta P} Z_k^{\frac 14 n}.
\end{equation}
Taking logs and dividing by $nk\Delta$, we get
\[
\frac 1{nk\Delta}\log\Lambda(\ph,nk\Delta,4\zeta,A_{nk\Delta}^\GGG)
\leq \frac{\log k}{kn\Delta}
+ \frac{\log(20\Delta e^{Q\Delta} C_U)}{k\Delta}
+ \frac{\log Z_k}{4k\Delta}
+ \frac 34 P.
\]
By \eqref{eqn:limG},
%By Lemma \ref{lem:limsupislim} and the fact that $4\zeta < \eps$, 
the left-hand side converges to $P$ as $n\to\infty$.
Sending $n\to\infty$ along multiples of $4$, we get
\[
P \leq \frac{\log(20\Delta e^{Q\Delta} C_U)}{k\Delta}
+ \frac{\log Z_k}{4k\Delta}
+ \frac 34 P.
\]
Multiplying by $4 k\Delta$ and then subtracting $3k\Delta P$ from both sides gives
\[
k\Delta P \leq 4\log(20\Delta e^{Q\Delta} C_U) + \log Z_k,
\]
which implies \eqref{eqn:C-Delta}: writing $C_\Delta := e^{-4Q\Delta}(20\Delta C_U)^{-1}$, for every $k\geq 4$, there exists $t \in [(k-4)\Delta,(k+4)\Delta]$ such that
\begin{equation}\label{eqn:k-Delta}
\Lambda(\ph,t,2\zeta,(\AO)_t) \geq C_\Delta e^{k\Delta P}.
\end{equation}
To complete the proof of Theorem \ref{thm:counting2}, we use the supermultiplicativity bound in Lemma \ref{lem:spec-sum-Bow}. Take $\rho \in (0, \zeta/2]$ sufficiently small that $A^{-\rho}\neq\emptyset$, and write $\tau'$ for the corresponding transition time in the specification property on $\AO$.
Given any $T > 2\tau'$, we can take $k = \lfloor\frac{T-2\tau'}\Delta\rfloor - 4$ and then fix $t\in [(k-4)\Delta,(k+4)\Delta]$ such that \eqref{eqn:k-Delta} holds. 
Observe that
%\begin{equation}\label{eqn:k4Delta}
\[
%T-2\tau'-\Delta \leq (k+4)\Delta \leq T-2\tau',
T-2\tau' \in [(k+4)\Delta,(k+5)\Delta],
\]
%\end{equation}
which implies that
\begin{equation}\label{eqn:T-t}
%T-2\tau' - 9\Delta \leq t \leqT-2\tau',
T- 2\tau' - t \in [0, 9\Delta].
\end{equation}
%so in particular, $t \leq T-2\tau'$, 
Thus $r := (T-t)/2 \geq \tau'$, and we have
%and thus the quantity 
%$r = (T - t)/2$
%satisfies $r \in [\tau',\tau'+\frac 92\Delta]$.
%We obtain
\[
\Spec_\rho^r( (A^{-\rho},0), ((\AO)_{t},t), (A^{-\rho},0))
\subset A_T,
\]
so Lemma \ref{lem:spec-sum-Bow} 
and \eqref{eqn:k-Delta} give
\begin{equation}\label{eqn:T-and-t}
\Lambda(\ph,T,\zeta,A_T) 
\geq e^{-3(\Vt+Q)} \Lambda(\ph,t,2\zeta,(\AO)_{t}) \\
\geq e^{-3(\Vt+Q)} C_\Delta e^{t P}.
\end{equation}
From \eqref{eqn:T-t}, we see that
\[
e^{tP} \geq e^{(T-2\tau')P} e^{-9\Delta|P|};
\]
using this in \eqref{eqn:T-and-t} 
and writing $C_L = e^{-3(\Vt+Q)} C_\Delta e^{-2\tau'P} e^{-9\Delta|P|}$ proves the lower bound in \eqref{eqn:unif-counting}.

This completes the proof of Theorem \ref{thm:counting2}. Theorem \ref{thm:uniformgeneral} is an immediate consequence of Lemma \ref{upper} and Theorem \ref{thm:counting2}.

\subsection{Counting estimates on periodic orbits} \label{sec:periodicorbits}
Fix $\alpha>0$ and let  $A':= \AO[\alpha]$. For each $c \in \Per(A,T,\alpha)$, let $x=x(c) \in A \cap c$. We know that $f_Tx \in A'$.  Suppose that there exists $\zeta^\ast$ so that $\Per(A,T,\alpha)$ is $(T, \zeta^\ast)$-separated. Let $\eps>0$ be as in Convention \ref{note:scales} for the reference set $A'$ and assume without loss of generality that $\zeta^\ast < \fraceps$. By Lemma \ref{lem:L}, there exists $V>0$ so that
\[
\Phi(x, T)- V \leq \Phi (c) \leq \Phi(x,T) + V.
\]
The set $E = \{x(c): c \in \Per(A,T,\alpha)\}$ is $(T, \zeta^\ast)$-separated and the upper bound in  \eqref{eqn:unif-count-per} follows from Theorem \ref{thm:uniformgeneral} because we have 
\[
 \sum_{c\in \Per(A,T, \alpha)} e^{\Phi(c)} \leq e^V \sum_{x \in E}e^{\Phi(x,T)} \leq e^V C_U (A', \zeta^\ast) e^{TP(\ph)}.
\]

Now assume the periodic specification property. We adapt the proof of \cite[Proposition 6.4]{BCFT}. Let $3 \zeta \in (0, \eps/8)$ and let $E_T$ be a maximal $(T, 3\zeta)$-separated set for $A_T$ so that from Theorem \ref{thm:uniformgeneral} we have
\[
\sum_{x\in E_T} e^{\Phi(x,T)} \geq C_L(A,6\zeta) e^{-Q} e^{TP(\ph)}.
\]
Using the periodic specification property, there exists $T_0$ so that we can define a map $x \to x'$ so that  
\[
f_{T+\tau} x' = x' \mbox{ for some } \tau \in (T_0- \alpha, T_0] \mbox{ and }d_T(x, x') < \zeta. 
\]
The map is injective and the image $E'_T$ is  $(T+T_0, \zeta)$-separated for $A'_{T+T_0}$. We have 
\begin{equation} \label{eq:per1}
\sum_{x'\in E'_T} e^{\Phi(x',T+T_0)} \geq C_L(A,6\zeta) e^{-Q-Q(A', \zeta)-V} e^{TP(\ph)}.
\end{equation}
Since $E'_T$ is $(T+T_0, \zeta)$-separated, each $c  \in \Per(A,T+T_0,\alpha)$ contains at most $(T+T_0)/ \zeta$ elements of $E'_t$. Since $\Phi(c) \geq \Phi(x',T+T_0) - V$, we obtain 
\begin{equation} \label{eq:per2}
 \sum_{c\in \Per(A,T, \alpha)} e^{\Phi(c)} \geq  \frac{e^{-V} \zeta}{T+T_0} \sum_{x'\in E'_T} e^{\Phi(x',T+T_0)},
\end{equation}
and the lower  bound in  \eqref{eqn:unif-count-per} on $\Per(A,T,\alpha)$ follows from \eqref{eq:per1} and \eqref{eq:per2}.
\section{Misiurewisz measures and the Gibbs property} \label{sec:MandGibbs}

\subsection{Misiurewicz  measures via tightness} \label{sec:Bow-Mis-exists}

We prove Theorem \ref{thm:mainGibbs} as a consequence of Theorem \ref{thm:mainGibbsgeneral} below. This result does not require the full strength of our expansivity proeprty, and we work with \eqref{eqn:P-eps} in place of Condition \ref{cond:E}.
To construct the measure which is our candidate to be an equilibrium state, we formalize the standard construction of Misiurewicz that appears in the proof of the variational principle \cite[Theorem 9.10]{pW82}.
This construction involves taking a weak* limit point of an appropriate family of measures; in order to do this in the non-compact setting, we need to prove that the family of measures is tight. First we describe the measures.

\begin{definition}\label{def:Bow-Mis} 
%Let $X,\FFF,\DDD,\ph$ satisfy Condition \ref{cond:basic}, \ref{cond:S} and \ref{cond:VP}. 
Given $t>0$ and a finite set $E \subset X$, consider the measures
\begin{equation}\label{eqn:sigma-Et}
\sigma^E_t := \frac{\sum_{x\in E} e^{\Phi(x,t)} \delta_x} {\sum_{y\in E} e^{\Phi(y,t)}}
\quad\text{and}\quad
\mu^E_t := \frac 1t \int_0^t (f_s)_* \sigma^E_t \,ds.
\end{equation}
Fix $A\in \RRR$ and $\zeta>0$. For each $t>0$, let $E_t \subset A_t$ be a $(t,\zeta)$-separated set that is maximal with respect to inclusion.
 We call $(\mu_{t})$ a \emph{Misiurewicz   family}  for $\varphi$ with respect to $A$ (for the flow $\FFF$) at scale $\zeta$. A weak*-accumulation measure of $(\mu_{t})$ is called a \emph{Misiurewicz  measure} for $\varphi$ with respect to $A$ (for the flow $\FFF$) at scale $\zeta$.
\end{definition}

An equivalent definition of $\mu_t^E$ in \eqref{eqn:sigma-Et} is as follows: given $x\in X$ and $t>0$,  denote the corresponding empirical measure by $\emp{x,t} = \int_0^t (f_s)_* \delta_x \,ds$,
and given a finite set $E \subset A_t$, let
\[
\mu_t^E := \sum_{x\in E} p_x^E \emp{x,t},
\quad\text{where } p_x^E := e^{\Phi(x,t)} \Lambda(\ph,t,E)^{-1},
\]
where we recall the partition sum notation $\Lambda(\ph,t,E) = \sum_{y\in E} e^{\Phi(y,t)}$ from \eqref{eqn:LET}.

\begin{theorem}\label{thm:mainGibbsgeneral}
Let $X,\FFF,\DDD,\ph$ satisfy Conditions \ref{cond:basic}, \ref{cond:S}, and \ref{cond:B}. 
Suppose that $\ph$ is SPR w.r.t.\ $A\in \RRR$, as in Condition \ref{cond:SPR}, and that $\eps \in (0,\rmet(A)]$ is such that $P^{\AO}(\ph, \eps) = P(\varphi)$, as in \eqref{eqn:P-eps}. Then the following are true.%
\begin{enumerate}[label=\upshape{(\alph{*})}]
\item\label{thm:tight} 
Every Misiurewicz family for $\ph$ w.r.t.\ $A$ at scale $\zeta \in(0, \fraceps)$ is tight. 
In particular, there is a weak*-accumulation point, so the set of Misiurewicz  measures is non-empty;
\item\label{thm:Gibbs} 
%For $\eps_0\in (0,\eps/10)$ such that $A^{-\eps_0}\neq \emptyset$,
If $A \in \RRR$ contains an $\epso$-ball for some $\epso \in (0,\fracfraceps)$, then
any Misiurewicz  measure $\mu$ for $\ph$ w.r.t.\ $A$ at scale $\epso$ is an $\FFF$-invariant probability measure that satisfies the lower Gibbs property at scale $\twoepso$ on any $A'\in \RRR$:
that is, there exists $G_L = G_L(A') > 0$ such that for all $t\geq 0$ and $x\in A'_t$, we have
\begin{equation}\label{eqn:GL}
\mu(B_t(x, \twoepso)) \geq G_L e^{-tP + \Phi(x, t)}.
\end{equation}
If $A'\in \RRR$ is such that $\rmet(A\cup A') \geq \epso$, then $\mu$ also satisfies the upper Gibbs property at scale $\halfepso$ on $A'$: there exists $G_U = G_U(A')>0$ such that for every $t\geq 0$ and $x\in A'_t$, we have
\begin{equation}\label{eqn:GU}
\mu(B_t(x,\halfepso)) \leq G_Ue^{-tP+\Phi(x,t)}.
\end{equation} 
\end{enumerate}
\end{theorem}

In the remainder of this section, we prove Theorem \ref{thm:mainGibbsgeneral}\ref{thm:tight}.
Given $\eta>0$, let $\gamma = \eta/2$ and use Theorem \ref{thm:exp-rare} to choose $\Delta_0$ sufficiently large and $\beta>0$ sufficiently small that \eqref{eqn:exp-rare} holds for every $\Delta\geq \Delta_0$; in particular, if $E = E_t$ is $(t,\zeta)$-separated, then
\begin{equation}\label{eqn:sum-AtB}
\sum_{x\in A_t^\BBB \cap E} e^{\Phi(x,t)} \leq \Lambda(\ph,t,\zeta,A_t^\BBB) \leq C_B e^{-\beta t} e^{Pt}.
\end{equation}
Recalling \eqref{eqn:unif-counting-E}, we have $\Lambda(\ph,t,E) \geq C_L e^{Pt}$, for all $t \geq T_0(A,\zeta)$, so \eqref{eqn:sum-AtB} gives
\[
\sum_{x\in A_t^\BBB \cap E} p_x^E \leq C_L C_B e^{-\beta t}.
\]
Let $t_0 \geq T_0(A,\zeta)$ be sufficiently large that $C_L C_B e^{-\beta t_0} < \gamma$. Then for every $\Delta \geq \Delta_0$ and $t\geq t_0$, the compact set $K = \AO[\Delta]$ satisfies
%Let $K = \AO[\Delta]$. Then we have $\emp{x,t}(K^c) \leq \gamma$ for all $x\in A_T^\GGG$, and thus
\begin{equation}\label{eqn:muKc}
\mu_t(K^c) \leq \sum_{x\in A_t^\GGG \cap E} p_x^E \emp{x,t}(K^c)
+ \sum_{x\in A_t^\BBB \cap E} p_x^E
< \gamma + \gamma = \eta,
\end{equation}
where we use the fact that $\emp{x,t}(K^c) \leq \gamma$ for all $x\in A_T^\GGG$ by definition. Let $\Delta = \max(\Delta_0,t_0)$, and observe that for all $t \in (0,t_0)$, we have $\mu_t(K^c) = 0$. Thus $\mu_t(K^c) < \eta$ for every $t>0$, which completes the proof.

\subsection{Invariance and Gibbs bounds}\label{sec:Gibbs} 
Now we prove Theorem \ref{thm:mainGibbsgeneral}\ref{thm:Gibbs}.
Fix $A\in \RRR$ and $\eps_0 \in (0,\fracfraceps)$ such that $\ph$ is SPR on $A$, and $A^{-\eps_0}$ has nonempty interior, and is thus a member of $\RRR$. Suppose that $\mu$ is a Misiurewicz measure for $\ph$ w.r.t.\ $A$, so there exist times $t_\ell\to\infty$ such that the measures $\mu_{t_\ell} := \mu_{t_\ell}^{E_{t_\ell}}$ from \eqref{eqn:sigma-Et} converge to $\mu$ in the weak* topology. 
Invariance of $\mu$ follows from the usual argument:
for any bounded continuous function $g \colon X\to \RR$, and every $T>0$, we have
\begin{align*}
\int g\circ f_T \,d\mu - \int g \,d\mu
&= \lim_{\ell\to\infty} \int (g\circ f_T - g) \,d\mu_{t_\ell} \\
&= \lim_{\ell\to\infty} \sum_{x\in E_{t_\ell}} p_x^{E_{t_\ell}} \frac 1{t_\ell} \big(G(f_t x, t_\ell) - G(x, t_\ell) \big),
\end{align*}
and since $g$ is bounded, we can use the estimate
\[
\big|G(f_T x, t_\ell) - G(x, t_\ell) \big|
= \big| G(f_{t_\ell}x, T) - G(x,T)\big|
\leq 2T \|g\|,
\]
which is illustrated in Figure \ref{fig:Mis-inv}, to conclude that
\[
\Big|\int g\circ f_T \,d\mu - \int g \,d\mu\Big|
\leq \limsup_{\ell\to\infty} \frac {2T}{t_\ell}\|g\| = 0.
\]

\begin{figure}[htbp]
\begin{tikzpicture}
\def\a{0}
\def\b{2}
\def\c{7}
\def\d{9}
\def\y{1}
\def\z{.5}
\draw[red!90!black,ultra thick] (\a,0)--(\b,0) (\a,\z)--(\c,\z);
\draw[blue,ultra thick] (\c,0)--(\d,0) (\b,\y)--(\d,\y);
\foreach \p in {(\a,0), (\b,0), (\c,0), (\d,0), (\a,\z), (\b,\y), (\c,\z), (\d,\y)}
{ \fill \p circle (2pt); }
%\foreach \x / \t in {\a /$0$, \b /$T$, \c /$t_\ell$, \d / $t_\ell+T$}
%{ \draw[dotted] (\x,1.2)--(\x,-.2) node[below]{\t};
\draw[dotted] (\a,1.2)--(\a,-.2) node[below]{$0$};
\draw[dotted] (\b,1.2)--(\b,-.2) node[below]{$T$};
\draw[dotted] (\c,1.2)--(\c,-.2) node[below]{$t_\ell$};
\draw[dotted] (\d,1.2)--(\d,-.2) node[below]{$t_\ell+T$};
\end{tikzpicture}
\caption{Proving invariance of a Misiurewicz measure.}
\label{fig:Mis-inv}
\end{figure}
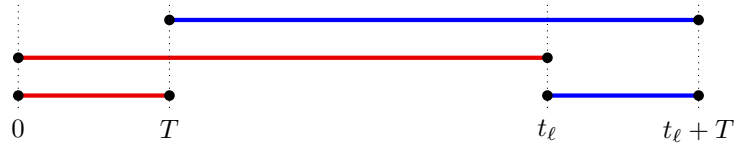

It remains to establish the Gibbs property of $\mu$.
In what follows, it will be convenient to write the measures in \eqref{eqn:sigma-Et} in terms of partition sums: given $T>0$, a finite set $E\subset X$, and a measurable set $Y \subset X$, we have
\begin{equation}\label{eqn:sig-E}
(f_s)_* \sigma_T^E(Y)
= \frac{\Lambda(\ph,T,E \cap f_{-s}(Y))}
{\Lambda(\ph,T,E)}
\quad\text{for every } s\in [0,T].
\end{equation}
We start by proving the upper Gibbs bound \eqref{eqn:GU} under the assumption that $\epso \leq \rmet(A \cup A')$.
Fix $T>0$.
Since $E_T$ is maximal $(T,\eps_0)$-separated in $A_T$ and $\eps_0 < \halffraceps$, we obtain a lower bound on $\Lambda(\ph,T,E_T)$ from Theorem \ref{thm:uniformgeneral}.
 Together with \eqref{eqn:sig-E}, this shows that for any measurable $Y \subset X$ and any $s\in [0,T]$, we have
\begin{equation}\label{eqn:fs-sig-leq}
(f_s)_*\sigma_T(Y) \leq \Lambda(\ph,T,E_T\cap f_{-s}(Y)) C_L^{-1} e^{-TP},
\end{equation}
where $C_L = C_L(A,\eps_0)$.
Fixing $t\geq 0$ and $x\in A'_t$, we will apply \eqref{eqn:fs-sig-leq} with $Y = B_t(x,\eps_0/2)$. We need an upper bound on $\Lambda(\ph,T,E_T \cap f_{-s}(Y))$, which will be provided by Lemma \ref{lem:submult}.

More precisely, given $s\in [0,T-t]$, we let $A'' = A\cup A'$, and since $\epso \leq \rmet(A'')$, we can apply Lemma \ref{lem:submult} with $Z = E_T \cap f_{-s}(Y)$ and $\bt = (s,t,r)$, where $r = T-(s+t)$ as shown in Figure \ref{fig:Gibbs}(a), to conclude that writing $\rho = \eps_0/2$, we have
\begin{align*}
\Lambda(\ph,T,E_T \cap f_{-s}(Y))
&\leq e^{3Q}
\Lambda(\ph,s,\rho,A''_s) \Lambda(\ph,t,\rho,Y) \Lambda(\ph,r,\rho,A''_r) \\
&\leq e^{3Q} C_U(A'',\rho)^2 e^{sP} e^{rP} \cdot e^Q e^{\Phi(x,t)},
\end{align*}
where the second inequality uses the upper bound in Theorem \ref{thm:uniformgeneral} twice, together with the observation that $\{x\}$ is a maximal $(t,\eps_0/2)$-separated set in $Y = B_t(x,\eps_0/2)$.
Combining this bound with \eqref{eqn:fs-sig-leq} gives
\[
(f_s)_* \sigma_T(B_t(x,\eps_0/2))
\leq C_L^{-1} e^{-TP} e^{4Q} C_U^2 e^{(s+r)P} e^{\Phi(x,t)}
= G_U e^{-tP + \Phi(x,t)},
\]
where $G_U := C_L(A,\eps_0)^{-1} C_U(A'',\rho)^2 e^{4Q(A'')}$ and we use  $-T + s + r = -t$.
This bound is valid for all $s\in [0,T-t]$, so
\[
\mu_T(B_t(x,\eps_0/2)) \leq \frac {T-t}{T} G_U e^{-tP + \Phi(x,t)} + \frac{t}{T}.
\]
Sending $T\to\infty$ along the sequence $T=t_\ell$ proves \eqref{eqn:GU}.

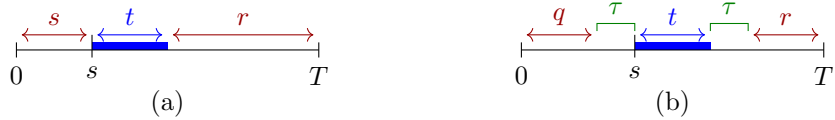
\begin{figure}[htbp]
\qquad
\begin{tikzpicture}%
[myarr/.style={{<[sep=2pt]}-{>[sep=2pt]}}]
\def\a{1}
\def\b{2}
\fill[blue] (\a,0) rectangle (\b,.1);
\draw[myarr,red!60!black] (0,.2) -- (\a,.2) node[pos=0.5,above]{$s$};
\draw[myarr,blue] (\a,.2) -- (\b,.2) node[pos=0.5,above]{$t$};
\draw[myarr,red!60!black] (\b,.2) -- (4,.2) node[pos=0.5,above]{$r$};
\foreach \x in {0, 4}
{ \draw (\x,-.1) -- ++(0,.2); }
\node[below] at (0,-.1){$0$};
\node[below] at (4,-.1){$T$};
\draw (0,0) -- (4,0);
\draw (\a,.2) -- (\a,-.1) node[below]{$s$};
\node at (2,-0.7){(a)};
\end{tikzpicture}
\hfill
\begin{tikzpicture}%
[myarr/.style={{<[sep=2pt]}-{>[sep=2pt]}}]
\def\a{1.5}
\def\b{2.5}
\def\c{1}
\def\d{3}
\fill[blue] (\a,0) rectangle (\b,.1);
\draw[myarr,red!60!black] (0,.2) -- (\c,.2) node[pos=0.5,above]{$q$};
\draw[myarr,blue] (\a,.2) -- (\b,.2) node[pos=0.5,above]{$t$};
\draw[myarr,red!60!black] (\d,.2) -- (4,.2) node[pos=0.5,above]{$r$};
\draw[green!50!black] (\c,.25)--(\c,.35)--(\a,.35)
node[pos=0.5,above]{$\tau$} --(\a,.25);
\draw[green!50!black] (\b,.25)--(\b,.35)--(\d,.35)
node[pos=0.5,above]{$\tau$} --(\d,.25);
\foreach \x in {0, 4}
{ \draw (\x,-.1) -- ++(0,.2); }
\draw (\a,.2) -- (\a,-.1) node[below]{$s$};
\node[below] at (0,-.1){$0$};
\node[below] at (4,-.1){$T$};
\draw (0,0) -- (4,0);
\node at (2,-0.7){(b)};
\end{tikzpicture}
\qquad{}
\caption{Bookkeeping in the proofs of the Gibbs bounds.}
\label{fig:Gibbs}
\end{figure}

Now we prove the lower Gibbs bound \eqref{eqn:GL}.
As before, fix $T>0$ and let $Y = B_t(x,\twoepso)$.
Combining \eqref{eqn:sig-E} and the upper bound in Theorem \ref{thm:uniformgeneral}, we get
\begin{equation}\label{eqn:fs-sig-geq}
(f_s)_*\sigma_T(Y) \geq \Lambda(\ph,T,E_T\cap f_{-s}(Y)) C_U^{-1} e^{-TP},
\end{equation}
where $C_U = C_U(A,\epso)$. 
Let $\rho = \epso$, and let $\tau$ be the transition time in the specification property for scale $\rho$.

Fixing $s\in [\tau,T-\tau]$, put $q=s-\tau$ and $r=T-(s+t)-\tau$ as in Figure \ref{fig:Gibbs}(b), so $q + t + r = T-2\tau$. Writing $\bt = (q,t,r)$ and $\bZ = ((A^{-\epso})_q,x,(A^{-\epso})_r)$, we have
\begin{equation}\label{eqn:spec-in}
\Spec_\rho^\tau(\bZ,\bt) \subset A_T \cap f_{-s}(B_t(x,\epso)).
\end{equation}
We claim that $E_T \cap f_{-s}(Y)$ is $(T,\epso)$-spanning for $A_T \cap f_{-s}(B_t(x,\epso))$. Indeed, since $E_T$ is $(T,\epso)$-spanning for $A_T$, we see that given any $y\in A_T \cap f_{-s}(B_t(x,\epso))$, there exists $z\in E_T$ such that $y \in B_T(z,\epso)$. This implies that
\[
f_s(z) \in B_t(f_s(y),\epso) \subset B_t(x,\twoepso) = Y,
\]
so $z\in E_T \cap f_{-s}(Y)$, verifying the claim.

Since $E_T \subset A_T$ is a $(T,\epso)$-separated set, Lemma \ref{lem:all-sep} now implies that
\[
\Lambda(\ph,T,E_T \cap f_{-s}(Y))
\geq e^{-Q} \Lambda(\ph,T,\twoepso,A_T \cap f_{-s}(B_t(x,\epso))).
\]
Combining this with \eqref{eqn:fs-sig-geq} and \eqref{eqn:spec-in} gives
\begin{equation}\label{eqn:geq-spec}
(f_s)_* \sigma_T(Y)
\geq e^{-Q} C_U^{-1} e^{-TP} \Lambda(\ph,T,2\eps_0,\Spec_\rho^\tau(\bZ,\bt)).
\end{equation}
Writing $\zeta = 2\eps_0$ so that $\zeta + 2\rho = 4\eps_0$, we can use the supermultiplicativity in Lemma \ref{lem:spec-sum-Bow}   to get
\begin{align*}
\Lambda(\ph,T,\zeta,\Spec_\rho^\tau(\bZ,\bt))
&\geq e^{-3(\Vt+Q)} \Lambda(\ph,q,\zeta+2\rho,A''_q)
e^{\Phi(x,t)} \Lambda(\ph,r,\zeta+2\rho,A''_r) \\
&\geq e^{-3(\Vt+Q)} C_L^2 e^{(q+r)P} e^{\Phi(x,t)},
\end{align*}
where $C_L = C_L(A'',4\eps_0)$ is from Theorem \ref{thm:uniformgeneral}. Combining this with \eqref{eqn:geq-spec}, we get
\[
(f_s)_*\sigma_T(Y)
\geq 
K e^{-TP} e^{(q+r)P} e^{\Phi(x,t)},
\]
where $K = e^{-Q} C_U^{-1} e^{-3(\Vt+Q)} C_L^2$. Integrating gives
\[
\mu_T(B_t(x,\eps_0)) \geq \frac{T-2\tau}{T}K e^{-(T-(q+r)) P} e^{\Phi(x,t)},
\]
and since $T = q + r + t + 2\tau$, we have $e^{-(T-(q+r)) P} \geq e^{-2\tau P} e^{-tP}$, so sending $T\to\infty$ along the sequence $T=t_\ell$ proves \eqref{eqn:GL} with $G_L = K e^{-2\tau P}$.

An analogous argument to the proof just given for \eqref{eqn:GL} implies the following `joint Gibbs property', which is essential for proving ergodicity of $\mu$. For details, see e.g. \cite[Lemma 4.17]{CT16}.

\begin{lemma} \label{lem: DoubleLowerGibbs}
For any $A'\in \RRR$ there exists a constant $H_L$, and a time $T>0$ such that for any $t_1,t_2>0$, $x_1\in A'_{t_1}$, $x_2\in A'_{t_2}$ and $t>T$, we have
\[
\mu(B_{t_1}(x_1,2 \eps_0)\cap f_{-t-t_1}(B_{t_2}(x_2,2 \eps_0)))\geq H_Le^{-(t_1+t_2)P+\Phi(x_1,t_1)+\Phi(x_2,t_2)}.
\]
\end{lemma}

\begin{remark}
The scales at which we obtained the upper and lower Gibbs properties of $\mu$ are distinct, and do not give the upper and lower Gibbs property simultaneously at any fixed scale. We have not proved that the lower Gibbs property applies at scales smaller than $\eps_0/2$; or that the upper Gibbs property applies at scales larger than $2\eps_0$; or that the upper Gibbs property applies for reference sets $A'$ with $\rmet(A') < \epso$. When the equilibrium state is known to be unique, it will follow that any valid choice of $\epso$ leads to the same $\mu$, implying that
both the lower and upper Gibbs bounds hold for all sufficiently small scales and all reference sets. Here, ``sufficiently small'' will depend on the reference set.
\end{remark}

\section{Adapted partitions and equilibrium states}\label{sec:eq-st}

\subsection{Overview}

Now we turn our attention to Theorems \ref{thm:existsES} and \ref{thm:mainES}, which claim:
\begin{itemize}
\item every Misiurewicz measure is an equilibrium state;
\item every Misiurewicz measure is ergodic;
\item no other ergodic measure can be an equilibrium state.
\end{itemize}
In the compact setting, each of these three arguments uses the concept of a partition that is \emph{adapted} to a given $(n,\eps)$-separated set of points. In this section, we describe how this concept must be modified in the non-compact setting to allow $\eps$-separated sets of \emph{orbit segments}, possibly varying in length, and then use it to prove 
that every Misiurewicz measure is an equilibrium state. More precisely, we prove the following, which implies Theorem \ref{thm:existsES}:

\begin{theorem}\label{thm:Mis-ES}
Let $X,\FFF,\DDD,\ph$ satisfy Conditions \ref{cond:basic}, \ref{cond:S},  \ref{cond:B}, and \ref{cond:UESB}. Suppose that $\ph$ is SPR w.r.t.\ $A\in \RRR$, and that $\eps \in (0,\rmet(A^{L^*})]$ is such that $P^{\AO}(\ph,\eps) = P$. Fix $\epso \in (0,\eps/20)$ such that $A^{-\epso} \neq \emptyset$. 
Let $\mu$ be any Misiurewicz measure for $\ph$ w.r.t.\ $A$ at scale $\epso$.
Then $\mu$ is an equilibrium state for $\ph$.
\end{theorem}

We give the basic definitions and properties in \S\S\ref{sec:sepandadaptedbasics}--\ref{sec:Nthreturns}, and in \S\ref{sec:Bow-Mis-ES}, we prove Theorem \ref{thm:Mis-ES} following Misiurewicz's approach.

The modified versions of adaptedness that we introduce in this section will also play a crucial role in \S\ref{sec:uniform-katok} and \S\ref{sec:eu-other} when we define a certain induced map and use it to prove ergodicity and uniqueness.

\subsection{Adapted partitions and separated sets in the space of orbit segments}\label{sec:sepandadaptedbasics} 

In the compact case, partitions \emph{adapted} to a finite $(t, 2\zeta)$-separated set $E$ have a key role in the thermodynamic formalism. Adaptedness means that each partition element $w$ satisfies 
\[
B_t(x, \zeta) \subset w \subset \overline B_t(x, 2\zeta)
\] 
for some $x\in E$. In the classical theory, $2\zeta$ is an expansivity constant, so the diameter of $f_{t/2} w$ goes to $0$ as $t \to \infty$, and for a Gibbs measure $\mu$, we obtain a lower bound on $\mu(w)$ from the lower Gibbs property. For our analysis, we must extend these ideas to countable collections of orbit segments $(x_k, t_k)$. The varying times usually come from inducing on a reference set $A$, and we will generally discretize and reduce to consider only integer times.

\begin{definition}\label{def:adapted-part}
Fix $S\subset X$ and let $\mathcal{E} \subset S\times [0,\infty)$ be countable. Let $\zeta>0$ and $K\geq2$. We say that a countable partition $\xi$ of $S$ is $(\zeta, K \zeta)$-\emph{adapted} for $\mathcal{E}$ if $\xi$ and $\mathcal{E}$ can be indexed as $\xi = \{w_j\}_{j\in \NN}$ and $\mathcal{E} = \{(x_j,m_j)\}_{j\in \NN}$ such that for every $j\in \NN$, we have
\begin{equation}\label{adapted-inclusion0}
B_{m_j}(x_j,\zeta) \cap S \subset w_j \subset \overline B_{m_j}(x_j,K\zeta) \cap S.
\end{equation}
\end{definition}

To build adapted partitions, we need a notion of $\zeta$-separated for orbit segments:

\begin{definition}\label{def:eps-sep}
A collection of orbit segments $\mathcal{E} \subset X\times [0,\infty)$ is \emph{$\zeta$-separated}
if for every $(x,t) \in \mathcal{E}$, the only $(y,s) \in \mathcal{E}$ with $y\in \overline B_t(x,\zeta)$ is $(x,t)$ itself. 
\end{definition}

\begin{remark}
Observe that $E \subset S$ is a $(t, \zeta)$-separated set if and only if
$E \times \{t\}$ is $\zeta$-separated in the sense of Definition \ref{def:eps-sep}.
It is well known that if $S$ is the phase space for a compact dynamical system and $E \subset S$ is a $(t, 2\zeta)$-separated set that is maximal with respect to inclusion, then $E$ is $2\zeta$-spanning and there exists a $(\zeta, 2\zeta)$-adapted partition for $E$, and hence for $E\times \{t\}$.
One way to construct such a partition is to assign each point of $S$ to the nearest element of $E$ (breaking ties arbitrarily). Another way, which we will follow in Proposition \ref{prop:slicedadapted} below, is to assign the ``core'' $B_t(x,\zeta)$ to $x$ for each $x\in E$, and then use a ``greedy algorithm'' to partition the remaining points of $S$ by ordering the points of $E$ and then assigning to each in turn all unassigned points in $\overline B_t(m,2\zeta)$.
\end{remark}

Definition \ref{def:eps-sep} can be reformulated as follows: $\mathcal{E} \subset X\times [0,\infty)$ is $\zeta$-separated if and only if
\begin{equation} \label{eq:definingforseparated}
    d_{\min\{t_1,t_2\}}(x,y)>\zeta \quad  \text{ for every }       (x,t_1), (y,t_2)\in \mathcal{E}.
\end{equation}

\begin{lemma}\label{lem:disjoint}
If $\mathcal{E} \subset S\times [0,\infty)$ is $2\zeta$-separated, then the Bowen balls $\{ \overline B_t(x,\zeta) : (x,t) \in \mathcal{E} \}$ are pairwise disjoint.
\end{lemma}
\begin{proof}
Let $(x,t_1), (y,t_2) \in \mathcal{E}$ be distinct, and assume without loss of generality that $t_2 \geq t_1$. Then
\[
B_{t_1}(x,\zeta) \cap B_{t_2}(y,\zeta)
\subset B_{t_1}(x,\zeta) \cap B_{t_1}(y,\zeta) = \emptyset,
\]
where the second inequality uses \eqref{eq:definingforseparated}.
\end{proof}

\begin{lemma}\label{lem:uniformpartitionsumA}
Suppose that $\ph$ is SPR on some $A\in \RRR$, and that $\eps \in (0,\rmet(\AO)]$ satisfies $P^{\AO}(\ph, \eps) = P$.
Fix $\eps_0\in (0,\fracfraceps)$ such that $A^{-\eps_0}\neq \emptyset$.
For $A' \in \RRR$, let $G_L = G_L(A')$, where $G_L$ is the lower Gibbs constant from Theorem \ref{thm:mainGibbsgeneral}\ref{thm:Gibbs}. Then for any $Z\subset A'$ and any $\fourepso$-separated collection of orbit segments $\EE \subset Z \times [0,\infty)$ such that $f_t(x) \in A'$ for all $(x,t) \in \EE$, we have
\begin{equation}\label{eqn:induced-sums}
\sum_{(x,t) \in \EE} e^{\Phi(x,t) - tP(\ph)} \leq G_L^{-1} \mu(Z^{\twoepso}).
\end{equation}
\end{lemma}
\begin{proof}
By Theorem \ref{thm:mainGibbsgeneral}\ref{thm:tight}, there exists a Misiurewicz measure $\mu$ at scale $\eps_0$ w.r.t.\ $A$.  For every $(x,t) \in \EE$, the lower Gibbs property for $\mu$ from Theorem \ref{thm:mainGibbsgeneral}\ref{thm:Gibbs} can be written as $e^{\Phi(x,t) - tP(\ph)} \leq G_L^{-1} \mu(B_t(x,\twoepso))$. Since the Bowen balls
$\{B_t(x,\twoepso) : (x,t) \in \EE\}$ are pairwise disjoint by Lemma \ref{lem:disjoint}, and are all contained in $Z^{\twoepso}$, summing over $(x,t) \in \EE$ proves the proposition.
\end{proof}

We also introduce the notion of spanning for a collection of orbit segments.
\begin{definition}\label{def:span-segmentsset}
Given a collection of orbit segments $\mathcal{E}  \subset X \times [0, \infty)$, we say that $\mathcal{E}$ is \emph{$\zeta$-spanning} for a set $S \subset X$ if $S \subset  \bigcup_{(x,t) \in \mathcal{E}} \overline B_t(x,\zeta)$.
\end{definition}

We investigate how spanning and separated are dual notions in this setting. We consider separated sets which are maximal in the following sense.

\begin{definition}\label{def:maximalinclusion}
Given $\mathcal{S} \subset X\times [0,\infty)$,
a $\zeta$-separated collection of orbit segments $\mathcal{E} \subset \mathcal{S}$ is \emph{maximal with respect to inclusion in $\mathcal{S}$}, or simply \emph{maximal in $\mathcal{S}$}, if there does not exist $(z,t) \in \mathcal{S} \setminus \mathcal{E}$ such that $\mathcal{E} \cup \{(z,t)\}$ is  $\zeta$-separated for $S$. In the case when $\mathcal{S} = S\times [0,\infty)$ for some $S \subset X$, we will refer to such an $\mathcal{E}$ as \emph{maximal for $S$}.
\end{definition}

To prove that ``maximal separated implies spanning'', we must restrict to countable collections of orbit segments $\mathcal{E} \subset S \times [0,T]$ for some fixed $T>0$.

\begin{lemma}\label{lem:septospan}
Fix $S\subset X$ and let $\mathcal{E} \subset S\times [0,T]$ be $\zeta$-separated. Then the following are equivalent:
\begin{enumerate}[label=\upshape{(\alph{*})}]
\item\label{it:ST} $\mathcal{E}$ is maximal in $S\times [0,T]$;
\item\label{it:Sinf} $\mathcal{E}$ is maximal in $S\times [0,\infty)$;
\item\label{it:S} $\mathcal{E}$ is maximal for $S$;
\item\label{it:span} $\mathcal{E}$ is $\zeta$-spanning for $S$.
\end{enumerate}
\end{lemma}
\begin{proof} 
The equivalence of \ref{it:Sinf} and \ref{it:S} is immediate from Definition \ref{def:maximalinclusion}.
We prove \ref{it:Sinf} $\Rightarrow$ \ref{it:ST} $\Rightarrow$ \ref{it:span} $\Rightarrow$ \ref{it:Sinf} by showing the contrapositives,
starting with the observation that \ref{it:Sinf} implies \ref{it:ST}, since any $(z,t) \in (S\times [0,T]) \setminus \mathcal{E}$ violating maximality would also lie in $(S\times [0,\infty)) \setminus \mathcal{E}$.

To see that \ref{it:ST} implies \ref{it:span}, suppose $\mathcal{E}$ is not $\zeta$-spanning for $S$, so there exists $z\in S$ such that 
$d_t(x,z) > \zeta$ for every $(x,t) \in \mathcal{E}$.
Since $\min(t,T) = t$, it follows from the characterization of $\zeta$-separated sets in \eqref{eq:definingforseparated} that $\mathcal{E} \cup \{(z,T)\}$ is $\zeta$-separated, so $\mathcal{E}$ is not maximal in $S\times [0,T]$.

Finally, for \ref{it:span} implies \ref{it:Sinf}, suppose $(z,s) \in S\times [0,\infty)$ is such that $\mathcal{E} \cup \{(z,s)\}$ is $\zeta$-separated. Then $z\notin \overline{B}_t(x,\zeta)$ for every $(x,t) \in \mathcal{E}$, so $\mathcal{E}$ is not $\zeta$-spanning.
\end{proof}

We note that if $\mathcal{E} \subset S\times [0, \infty)$ is a $\zeta$-separated set that is maximal with respect to inclusion
but contains $(x,t)$ with arbitrarily large values of $t$, then the above argument fails, since there is no $T$ for which we always have $\min(t,T) = t$.
Thus in order to conclude that a maximal separated collection of orbit segments is spanning, we need their lengths to be uniformly bounded.
As we will soon see, we can extend this to the setting of \emph{graded sets}, which
arises naturally in our arguments. There are two constructions which are useful in this context, which we introduce in the next two subsections.

\subsection{Graded sets and adapted partitions} \label{sec:gradedsep}

The following concept will be used frequently in the remaining proofs.

\begin{definition}\label{def:graded-set}
A \emph{graded set} in $X$ consists of a pair $(S,\tau)$, where $S\subset X$, and $\tau\colon S\to (0,\infty)$ is such that $\tau(S)$ is countable and well-ordered. If $\tau(S) \subset \NN$, then we will call $(S,\tau)$ an \emph{$\NN$-graded set}.
\end{definition}

Given a graded set $(S,\tau)$, we can write
\begin{equation}\label{eqn:St}
\tau(S) = \{ t_1 < t_2 < t_3 < \cdots \}
\quad\text{and}\quad
S_n := \tau^{-1}(t_n).
\end{equation}
Writing $\bS = (S_n)_n$ and $\bt = (t_n)_n$, the descriptions of the graded set as $(S,\tau)$ and as $(\bS,\bt)$ are equivalent, and we will pass back and forth between them without further comment. 

In the proof of existence of equilibrium states, we will consider graded sets determined by a first return time function $\tau$ for the time-$1$ map $F$. In \S \ref{sec:uniform-katok}, we will need to work with a more carefully selected function $\tau$.

A graded set $(S,\tau)$ determines a collection of orbit segments $\mathcal{S} \subset S\times [0,\infty)$ as the graph of the function $\tau$:
\begin{equation}\label{eqn:SSS}
\mathcal{S} := \{ (x,\tau(x)) : x\in S \} = \bigcup_{n=1}^\infty (S_n \times \{t_n\}),
\end{equation}
where the second description uses the correspondence in \eqref{eqn:St} between $(S,\tau)$ and $(\bS,\bt)$.
Given $n\in \NN$, we will also consider the subset $S_{\leq n} \subset S$  given by
\begin{equation}\label{eqn:S-leq-n}
S_{\leq n} := \{ x\in S : \tau(x) \leq t_n \}
= \bigcup_{k=1}^n S_k,
\end{equation}
and the subcollection $\mathcal{S}_{\leq n} \subset \mathcal{S}$ given by
\begin{equation}\label{eqn:S-leq-n2}
\mathcal{S}_{\leq n} := \{ (x,\tau(x)) : x\in S_{\leq n}\}
= \bigcup_{k=1}^n (S_k \times \{t_k\})
\subset
S_{\leq n} \times [0,t_n].
\end{equation}

\begin{definition}\label{def:max-graded}
Let $(S,\tau)$ be a graded set,
and $\mathcal{S}$ the corresponding collection of orbit segments as in \eqref{eqn:SSS}. Given $\zeta>0$, a $\zeta$-separated collection of orbit segments $\mathcal{E} \subset \mathcal{S}$ is \emph{maximal for the grading $\tau$} if for every $n\in \NN$, the $\zeta$-separated collection $\mathcal{E} \cap \mathcal{S}_{\leq n}$ is maximal in $\mathcal{S}_{\leq n}$.
\end{definition}

The condition of being maximal for a grading can be reformulated as follows: for every $n\in \NN$
and every $z\in S$ with $\tau(z) \leq t_n$, the collection $\mathcal{E} \cup \{(z,\tau(z))\}$ is \emph{not} $\zeta$-separated. We will see momentarily that there there exist collections $\mathcal{E}$ satisfying this condition, and that every such collection is spanning. With this in mind, it is worth mentioning two closely related conditions, which we do not use.
\begin{itemize}
\item Maximality for a grading is weaker than the condition that each $\mathcal{E} \cap (S_n \times \{t_n\})$ is maximal in $S_n \times \{t_n\}$; we have no guarantee that there exists a $\zeta$-separated $\mathcal{E}$ satisfying this stronger condition. (This condition will appear in \S\ref{sec.slicedadapted} as \emph{maximality on slices}, where there is no requirement that $\mathcal{E}$ itself is $\zeta$-separated.)
\item Maximality for a grading is stronger than the requirement that $\mathcal{E}$ is maximal in $\mathcal{S}$; a $\zeta$-separated $\mathcal{E}$ satisfying this weaker condition need not be $\zeta$-spanning.
\end{itemize}

\begin{lemma} \label{lem:maxgradedspans}
Let $(S,\tau)$ be a graded set, and $\mathcal{S}$ the corresponding collection of orbit segments as in \eqref{eqn:SSS}. Then for every $\zeta>0$,
\begin{enumerate}
\item there exists a $\zeta$-separated collection of orbit segments $\mathcal{E} \subset \mathcal{S}$ that is maximal for the grading $\tau$; and
\item any such collection is $\zeta$-spanning for $S$.
\end{enumerate}
\end{lemma}
\begin{proof}
The proof of existence follows the usual approach, proceeding inductively in $n$: 
given $n\geq 1$ and assuming $E_1,\dots, E_{n-1}$ have been chosen, let $E_n \subset S_n$ be maximal with respect to inclusion among all subsets of $S_n$ with the property that $\bigcup_{k=1}^n (E_k \times \{t_k\})$ is $\zeta$-separated. Then take $\mathcal{E} = \bigcup_{n=1}^\infty (E_n \times \{t_n\})$.

For the second claim, observe that given any $z\in S$, there exists $n$ such that $z \in S_n$, and applying Lemma \ref{lem:septospan} to the set $\mathcal{E} \cap \mathcal{S}_{\leq n}$, we see that there exists $(x,t) \in \mathcal{E}$ such that $z\in \overline B_t(x,\zeta)$.
\end{proof}

Given a graded set $(S,\tau)$ with corresponding collection of orbit segments $\mathcal{S}$ as in \eqref{eqn:SSS}, there is a one-to-one correspondence between collections $\mathcal{E} \subset \mathcal{S}$ and subsets $E \subset S$; the former can be viewed as the graph of the function $\tau$ over the latter. 

\begin{proposition}\label{prop:adapted}
Given a graded set $(S,\tau)$ with $\mathcal{S}$ as in \eqref{eqn:SSS}, and a $2\zeta$-separated collection $\mathcal{E} \subset \mathcal{S}$ that is maximal for the grading $\tau$, there exists a  partition $\xi$ of $S$ that is $(\zeta, 2\zeta)$-adapted for $\mathcal{E}$.
\end{proposition}
\begin{proof} 
We write $\mathcal{E}$ as $\{(x_i, t_i) : i \in \NN\}$ with $t_{i+1} \geq t_i$ for all $i \in \NN$.
The balls $B_{t_i}(x_i,\zeta)$ are disjoint by Lemma \ref{lem:disjoint} and we have $S \subset \bigcup_{(x_i,t_i) \in E} \overline B_{t_i}(x_i, 2 \zeta)$ by Lemma \ref{lem:maxgradedspans}. To find partition elements as required, we use the following recursive construction. To construct $w_j$, we chop out all $B_{t_i}(x_i, \zeta)$ with $i \neq j$, and all points already assigned to a partition element. That is, we write
\[
B_j^\ast = \bigcup_{i \in \NN \setminus \{j\}} B_{t_i}(x_i, \zeta),
\] 
and we let $w_1= (S \cap \overline B_{t_1}(x_1, 2 \zeta)) \setminus B_1^\ast$. We let $w_2= (S \cap \overline B_{t_2}(x_2, 2 \zeta)) \setminus (B_2^\ast \cup w_1)$. Proceeding recursively, given $w_{j-1}$, we set
\[
w_j = (S \cap \overline B_{t_j} (x_j, 2 \zeta)) \setminus \Big( B_j^\ast \cup \bigcup_{i=1}^{j-1}w_i \Big).
\]
It is clear that $w_j \subset \overline B_{t_j} (x_j, 2 \zeta)$, and we have $B_{t_j} (x_j, \zeta) \subset w_j$ because $\mathcal{E}$ is $2\zeta$-separated and so $B_{t_j} (x_j, \zeta) \cap B_j^\ast = \emptyset$ by Lemma \ref{lem:disjoint}. We have $B_{t_j} (x_j, \zeta) \cap w_i = \emptyset$ when $i \neq j$ by definition. Thus, $B_{t_j}(x_j, \zeta) \cap \big(B_j^\ast \cup \bigcup_{i=1}^{j-1}w_i\big) = \emptyset$. It follows from this and the definition of $w_j$ that $S \cap B_{t_j}(x_j, \zeta) \subset w_j$.

We let $\xi =\{w_i : i\in \NN\}$.  To see that $\xi$ is a partition of $S$, consider $z \in S$. If $z \in \bigcup_i B_{t_i}(x_i, \zeta)$, then we are done. Suppose $z \notin \bigcup_i B_{t_i}(x_i, \zeta)$. Since $\mathcal{E}$ is $2\zeta$-spanning by Lemma \ref{lem:maxgradedspans}, $z \in B_{t_j}(x_j, 2 \zeta)$ for some $j$. Note that
\[
S \cap B_{t_j}(x_j, 2 \zeta) = w_j \sqcup \left ((S \cap B_j^\ast) \cup \bigcup_{i=1}^{j-1}w_i \right ).
\]
Since $z \notin B_j^\ast$, it follows that $z \in w_j$ or $z \in \bigcup_{i=1}^{j-1}w_i$. Thus $\xi$ is a partition.
\end{proof}

\begin{remark}
It is sometimes useful to index a maximal (for a grading) $2\eps$-separated set $E$ and the corresponding adapted partition as follows. Write
\[
E_n = \{(x^n_i, t_n)\}_i 
\quad\text{and}\quad
\underline x^n_i:= (x^n_i, t_n),
\quad\text{so}\quad
E = \bigcup_{n=1}^\infty \{ \underline x^n_i\}.
\]
Given $x \in S$, let $\xi(x)$ be the partition element containing $x$. The adaptedness condition \eqref{adapted-inclusion0} with $K=2$ is equivalent to the condition that for all $i, n$ we have
\begin{equation}\label{adapted-inclusionv2}
S \cap B_{t_n}(x^{n}_i,\eps) \subset \xi(x^{n}_i) \subset \overline B_{t_n}(x^n_i,2\eps).
\end{equation}
We write $w^n_i:=\xi(x^{n}_i)$ so that $w^n_i$ is the partition element associated to $\underline x^n_i$, and we observe that if $x^n_i\in S^{-\eps}$, then $B_{t_n}(x^{n}_i,\eps) \subset \xi(x^{n}_i)$. We may also write
\[
\xi_n := \{ \xi(x^n_i): \underline x^n_i\in E_n \}.
\]
\end{remark}

\begin{remark} \label{rem:overflow}
We do not require $\xi(x^n_{i}) \subset S_n$. We expect that in many cases, $\xi(x^n_{i})$ will `overflow' and intersect both the sets $\bigcup_{m>n} S_m$ and $\bigcup_{m<n} S_m$; see Figure \ref{fig:graded-sliced}.
\end{remark}

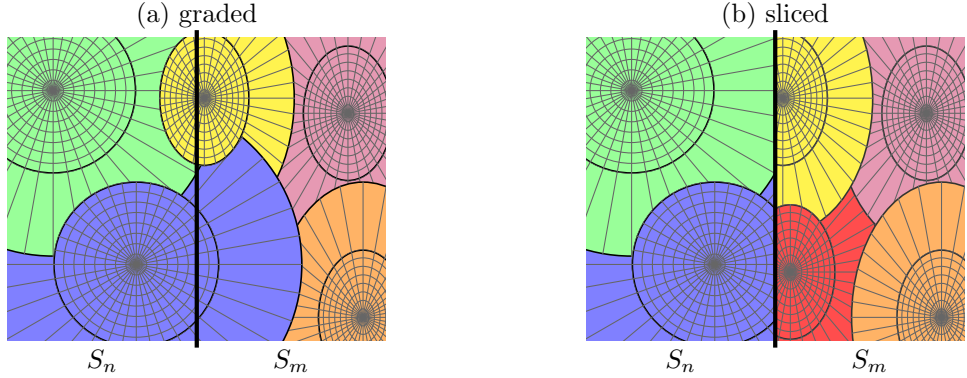
\begin{figure}[htbp]
\def\npts{%
 (1.2,1) /blue!50,
 (.1,3.3) /green!40}
\def\mpts{%
 (4,3) /purple!40,
 (2.1,3.2) /yellow!80,
 (4.2,0.3) /orange!60}
\edef\xpts{%
 (2.2,0.9) /red!70, \mpts}
\def\outern{ circle(2.2); }
\def\midn{ circle(1.1); }
\def\innern{ circle(\t*1.1); }
\def\outerm{ ellipse(1.2 and 1.8 ); }
\def\midm { ellipse(.6 and .9); }
\def\innerm{ ellipse(\t*.6 and \t*.9); }
\def\spokes{ \foreach \t in {0,10,...,350} {\draw[black!60] \p -- ++(\t:3);} }

\begin{tikzpicture}
\begin{scope}[draw=black]
\clip (-.5,0) rectangle (4.5,4);
\foreach \p / \c in \mpts
{ \begin{scope} \clip \p \outerm;
\filldraw[fill=\c, very thick] \p \outerm;
\spokes \end{scope} }

\foreach \p / \c in \npts
{ \begin{scope} \clip \p \outern;
\filldraw[fill=\c, very thick] \p \outern;
\spokes \end{scope} }

\foreach \p / \c in \npts
{ \begin{scope} \clip \p \midn;
\filldraw[fill=\c, very thick] \p \midn;
\foreach \t in {0,.125,...,1} {\draw[black!60] \p \innern;}
\spokes \end{scope} }

\foreach \p / \c in \mpts
{ \begin{scope} \clip \p \midm;
\filldraw[fill=\c, very thick] \p \midm;
\foreach \t in {0,.125,...,1} {\draw[black!60] \p \innerm;}
\spokes \end{scope} }
\end{scope}

\draw[ultra thick] (2,-.1) -- (2,4.1);
\draw (0.75,0) node[below]{$S_n$}
(3.25,0) node[below]{$S_m$};
\node[above] at (2,4) {(a) graded};
\end{tikzpicture}
\hfill
\begin{tikzpicture}
\begin{scope}[draw=black]
\clip (-.5,0) rectangle (2,4);
\foreach \p / \c in \npts
{ \begin{scope} \clip \p \outern;
\filldraw[fill=\c, very thick] \p \outern;
\spokes \end{scope} }

\foreach \p / \c in \npts
{ \begin{scope} \clip \p \midn;
\filldraw[fill=\c, very thick] \p \midn;
\foreach \t in {0,.125,...,1} {\draw[black!60] \p \innern;}
\spokes \end{scope} }
\end{scope}

\begin{scope}[draw=black!80]
\clip (2,0) rectangle (4.5,4);
\foreach \p / \c in \xpts
{ \begin{scope} \clip \p \outerm;
\filldraw[fill=\c, very thick] \p \outerm;
\spokes \end{scope} }

\foreach \p / \c in \xpts
{ \begin{scope} \clip \p \midm;
\filldraw[fill=\c, very thick] \p \midm;
\foreach \t in {0,.125,...,1} {\draw[black!60] \p \innerm;}
\spokes \end{scope} }
\end{scope}

\draw[ultra thick] (2,-.1) -- (2,4.1);
\draw (0.75,0) node[below]{$S_n$}
(3.25,0) node[below]{$S_m$};
\node[above] at (2,4) {(b) sliced};
\end{tikzpicture}
\caption{Graded and sliced adapted partitions. Dots represent initial points of orbit segments in $\mathcal{E}$, and the concentric circles or ellipses around each dot represent the ``core'' $B_{t_i}(x_i,\zeta)$ that must be attached to it. These cores must be disjoint in the first picture, but not the second. The picture is simplified: $S_n$ is depicted as connected and adjacent to $S_{m}$, neither of which need actually be true, and the shape of the Bowen balls reflects expansion only, ignoring any stable directions.}
\label{fig:graded-sliced}
\end{figure}

\subsection{Sliced adapted partitions} \label{sec.slicedadapted} 

The `graded' adapted partitions from the previous section will play an important role in the proof of uniqueness of the equilibrium state. For the proof of existence, we will need another kind of partition that avoids the `overflowing' behavior mentioned in Remark \ref{rem:overflow} and illustrated in Figure \ref{fig:graded-sliced}(a). The idea is to take a maximal $(t_n, \zeta)$-separated set for each of the $S_n$ and `slice' the Bowen balls by taking their intersection with $S_n$. The slicing is a simple way to ensure that Bowen balls at each level do not intersect the Bowen balls at other levels, however by intersecting with $S_n$ (not just with $S$ as allowed by \eqref{adapted-inclusion0}) we do not obtain an adapted partition in the sense of Definition \ref{def:adapted-part}.

\begin{definition} \label{def:slicedmaximalinduced}
Let $(S,\tau)$ be a graded set and $\mathcal{S}$ as in \eqref{eqn:SSS}. Given a collection $\mathcal{E} \subset \mathcal{S}$, we say that $\mathcal{E}$ is \emph{maximal $\zeta$-separated on slices} for $(S,\tau)$ if for every $n\in \NN$, the collection $\mathcal{E} \cap (S_n \times \{t_n\})$ is maximal $\zeta$-separated for $S_n\times \{t_n\}$.
\end{definition}

We emphasize that a collection $\mathcal{E}$ satisfying Definition \ref{def:slicedmaximalinduced} need not be $\zeta$-separated in the sense of Definition \ref{def:eps-sep}, since Bowen balls associated to orbit segments with different lengths are allowed to intersect.

\begin{definition}\label{def:adapted} 
Let $(S,\tau)$ be a graded set with $\mathcal{S}$ as in \eqref{eqn:SSS}, and let $\mathcal{E} \subset \mathcal{S}$ be a countable collection of orbit segments. Given $\zeta>0$ and $k_2>k_1>0$,  we say that a partition $\xi$ of $S$ is a \emph{$\tau$-sliced $(k_1 \zeta, k_2 \zeta)$-adapted partition} for $\mathcal{E}$ if we can write $\xi = \{w_j\}_{j\in \NN}$ and $\mathcal{E} = \{(x_j,m_j)\}_{j\in \NN}$ so that for every $j\in \NN$, we have 
\begin{equation}\label{adapted-inclusion-sliced}
B_{m_j}(x_j,k_1\zeta) \cap S_{m_j} \subset w_j \subset \overline B_{m_j}(x_j,k_2\zeta) \cap S_{m_j}.
\end{equation}
\end{definition}

Given a graded set $(S,\tau)$, with $\bS$ and $\bt$ as in \eqref{eqn:St}, we write $\partial\bS := \bigcup_{n\in \NN} \partial S_n$. Similarly, given a partition $\xi = \{w_j\}_{j\in \NN}$, we write $\partial\xi := \bigcup_{j\in \NN} \partial w_j$.

\begin{proposition}\label{prop:slicedadapted}
Let $(S,\tau)$ be a graded set with $\mathcal{S}$ as in \eqref{eqn:SSS}, and suppose $m\in \Mf$ satisfies $m(\partial\bS)=0$.
Let $\mathcal{E} \subset \mathcal{S}$ be maximal $4\zeta$-separated on slices. Then there exists a partition $\xi$ of $S$ which is a $\tau$-sliced $(\zeta,5\zeta)$-adapted partition for $\mathcal{E}$ satisfying $m(\partial \xi)=0$.
\end{proposition}
\begin{proof}
Let $\mathcal{E}= \bigcup_{n\in\NN} \mathcal{E}_n$ with $\mathcal{E}_n=\{(x^n_{i}, t_n)\}_{i\in I_n}$ a maximal $4\zeta$-separated set for $S_n$.
For a fixed $n$, the balls $B_{t_n}(x^n_{i},2\zeta)$ are mutually disjoint by Lemma \ref{lem:disjoint}. By Lemma \ref{lem:septospan}, $\mathcal{E}_n$ is $4\zeta$-spanning for $S_n$. 
For each $i\in I_n$, we choose  constants $r^n_{i}\in (\zeta,2\zeta]$ and $R^n_{i}\in (4\zeta,5\zeta)$ for each $i$, such that
\[
m(\partial B_{t_n}(x^n_{i},r^n_{i})\cup\partial B_{t_n}(x^n_{i},R^n_{i}))=0.
\]
Observe that $B_{t_n}(x^n_{i},r^n_{i})$ are mutually disjoint, and $S_n\subset \bigcup_{i}B_{t_n}(x^n_{i},R^n_{i})$. We construct the partition $\xi^n=\{w^n_{i}\}$ for $S_n$ iteratively by defining 
\[
w^n_{1}:=
S_n\cap  \overline B_{t_n}(x^n_{1},R^n_{1})\setminus\Big(\bigcup_{j\geq 2}B_{t_n}(x^n_{j},r^n_{j})\Big),
\]
and given $(w^n_{j})_{j=1}^{i-1}$, we define
\[
w^n_{i}:=S_n\cap \overline B_{t_n}(x^n_{i},R^n_{i})\setminus\Big(\bigcup_{j< i}w^n_{j}\cup\bigcup_{j>i}B_{t_n}(x^n_{j},r^n_{j})\Big).
\]
It  follows immediately from the construction that 
$\xi := \{ w_i^n : n\in \NN, i\in I_n \}$ is a $\tau$-sliced $(\zeta,5\zeta)$-adapted partition for $\mathcal{E}$. Furthermore, for every $n,i$, we have
\begin{multline*}
    m(\partial w^n_{i})
    \leq m(\partial S_n)+m(\partial B_{t_n}(x^n_{i},R^n_{i})) \\
    +\sum_{j< i}m(\partial w^n_{j})+\sum_{j>i}m(\partial B_{t_n}(x^n_{j},r^n_{j}))=\sum_{j< i}m(\partial w^n_{j}).
\end{multline*}
Since $m(\partial w^n_{1})\leq m(\partial S_n)+m(\partial B_{t_n}(x^n_{1},R^n_{1}))+m(\sum_{j\geq 2}\partial B_{t_n}(x^n_{j},r^n_{j}))=0$, we see that $m(\partial w^n_{i})=0$ for all $i$.
\end{proof}

We note that scales and argument for the above lemma would simplify if we did not need to obtain $m(\partial \xi)=0$. It is easy to verify the following statement.

\begin{proposition} \label{prop:sliceadaptednormal}
Let $(S,\tau)$ be a graded set and $\mathcal{S}$ be as in \eqref{eqn:SSS}. Let $\mathcal{E} \subset \mathcal{S}$ be maximal $2\zeta$-separated on slices. 
Then there exists a partition $\xi$ of $S$ which is a $\tau$-sliced $(\zeta,2\zeta)$-adapted partition for $\mathcal{E}$.
\end{proposition}

\subsection{Adapted partitions for returns of time-$1$ maps} \label{sec:Nthreturns}

The above formalism can be applied to $F=f_1$, the time-$1$ map of the flow.

Given $A\in \RRR$ and $x\in A$, enumerate the set of return times by
\begin{equation}\label{eqn:btau}
\btau_A(x) := \{ k \in \NN : F^k(x) \in A \}
= \{ \tau_A^1(x) < \tau_A^2(x) < \tau_A^3(x) < \cdots \}.
\end{equation}
Thus $\tau_A^n(x)$ denotes the $n$th return time to the set $A$, provided it exists: for some $x\in A$, the set $\btau_A(x)$ may be finite, or even empty. However, writing
\[
\ret{A}{N} := \{ x\in A : \#\btau_A(x) \geq N \}
\]
for the set of points in $A$ that have at least $N$ returns to $A$, Poincar\'e recurrence gives $m(\ret{A}{N})= m(A)$ for every $m\in \Mf$ and $N\geq 1$. 

Writing $S = \ret{A}{1}$ and $\tau=\tau_A^1$, we obtain an $\NN$-graded set $(S,\tau)$.
For each $n\in \NN$, we will write
\[
A_n^* := \{ x\in A : \tau_A^1(x) = n \}
\]
for the corresponding slice in this graded set. (This is the same as the set $S_n$ in our more general notation.)

\begin{lemma}\label{lem:bdry-null}
If $(S,\tau) = (A^1,\tau_A^1)$ is the graded set associated to $A \in \RRR$, then for every $m\in \Mf$ satisfying $m(\partial A) = 0$, we have $m(\partial\bS) = 0$, where $\bS$ is as in \eqref{eqn:St}.
\end{lemma}
\begin{proof}
Since $\partial\bS = \bigcup_{n\in\NN} \partial A_n^*$, it suffices to show that $m(\partial A_n^*)=0$ for all $n\in\NN$. To this end, we observe that
\[
A_n^* = A \cap F^{-1}(X\setminus A) \cap F^{-2}(X\setminus A) \cap \cdots \cap F^{-(n-1)}(X\setminus A) \cap F^{-n}(A),
\]
from which we deduce that
\[
\partial A_n^* \subset \bigcup_{i=0}^n \partial(F^{-i} A).
\]
For every $i\geq 0$, we have
\[
m(\partial(F^{-i} A)) = m(F^{-i}(\partial A)) = m(\partial A) = 0,
\]
where the first equality is because $F$ is a homeomorphism, and the second is because $m$ is $F$-invariant. This proves the lemma.
\end{proof}

Combining Proposition \ref{prop:slicedadapted} and Lemma \ref{lem:bdry-null} proves the following.

\begin{lemma} \label{lem:nullboundary}
 Let $A\in \RRR$ and let $m\in \mathcal{M}_{F}$ with $m(A)>0$ and $m (\partial A)=0$.  Let $\mathcal{E}$ be maximal $4\zeta$-separated on slices for $(\ret{A}{1}, \tau^1_A)$. Then there is a $\tau^1_A$-sliced $(\zeta,5\zeta)$-adapted partition $\xi_A^1$ for $\mathcal{E}$ of $\ret{A}{1}$ satisfying $m(\partial \xi_A^1)=0$.
 \end{lemma}

We define
\begin{equation} \label{eqn:xiA}
\xi=\xi(A) := \xi_A^1 \cup \{X\setminus \ret{A}{1}\}.
\end{equation}
We observe from Lemma \ref{lem:nullboundary} that $\xi$ is a partition for $X$ satisfying $m(  \partial \xi) =0$. Using our assumption that $h=h_{GS}<\infty$, we will show that the partition $\xi$ has finite entropy for a broad class of measures, extending \cite[Proposition 2.11]{fR18} and \cite[Proposition 6.3]{fL13}.

To control the entropy of $\xi$, start by considering the partition
\begin{equation}\label{eqn:alpha}
\alpha := \{ X\setminus \ret{A}{1}, A_1^*, A_2^*, A_3^*, \dots \},
\end{equation}
which is refined by $\xi$. We will obtain bounds on $H_\nu(\alpha)$ and $H_\nu(\xi|\alpha)$ in terms of $\sum_{n\in \NN} n\nu(A_n^*)$. First, we need some elementary facts.

\begin{definition}\label{def:HE}
A \emph{probability vector} on $\NN_0$ is $\bp = (p_0, p_1, p_2, \dots)$ such that $p_n \geq 0$ and $\sum_{n=0}^\infty p_n = 1$. Given a probability vector $\bp$, the \emph{entropy} and \emph{mean} of $\bp$ are, respectively,
\begin{equation}\label{eqn:HE-def}
H(\bp) := \sum_{n\in \NN_0} -p_n \log p_n
\quad\text{and}\quad
M(\bp) := \sum_{n\in \NN} n p_n.
\end{equation}
\end{definition}

One or both of the quantities in \eqref{eqn:HE-def} may be infinite. However, they satisfy the following relationship.

\begin{lemma}\label{lem:HE}
Given any probability vector $\bp$ on $\NN_0$, we have
\begin{equation}\label{eqn:HE}
H(\bp) \leq (\log 2) (1+M(\bp)).
\end{equation}
\end{lemma}
\begin{proof}
Observe that for every probability vector $\bq$, the inequality $\log t \leq t-1$ gives
\[
\sum_{n\in \NN_0} (-p_n \log p_n + p_n \log q_n)
= \sum_{n\in \NN_0} p_n \log \frac{q_n}{p_n}
\leq \sum_{n\in \NN_0} p_n \Big( \frac{q_n}{p_n} - 1 \Big) = 0.
\]
Taking $q_n = 2^{-(n+1)}$ for all $n\in \NN_0$ gives a probability vector for which we can use this inequality to conclude that
\[
H(\bp) \leq \sum_{n\in \NN_0} -p_n \log q_n
= \sum_{n\in \NN_0} p_n \cdot (n+1) \log 2
= (M(\bp)+1) \log 2.\qedhere\]
\end{proof}

Now we enumerate the elements of the partition $\alpha$ from \eqref{eqn:alpha} as $w_0 = X\setminus \ret{A}{1}$ and $w_n = A_n^*$ for $n\in \NN$. 
Given any Borel probability measure $\nu$ on $X$, we can apply Lemma \ref{lem:HE} to the probability vector given by $p_n = \nu(w_n)$ to obtain
\begin{equation}\label{eqn:H-alpha}
H_\nu(\alpha) \leq (\log 2)\Big( 1 + \sum_{n\in \NN} n \nu(A_n^*) \Big).
\end{equation}

\begin{proposition}\label{prop:fin-ent-1}
Let $\tilde\xi$ be any $\tau_A^1$-sliced $(\zeta,5\zeta)$-adapted partition for $\mathcal{E}$ of $\ret{A}{1}$.
Let $\xi = \tilde\xi \cup \{X \setminus \ret{A}{1}\}$, and let $\alpha$ be as in \eqref{eqn:alpha}. Then for every $h' > h = h_{GS}$, there exists $C_5>0$ such that for every Borel probability measure $\nu$ on $X$, we have
\begin{equation}\label{eqn:Hnu-xi}
H_\nu(\xi) \leq \log(2C_5) + (h' + \log 2) \sum_{n\in \NN} n\nu(A_n^*).
\end{equation}
In particular, if $\nu$ is $F$-invariant, then Kac's lemma gives
\begin{equation}\label{eqn:Hnu-inv}
H_\nu(\xi) \leq \log(2C_5) + h' + \log 2 < \infty.
\end{equation}
\end{proposition}
\begin{proof}
We have $H_\nu(\xi) = H_\nu(\alpha) + H_\nu(\xi|\alpha)$, and since the first quantity is controlled by \eqref{eqn:H-alpha}, we consider the second. Recall that
\begin{equation}\label{eqn:H-xi-alph}
H_\nu(\xi|\alpha) = \sum_{n\in \NN_0} \nu(w_n) H_{\nu_{w_n}}(\xi),
\end{equation}
where $\nu_{w_n}$ is the normalized restriction of $\nu$ to $w_n$.

As in Lemma \ref{lem:nullboundary}, every element of the partition $\xi$ (except for $X\setminus \ret{A}{1}$) is associated to an orbit segment in $\mathcal{E}$, which is maximal $4\zeta$-separated on slices for $(\ret{A}{1},\tau_A^1)$.
Since $F^n(A_n^*) \subset A$, from the definition of $h=h_{GS}$ we see that there exists $C_5>0$ such that each cell $w_n$ of $\alpha$ contains at most $C_5 e^{nh'}$ cells of $\xi$. It follows that
\[
H_{\nu_{w_n}}(\xi) \leq \log(C_5 e^{nh'}) = nh' + \log C_5.
\]
Combining this with \eqref{eqn:H-xi-alph} and writing $M = \sum_{n\in \NN} n\nu(A_n^*)$ gives
\[
H_\nu(\xi|\alpha) \leq \sum_{n\in \NN_0} \nu(w_n)(nh' + \log C_5) = Mh' + \log C_5.
\]
Since \eqref{eqn:H-alpha}
gives $H_\nu(\alpha) \leq (\log 2)(1+M)$, we now have
\[
H_\nu(\xi) \leq \log 2 + (\log 2)M + Mh' + \log C_5,
\]
which proves \eqref{eqn:Hnu-xi}. For \eqref{eqn:Hnu-inv}, it suffices to observe that Kac's lemma gives $M=1$ whenever $\nu$ is $F$-invariant.
\end{proof}

We require one further entropy estimate along the lines of Proposition \ref{prop:fin-ent-1}.

\begin{lemma}\label{lem:eta-n}
With $\alpha$ as in \eqref{eqn:alpha} and $\xi$ as in Proposition \ref{prop:fin-ent-1}, consider for each $n\in \NN$ the set
\begin{equation}\label{eqn:Zn}
Z_n := \bigcup_{k>n} A_k^* = \{x\in \ret{A}{1} : \tau_A^1(x) > n \}
\end{equation}
and the finite partition
\begin{equation}\label{eqn:eta-n}
\eta_n := \{Z_n\} \cup \{ u \in \xi : u \cap Z_n = \emptyset \}.
\end{equation}
Suppose that $h' \geq 0$ and $n\in \NN$ are such that $\Lambda(0,k,4\zeta,A_k) \leq e^{kh'}$ for all $k > n$. Then for every Borel probability measure $\nu$ on $X$, we have
\begin{equation}\label{eqn:H-eta}
H_\nu(\xi | \eta_n)
\leq \nu(Z_n) \log 2 + (h' + \log 2) \sum_{k > n} k \nu(A_k^*) =: r_n(\nu).
\end{equation}
\end{lemma}
\begin{proof}
For every $u\in \eta_n \setminus \{Z_n\}$, we have $u\in \xi$, so $H_{\nu_u}(\xi) = 0$, and we have
\[
H_\nu(\xi|\eta_n) = \sum_{u\in \eta_n} \nu(u) H_{\nu_u}(\xi)
= \nu(Z_n) H_{\nu_{Z_n}}(\xi)
= \nu(Z_n) \big( H_{\nu_{Z_n}}(\alpha) + H_{\nu_{Z_n}}(\xi|\alpha) \big).
\]
Writing $s_n := \sum_{k>n} k\nu(A_k^*)$, Lemma \ref{lem:HE} gives
\[
H_{\nu_{Z_n}}(\alpha)
\leq \log 2 + (\log 2) \sum_{k>n} k \nu_{Z_n}(A_k^*)
= \log 2 + (\log 2) \sum_{k>n} k \frac{\nu(A_k^*)}{\nu(Z_n)}
= \frac{s_n \log 2}{\nu(Z_n)}.
\]
Combining the previous two computations gives
\begin{equation}\label{eqn:H-nu-partway}
H_\nu(\xi|\eta_n) \leq (\nu(Z_n) + s_n) \log 2 + \nu(Z_n) H_{\nu_{Z_n}}(\xi|\alpha).
\end{equation}
Arguing as in Proposition \ref{prop:fin-ent-1}, we see that for each $k>n$, the cell $w_k \in \alpha$ contains at most $e^{kh'}$ cells of $\xi$, so
$H_{\nu_{w_k}}(\xi) \leq \log e^{kh'} = kh'$, and we have
\[
H_{\nu_{Z_n}}(\xi|\alpha)
= \sum_{k>n} \nu_{Z_n}(w_k) H_{\nu_{w_k}}(\xi)
\leq \sum_{k>n} \frac{\nu(A_k^*)}{\nu(Z_n)} kh'
= \frac{h' s_n}{\nu(Z_n)}.
\]
Together with \eqref{eqn:H-nu-partway}, this implies that
\[
H_\nu(\xi|\eta_n)
\leq (\nu(Z_n) + s_n) \log 2 + h' s_n,
\]
which proves \eqref{eqn:H-eta}.
\end{proof}

\subsection{Misiurewicz measures are equilibrium states}\label{sec:Bow-Mis-ES}

We now have the tools we need to prove Theorem \ref{thm:Mis-ES}.
This part of the proof generalizes the standard argument
that Misiurewicz's construction gives equilibrium states, using a sliced adapted partition for return times together with the finite entropy results in the previous section. Recall from \S\ref{sec:Bow-Mis-exists}
that the Misiurewicz  measure $\mu$ is the weak*-limit of a sequence of measures $\mu_{t_k}$, where $\mu_t$ is given by  \eqref{eqn:sigma-Et} in terms of a maximal $(t,\eps_0)$-separated set $E_t \subset (A_0)_t$. 

We now work in discrete-time.
Consider the integers $n_k := \lceil t_k \rceil$.
We define a probability vector on each $E_{t_k}$ by
\[
p_x^k = \frac{e^{\Phi(x,n_k)}}{\Lambda(\psi,n_k,E_{t_k})},
\]
we define a discrete-time Misiurewicz sequence by:
\begin{equation}\label{eqn:disc-Mis}
\bsig_k := \sum_{x\in E_{t_k}} p_x^k \delta_x
\quad\text{and}\quad
\bmu_k := \frac 1{n_k} \sum_{i=0}^{n_k -1} F_*^i \bsig_k.
\end{equation}
We assume without loss of generality that the sequence $\bmu_k$ weak* converges to a limit $\bmu$.  If not, we use tightness to pass to a convergent subsequence and reindex. Then 
\begin{equation} \label{eq:mumu1}
    \mu=\lim_{k\to \infty}\mu_{t_k}=\lim_{k\to \infty}\int_0^1(f_s)_* \bmu_k ds=\int_0^1(f_s)_* \bmu \,ds.
\end{equation}
We fix a reference set $A \in \RRR$ that contains a neighborhood of $f_{[-1,1]}A_0$ and satisfies
\begin{equation}  \label{eqn:emptyboundary}
\bmu(A)>0, \quad \bmu(\partial A)=0.
\end{equation}
Let $\xi_A^1$ be any $\tau_A^1$-sliced $(\eps_0, 5 \eps_0)$-adapted partition $\xi_A^1$ for $\ret{A}{1}$ with $\bmu(\partial \xi_A^1)=0$; such a partition exists by Lemma \ref{lem:nullboundary}. As in \eqref{eqn:xiA}, let $\xi= \xi_A^1 \cup \{X\setminus \ret{A}{1}\}$. For a partition $\xi$ and a point $x\in X$, we write $\xi(x)$ for the partition element which contains $x$.

\begin{lemma}\label{lem:Mis-bounds}
Let $\rho = \eps_0/2$ and let $T_0$ be as in Theorem \ref{thm:uniformgeneral}. Then there exists $C_M>0$ such that for every $k\in \NN$ satisfying $n_k>T_0$, every $i\in \{0,1,\dots, n_k-1\}$, and every $\ell\in \NN$, we have
\begin{equation}\label{eqn:Mis-bounds}
F_*^i \bsig_k (A_\ell^*) \leq C_M \Lambda^*(\ph,\ell,\rho,A_\ell) e^{-\ell P}.
\end{equation}
\end{lemma}
\begin{proof}
Since $\bsig_k(A \cap F^{-{n_k}} A) = 1$, we have
\[
F_*^i \bsig_k(A_\ell^*)
= \bsig_k(F^{-i} A_\ell^*)
= \bsig_k(A \cap F^{-i} A_\ell^* \cap F^{-{n_k}} A).
\]
If $i < {n_k} < i+\ell$, then this set is empty by the definition of $A_\ell^*$. Otherwise, we write $j = {n_k} - i - \ell$ and use the definition %\eqref{eqn:s-mu} 
\eqref{eqn:disc-Mis}
and the submultiplicativity Lemma \ref{lem:submult} with $\rho = \eps_0/2$ (which we can do since $\epso \leq \rmet(A)$) to get
\[
F_*^i \bsig_k(A_\ell^*)
\leq \frac{\Lambda(\ph,i,\rho,A_i)
\cdot \Lambda^*(\ph,\ell,\rho,A_\ell)
\cdot \Lambda(\ph,j,\rho,A_j)}
{\Lambda(\ph,{n_k},\eps_0,A_{n_k})}
\]
Now \eqref{eqn:Mis-bounds} follows from the upper and lower bounds in \eqref{eqn:CU} and \eqref{eqn:unif-counting}.
\end{proof}

When $\ph$ is $P^*$-SPR for some $P^* \in [0,P)$, Lemma \ref{lem:Mis-bounds} and \eqref{eqn:CF} give
\begin{equation}\label{eqn:beta-decay}
F_*^i \bsig_k(A_\ell^*) \leq C_M C_F e^{nP^*} e^{-nP} = C_M C_F e^{-\ell\beta},
\end{equation}
for every ${n_k} > T_0$, every $i\in \{0,1,\dots, {n_k}-1\}$, and every $\ell\in \NN$,
where $\beta = P - P^* > 0$. Combining this with Proposition \ref{prop:fin-ent-1}, we see that fixing $h' > h = h_{GS}$, there exists $C_5>0$ such that for every such ${n_k}$ and $i$, we have
\begin{equation}\label{eqn:H0}
H_{\bsig_k}(F^{-i} \xi)
= H_{F_*^i \bsig_k}(\xi)
\leq \log(2C_5) + (h' + \log 2) \sum_{\ell=1}^\infty \ell C_M C_F e^{-\ell\beta} =: H_0.
\end{equation}
Observe that $H_0$ is independent of ${n_k}$ and $i$. We are now in a position to prove that every Misiurewicz measure is an equilibrium state.

\begin{proof}[Proof of Theorem \ref{thm:Mis-ES}]
By \eqref{eq:mumu1}, to show that $\mu$ is an equilibrium state, it suffices to prove that
\begin{equation}\label{eqn:ES-geq}
h_{\bmu}(F,\xi) + \int \psi\,d\bmu \geq P.
\end{equation}
Recall that
\begin{equation}\label{eqn:lim-q}
h_{\bmu}(F,\xi)
= %\lim_{q\to\infty} 
\inf_{q\in \NN} \frac 1q H_{\bmu} (\xi^q),
\quad\text{where }
\xi^q := \bigvee_{i=0}^{q-1} F^{-i} \xi.
\end{equation}
For every $k\in \NN$, each cell of $\xi^{n_k}$ contains at most one point in $E_{t_k}$, and thus
\begin{equation}\label{eqn:H-sum}
H_{\bsig_k}(\xi^{n_k}) + \int \Phi(x,n_k) \,d\bsig_k(x)
= \sum_{x\in E_{t_k}} \big(-p_x^k \log p_x^k + p_x^k \Phi(x,n_k) \big).
\end{equation}
By our choice of $p_x^k$, 
we have
\begin{equation}\label{eqn:max-sum}
\sum_{x\in E_{n_k}} \big(-p_x^k \log p_x^k + p_x^k \Phi(x,n_k) \big)
= \log \Lambda(\ph,n_k,E_{t_k})
\end{equation}
Recalling the lower bound in Theorem \ref{thm:uniformgeneral}, we see that when $k$ is sufficiently large that $t_k \geq T_0$, we have
\begin{equation}\label{eqn:Lnt-geq}
\log \Lambda(\ph,n_k,E_{t_k})
\geq \log \Lambda(\ph,t_k,E_{t_k}) - \log \|\ph|_A\|
\geq t_k P - \log (C_L^{-1} \|\ph|_A\|)
\end{equation}
Writing $C = \log (C_L^{-1}\|\ph_A\|)$, we combine \eqref{eqn:H-sum}--\eqref{eqn:Lnt-geq} to get
\begin{equation}\label{eqn:H+geq}
H_{\bsig_k}(\xi^{n_k}) + \int \Phi(x,n_k) \,d\bsig_k(x)
\geq t_k P - C.
\end{equation}
Rearranging and dividing both sides by ${n_k}$ gives
\begin{equation}\label{eqn:P+leq}
\begin{aligned}
\frac{t_k}{n_k} P - \frac{C}{n_k}
&\leq \frac 1{n_k} H_{\bsig_k}(\xi^{n_k})
+ \int \frac 1{n_k} \Phi(x,n_k) \,d\bsig_k(x) \\
&= \frac 1{n_k} H_{\bsig_k}(\xi^{n_k})
+ \int \psi \,d\bmu_{k}.
\end{aligned}
\end{equation}
Since $\bmu_k \to \bmu$ in the weak* topology and we assumed that $\sup \psi<\infty$, we have $\limsup_{k\to \infty}\int \psi d\mu_{k} \leq \int \psi d\mu$. Thus it follows from \eqref{eqn:P+leq} that

\begin{equation}\label{eqn:P-leq-llim}
P \leq \llim_{k\to\infty} \frac 1{n_k} H_{\bsig_{k}} (\xi^{{n_k}}) + \int \psi \,d\bmu.
\end{equation}
All that remains for us to prove is that
\begin{equation}\label{eqn:q-and-ell}
\text{for every $q\in \NN$, we have }
\frac 1q H_{\bmu}(\xi^q)
\geq \llim_{k\to\infty} \frac 1{{n_k}} H_{\bsig_{k}}(\xi^{{n_k}}),
\end{equation}
since combining \eqref{eqn:q-and-ell} with \eqref{eqn:P-leq-llim} and \eqref{eqn:lim-q} proves \eqref{eqn:ES-geq}.

To carry this out, we will fix $q\in \NN$ and consider the measures
\[
\nu_k^q := \frac 1{{n_k}-q} \sum_{i=0}^{{n_k}-q-1} F_*^i \bsig_k, 
\]
which will be needed in \eqref{eqn:need-nuq}.
Since $\bmu_{k} \to \bmu$,
we have $\nu_{{k}}^{q} \to \bmu$ for each fixed $q\in \NN$. We prove \eqref{eqn:q-and-ell} in two steps:
\begin{enumerate}
\item relate $\frac 1q H_{\nu_k^q}(\xi^q)$ and $\frac 1{n_k} H_{\bsig_k}(\xi^{n_k})$ via the standard argument of Misiurewicz;
\item relate $H_{\nu_k^q}(\xi^q)$ and $H_{\nu_k^q}(\eta_\ell^q)$, where $\eta_\ell$ is the finite partition from \eqref{eqn:eta-n}.
\end{enumerate}
The second step is necessary because the partition $\xi^q$ is infinite, so we cannot immediately conclude that $H_{\bmu}(\xi^q)$ is the limit of $H_{\nu_k^q}(\xi^q)$. 

To carry out the first step, we start by using concavity of the entropy function $\nu \mapsto H_\nu(\xi^q)$ to obtain
\begin{equation}\label{eqn:from-conc}
H_{\nu_k^q}(\xi^q)
\geq \frac 1{{n_k}-q} \sum_{i=0}^{{n_k}-q-1} H_{F_*^i \bsig_k}(\xi^q)
= \frac 1{{n_k}-q} \sum_{i=0}^{{n_k}-q-1} H_{\bsig_k}(F^{-i}\xi^q),
\end{equation}
Thus we must relate this last quantity to $H_{\bsig_k}(\xi^{n_k})$. The necessary bookkeeping is illustrated in Figure \ref{fig:Mis}.

\begin{figure}[htbp]
\begin{tikzpicture}
\fill[black!25!white] (0,0)--(0,2)--(2,0)--cycle;
\draw (.7,.7) node{$I$};
\fill[black!25!white] (8,2)--(9.5,.5)--(9.5,2)--cycle;
\draw (9,1.5) node{$I$};
\fill[black!25!white] (8,0)--(7.5,.5)--(9.5,.5)--(9.5,0)--cycle;
\draw (8.5,0.25) node{$I$};
\draw (0,0) node[below]{$0$} rectangle (9.5,2);
\draw (0,2) -- (2,0) node[below]{$q$};
\draw (2,2) -- (4,0) node[below]{$2q$};
\draw (4,2) -- (6,0);
\draw (6,2) -- (8,0);
\draw (8,2) -- (9.5,0.5);
\draw[red] (7.5,0) node[below]{${n_k} - q$} -- (7.5,2);
\draw[blue,<->] (10,0) -- (10,2);
\draw[blue] (10,1) node[right]{$q$};
\draw[->] (-.5,2) -- (-.5,1.5) node[below]{$j$};
\draw[->] (0,2.3) -- (1,2.3) node[right]{$i$};
\draw (9.5,0) node[below]{${n_k}$};
\draw (7.5,.5)--(9.5,.5);
\draw (2,1) node{$r=0$} (4,1) node{$r=1$} (6,1) node{$\dots$} (8,1) node{$r=a-1$};
%\draw[blue,thick] (2,0)--(4,0) (2,2)--(4,2);
%\draw[green,thick] (4,0)--(6,0) (4,2)--(6,2);
\end{tikzpicture}
\caption{Misiurewicz's argument.}
\label{fig:Mis}
\end{figure}
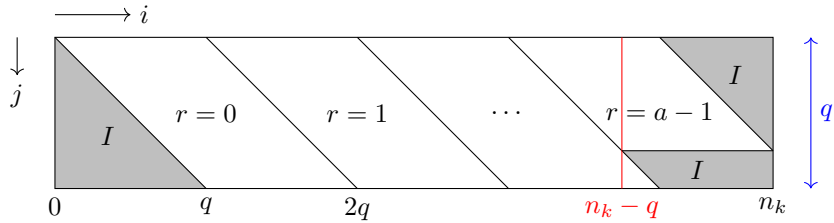

The rectangle in the figure represents $\{0,1,\dots, {n_k}-1\} \times \{0,1,\dots, q-1\}$. Given $j\in \{0,1,\dots, q-1\}$, the different regions the $j$th row passes through indicate a way of representing $\xi^{n_k}$:
writing $a(j) = \lfloor\frac{{n_k}-j}q\rfloor$, we have
\begin{equation}\label{eqn:jth-decomp}
\{0,1,\dots,{n_k}-1\} = I_j \sqcup \bigsqcup_{r=0}^{a(j)-1} (j + rq + \{0,1,\dots,q-1\}),
\end{equation}
where the set $I_j = [0,q) \cap [{n_k} - j-a(j)q, {n_k}) \cap \ZZ$ corresponds to the part of the shaded region $I$ that lies in the $j$th row. 
The decomposition in \eqref{eqn:jth-decomp} gives
\begin{equation}\label{eqn:jth-decomp-1}
\xi^{n_k} = \Big( \bigvee_{i\in I_j} F^{-i} \xi \Big)  \vee \Big( \bigvee_{r=0}^{a(j)-1} F^{-(j+rq)} \xi^q \Big).
\end{equation}
Recalling from \eqref{eqn:H0} that $H_{\bsig_k}(F^{-i} \xi) \leq H_0 < \infty$ for every ${n_k} > i \geq 0$, where $H_0$ is independent of ${n_k}$ and $i$,  we see from \eqref{eqn:jth-decomp-1} that
\begin{equation}\label{eqn:H-j}
H_{\bsig_k}(\xi^{n_k}) \leq (\#I_j) H_0 + \sum_{r=0}^{a(j)-1} H_{\bsig_k}(F^{-(j+rq)} \xi^q).
\end{equation}
Since $\#I_j \leq 2q$ for every $j$, we can sum \eqref{eqn:H-j} over $j\in \{0,1,\dots, q-1\}$ and obtain
\begin{equation}\label{eqn:qHs}
q H_{\bsig_k}(\xi^{n_k}) \leq 2q^2 H_0 + \sum_{i=0}^{{n_k}-q-1} H_{\bsig_k}(F^{-i} \xi^q).
\end{equation}
Dividing both sides of \eqref{eqn:qHs} by $q{n_k}$ and using \eqref{eqn:from-conc} gives
\begin{equation}\label{eqn:ell-H-leq}
\frac 1{n_k} H_{\bsig_k}(\xi^{n_k}) \leq \frac{2qH_0}{{n_k}} + \frac {{n_k}-q}{q{n_k}} H_{\nu_k^q}(\xi^q).
\end{equation}

Now we come to the second step in the proof of \eqref{eqn:q-and-ell}.
Since $\nu_k^q \to \bmu$,
we want to send ${n_k}\to \infty$ in \eqref{eqn:ell-H-leq} and obtain an inequality involving $H_{\bmu}(\xi^q)$. However, matters are complicated by the fact that $\xi$, and hence $\xi^q$, is infinite. 
Thus we must approximate $\xi$ with the finite partitions $\eta_\ell$ from \eqref{eqn:eta-n}. Observe that
\begin{equation}\label{eqn:Hnuq-0}
H_{\nu_k^q}(\xi^q|\eta_\ell^q)
\leq \sum_{j=0}^{q-1} H_{\nu_k^q} (F^{-j} \xi|\eta_\ell^q)
\leq \sum_{j=0}^{q-1} H_{\nu_k^q} (F^{-j} \xi | F^{-j} \eta_\ell).
\end{equation}
The conditional entropy estimate \eqref{eqn:H-eta} from Lemma \ref{lem:eta-n} gives
\begin{equation}\label{eqn:Hnuq-1}
H_{\nu_k^q} (F^{-j} \xi | F^{-j} \eta_\ell) = 
H_{F_*^j \nu_k^q}(\xi|\eta_\ell)
\leq (h' + 2\log 2) \sum_{m>\ell} m F_*^j \nu_k^q(A_m^*).
\end{equation}
For every $m\in \NN$, the tail estimate \eqref{eqn:beta-decay} gives $F_*^i \bsig_k(A_m^*) \leq C_M C_F e^{-m\beta}$ for every ${n_k} >T_0$ and every $i \in \{0,1,\dots, {n_k}-1\}$. Thus for every $0\leq j < q$, we have
\begin{equation}\label{eqn:need-nuq}
F_*^j \nu_k^q(A_m^*)
= \frac 1{{n_k}-q} \sum_{i=0}^{{n_k}-q-1} F_*^{i+j} \bsig_k(A_m^*)
\leq C_M C_F e^{-m\beta}.
\end{equation}
In the previous step,  it was important that we use $\nu_k^q$ and not $\bmu_{k}$, since we do not have control on $F_*^i\bsig_k(A_m^*)$ when $i > {n_k}$.
Combining \eqref{eqn:need-nuq} with \eqref{eqn:Hnuq-0} and \eqref{eqn:Hnuq-1} gives
\begin{equation}\label{eqn:Hnuq}
H_{\nu_k^q}(\xi^q|\eta_\ell^q) \leq q(h'+2\log 2) \sum_{m>\ell} m C_M C_F e^{-m\beta}.
\end{equation}
Writing $R_\ell := (h' + 2\log 2) \sum_{m>\ell} m C_M C_F e^{-m\beta}$, we conclude that
\[
H_{\nu_k^q}(\xi^q) \leq H_{\nu_k^q}(\eta_\ell^q) + qR_\ell,
\]
where $\lim_{n\to\infty} R_\ell = 0$.
Using this together with \eqref{eqn:ell-H-leq} gives
\begin{equation}\label{eqn:leq-3-term}
\frac 1{n_k} H_{\bsig_k}(\xi^{n_k})
\leq \frac{2qH_0}{n_k} + \frac{{n_k}-q}{q{n_k}} H_{\nu_k^q}(\eta_\ell^q) + \frac{{n_k}-q}{{n_k}} R_\ell.
\end{equation}
For every choice of $q,\ell\in\NN$,
the partition $\eta_\ell^q$ is finite, has $\bmu(\partial \eta_\ell^q) \leq \bmu(\partial \xi^q) = 0$, and is refined by $\xi^q$, from which we conclude that
\begin{equation}\label{eqn:H-limit}
\ulim_{k\to\infty} H_{\nu_k^q}(\eta_\ell^q) = H_{\bmu}(\eta_\ell^q) \leq H_{\bmu}(\xi^q).
\end{equation}
Now we can prove \eqref{eqn:q-and-ell}: for every $q,\ell\in \NN$, \eqref{eqn:leq-3-term} and \eqref{eqn:H-limit} give
\[
\ulim_{k\to\infty} \frac 1{{n_k}} H_{\sigma_{{n_k}}}(\xi^{{n_k}})
\leq \frac 1q H_{\bmu}(\xi^q) + R_\ell,
\]
and sending $\ell\to\infty$ establishes \eqref{eqn:q-and-ell} since $R_\ell\to 0$. This completes the proof of %Proposition \ref{prop: MuBeingEquilibrium}.
Theorem \ref{thm:Mis-ES}.
\end{proof}

%%%%%%%%%%%%%%%%%%%%%%

 \section{Uniform Katok estimates}\label{sec:uniform-katok} 

Now we turn our attention to Theorem \ref{thm:mainES}. From now on, we will assume that $X,\FFF,\DDD,\ph$ satisfy all of Conditions \ref{cond:basic}--\ref{cond:UESB}.
By Corollaries \ref{cor:ES-A} and \ref{cor:transferSPR}, we can fix $A_0 \subset \RRR$ such that 
\begin{itemize}
\item $\ph$ is SPR w.r.t.\ $A_0$, and 
\item for every equilibrium state $\nu$, we have $\nu(A_0)>0$.
\end{itemize}
We also consider the reference set
\[
\wtA := \big(f_{[-1,1]} A_0)^{1+L^*},
\]
and let $\rexp(\wtA) \in (0,\rmet(\wtA)]$ be an associated expansivity scale. Fix $\eps \in (0, \frac 12 \rexp(\wtA))$ and $\eps_0 \in (0,\eps/20)$ such that $A_0^{-\eps_0} \neq \emptyset$.

To prove Theorem \ref{thm:mainES}, we will let $\mu$ be a Misiurewicz measure at scale $\eps_0$ for $\ph$ w.r.t.\ $A_0$, noting that $\mu$ is an equilibrium state for $\ph$ by Theorem \ref{thm:existsES}, and we will prove in \S\ref{sec:eu-other} that $\mu$ is in fact the \emph{unique} equilibrium state for $\ph$. In order to do this, we will use the following result, to whose proof this section is devoted.

\begin{theorem}\label{thm:uniform-katok}
Let $X,\FFF,\DDD,\ph$ satisfy Conditions \ref{cond:basic}--\ref{cond:UESB}.
Given $\Azero, \wtA$ as above, there exist $K>0$ and $N_0\in \NN$ such that the following is true.
For every equilibrium state $\nu$ and for every $N\in \NN$, there exists a $\fourepso$-separated collection of orbit segments
\[
\GGG^N = \{ (x_i^N, m_i^N) \}_{i\in I_N} \subset \wtA \times \NN
\]
with the properties that 
\[
\text{for every $i\in I_N$, we have
$x_i^N \in \wtA_{m_i^N}$ and $m_i^N \geq N$},
\]
and that given any flow-invariant set $Y\subset X$ such that $\nu(Y)>0$, there exists $\gamma>0$ such that the following is true: if $Z\subset Y$ is any set satisfying $\nu(\wtA \cap (Y \setminus Z)) < \gamma$, then the set of indices defined by
\begin{equation}\label{eqn:JN}
J_N := \{ i\in I_N : B_{m_i^N}(x_i^N,\tenepso) \cap Z \neq \emptyset \}
\end{equation}
has the property that
\begin{equation}\label{eqn:uniform-katok}
\sum_{i\in J_N} e^{\Phi(x_i^N,m_i^N) - m_i^N P(\ph)}
\geq K
\quad\text{for all } N\geq N_0.
\end{equation}
\end{theorem}

\begin{description}
\item[\S\ref{sec:reference sets}] In Proposition \ref{prop:mainadapted}, we fix a reference set $\Abiggest$ and choose $\GGG^N$, along with a larger collection of orbit segments $\mathcal{E}^N$, 
in such a way that we can construct a partition $\xi_N$ of $\Abiggest$ using a ``hybrid'' of the graded and sliced constructions from \S\ref{sec:eq-st}, producing a family of induced maps $\barfn \colon \Abiggest \to \Abiggest$ with good properties.
\item[\S\ref{sec:Zwei}] In Theorem \ref{thm:Zwei}, we use results of Zweim\"uller \cite{rZ05} to relate a flow-invariant measure on $X$ with a $\barfn$-invariant measure on $\Abiggest$.
\item[\S\ref{sec:Zwei-ref}] We apply the machinery from Zweim\"uller to the induced maps from \S\ref{sec:reference sets} in the case when $\nu$ is an equilibrium state, paying particular attention to the Radon--Nikodym derivative of the resulting measure on $\Abiggest$.
\item[\S\ref{sec:entropyestimate}] In Lemma \ref{lem:adapted-generates},
we show that for an equilibrium state $\nu$,
the partition $\xi_N$ sees all the entropy for the induced measure $\bnun$ on $\Abiggest$:  $h_{\bnun}(\barfn) = h_{\bnun}(\barfn, \xi_N)$.
\item[\S\ref{sec:uniformlowerestimate}] In Proposition \ref{prop.uniformkatok}, we prove a uniform (lower) Katok estimate on partition sums over sets of positive induced measure for an equilibrium state, and then use this to deduce Theorem \ref{thm:uniform-katok}.
\end{description}

\subsection{Partitions of a reference set}  \label{sec:reference sets}

In this section, we work in discrete time, writing $F = f_1$.
Given $\Azero$ and $\wtA$ as above, let
\begin{equation}\label{eqn:MFA0}
\begin{aligned}
\MF^{\Azero} &:= 
\Big\{ \nu \in \MF : \nu\Big( \bigcup_{n=0}^\infty F^n(\Azero)\Big) = 1 \Big \} \\
&= \{ \nu \in \MF : \text{almost every ergodic component of $\nu$} \\
&\qquad\qquad\qquad\qquad\qquad\qquad\text{gives positive weight to $\Azero$}\}.
\end{aligned}
\end{equation}
Fixing $\nu\in \MF^{\Azero}$, which need not be ergodic, we now choose reference sets on which we will construct a family of partitions and an induced map. To this end, fix $\justA,\Abiggest \in \RRR$ such that
\begin{equation}\label{eqn:nesting}
(f_{[-1,1]} \Azero)^{\tenepso}
\subset \justA
\quad\text{and}\quad
\justA^\twentyepso \subset \Abiggest \subset \wtA,
\end{equation}
%and also satisfying
%\begin{equation} \label{eq:newcondition0}
%\nu(\partial A\cup \partial (A^{\tenepso})  \cup \partial \Abiggest)=0.
%\end{equation}
Thus we have four reference sets that we consider: $\Azero \subset \justA \subset \Abiggest \subset \wtA$. 

The intuition for our construction is that for each $N\in \NN$, we define a partition $\xi_N$ of $\Abiggest$ that is graded-adapted where orbits enter $A$ fast enough, and sliced-adapted where they do not; the SPR property will guarantee that this latter part is ``small'' in an appropriate sense. The induced map will be $\barfn = F^{\rho_N}$, where $\rho_N \colon \Abiggest \to \NN \cup \{\infty\}$ is constant on elements of $\xi_N$ and is uniformly bounded in terms of $2N^{\text{th}}$ return times.

When we work with induced transformations, it is convenient to have pressure equal to $0$, so we define a normalized potential by
\begin{equation}\label{eqn:bph}
\bph(x) = \int_0^1 \ph(f_s x)\,ds - P(\ph).
\end{equation}
We write the Birkhoff sums of $\bph$ as
\[
\bPh(x, n) = \sum_{i=0}^{n-1} \bph(F^{i-1}x)
= \Phi(x,n) - nP(\ph),
\]
and use the following notation: given a partition $\xi = \{w_i\}_i$
and a countable collection of orbit segments $\EE=\{(x_i, n_i)\}_i$
such that $x_i \in w_i$ for all $i$, write
\[
\bPh(w_i) = \sup_{x\in w_i}{\overline \Phi (x, n_i)}
\quad\text{and}\quad
\Lambda(\xi, \bph):=\sum_{w_i \in \xi}e^{\bPh(w_i)}.
\]
Given $\Azero,\wtA$ as at the beginning of \S\ref{sec:uniform-katok}, let $Q = Q(\wtA)$ be the corresponding constant in the Bowen property, and $G_L = G_L(\wtA)$ the constant in the lower Gibbs bound from \eqref{eqn:GL} in Theorem \ref{thm:mainGibbsgeneral}\ref{thm:Gibbs}.
Fix $L>0$ sufficiently large that $\wtA \subset (\Azero)^L$, and since the SPR property is satisfied w.r.t.\ $\Azero$, recall from Lemma \ref{lem:SPRnoL} that there exist $C_F>0$ and $P^* < P(\ph)$ such that
\begin{equation}\label{eqn:A0mL}
\Lambda(\ph,m,\fourepso, (\Azero)_{m,L}^*)
\leq C_F e^{mP^*} \quad\text{for all } m\in \NN.
\end{equation}
Now consider the constants
\begin{equation}\label{eqn:betaR*}
\beta := P(\ph) - P^*,
\quad
C_* := \frac{C_F}{1-e^{-\beta}},
\quad\text{and}\quad
R := e^Q (G_L^{-1} + C_*).
\end{equation}

\begin{proposition}\label{prop:mainadapted}%
With $\Azero,\wtA$ as at the beginning of \S\ref{sec:uniform-katok}, let $\beta,C_*,R>0$ be given by \eqref{eqn:betaR*}.
Then given any $\nu \in \MF^{\Azero}$, any $\justA,\Abiggest$ satisfying \eqref{eqn:nesting}, and any $N\in \NN$, there exist:
\begin{itemize}
\item a countable partition $\xi_N = \{w_i^N\}_{i\in I_N^\xi}$ of $\Abiggest$,
\item a collection of orbit segments $\EN = \{(x_i^N,m_i^N)\}_{i\in I_N^\xi} \subset \Abiggest \times \NN$, and
\item a partition of the index set $I_N^\xi = I_N \sqcup I_N^*$,
\end{itemize}
such that the following are true.
\begin{enumerate}[label=\upshape{(H\arabic{*})},leftmargin=*,itemsep=0.5ex]
\item\label{H1} For all $i \in I_N^\xi$, we have $x_i^N \in w^{N}_{i}\subset B_{m^{N}_{i}}(x^{N}_{i},\tenepso)$
and $F^{m_i^N}(w_i^N) \subset \Abiggest$.
    
\item \label{H2}  
The sets $\{B_{m_i^N}(x_i^N, \twoepso)\}_{i\in I_N}$ are disjoint,%\foot{Not sure if we actually need the second half of \ref{H2} any more.} 
%and for every $i\in I_N^\xi$ satisfying $w_N^i \cap \Azero \neq \emptyset$, we have $i\in I_N$ and 
%$B_{m^{N}_{i}}(x^{N}_{i},\twoepso)\subset w^{N}_{i}$.

\item\label{H3}  
The function
$\rho_N := \sum_i m_i^N \one_{w_i^N}$ 
satisfies $N\leq \rho_N \leq \tau_{\justA}^{2N}$ on $\Abiggest$.
    
\item  \label{H4}  The function $\rho_N \colon \Abiggest \to \NN \cup \{\infty\}$ is $\nu$-integrable:
$\int_{\Abiggest} \rho_N \,d\nu < \infty$.
    
\item\label{H5}  \label{eq:partitionsumestimatefinal} For all $N \in \NN$, we have $\Lambda(\xi_N, \bph)\leq R$
and $\sum_{i\in I_N^*} e^{\bPh(x_i^N,m_i^N)} \leq C_* e^{-\beta N}$.
\end{enumerate}
\end{proposition}

Since $f^{m_i^N}(w_i^N) \subset \Abiggest$ and $\rho_N|_{w_i^N} \equiv m_i^N$, the proposition provides  the following \emph{induced transformation}:
\begin{equation}\label{eqn:induced}
\barfn \colon \Abiggest \to \Abiggest,
\qquad
\barfn(x) = F^{\rho_n(x)}(x).
\end{equation}

\begin{proof}%[Proof of Proposition \ref{prop:mainadapted}]
Throughout the proof, we will assume that all sets are intersected with  
$\justA^{(\infty)}=\justA^{(\infty)}_F:=\bigcap_{n\geq 1}\bigcup_{k\geq n}F^{-k}A$,
in order to simplify notation. This set has full $\nu$-measure since $\nu \in \MF^{\Azero}$.

Consider the set $\Ap := A^{\tenepso}$,
which satisfies ${\Ap}^{\tenepso} \subset \Abiggest$.
Observe that $\tau_{\Ap}^N(x)\leq \tau_{\justA}^N(x)<\infty$ for $\nu$-a.e.\ $x$ since 
almost every ergodic component of $\nu$ gives positive weight to $\Azero \subset \justA$. By applying Kac's formula $N$ times to each of these ergodic components, we have
\begin{equation} \label{eq:Kacnthreturn}
   \int_{\justA}  \tau_{\Ap}^N \,d\nu \leq   \int_{\justA} \tau_{\justA}^N\,d\nu=N.
\end{equation}
It is possible that $\int_{\Abiggest}  \tau_{\Ap}^N \,d\nu$ could be infinite, 
so to obtain an integrable inducing function, we must proceed more carefully.
Start by partitioning $\Abiggest$ into the two sets
\[
\Agn := \{x\in \Abiggest\setminus \justA : \tau^1_{\justA}(x) > N\}
\quad\text{and}\quad
\Aln := \Abiggest \setminus \Agn.
\]
We will define $\xi_N$ as a sliced adapted partition on $\Agn$, and as a graded adapted partition on $\Aln$.

To deal with $\Aln$,
define a function $\widehat{\rho}\colon \Aln \to \NN$ by
\[
\widehat{\rho}(x)= 
\begin{cases}
\tau_{\Ap}^N(x) & \text{if } x\in \justA, \\
\tau^1_{\justA}(x) + \tau_{\Ap}^N(F^{\tau_{\justA}(x)}x) & \text{if } x\notin \justA.
\end{cases}
\]
By Lemma \ref{lem:maxgradedspans}, there exists a maximal graded $\fourepso$-separated collection of orbit segments $\Eln$ for the pair $(\Aln,\widehat{\rho})$, which we index as
$\Eln=\{(x_{i},m_{i})\}_{i\in I_N}$, observing that $m_{i}=\widehat{\rho}(x_{i})$ for all $i\in I_N$.  The disjointness claim in \ref{H2} follows from Lemma \ref{lem:disjoint}.

By Proposition \ref{prop:adapted}, there exists a partition $\xi_N^G = \{w_i\}_{i\in I_N}$ that is  $(\twoepso,\fourepso)$-adapted to $\Eln$, so
\begin{equation} \label{eqwin'adapted}
B_{m_{i}}(x_{i},\twoepso)\cap \Aln \subset w_{i} \subset \overline B_{m_{i}}(x_{i},\fourepso)\cap \Aln
\end{equation}
for all $i\in I_N$. 
%If there exists $z\in w_i \cap \Azero$, then for every $y\in B_{m_i}(x_i,\twoepso)$, we have
%\[
%D(y,z) \leq D(y,x_i) + D(x_i,z) < \sixepso
%\quad\Rightarrow\quad
%y\in \Azero^{\sixepso} \subset \justA \subset \Aln,
%\]
%which together with \eqref{eqwin'adapted} implies that
%\begin{equation}\label{eqn:A0-subset}
%\text{if $w_i \cap \Azero \neq \emptyset$, then $B_{m_i}(x_i,\twoepso) \subset w_i$}.
%\end{equation}

Given $x\in w_i$, \eqref{eqwin'adapted} gives $d_{m_i}(x, x_i) \leq \fourepso$, 
and so if $k\in\{0, \ldots, m_i\}$ is such that  $F^kx \in \justA$, then $F^k x_i \in \Ap$.
Since $m_i \leq \tau_{\Ap}^N(x_i)$, we conclude that
\begin{equation} \label{eq:returncontrol}
\tau_{\Ap}^N(x_i) \leq \tau_{\justA}^N(x).
\end{equation}
As described in \ref{H3}, we define $\rho_N$ on $\Aln$ by 
\begin{equation} \label{eq:deftaunleqn}
    \rho_N(x):=\widehat{\rho}(x_{i})=m_{i}
\quad\text{ for all } x\in w_i.
\end{equation}
Given any $x\in w_i \cap \justA$, we can now rewrite \eqref{eq:returncontrol} as $\rho_N(x) = \tau_{\Ap}^N(x_i) \leq \tau_{\justA}^N(x)$. 
Given $x\in w_i \setminus \justA$, we can also use \eqref{eq:returncontrol} to get
\[
\rho_N(x)=\tau^1_{\justA}(x_{i})+\tau_{\Ap}^N(F^{\tau^1_{\justA}(x_{i})}x_{i})\leq N+\tau_{\justA}^N(F^{\tau^1_{\justA}(x_{i})}x)\leq \tau_{\justA}^{2N}(x).
\]
We conclude that
\begin{equation} \label{eq:taunxsizeleq}
    N\leq\rho_N(x)\leq  \tau_{\justA}^{2N}(x) \quad \text{ for all }x\in \Aln,
\end{equation}
which verifies \ref{H3} for $\rho_N|_{\Aln}$.

To see that $\rho_N$ is $\nu$-integrable over $\Aln$, we consider for each $j\in \{1,\dots, N\}$ the set $Z_j := \{ x\in \Aln \setminus \justA : \tau_\justA(x) = j\}$, and obtain
\begin{equation} \label{eq:taunintegrable<}
    \begin{aligned}
        \int_{\Aln}\rho_N\,d\nu
        &\leq \int_{\justA}
\tau_{\justA}^{2N}\,d\nu+\int_{\Aln\setminus \justA}\big( \tau^1_{\justA}(x)+\tau_{\justA}^{2N-1}(F^{\tau_{\justA}(x)}x)
\big)\,d\nu(x) \\
        &\leq 2N+N\nu(\Aln)+\sum_{j=1}^N\int_{Z_j} \tau_{\justA}^{2N-1}(F^jx)\,d\nu(x) \\
        &\leq 2N+N\nu(\Aln)+N\int_{\justA}\tau_{\justA}^{2N-1}\,d\nu \\
        &\leq 2N^2+2N<\infty.
    \end{aligned}
\end{equation}
To bound $\Lambda(\xi^G_N, \bph)$, observe that by  Lemma \ref{lem:uniformpartitionsumA} and the Bowen property of $\ph$ over $\Abiggest$ at scale $\fourepso$, we have
\begin{equation} \label{eq:LambdaAleqnupperbound}
\Lambda(\xi^G_N, \bph) \leq e^Q G_L^{-1}
\quad\text{for all } N.
\end{equation}
If $\nu(\Abiggest\setminus \justA)=0$, then $\xi_N^G$ is a partition $\nu$-mod $0$ for $\Abiggest$ and we have shown that it satisfies \ref{H1}--\ref{H5}. In general, we expect that $\nu(\Abiggest\setminus \justA)>0$, and for the rest of the proof we now assume this to be the case.

To produce a suitable partition on $\Agn$, we must used a sliced adapted partition, since the graded adapted partition construction from above would not necessarily give the required integrability. Let
\begin{equation} \label{eq:deftaun>n}
    \rho_N(x):=N+\tau^1_{\Abiggest}(F^N(x)),
\end{equation}
which is the first time after $N$ that the orbit of $x$ enters $\Abiggest$. 
We have
\begin{equation} \label{eq:taunxsize>}
    N\leq \rho_N(x)\leq \tau_{\Abiggest}^{N+1}(x)\leq \tau_{\Acomplement}^{N+1}(x).
\end{equation}
Using the ergodic decomposition of $\nu$, there exist measures $\nu_0,\nu_1 \in \MF$ and $\lambda \in (0,1]$ such that
\begin{itemize}
\item $\nu= (1-\lambda)\nu_0 + \lambda \nu_1$,
\item $\nu_0(\Acomplement)=0$, and
\item almost every ergodic component of $\nu_1$ gives positive weight to $\Acomplement$.
\end{itemize}
Applying Kac's formula to the ergodic components of $\nu_1$, we obtain
\[
\int_{\Acomplement}\tau_{\Acomplement}^{N+1}(x)\,d\nu_1 = N+1.
\]
We see that $\rho_N$ is $\nu$-integrable over $\Agn$ since the above computations yield that
\begin{equation} \label{eq:taunintegrable>}
\int_{\Agn}\rho_N\,d\nu\leq N+1<\infty.
\end{equation}
Let $\Egn:=\{(x_i, m_i)\}_{i\in I_N^*}$ be a maximal $\fourepso$-separated set on slices for  $(\Agn,\rho_N)$, and observe that $m_i=\rho_N(x_i)$. 
Given $m\in \NN$, consider the set
\[
\Agn^m = \{ y\in \Agn : \rho_N(y) = m \}.
\]
By Proposition \ref{prop:sliceadaptednormal}, there exists a partition $\xi_N^* = \{w_i\}_{i\in I_N^*}$ of $\Agn$ that is sliced $(\twoepso,\fourepso)$-adapted for $\Egn$, so that
\begin{equation} \label{eqwin''adapted}
B_{m_i}(x_i,\twoepso)\cap \Agn^m \subset w_{i} \subset \overline B_{m_i}(y_{i},\fourepso)\cap \Agn^m
\quad\text{for all } i\in I_N^*.
\end{equation}
Given $x\in \Agn^m$, we have $x \in \Abiggest \subset \wtA \subset (\Azero)^L$, and similarly for $f_{m}(x)$. At the same time, the definition of $\Agn$ implies that $f_t\notin \Azero$ for all $t\in [0,N]$, and \eqref{eq:deftaun>n} gives $f_t \notin \Azero$ for all $t\in [N,m]$ since $\justA$ contains a neighborhood of $(\Azero)_1^{\mathcal{O}}$, so recalling \eqref{eqn:ATL}, we conclude that
\[
\Agn^m \subset (\Azero)_{m,L}^*.
\]
Writing $I_N^m := \{ i\in I_N^* : m_i^N = m \}$
and using \eqref{eqn:A0mL}, we obtain
\begin{equation} \label{eq:E>npartitionsum}
\sum_{i\in I_N^m} e^{\bPh(x_i, m)}
\leq \Lambda( \ph,m,\fourepso,A^*_{m,L})e^{-mP} \leq C_{F} e^{m(P^*-P)} = C_F e^{-\beta m}.
\end{equation}
Since $I_N^* = \bigsqcup_{m=N}^\infty I_N^m$, we can sum over $m$ to obtain
\begin{equation} \label{eq:LambdaA>nupperbound}
\sum_{i\in I_N^*} e^{\bPh(x_i,m_i)}
\leq \sum_{m=N}^\infty C_F e^{-\beta m}
= \frac{C_F e^{-\beta N}}{1-e^{-\beta}}.
\end{equation}
To conclude the proof of Proposition \ref{prop:mainadapted},
We let $\EN:=\Eln\cup \Egn$ and
let $\xi_N$ be the partition of $\Abiggest$ given by $\xi^G_N \cup \xi^*_N$, indexed by $I_N^\xi := I_N \sqcup I_N^*$.

By \eqref{eqwin'adapted} and \eqref{eqwin''adapted}, $\xi_N$ satisfies \ref{H1}. By the discussion after \eqref{eqwin'adapted}, $\xi_N$ satisfies \ref{H2}.   By construction, $\rho_N$ is constant on each partition element and \ref{H3}  follows from \eqref{eq:taunxsizeleq} and \eqref{eq:taunxsize>}. The integrability condition \ref{H4} follows from \eqref{eq:taunintegrable<} and \eqref{eq:taunintegrable>}, which together give
\begin{equation}\label{eqn:int-bound}
\int_{\Abiggest} \rho_N \,d\nu \leq 2N^2+3N+1.
\end{equation}
Finally, \ref{H5} holds as a consequence of \eqref{eq:LambdaAleqnupperbound} and \eqref{eq:LambdaA>nupperbound}.
\end{proof}

\subsection{Measures for induced maps}\label{sec:Zwei}

The function $\rho_N$ defined in Proposition \ref{prop:mainadapted} 
gives an induced map $\barfn := F^{\rho_N}$ from the reference set $\Abiggest$ to itself, see \eqref{eqn:induced}. 
The inducing times $\rho_N$ are not $N^{th}$ return times and thus the restriction of an invariant measure $m$ to ${\Abiggest}$ is not necessarily $\barfn$-invariant. 
We use results of Zweim\"uller \cite{rZ05}, which we present and sharpen in this section, to obtain an invariant measure for the induced map. We need a quantitative result comparing this measure to $m$, which goes beyond the statement in \cite{rZ05}, but uses the same essential proof ideas. The results in this subsection are abstract ergodic theory. In particular, we do not require that any expansivity or specification properties are satisfied, or even that $X$ be a metric space. We deal with both the ergodic and non-ergodic cases, which is why our condition \eqref{eqn:Bf+m} below is  stronger than the corresponding condition in \cite{rZ05}, which only requires that $m(Y)>0$.

\begin{definition}
Given a probability space $(X,\BBB,m)$ and an invertible measure-preserving transformation $f\colon X\to X$, let
\begin{equation}\label{eqn:Bf+m}
\BBB_f^+(m) := 
\Big\{ Y\in \BBB : m\Big(\bigcup_{n=0}^\infty f^n(Y)\Big) = 1 \Big\}.
\end{equation}
Equivalently, $\BBB_f^+(m)$ consists of those sets $Y\in \BBB$ with the property that almost every ergodic component of $m$ gives positive weight to $Y$.

Given $Y\in \BBB_f^+$, an \emph{inducing time function (mod $m$)} is any function $\tau \colon Y\to \NN \cup \{\infty\}$ such that for $m$-a.e.\ $x\in Y$, we have $\tau(x) < \infty$ and $f^{\tau(x)}(x) \in Y$. 
\end{definition}

The simplest inducing time function to deal with is the first return time
$\tau_Y(x) := \min \{ n\in \NN : f^n(x) \in Y\}$,
which gives the first return map $f^{\tau_Y}$. 
This map preserves the restriction of $m$ to $Y$, and under the assumption that $m(\bigcup_{n=0}^\infty f^n(Y))=1$, this restriction induces the original measure $m$ 
in the sense that for every $\psi \in L^1(X,m)$, we have
\[
\int_X \psi(x) \,dm(x) = \int_Y \Psi(x,\tau_Y(x)) \,dm(x).
\]
In particular, taking $\psi\equiv 1$, we have $\int_Y \tau_Y \,dm = 1$.
Normalizing the restriction of $m$ to $Y$ gives the probability measure $m_Y$ on $Y$ defined by
\[
m_Y(E):=\frac{m(E\cap Y)}{m(Y)},
\]
which is invariant under the first return map.

For more general integrable inducing times, we can recover a similar story, although we must replace ``the restriction of $m$ to $Y$'' by a suitable measure whose definition is not quite so immediate. 
This measure is provided by the following results; we start with the ergodic case, which is mostly due to Zweim\"uller \cite{rZ05}, and then use the ergodic decomposition to extend the result to non-ergodic measures.

\begin{theorem}\label{thm:Zwei}
Let $(X,\BBB,f,m)$ be an ergodic probability measure-preserving transformation. Fix $Y\in \BBB$ such that $m(Y)>0$, and let $\tau \colon Y\to \NN \cup \{\infty\}$ be an inducing time function satisfying $m(\tau) := \int_Y \tau \,dm < \infty$. Then there exists a finite $f^\tau$-invariant measure $m_{\tau}\ll m$ that induces $m$ in the sense that for every $\psi\in L^1(X,m)$, we have
\begin{equation}\label{eqn:induces}
\int_X \psi(x) \,dm(x)=\int_Y \Psi(x,\tau(x)) \,dm_{\tau}(x),
\end{equation}
and in particular, $\int_Y \tau \,dm_\tau = 1$.
Writing $\min\tau:=\min\{\tau(y):y\in Y\}$, the measure $m_\tau$ satisfies
\begin{equation} \label{eqn:RN-control0}
m_\tau \leq \frac{m}{(\min\tau)(m(Y))}.
\end{equation}
For the probability normalization $\overline{m}_{\tau} := \frac{m_{\tau}}{m_{\tau}(Y)}$, we have Abramov's formula:
\begin{equation}\label{eqn:Abramov}
h_{\overline{m}_{\tau}}(f^{\tau})=h_{m}(f)\int_Y \tau \,d\overline{m}_{\tau},
\end{equation}
which together with \eqref{eqn:induces} gives
\begin{equation}\label{eqn:induce-P}
h_{\overline{m}_\tau}(f^\tau) + \int_Y \Psi(x,\tau(x))\,d\overline{m}_\tau(x)
= \Big( \int_Y \tau \,d\overline{m}_\tau \Big) \Big( h_m(f) + \int_X \psi \,dm \Big).
\end{equation}
Now suppose that in addition to the above hypotheses, there exist $M\in \NN$ and $Z\subset Y$ such that $m(Z)>0$ and
\begin{equation}\label{eqn:inducing-return}
\tau(x) \leq \tau_Z^M(x)
\quad\text{for all } x\in Y,
\end{equation}
where $\tau_Z^M(x)$ denotes the $n^{\text{th}}$ return time to $Z$.
Then we have
\begin{equation}\label{eqn:mtY}
m_\tau(Y) \geq \frac 1M m(Z),
\end{equation}
and can conclude from \eqref{eqn:RN-control0} and \eqref{eqn:mtY} that
\begin{equation}\label{eqn:RN-control}
\overline{m}_\tau \leq \frac{M}{(\min\tau)(m(Z))} m.
\end{equation}
\end{theorem}

\begin{remark}
The inducing relationship \eqref{eqn:induces} is equivalent to the statement that for every $E\in \BBB$, we have
\[
m(E) = \sum_{n=0}^\infty m_\tau( \{ x\in Y \cap f^{-n}(E) : \tau(x) > n \}).
\]
Writing $Z_n := \{ x\in Y : \tau(x) > n \}$, this is equivalent to the condition that
\[
m = \sum_{n=0}^\infty f_*^n (m_\tau|_{Z_n}).
\]
\end{remark}

\begin{proof}[Proof of Theorem \ref{thm:Zwei}]
Given $(X,\BBB,f,m)$ and $(Y,\tau)$ as in the statement, the existence of $m_{\tau}$ satisfying \eqref{eqn:induces} and \eqref{eqn:Abramov} is proved in \cite[Theorems 1.1 and 5.1]{rZ05}. 

Since $\int_Y \tau\,dm_\tau = 1$, we have $\int_Y \tau \,d\overline{m}_\tau = (m_\tau(Y))^{-1}$. Multiplying both sides of \eqref{eqn:induces} by this quantity gives
\[
\Big(\int_Y \tau \,d\overline{m}_\tau\Big) \Big( \int_X \psi\,dm \Big)
= \int_Y \Psi(x,\tau(x)) \,d\overline{m}_\tau(x).
\]
Adding this to \eqref{eqn:Abramov} gives \eqref{eqn:induce-P}.

Now we review Zweim\"uller's construction of $m_{\tau}$ and show that it implies \eqref{eqn:RN-control0}. To construct $\overline{m}_{\tau}$, Zweim\"uller
works with the \emph{$\tau$-separating extension} $(\widehat{X},\widehat{\mathcal{B}},\widehat{f},\widehat{m})$ of $(X,\BBB,f,m)$, which we now describe.
Interpret a triple $(x,n,k) \in X\times \NN_0 \times \NN_0$ as carrying the following information: ``we are currently at the point $x$, we are currently going through an inducing return of length $n$, in which we have taken $k$ steps so far''. Extending $\tau$ to $X$ by setting $\tau|_{X\setminus Y} \equiv 0$, define sets $\widehat{X}_k \subset X\times \NN_0 \times \NN_0$ inductively by
\begin{align*}
\widehat{X}_0 &:= \{ (x,\tau(x),0) : x\in X \}, \\
\widehat{X}_{k}
&:= \{ (f(x), n, k) : (x,n,k-1) \in \widehat{X}_\ell
\text{ and }
k < n \}.
\end{align*}
Let $\widehat{X} := \bigcup_{k\in \NN_0} \widehat{X}_k$, and define $\widehat{f} \colon \widehat{X}\to \widehat{X}$ by
\[
\widehat{f}(x,n,k) :=
\begin{cases}
(f(x),n,k+1) &\text{if } k < n-1, \\
(f(x), \tau(f x), 0) &\text{if } k=n-1.
\end{cases}
\]
Writing  $\pi\colon X\times \NN_0\times \NN_0\mapsto X$ for projection to the first coordinate, we see that $\pi \circ \widehat{f} = f\circ \pi$. Now putting
\[
\widehat{Y}_0
:= \widehat{X}_0 \cap \pi^{-1}(Y)
\quad\text{and}\quad
\widehat{Y} := \bigcup_{n=0}^\infty \widehat{f}^n(\widehat{Y}_0),
\]
we can consider the nonsingular transformation $(\widehat{Y},\widehat{f},\widehat{m})$, where
$\widehat{m}$ denotes the product measure of $m$ with counting measure on $\NN_0\times\NN_0$.
The map $\pi \colon \widehat{Y}_0 \to Y$ gives an isomorphism between the first return map to $\widehat{Y}_0$ and our original induced transformation $(Y,m,f^\tau)$.
Consequently, as described in \cite[(3.2)]{rZ05}, proving the existence of $m_{\tau}$  is equivalent to finding a finite $\widehat{f}$-invariant measure $\widehat{\nu} \ll \widehat{m}$ on $\widehat{Y}$ such that $\widehat{\nu}(\widehat{Y}_0) > 0$.

A finite measure $\widehat{\nu} \ll \widehat{m}$ is $\widehat{f}$-invariant if and only if its density function $\widehat{h} := d\widehat{\nu} / d\widehat{m}$ is a fixed point of the (dual) transfer operator $\widehat{\mathcal{L}}$ of $\widehat{f}$ acting on $L^1(\widehat{A},\widehat{m})$, which is defined by
\[
\int_{\widehat{Y}} \widehat{g} \cdot \widehat{\mathcal{L}}(\widehat{u}) \,d\widehat{m}
= \int_{\widehat{Y}} (\widehat{g}\circ \widehat{f}) \cdot \widehat{u} \,d\widehat{m} 
\quad \text{ for all }\widehat{g}\in L^{\infty}(\widehat{A},\widehat{m}) \text{ and } \widehat{u}\in L^{1}(\widehat{A},\widehat{m}).
\]
As shown in the proof of 
\cite[Theorem 1.1]{rZ05}, such a function $\widehat{h}$ can be obtained as a weak limit of (a subsequence of) $(nm(Y))^{-1} \sum_{j=0}^{n-1} \widehat{\mathcal{L}}^j \one_{\widehat{Y}_0}$, and then we define $m_\tau$ by
\begin{equation} \label{eq:ml}
m_\tau(E) := \int_{\pi^{-1}(E) \cap \widehat{Y}_0} \widehat{h} \,d\widehat{m}
= \int_{\widehat{Y}_0} \widehat{h} \one_{\pi^{-1}(E)} \,d\widehat{m}.
\end{equation}
It follows from \cite[(3.3)]{rZ05} that $m_\tau$ induces $m$ in the sense of \eqref{eqn:induces}.

To prove \eqref{eqn:RN-control0}, we start with the following observation: by the definition of $\widehat{\mathcal{L}}$, given $j\geq 0$ and $\widehat{x}\in \widehat{Y}_0$, there is at most one $k\in \{j,j+1,\cdots,j+\min\tau-1\}$ with
\[
\widehat{\mathcal{L}}^k(\one_{\widehat{Y}_0})(\widehat{x})\neq 0.
\]
From \cite[(3.5)]{rZ05}, we have $0\leq \widehat{\mathcal{L}}^k(\one_{\widehat{Y}_0}) \leq \one_{\widehat{X}}$ for all $k\geq 0$, and it follows that
\begin{equation} \label{eqhathsize}
0\leq \widehat{h}|_{\widehat{Y}_0} \leq \frac{1}{(\min\tau) m(Y)}.
\end{equation}
Together with \eqref{eq:ml}, this implies \eqref{eqn:RN-control0}.

Now suppose there exist $M\in \NN$ and $Z\subset Y$ such that $m(Z)>0$ and $\tau \leq \tau_Z^M$.
Putting $\psi = \one_Z$, this implies that
\[
\Psi(y,\tau(y)) = \#\{ k \in \{0,\dots, \tau(x) - 1\} : f^k(y) \in Z \} \leq M
\]
for every $y\in Y$, so  \eqref{eqn:induces-mx} gives
%\begin{equation}
\[
m(Z) = \int_Y  \Psi(y,\tau(y)) \,dm_\tau(y) \leq M \cdot m_\tau(Y),
\]
%\end{equation}
which implies \eqref{eqn:mtY}. Combining \eqref{eqn:RN-control0} and \eqref{eqn:mtY} gives \eqref{eqn:RN-control}.
\end{proof}

\begin{comment}
\begin{remark}
The upper bound in \eqref{eqn:RN-control} will be used when we prove uniqueness in \S\ref{sec:uniqueness}. Prior to that, we will need a similar bound that applies without the assumption of ergodicity; this is provided by \eqref{eqn:RN-control-nonerg} below, which is used in \S\ref{sec:ergodic} as part of the argument to establish ergodicity of the Misiurewicz measure.
\end{remark}
\end{comment}

If we remove the assumption of ergodicity in Theorem \ref{thm:Zwei}, then the natural candidate for an inducing measure is obtained as follows. Writing the ergodic decomposition of $m$ as $\int_X m^x \,dm(x)$, where each $m^x \in \MF$ is ergodic, we see that if $Y\in \BBB_f^+(m)$, then $m^x(Y)>0$ for each $x$, and if $\tau \colon Y \to \NN \cup \{\infty\}$ satisfies $m(\tau) < \infty$, then since $m(\tau) = \int_X m^x(\tau) \,dm(x)$, we have $m^x(\tau) < \infty$ for $m$-a.e.\ $x$. Thus we can apply Theorem \ref{thm:Zwei} to $(X,\BBB,f,m^x)$ for $m$-a.e.\ $x$, obtaining a family of measures $m_\tau^x \ll m^x$ on $Y$ such that for every $\psi \in L^1(X,m)$, we have
\begin{equation}\label{eqn:induces-mx}
\int_X \psi(y) \,dm^x(y)
= \int_Y \Psi(y,\tau(y)) \,dm_\tau^x(y),
\end{equation}
%\frac{dm_\tau^x}{d(m^x)_Y} &\leq \frac 1{\min\tau}.
Now we define the natural candidate for an inducing measure by
\begin{equation}\label{eqn:mtau}
m_\tau  := \int_X m_\tau^x \,dm(x).
\end{equation}
The probability normalizations $\overline{m}_{\tau} := \frac{m_{\tau}}{m_{\tau}(Y)}$ and $\overline{m}_\tau^x := \frac{m_\tau^x}{m_\tau^x(Y)}$ are related by
\begin{equation}\label{eqn:norm-decomp}
\overline{m}_\tau(E) = \int_X \frac{m_\tau(Y)}{m_\tau^x(Y)} \overline{m}_\tau^x(E) \,dm(x).
\end{equation}

\begin{corollary}\label{cor:Zwei}
Let $(X,\BBB,f,m)$ be a probability measure-preserving transformation, which need not be ergodic. 
Fix $Y\in \BBB_f^+(m)$, and let $\tau \colon Y\to \NN \cup \{\infty\}$ be an inducing time function satisfying $m(\tau) < \infty$.
Let $m_\tau^x$ be the inducing measures obtained by applying Theorem \ref{thm:Zwei} to the ergodic components of $m$, and let $m_\tau$ be the measure on $Y$ defined by \eqref{eqn:mtau}. Then $m_\tau$ is $f^\tau$-invariant, $m_\tau \ll m$, and $m_\tau$ induces $m$ in the sense that for every $\psi\in L^1(X,m)$, we have \eqref{eqn:induces}.

Now suppose that in addition to the above hypotheses, there exist $M\in \NN$, $Z\subset Y$, and $c>0$ such that 
\eqref{eqn:inducing-return} holds and $m$-a.e.\ ergodic component of $m$ satisfies $m^x(Z)\geq c$.
Then %there exists an invariant set $P \in \BBB$ such that $m(P) > \frac 12 m(Z)$ and 
\begin{equation}\label{eqn:RN-control-nonerg}
%\overline{m}_\tau(E) \leq \frac{8M}{(\min\tau) (m(Z))^3} \cdot m(E)
\overline{m}_\tau(E) \leq \frac{M}{c^2 (\min\tau) } \cdot m(E)
\quad\text{for all measurable } E. %\subset P.
\end{equation}
\end{corollary}

\begin{proof}
%Since $Y\in \BBB_f^+(m)$, we have $m^x(Y) > 0$ for a.e.\ $x$, and so we can apply Theorem \ref{thm:Zwei} to each of the ergodic systems $(X,\BBB,f,m^x)$ to obtain a family of measures $m_\tau^x$ on $Y$ such that for every $x\in X$ and
%As in the statement of the theorem, we define a measure $m_\tau$ on $Y$ by
Invariance of $m_\tau$ follows from invariance of $m$ together with invariance of $x\mapsto m_\tau^x$, and the fact that $m_\tau\ll m$ is an immediate consequence of \eqref{eqn:mtau} together with $m_\tau^x\ll m^x$. To see that $m_\tau$ induces $m$ in the sense of \eqref{eqn:induces}, we integrate \eqref{eqn:induces-mx} to get
\begin{align*}
\int_X \psi(y) \,dm(y)
&= \int_X \int_X \psi(y) \,dm^x(y) \,dm(x) \\
&= \int_X \int_Y \Psi(y,\tau(y)) \,dm_\tau^x(y) \,dm(x)
= \int_Y \Psi(y,\tau(y)) \,dm_\tau(y),
\end{align*}
where the last step uses Fubini's theorem and \eqref{eqn:mtau}.

Now suppose there exist $M\in \NN$, $Z\subset Y$, and $c>0$ such that
$\tau\leq \tau_Z^M$ and $m^x(Z)\geq c$ for $m$-a.e.\ $x$.
\begin{comment}
$m(Z)>0$ and $\tau \leq \tau_Z^M$.
Consider the invariant set
\[
P := \{x\in X : m^x(Z) \geq m(Z)/2\},
\]
and observe that
\[
m(Z) = \int_P m^x(Z) \,dm(x) + \int_{X\setminus P} m^x(Z) \,dm(x)
< m(P) + \frac{m(Z)}2,
\]
from which we conclude that $m(P) > \frac 12 m(Z)$, verifying the first claim regarding $P$. 
\end{comment}
Recalling \eqref{eqn:mtau}, we see that for %any measurable $E \subset P$, we have
any $E\in \BBB$, we have
\begin{equation}\label{eqn:barmt}
\overline{m}_\tau(E)
= \frac{m_\tau(E)}{m_\tau(Y)}
= \frac{\int_X m_\tau^x(E) \,dm(x)}{\int_X m_\tau^x(Y) \,dm(x)}.
%\leq \frac{\int_P m_\tau^x(E) \,dm(x)}{\int_P m_\tau^x(Y) \,dm(x)},
\end{equation}
%where the last step uses the fact that $m^x(E) = 0$ for all $x\in X\setminus P$, since $P\supset E$ is invariant.
\begin{comment}
For every $x\in P$, we can use \eqref{eqn:mtY} to get
\[
m_\tau^x(Y) \geq \frac 1M m^x(Y) \geq \frac 1{2M} m(Z).
\]
Combining this with \eqref{eqn:RN-control0}, we get
\[
\frac{m_\tau^x(E)}{m_\tau^x(Y)}
\leq \frac{ (\min\tau)^{-1} (m^x(Y))^{-1} m^x(E)}{(2M)^{-1} m(Z)}
\leq \frac{4M}{(\min\tau)(m(Z))^2} m^x(E).
\]
***************
\end{comment}
To bound the denominator, use \eqref{eqn:mtY}
and the assumption $m^x(Z)\geq c$ to get
%, the definition of $P$, and the estimate $m(P) > \frac 12 m(Z)$ to get
\begin{equation}\label{eqn:denom}
%\begin{aligned}
\int_X m_\tau^x(Y) \,dm(x) \geq \frac{1}{M} \int_X m_\tau^x(Z) \,dm(x)
%\geq \frac 1M \int_X m^x(Y) \,dm(x) 
%&\geq \frac 1{2M} \int_P m(Z) \,dm(x)
%= \frac{m(P) m(Z)}{2M}
%> \frac{(m(Z))^2}{4M}.
\geq \frac c M.
%\end{aligned}
\end{equation}
%When $E \subset P$, we can use invariance of $P$ to conclude that $m^x(E) = 0$ for all $x\in X \setminus P$, so that the numerator in \eqref{eqn:barmt} admits the following bound, where 
For the numerator in \eqref{eqn:barmt},
we apply \eqref{eqn:RN-control0} to each $m^x$:
\begin{equation}\label{eqn:numer}
\begin{aligned}
\int_X m_\tau^x(E) \,dm(x) %=
%\int_P m_\tau^x(E) \,dm(x) 
&\leq \frac 1{\min\tau} \int_X\frac{m^x(E)}{m^x(Y)} \,dm(x) \\
&\leq \frac{1}{c(\min\tau)} \int_X m^x(E) \,dm(x)
= \frac{m(E)}{c(\min\tau)}.
\end{aligned}
\end{equation}
Combining \eqref{eqn:barmt}, \eqref{eqn:denom}, and \eqref{eqn:numer} gives \eqref{eqn:RN-control-nonerg}.
\end{proof}

\subsection{Induced measures on reference sets}\label{sec:Zwei-ref}

We now combine the results of \S\ref{sec:reference sets} and \S\ref{sec:Zwei}, applying Theorem \ref{thm:Zwei}  to the setting of Proposition \ref{prop:mainadapted}.

To this end, let $X,\FFF,\DDD,\ph$ satisfy Conditions \ref{cond:basic}--\ref{cond:UESB}, and let $\Azero$, $\wtA$, $\rexp$, $\eps$, and $\eps_0$ be as described at the start of \S\ref{sec:uniform-katok}. 

Given $\nu \in \Mf^\ph \cap \MF^{\Azero}$, let $\justA$ and $\Abiggest$ satisfy \eqref{eqn:nesting}. %and \eqref{eq:newcondition0}. 
We define a sequence of induced maps $\barfn$ and inducing probability measures $\bnun$ on $\Abiggest$:
\begin{itemize}
\item Proposition \ref{prop:mainadapted} provides $\mathcal{E}^N$, $\xi_N$, and $\rho_N$, and $\barfn \colon \Abiggest\to \Abiggest$ is defined $\nu$-a.e.\ by $\barfn(x) = F^{\rho_N(x)}(x)$, as in \eqref{eqn:induced};
\item 
since \ref{H1} gives $\int_{\Abiggest} \rho_N\,d\nu < \infty$, we can apply Theorem \ref{thm:Zwei}, and more generally Corollary \ref{cor:Zwei}, with $m=\nu$, $Y=\Abiggest$, and $\tau = \rho_N$ to obtain the measure $\bnun = \overline{m}_\tau$.
\end{itemize}

\begin{lemma} \label{lem:induced-es}
Let $\nu \in \Mf^\ph \cap \MF^{\Azero}$ be ergodic, and let $\justA,\Abiggest,\rho_N,\barfn,\bnun$ be as above.
Then $\bnun$ is a $\barfn$-invariant probability measure on $\Abiggest$ satisfying the following:
\begin{align}
\label{eqn:P-bnun}
h_{\bnun}(\barfn) + \int \bPh(\cdot,\rho_N) \,d\bnun
&= \bnun(\rho_N)
\Big( h_\nu(F) + \int\ph\,d\nu - P(\ph) \Big), \\
\label{eqn:uniformabsolutecontinuity}
%\frac{d\bnun}{d\nu_{\Abiggest}} &\leq \frac{2}{\nu(\justA)}.
\bnun &\leq \frac{2}{\nu(A)} \nu.
\end{align}
\end{lemma}
\begin{proof}
Observe that \eqref{eqn:P-bnun} is just \eqref{eqn:induce-P} from Theorem \ref{thm:Zwei} transcribed into our present notation and evaluated for $\psi = \overline{\ph} = \ph - P(\ph)$. For \eqref{eqn:uniformabsolutecontinuity},
we write $Z = \justA$ and $M = 2N$, and observe that \ref{H3} gives \eqref{eqn:inducing-return},  so \eqref{eqn:RN-control} applies, giving
\[
\bnun \leq \frac{2N}{\min\rho_N} \cdot \frac{\nu}{\nu(A)}.
\]
Since $\rho_N \geq N$ by \ref{H3}, this implies \eqref{eqn:uniformabsolutecontinuity}.
\end{proof}

By definition, every equilibrium state for $\ph$ must lie in $\Mf^\ph$, and $\Azero$ was chosen at the start of \S\ref{sec:uniform-katok} to have the property that $\MF^{\Azero}$ contains all equilibrium states.
Thus the construction of $\justA,\Abiggest,\rho_N,\barfn,\bnun$ preceding Lemma \ref{lem:induced-es} can be carried out for every equilibrium state $\nu$.

\begin{lemma}\label{lem:induced-es-nonerg}
Let $\nu$ be an equilibrium state, which need not be ergodic, and let $\justA,\Abiggest,\rho_N,\barfn,\bnun$ be as above.
Then
\begin{equation}\label{eqn:itsstillzero!}
h_{\bnun}(\barfn) + \int \bPh(x, \rho_N(x)) \,d\bnun(x) = 0,
\end{equation}
and if every a.e.\ ergodic component of $\nu$ satisfies $\nu^x(A) \geq c > 0$, then we have
%there exists a 
%a flow-invariant measurable set $P \subset X$ such that $\nu(P) > \frac 12 \nu(\justA)$ such that
\begin{equation}\label{eqn:NU-absolutecontinuity}
%\frac{d\bnun}{d\nu_{\Abiggest}}\Big|_P \leq \frac{16}{(\nu(\justA))^3}.
\bnun \leq 2c^{-2} \nu.
\end{equation}
\end{lemma}
\begin{proof}
Recalling that \eqref{eqn:norm-decomp} gives the ergodic decomposition of $\bnun$, \eqref{eqn:itsstillzero!} follows by applying \eqref{eqn:P-bnun} from  Lemma \ref{lem:induced-es} to each ergodic component of $\nu$. 
For \eqref{eqn:NU-absolutecontinuity}, we argue just as in the proof of Lemma \ref{lem:induced-es}, taking $Z=\justA$ and $M=2N$, so that \ref{H3} gives \eqref{eqn:inducing-return} and allows us to apply the second half of Corollary \ref{cor:Zwei}, using the fact that $\frac{M}{\min\tau} = \frac{2N}{\rho_N} \leq 2$ since $\rho_N\geq N$.
%, obtaining an invariant set $P\in \BBB$ such that $\nu(P) >\frac 12 \nu(\justA)$ and such that for every measurable $E \subset P$, \eqref{eqn:RN-control-nonerg} gives
%\[
%\bnun(E) \leq \frac{8 \cdot 2N}{(\min\rho_N)(\nu(\justA))^3} \nu(E),
%\]
%which implies \eqref{eqn:NU-absolutecontinuity} since $\rho_N \geq N$.
\end{proof}

\subsection{Entropy estimates}  \label{sec:entropyestimate}

Let $\nu$ be an equilibrium state, which need not be ergodic, and let $\justA,\Abiggest,\rho_N,\barfn,\bnun$ be as in \S\ref{sec:Zwei-ref}. We first prove that the partition $\xi_N = \{w_i^N\}_i$ has finite entropy, and then show that it sees all the entropy of $(\Abiggest,\barfn,\bnun)$.

\begin{lemma} \label{lem:finite-entropy-adapted}
The partition $\xi_N$ satisfies $H_{\bnun}(\barfn,\xi_N)<\infty$.
\end{lemma}
\begin{proof}
First we show that
$H_{\bnun}(\barfn,\xi_N)+\int \bPh(\cdot, \rho_N)  \,d\bnun<\infty$.
Writing
\begin{equation}\label{eqn:piqi}
p_i = \bnun(w^N_i)
\quad\text{and}\quad
q_i = \bPh(w_i),
\end{equation}
we can use the  Bowen property and \ref{H5} from Proposition \ref{prop:mainadapted} to get
\begin{align*}
H_{\bnun}(\barfn,\xi_N)
+ \int \bPh(\cdot, \rho_N)  \,d\bnun
&\leq \sum p_i(q_i-\log p_i) 
\leq \log \sum e^{q_i} \\
&\leq \log e^Q\sum e^{\bPh(x^N_i, m^N_i)}
\leq Q+\log R(\Abiggest)<\infty. 
\end{align*}
Thus it suffices to show that $\int \bPh(\cdot, \rho_N)  \,d\bnun>-\infty$. To this end, observe that
since $\nu$ is an equilibrium state, $h_{GS}<\infty$, and $P(\varphi)>-\infty$, we have
$\int \ph \,d\nu>-\infty$, and thus $\int \bar\ph \,d\nu = \int \ph \,d\nu - P(\varphi) > -\infty$.
Since $\int \rho_N \,d\bnun<\infty$, we can use the inducing relationship \eqref{eqn:induces} to get
\[
\int \bPh(\cdot, \rho_N)  \,d\bnun=\int \rho_N \,d\bnun \int \bph \,d\nu > -\infty,
\]
which completes the proof.
\end{proof}

\begin{proposition}\label{lem:adapted-generates}
The  partition $\xi_N$ has the property that $h_{\bnun}(\barfn) = h_{\bnun}(\barfn,\xi_N)$.%
\end{proposition}

%We prove Proposition \ref{lem:adapted-generates} in the case that $\overline{\nu_N}$ is ergodic under $\overline{f_N}$. For general $\overline{\nu_N}$, we run the argument over its ergodic components $\overline{\nu_{N,x}}$, $x\in \Abiggest$, satisfying 
%\begin{equation} \label{eq:nuNxintegrable}
%\int_{\Abiggest} \rho_N \,d \overline{\nu_{N,x}} < \infty.
%\end{equation}
%Since $\int_{\Abiggest} \rho_N \,d\overline{\nu_{N}}=\frac{1}{\nu_N(\Abiggest)}<\infty$, \eqref{eq:nuNxintegrable} is satisfied by $\overline{\nu_{N,x}}$ for $\overline{\nu_N}$-a.e $x$, and we obtain the statement of Proposition \ref{lem:adapted-generates} for $\overline{\nu_{N}}$ as long as it holds for each $\overline{\nu_{N,x}}$.

\begin{proof}
Fix $N$ and write $\barf = \barfn$. Given a partition $\mathcal{P}$ of $\Abiggest$ and a point $x\in \Abiggest$, let $\mathcal{P}(x)$ denote the element of $\mathcal{P}$ containing $x$, and consider for each $k\in \NN \cup \{\infty\}$ the partitions defined by
\[
\mathcal{P}^k := \bigvee_{j=0}^{k-1} \barf^{-j} \mathcal{P}
\quad\text{and}\quad
\mathcal{P}^{\pm k} := \bigvee_{j=-k}^{k} \barf^{-j} \mathcal{P}.
\]
We will consider a refining sequence of partitions starting with $\mathcal{P}_1 := \xi_N$. To define $\mathcal{P}_n$ for $n>1$, recall that by the expansivity property, we have
\[
\mathcal{P}_1^{\pm\infty}(x) \subset f_{[-\alpha,\alpha]}(x)
\quad\text{for $\bnun$-a.e.\ $x\in \Abiggest$}.
\]
Consider the measurable function $\mathbf{a} \colon \Abiggest \to [0,\alpha]$ defined by
\begin{equation}\label{eqn:aa}
\mathbf{a}(x) := \sup \{ t\in [0,\alpha] : f_t(x) \in \mathcal{P}_1^{\pm\infty}(x) \}.
\end{equation}
Given $n > 1$, define sets $Z_j^n$ for $j\in \NN$ by
\[
Z_j^n := \{ x\in \Abiggest : \mathbf{a}(x) \in [(j-1)2^{-n}, j2^{-n}) \}.
\]
Let $\mathcal{P}_n$ be the partition of $\Abiggest$ into the sets $\{w_i \cap Z_j^n : i\in I_N^\xi, j\in \NN\}$.  We have
\[
\mathcal{P}_1(x) \supset \mathcal{P}_2(x) \supset \mathcal{P}_3(x) \supset \cdots \supset \mathcal{P}_n(x) \supset \cdots.
\]
We claim that:
\begin{equation}\label{eqn:Pn-gen}
\text{$\bigcup_{n\in \NN} \mathcal{P}_n$ generates the Borel $\sigma$-algebra on $\Abiggest$.}
\end{equation}
It suffices to show that for a.e.\ $x\in \Abiggest$ and every $\delta>0$, there exist $n,k\in \NN$ such that $\mathcal{P}_n^{\pm k}(x) \subset B(x,\delta)$. To this end, start by fixing $n\in \NN$ sufficiently large that $f_{[-2^{-n}, 2^{-n}]}(x) \subset B(x,\delta/2)$. Then since $\mathcal{P}_n^{\pm\infty}(x) \subset f_{[-2^{-n}, 2^{-n}]}(x)$, there exists $k\in \NN$ such that 
\[
\mathcal{P}_n^{\pm k}(x) \subset B(f_{[-2^{-n},2^{-n}]}(x),\delta/2) \subset B(x,\delta),
\]
proving \eqref{eqn:Pn-gen}. By the Kolmogorov--Sinai Theorem \cite[Theorem 9.2.1]{VO16}, we have
\begin{equation}\label{eqn:h-is-lim}
h_{\bnu}(\barf) = \lim_{n\to\infty} h_{\bnu}(\barf,\mathcal{P}_n),
\end{equation}
so in order to prove Proposition \ref{lem:adapted-generates}, it suffices to show that
\begin{equation}\label{eqn:HPn}
h_{\bnu}(\barf,\mathcal{P}_n)
= h_{\bnu}(\barf,\mathcal{P}_1)
\quad\text{for all } n\in \NN.
\end{equation}
To this end, first observe that since $\rho_N$ is constant on $\mathcal{P}_1(x) \supset \mathcal{P}_1^{\pm\infty}(x)$, we have
\[
\barf(\mathcal{P}_1^{\pm\infty}(x))
= \mathcal{P}_1^{\pm\infty}(\barf(x)),
\]
and it follows that
\begin{equation}\label{eqn:a-inv}
\mathbf{a}(\barf(x)) = \mathbf{a}(x).
\end{equation}\
Fixing $n$, \eqref{eqn:a-inv} implies that $\barf(Z_j^n) = Z_j^n$ for all $j\in \NN$, and thus
$\mathcal{P}_n^k$ is the partition of $\Abiggest$ into sets of the form $w \cap Z_j^n$, where $w\in \mathcal{P}_1^k$. In particular, each element of $\mathcal{P}_1^k$ contains at most $2^n/\alpha$ elements of $\mathcal{P}_n^k$, and we conclude that
\[
H_{\bnu}(\mathcal{P}_n^k)
= H_{\bnu}(\mathcal{P}_1^k)
+ H_{\bnu}(\mathcal{P}_n^k \mid \mathcal{P}_1^k)
\leq H_{\bnu}(\mathcal{P}_1^k) + \log(2^n/\alpha).
\]
Dividing both sides by $k$ and sending $k\to\infty$ gives \eqref{eqn:HPn}, which together with \eqref{eqn:h-is-lim} completes the proof of Proposition \ref{lem:adapted-generates}.
\end{proof}

\subsection{Uniform lower estimate} \label{sec:uniformlowerestimate}

We prove a uniform lower bound on the partition sum over a positive measure set. We follow the basic strategy of \cite[Lemmas 3.8 and 4.18]{CT16} applied to our induced map $\barfn\colon \Abiggest\to \Abiggest$.

\begin{proposition} \label{prop.uniformkatok}
Let $\Azero \subset \wtA$ be chosen as at the beginning of \S\ref{sec:uniform-katok}, and let $Q,R$ be as in Proposition \ref{prop:mainadapted}. Given $\lambda \in (0,1)$, let
\begin{equation}\label{eqn:C-lambda}
C_\lambda :=e^{-Q}R (2R)^{-1/\lambda}.
\end{equation}
Let $\nu$ be an equilibrium state for $(X,\FFF,\ph)$, and fix $\justA, \Abiggest$ as in $\S \ref{sec:reference sets}$.
For each $N\in \NN$, let $\mathcal{E}^N, \xi_N, \rho_n, \barfn$ be given by Proposition \ref{prop:mainadapted},
and let $\bnun$ be the corresponding inducing measure on $\Abiggest$ provided by Corollary \ref{cor:Zwei}, just as in \S\ref{sec:Zwei-ref}.
Then given any collection of indices $J \subset I_N^\xi$ such that $\bnun(\bigcup_{i\in J} w_i^N) \geq \lambda$, we have
\[
\sum_{i\in J} e^{\bPh(x^N_i, m^N_i)} \geq C_\lambda.
\]
\end{proposition} 
\begin{proof}
Defining $p_i = \bnun(w_i^N)$ and $q_i = \bPh(w_i)$ as in \eqref{eqn:piqi}, Lemma \ref{lem:adapted-generates} gives 
\begin{gather*}
h_{\bnun}(\barfn) = h_{\bnun}(\barfn,\xi_N) \leq \sum -p_i \log p_i, \\
\int \bPh(\cdot, \rho_N)  \,d\bnun = \sum_{i \in I} \int_{w^N_i} \bPh(\cdot, m^N_i) \,d\bnun
\leq \sum_{i\in I} p_i q_i.
\end{gather*}
Using \eqref{eqn:itsstillzero!}, we obtain
\[
0 \leq \sum_{i\in I} p_i(q_i - \log p_i).
\]
We split this sum over the indices $J$ from the hypothesis of Proposition \ref{prop.uniformkatok} and their complement $J^c = I \setminus J$. Write $p(J) = \sum_{i\in J} p_i$ and $p(J^c) = \sum_{i\in J^c} p_i$; then we have
\begin{align*}
0 &\leq \sum_{i\in J} p_i(q_i - \log p_i) + \sum_{i\in J^c} p_i(q_i - \log p_i) \\
&= p(J) \sum_{i\in J} \frac{p_i}{p(J)} \Big(q_i - \log \frac{p_i}{p(J)} - \log p(J)\Big) \\
&\qquad+ p(J^c) \sum_{i\in J^c} \frac{p_i}{p(J^c)} \Big(q_i - \log \frac{p_i}{p(J^c)} - \log p(J^c)\Big) \\
&\leq \log 2 + p(J) \log \sum_{i\in J} e^{q_i} + p(J^c) \log \sum_{i\in J^c} e^{q_i}.
\end{align*}
where the last inequality follows from $-p(J)\log p(J)-p(J^c)\log p(J^c)\leq \log 2$ and applying the standard inequality $\sum_{i\in I}a_i(-\log a_i+b_i)\leq \log \sum_{i\in I}e^{b_i}$ for probability vectors $(a_i)$ twice. Using $p(J^c) = 1-p(J)$ and observing that $\sum_{i\in J^c} \leq \sum_{i \in I}$ when all terms are positive, we obtain
\begin{equation}\label{eqn:getting-there}
0 \leq \log 2 + p(J) \log \sum_{i\in J} e^{q_i} + (1-p(J)) \log \sum_{i \in I} e^{q_i}.
\end{equation}
Observe that
\begin{equation}\label{eqn:pJ}
p(J) = \bnun\Big(\bigcup_{i\in J} w_i\Big)
\geq \lambda,
\end{equation}
Applying \ref{H5}, we have 
\begin{align*}
0 &\leq \log 2 + p(J) \log \sum_{i\in J} e^{q_i} + (1-p(J)) \log R \\
&= \log (2R) + p(J) \log \Big( \frac 1{R} \sum_{i\in J} e^{q_i} \Big).
\end{align*}
Rearranging and using \eqref{eqn:pJ} gives
\[
\log \Big( \frac 1{R} \sum_{i\in J} e^{q_i} \Big) \geq -\frac{\log(2R)}{p(J)}
\geq -\frac{\log(2R)}\lambda,
\]
from which we deduce that
$
\sum_{i\in J} e^{q_i} \geq R (2R)^{-1/\lambda}.
$
Finally, applying the Bowen property gives 
\[
\sum_{i\in J} e^{\bPh(x^N_i, m^N_i)} \geq e^{-Q} \sum_{i\in J} e^{q_i}
\geq e^{-Q}R (2R)^{-1/\lambda},
\]
which completes the proof of Proposition \ref{prop.uniformkatok}.
\end{proof}

Now we can prove Theorem \ref{thm:uniform-katok}.
Let $Q,R,C_*,\beta$ be as in Proposition \ref{prop:mainadapted}, and as in \eqref{eqn:C-lambda}, let
$C_{1/2} = \frac 14 e^{-Q} R^{-1}$. Fix
\begin{equation}\label{eqn:choose-K}
K := \frac 12 C_{1/2} = \frac 18 e^{-Q} R^{-1},
\end{equation}
Now choose $N_0\in \NN$ sufficiently large that
\begin{equation}\label{eqn:N0}
C_* e^{-\beta N_0} \leq K.
\end{equation}
Given $\nu,Y$ as in the statement of Theorem \ref{thm:uniform-katok}, 
observe that since a.e.\ ergodic component of $\nu$ gives positive weight to $A$, we must have $\nu(Y\cap A)>0$ by invariance of $Y$. We claim that it suffices to fix $\gamma>0$ sufficiently small that \begin{equation}\label{eqn:gamma-small}
\frac{16\gamma}{(\nu(Y\cap A))^3} < \frac 12
\end{equation}
To this end, let
\[
P := \Big\{ x\in Y : \nu^x(Y\cap A) \geq \frac 12 \nu(Y\cap A) \Big\},
\]
and observe that
\[
\nu(Y\cap A) = \int_P \nu^x(Y\cap A) \,d\nu(x)
+ \int_{X\setminus P} \nu^x(Y\cap A) \,d\nu(x)
\leq \nu(P) + \frac 12 \nu(Y\cap A),
\]
from which we deduce that
\begin{equation}\label{eqn:nuYA}
\nu(P) > \frac 12 \nu(Y\cap A).
\end{equation}
Now the conditional measure $\nu_P$ is an equilibrium state, and has $\nu_P^x \geq \frac 12 \nu(Y\cap A)$ for a.e.\ ergodic component, so writing $\bnun^P$ for the corresponding inducing measure on $\Abiggest$ as in Lemma \ref{lem:induced-es-nonerg}, we see from \eqref{eqn:NU-absolutecontinuity} (with $c=\frac 12 \nu(Y\cap A)$) that
\begin{align*}
\bnun^P(Y\setminus Z) 
&\leq \frac 8{(\nu(Y\cap A))^2} \nu_P(Y\setminus Z)
\leq \frac{16}{(\nu(Y\cap A))^3} \nu\big(\wtA \cap (Y\setminus Z)\big) \\
&\leq \frac{16\gamma}{(\nu(Y\cap A))^3}
< \frac 12,
\end{align*}
where the last inequality uses the choice of $\gamma$ in \eqref{eqn:gamma-small}. We conclude that $\bnun^P(Z) > \frac 12$. 
%Let $C_{1/2}$ be the constant from Proposition \ref{prop.uniformkatok}, and put $K = \frac 12 C_{1/2}$. 
By \ref{H1}, the set of indices $J_N \subset I_N$ defined in \eqref{eqn:JN} has the 
property that 
\[
J_N \supset \{i\in I_N : w_i^N \cap Z \neq \emptyset \}.
\]
In particular, the set
\[
J_N^\xi := \{ i\in I_N^\xi = I_N \sqcup I_N^* : w_i^N \cap Z \neq \emptyset \}
\]
has the property that
\[
J_N^\xi \subset J_N \sqcup I_N^*.
\]
Recalling \eqref{eqn:choose-K}
and using Proposition \ref{prop.uniformkatok},  we get
\[
2K = C_{1/2} \leq 
\sum_{i\in J_N^\xi} e^{\bPh(x_i^N,m_i^N)}
= \sum_{i\in J_N} e^{\bPh(x_i^N,m_i^N)}
+ \sum_{i\in I_N^*} e^{\bPh(x_i^N,m_i^N)}.
\]
Recalling \ref{H5}, we get
\[
\sum_{i\in J_N} e^{\bPh(x_i^N,m_i^N)}
\geq 2K - \sum_{i\in I_N^*} e^{\bPh(x_i^N,m_i^N)}
\geq 2K - C_* e^{-\beta N}.
\]
For $N\geq N_0$, this is at least $K$ by \eqref{eqn:N0}, which proves Theorem \ref{thm:uniform-katok}.

\section{Ergodicity, uniqueness and other properties}\label{sec:eu-other}

In this section, we complete the proof of Theorem \ref{thm:mainES} by using the results from \S\ref{sec:uniform-katok} to show that $\mu$ is the unique equilibrium state. Our argument involves working with generic points in the sense of Birkhoff's ergodic theorem (along similar lines to the Hopf argument), and so we start by establishing a result allowing us to compare Birkhoff averages along nearby orbits.

\subsection{Comparison of Birkhoff averages}\label{sec:comparison}

\begin{proposition}\label{prop:exp-Birkhoff}
Let $X,\FFF,\DDD$ satisfy Conditions \ref{cond:basic} and \ref{cond:E}. Fix $\Azero,\bigref \in \RRR$ and $\zeta\in (0,\rexp(\bigref)]$ such that $(\Azero)^\zeta \subset \bigref$. Fix $\nu \in \Mf^{\Azero}$, and suppose that $\psi \in C(X)$ vanishes on $X\setminus \Azero$. Then for every $\delta>0$, there exist $G\subset \bigref$ and $T>0$ such that $\nu(G) > \nu(\bigref) - \delta$, and given any $x\in G$, any $t\geq T$, and any $y\in B_t(x,\zeta)$, we have
\begin{equation}\label{eqn:Birkhoff-close}
\Big| \frac 1t \Psi(x,t) - \frac 1t \Psi(y,t) \Big|
< \delta.
\end{equation}
\end{proposition}
\begin{proof}
Since $\psi$ is continuous and compactly supported, there exists $\delta_1>0$ such that
\begin{equation}\label{eqn:d1d3}
\text{for all $x,y\in X$ with $d(x,y) < \delta_1$, we have $|\psi(x) - \psi(y)| < \frac\delta 4$.}
\end{equation}
Fix $L>0$ sufficiently large that
\begin{equation}\label{eqn:L-large}
\frac{2\alpha\|\psi\|}{L} < \frac\delta 4.
\end{equation}
Fix $\delta_2>0$ sufficiently small that given any $x\in \bigref$ and $y\in X$ with $d(x,y) < \delta_2$, we have
\begin{equation}\label{eqn:d2-small}
d(f_t x, f_t y) < \delta_1
\quad\text{for all } t\in [-L,L].
\end{equation}
Let $\alpha= \alpha(B)$ be the lag time from the expansivity property. Given $n\in \NN$, let
\[
Z_n := \{ z\in \bigref : B_{[-n,n]}(z,\zeta)
\subset B(f_{[-\alpha,\alpha]} z, \delta_2) \},
\]
and let $a_n := \sqrt{\nu(\bigref\setminus Z_n)}$. By \eqref{eqn:BT-close} and Poincar\'e recurrence, we have $a_n \to 0$. Writing the ergodic decomposition of $\nu$ as $\int_X \nu^x \,d\nu(x)$, let
\[
Y_n := \{ x \in X : \nu^x(\bigref \setminus Z_n) > a_n \}.
\]
Observe that
\[
a_n^2 = \nu(\bigref\setminus Z_n) \geq \nu(Y_n) a_n
\quad\Rightarrow\quad
\nu(Y_n) \leq a_n,
\]
and that for $\nu$-a.e.\ $x\in X \setminus Y_N$, the Birkhoff ergodic theorem gives
\begin{equation}\label{eqn:leq-an}
\lim_{t\to\infty}
\frac 1t \int_0^t \one_{\bigref\setminus Z_n} (f_s x) \,ds \leq a_n.
\end{equation}
Fix $n\geq L$ sufficiently large so that
\begin{equation}\label{eqn:n-large}
4a_n\|\psi\| < \frac\delta4
\quad\text{and}\quad
2a_n < \delta,
\end{equation}
and given $T>0$, let 
\[
G_T := \Big\{ x\in \bigref : 
\frac 1t \int_0^t \one_{\bigref\setminus Z_n} (f_s x) \,ds
\leq 2a_n \text{ for all } t\geq T \Big\}.
\]
By \eqref{eqn:leq-an}, we have $\bigcup_T G_T \supset \bigref \setminus Y_n$, so there exists
\begin{equation}\label{eqn:T-large}
T \geq \frac 4\delta \cdot 6n\|\psi\|
\quad\text{such that }
\nu(G_T) > \nu(\bigref) - 2a_n > \nu(\bigref) - \delta,
\end{equation}
where the last inequality uses \eqref{eqn:n-large}.
Given any $x\in G_T$,  $t\geq T$, and $y\in B_t(x,\zeta)$, define $\Delta \colon [0,t] \to [0,\infty)$ by
\[
\Delta(s) := \psi(f_s x) - \psi(f_s y).
\]
We will bound $|\int_0^t \Delta(s)\,ds|$.
Let $k := \lfloor \frac{t-2n}{L} \rfloor$, and for each $j \in \{1,\dots, k\}$, let
\[
I_j := [ n + (j-1)L, n + jL)
\quad\text{and}\quad
\Delta_j := \int_{I_j} \Delta(s) \,ds,
\]
so that $[0,t] = [0,n) \cup \big( \bigcup_{j=1}^k I_j \big) \cup [n+kL, t]$, and
\begin{equation}\label{eqn:Delta}
\begin{aligned}
|\Psi(x,t) - \Psi(y,t)|
&\leq 
\int_0^n |\Delta(s)| \,ds
+ \int_{n+kL}^t |\Delta(s)| \,ds
+ \sum_{j=1}^k \Big|\int_{I_j} \Delta(s) \,ds\Big| \\
&\leq 6n\|\psi\| + \sum_{j=1}^k |\Delta_j|.
\end{aligned}
\end{equation}
Let $J  \subset \{1,\dots, k\}$ denote the set of those indices $j$ for which some $s\in I_j$ satisfies $f_s(x) \in Z_n$. Given such an $s$, we observe that $f_s(y) \in B_{[-n,n]}(f_s x,\delta)$, so by the definition of $Z_n$, there exists $\kappa \in [-\alpha, \alpha]$ such that $d(f_s(y), f_{s+\kappa}(x)) < \delta_2$. By \eqref{eqn:d2-small} and \eqref{eqn:d1d3}, this implies that
\[
\text{for every $r\in I_j$, we have $|\psi(f_r(y)) - \psi(f_{r+\kappa}(x))| < \frac \delta 4$}.
\]
Since $|\kappa|\leq \alpha$, we have
\[
\Big| \int_{I_j} \psi(f_r(x)) \,dr
- \int_{I_j} \psi(f_{r+\kappa}(x)) \,dr \Big| \leq 2\alpha \|\psi\|.
\]
Combining these two inequalities, we get
\begin{equation}\label{eqn:J-Delta}
\begin{aligned}
|\Delta_j|
&= \Big| \int_{I_j} \psi(f_r(y)) \,dr
- \int_{I_j} \psi(f_r(x)) \,dr \Big| \\
%&\leq 
%\Big| \int_{I_j} \psi(f_r(y)) \,dr
%- \int_{\kappa + I_j} \psi(f_r(x)) \,dr\Big| + 2\kappa \|\psi\| \\
&\leq \int_{I_j} |\psi(f_r(y)) - \psi(f_{r+\kappa}(x))| \,dr + 2\alpha\|\psi\|
\leq \frac{L\delta}4 + 2\alpha \|\psi\|.
\end{aligned}
\end{equation}
Let $J^c = \{1,\dots, k\}\setminus J$.
Given $j\in J$, we have $f_s(x) \in X \setminus Z_n$ for all $s\in I_j$. Observe that if $f_s(x) \notin \justA$, then $f_s(y) \notin \Azero$ since $(\Azero)^\zeta \subset \justA$, so $\psi(f_s(x)) = \psi(f_s(y)) = 0$ since $\psi$ vanishes outside of $\Azero$. We conclude that
\begin{align*}
\sum_{j\in J^c} |\Delta_j|
&\leq \sum_{j\in J^c} 2\|\psi\|\Leb\{ s\in I_j : f_s(x) \in \justA \setminus Z_n \} \\
&\leq 2\|\psi\| \Leb\{s\in [0,t] : f_s(x) \in \justA \setminus Z_n\}
\leq 4a_n t \|\psi\|,
\end{align*}
where the last inequality uses the definition of $G_T$. Combining this with \eqref{eqn:Delta} and \eqref{eqn:J-Delta}, we get
\begin{align*}
|\Psi(x,t) - \Psi(y,t)| 
&\leq 
6n\|\psi\| + \sum_{j\in J} |\Delta_j| + \sum_{j\in J^c} |\Delta_j| \\
&\leq
6n\|\psi\| + 
k\Big( \frac{L\delta}4 + 2\alpha\|\psi\|\Big)
+ 4a_n t \|\psi\|,
\end{align*}
and since $\max(Lk, T) \leq t$, we can divide both sides by $t$ to obtain
\[
\Big| \frac 1t \Psi(x,t) - \frac 1t \Psi(y,t) \Big|
\leq \frac{6n\|\psi\|}T + \frac\delta 4 + \frac{2\alpha\|\psi\|}L + 4a_n\|\psi\|.
\]
By our choices of $L$, $n$, and $T$ in \eqref{eqn:L-large}, \eqref{eqn:n-large}, and \eqref{eqn:T-large}, the sum on the right is less than $\delta$, which completes the proof.
\end{proof}

\subsection{Uniqueness}\label{sec:unique}

Once again, fix $\psi \in C(X)$ vanishing on $X\setminus\Azero$, and let
\[
\BBB := \Big\{ x\in X : \tps(x) := \lim_{t\to\infty} \frac 1t \Psi(x,t) \text{ exists} \Big\}.
\]
Given $c\in \RR$, let
\[
Y_c^- := \{ x\in \BBB : \tps(x) \leq c \}
\quad\text{and}\quad
Y_c^+ := \{ x\in \BBB : \tps(x) \geq c \}.
\]
We will deduce both ergodicity and uniqueness of $\mu$ from the following:

\begin{proposition}\label{prop:unique}
Let $\nu$ be any equilibrium state for $(X,\FFF,\ph)$, and let $\mu$ be the Misiurewicz measure constructed earlier. 
Supose that $\psi\in C(X)$ vanishes outside $\Azero$, and that $a,c\in \RR$ are such that $\nu(Y_a^-) >0$ and $\mu(Y_c^+)>0$. Then $c\leq a$. Similarly, if $\nu(Y_a^+)>0$ and $\mu(Y_c^-)>0$, then $c\geq a$.
\end{proposition}
\begin{proof}
It suffices to prove the first claim.
Let $K>0$ be the constant from Theorem \ref{thm:uniform-katok}.
Fix $\gamma>0$ sufficiently small that this corollary applies both to $(\nu,Y_a^-)$ and $(\mu,Y_c^+)$, and also such that
\begin{equation}\label{eqn:g-small}
\frac 14 K^2 H_L - 2\gamma > 0.
\end{equation}
Given $\delta,T>0$, let
\[
\BBB_T^\delta := \Big\{ x\in \BBB : \Big| \tps(x) - \frac 1t \Psi(x,t)\Big| < \delta \text{ for all } t\geq T \Big\}.
\]
Now fix $\delta>0$.
By Proposition \ref{prop:exp-Birkhoff} (with $\bigref = \wtA$) and the Birkhoff ergodic theorem, there exist $N\in \NN$ and $G\subset \wtA$ such that
\[
\nu(\wtA \setminus (G \cap \BBB_N^\delta)) < \gamma
\quad\text{and}\quad
\mu(\wtA \setminus (G \cap \BBB_N^\delta)) < \gamma,
\]
and such that \eqref{eqn:Birkhoff-close} holds for every $x\in G$, $t\geq N$, and $y\in B_t(x,\twentyepso)$.
Writing
\[
Z_a^- := Y_a^- \cap G \cap \BBB_N^\delta
\quad\text{and}\quad
Z_c^+ := Y_c^+ \cap G \cap \BBB_N^\delta,
\]
we have
\[
\nu\big(\wtA \cap (Y_a^- \setminus Z_a^-)\big) < \gamma
\quad\text{and}\quad
\mu\big(\wtA \cap (Y_c^+ \setminus Z_c^+)\big) < \gamma,
\]
and thus writing
\begin{align*}
J_N^- &:= \{ i\in I_N : B_{m_i^N}(x_i^N,\tenepso) \cap Z_a^- \neq \emptyset\}, \\
%\quad\text{and}\quad
J_N^+ &:= \{ i\in I_N : B_{m_i^N}(x_i^N,\tenepso) \cap Z_c^+ \neq \emptyset\},
\end{align*}
we can apply Theorem \ref{thm:uniform-katok} to obtain the bounds
\[
\sum_{i\in J_N^-} e^{\bPh(x_i^N, m_i^N)} \geq K
\quad\text{and}\quad
\sum_{i\in J_N^+} e^{\bPh(x_i^N, m_i^N)} \geq K.
\]
Given $M\in \NN$, let $J_{N,M}^{\pm} := \{ i \in J_N^{\pm} : m_i^N \leq M \}$. Fix $M$ sufficiently large that
\[
\sum_{i\in J_{N,M}^-} e^{\bPh(x_i^N, m_i^N)} \geq \frac K2
\quad\text{and}\quad
\sum_{i\in J_{N,M}^+} e^{\bPh(x_i^N, m_i^N)} \geq \frac K2.
\]
By \ref{H2}, the sets $B_{m_i^N}(x_i^N,\twoepso)$ are disjoint, 
and so writing
\[
U := \bigsqcup_{i\in J_{N,M}^-} 
B_{m_i^N}(x_i^N,\twoepso)
\quad\text{and}\quad
V := \bigsqcup_{j\in J_{N,M}^+} 
B_{m_j^N}(x_j^N,\twoepso),
\]
we can apply Lemma \ref{lem: DoubleLowerGibbs} to conclude that for every sufficiently large $t$, we have
\begin{align*}
\mu(U \cap f_{-t} V) 
& = \sum_{i\in J_{N,M}^-} \sum_{j\in J_{N,M}^+}
\mu\big(B_{m_i^N}(x_i^N,\twoepso) \cap f_{-t}
B_{m_j^N}(x_j^N,\twoepso)\big) \\
&\geq H_L 
\Big(\sum_{i\in J_{N,M}^-} e^{\bPh(x_i^N, m_i^N)} \Big)
\Big(\sum_{j\in J_{N,M}^+} e^{\bPh(x_j^N, m_j^N)} \Big)
\geq \frac 14 K^2 H_L.
\end{align*}
By our choice of $\gamma$ in \eqref{eqn:g-small}, we now have
\[
\mu((U \cap \BBB_N^\delta) \cap f_{-t}(V \cap \BBB_N^\delta))
\geq \frac 14 K^2H_L - 2\gamma > 0.
\]
Fix $x\in (U \cap \BBB_N^\delta) \cap f_{-t}(V \cap \BBB_N^\delta)$. 
Since $x\in U$, there exists $i\in J_{N,M}^- \subset J_N^-$ such that $x\in B_{m_i^N}(x_i^N,\twoepso)$, and by the definition of $J_N^-$, there exists
\[
y\in B_{m_i^N}(x_i^N,\tenepso) \cap Z_a^- \subset \BBB_N^\delta \cap G \cap Y_a^-.
\]
This gives  $x\in B_N(y,\twentyepso)$, and 
since $y\in G$, this implies that
\[
\Big|\frac 1N \Psi(x,N) - \frac 1N \Psi(y,N) \Big| < \delta.
\]
Since $x,y \in \BBB_N^\delta$, we have
\[
\Big|\tps(x) - \frac 1N \Psi(x,N)\Big| < \delta
\quad\text{and}\quad
\Big|\tps(y) - \frac 1N \Psi(y,N)\Big| < \delta.
\]
Finally, since $y\in Y_a^-$, we have $\tps(y) \leq a$, and combining this with three inequalities above gives
\begin{equation}\label{eqn:tpsx}
\tps(x) < a + 3\delta.
\end{equation}
Similarly, since $f_t(x) \in V \cap \BBB_N^\delta$, there exists $z\in \BBB_N^\delta \cap G \cap Y_c^+$ such that $f_t(x) \in B_N(z,\twentyepso)$, and an analogous argument gives
\[
|\tps(f_t(x)) - \tps(z)| < 3\delta
\quad\text{and}\quad \tps(z) \geq c
\quad\Rightarrow\quad
\tps(f_t x) > c - 3\delta.
\]
Combining this with \eqref{eqn:tpsx} and using invariance of $\tps$, we get
\[
c-3\delta < \tps(f_t(x)) = \tps(x) < a+3\delta
\quad\Rightarrow\quad
c < a+6\delta.
\]
Since $\delta>0$ was arbitrary, this implies that $c\leq a$, which completes the proof.
\end{proof}

Now we use Proposition \ref{prop:unique} to show that $\mu$ is the unique equilibrium state, and thus is ergodic.
It suffices to show that if $\nu$ is any equilibrium state, than $\nu = \mu$.
To this end, fix any $c\in \RR$ such that $\mu(Y_c^+) > 0$, and apply Proposition \ref{prop:unique} to get
\[
\essinf{\nu}\tps
= \inf \{ a\in \RR : \nu(Y_a^-) > 0\}
\geq c.
\]
Taking a supremum over all such $c$, we get
\[
\essinf{\nu}\tps
\geq \sup \{ c\in \RR : \mu(Y_c^+) > 0\}
= \esssup{\mu}\tps.
\]
A similar argument applies with the roles of $\nu$ and $\mu$ reversed, so we have
\[
\esssup{\mu}\tps
\leq \essinf{\nu}\tps
\leq \esssup{\nu}\tps
\leq \essinf{\mu}\tps
\leq \esssup{\mu}\tps.
\]
It follows that there exists $I(\psi) \in \RR$ such that
\[
\psi^*(x) = I(\psi)
\quad\text{$\nu$-a.e.\ and $\mu$-a.e.},
\]
and integrating gives
\[
\int \psi \,d\nu = \int \tps \,d\nu = I(\psi)
= \int \tps \,d\mu = \int \psi \,d\mu.
\]
We have now shown that $\int\psi\,d\nu = \int\psi \,d\mu$ for every continuous $\psi$ supported on $\Azero$. Since every ergodic component of both $\nu$ and $\mu$ gives positive weight to $\Azero$, this implies that $\nu=\mu$, and completes the proof of uniqueness and ergodicity.

\begin{comment}
First observe that applying the proposition with $\nu=\mu$ shows that 
given any $\psi \in C(X)$, 
there is an invariant set $G_\psi \subset X$ with $\mu(G_\psi)=1$ such that for every $x,y\in G_\psi$, the corresponding ergodic components satisfy $\mu^x(\psi) = \mu^y(\psi)$.
Since $\mu^x(A)>0$ and $\mu^y(A)>0$, having this equality for every $\psi \in C(X)$ supported on $A$ is enough to guarantee that $\mu^x = \mu^y$, so $\mu$ is ergodic.

For uniqueness, it suffices to apply Proposition \ref{prop:unique} when $\nu$ is any ergodic equilibrium state, and observe that $\nu(Y_a^-)>0$ if and only if $a\geq \nu(\psi)$, while $\mu(Y_c^+)>0$ if and only if $c\leq \mu(\psi)$, so the first part of the proposition implies that $\mu(\psi) \leq c \leq a \leq \nu(\psi)$, and similarly, the second part gives $\mu(\psi) \geq \nu(\psi)$. Since $\nu$ and $\mu$ are ergodic, this once again implies that $\nu=\mu$, so $\mu$ is the unique equilibrium state.
\end{comment}

\subsection{Gibbs at all small scales} \label{sec:the-rest} 

Now that we have proved uniqueness, the proof of Theorem \ref{thm:mainES} is completed by the following.

%We prove the additional properties of $\mu$ from Theorem \ref{thm:mainES}. 
%\begin{proposition}[Theorem \ref{thm:mainES} (3)] \label{prop:BMunique}
 %   Under condition \ref{cond:UESB}, for every $A\in \RRR$ over which $\ph$ is SPR, any Misiurewicz  measure for $\ph$ on $A$ at scale $\eps_0/2$ must be equal to $\mu$.
%\end{proposition}
%\begin{proof}
  %  The above result follows immediately from Corollary \ref{coro:BMbeingEquilibrium} and $\S \ref{sec:uniqueness}$.
%\end{proof}
%We check that we have the Gibbs property of $\mu$ at all small scales, as stated in the following proposition.

\begin{proposition} \label{prop:muGibbsallscales}
    For every $A\in \RRR$ and all $\delta\in (0,\halffracfraceps)$, there are constants $G_L$ and $G_U$ such that for all $t\geq 0$ and $x\in A_t$, we have
    \[
\frac{\mu(B_t(x,\delta))}{e^{tP(\ph)-\Phi(x,t)}}\in [G_L,G_U].
    \]
  % In particular, the measure $\mu$ satisfies lower Gibbs property at all scales.
\end{proposition}
\begin{proof}
  Recall from $\S \ref{sec:Gibbs}$ that $\mu$ was constructed on a sequence of $(n_\ell,\eps_0)$-separated sets $(E_{n_\ell})$, where $\eps_0$ is any fixed scale with $\eps_0 \in (0,\fracfraceps)$. Theorem \ref{thm:mainGibbsgeneral}\ref{thm:Gibbs} tells us that $\mu$ satisfies the lower Gibbs property at scale $2 \eps_0$ and the upper Gibbs property at scale $\eps_0/2$.  For any $\delta\in (0,\fracfraceps)$, we can find  $A_0\in \RRR$ on which $\varphi$ is SPR and $A_0^{-\delta} \neq \emptyset$. We repeat the construction to obtain a measure $\mu(\delta)$ satisfying the lower Gibbs property at scale $2 \delta$ and the upper Gibbs property at scale $\delta/2$ for all $A \in \RRR$. The measure $\mu(\delta)$ is an equilibrium state by Theorem \ref{thm:Mis-ES}. Since $\mu$ is the unique equilibrium state, $\mu(\delta)=\mu$, and this gives the Gibbs property for $\mu$ at all scales. %which concludes the proof of the proposition by our arbitrary choice on $\eps'\in (0,\fracfraceps)$.
\end{proof}

%%%%%%%%%%%%%%%%%%%%%%%
\section{Applications to geodesic flows} \label{s.application}
We apply our results to a large class of geodesic flows in negative curvature. In this section, we assume the hypotheses of Theorem \ref{thm:CAT}. We verify that the expansivity property of Condition \ref{cond:E} and the specification property of Condition \ref{cond:S} hold for our geodesic flows.  We then obtain Theorem \ref{thm:CAT} directly from Theorem \ref{thm:mainGibbs}. 

Theorem \ref{thm:CAT} has the hypotheses that the space $X$ is $\CAT(-a^2)$ for some $a>0$ or $X$ is a Hadamard manifold with negative curvature. We prove our results in the case that $X$ is a $\CAT(-1)$ space. It immediately follows that the theorem holds for $\CAT(-a^2)$ spaces for any $a>0$ by homothetic rescaling of the space. We then explain the required modifications to the proof to cover the case of Hadamard manifolds with negative curvature. 

We give criteria to check some of the other hypotheses of Theorem \ref{thm:CAT}. In the $\CAT(-1)$ case, we show that H\"older potentials satisfy our Bowen property. Finally, we establish verifiable criteria for the geodesic flow to have finite entropy.

\subsection{Geodesic flow for $\CAT(-1)$ spaces}
We refer the reader to background references including \cite{BH99, wB95, BPP19, tR03}. We follow the notation conventions of \cite{DT25}.

 Let $(X, d)$ be a proper geodesically complete $\CAT(-1)$
space, and let $\Gamma < \mathrm{Isom}(X)$ be a non-elementary discrete group of isometries of $X$ without torsion.  Let
$X_0=\Gamma\backslash X$ and let $\pi_X \colon X \to X_0$ denote the quotient map. The space $X_0$ is a proper connected geodesically
complete locally $\CAT(-1)$ space. Every such space arises this way, and $\Gamma$ is the fundamental group
acting as a group of isometries on the universal cover $X$. 
 We say that $c\colon\RR \to X$ is a geodesic line if it is an isometry onto its image. We let $GX$ be the space of geodesic lines in $X$.
A natural metric on $GX$ is the metric
\begin{equation} \label{eq:dgx}
	d_{G  X}(  c,  c') =  \int_{-\infty}^\infty d(  c(s),  c'(s)) e^{-2|s|}ds.
\end{equation}
The group $\Gamma$ acts naturally on $GX$ by isometries. Let $GX_0 = \Gamma \backslash GX$ and let $\pi_{GX} \colon GX \to GX_0$ denote the quotient maps.
For $c_1, c_2 \in GX_0$, we define $d_{GX_0}(c_1,c_2) = \inf_{\tilde{c}_1, \tilde{c}_2} d_{G X}(\tilde{c}_1,\tilde{c}_2)$,
where the infimum is taken over all lifts $\tilde{c}_1, \tilde{c}_2$ of $c_1, c_2$. 

Since $X$ is assumed to be proper, every closed ball in $X_0$ is compact. From this, it is easily verified that $(GX_0, d_{GX_0})$ is proper. As usual, $\RRR$ denotes the collection of all compact subsets of $GX_0$ with non-empty interior.

For $t \in \RR$, we define the geodesic flow $g_t \colon GX \to GX$ at time $t$ by 
\[
	(g_tc)(s) = g(t+s).
\]
It is easy to check from the definition of $d_{GX}$ that the flow is unit speed. For all $t\in\RR$, we have
\begin{equation}\label{eq:bddgrowth}
	e^{-2 |t|}d_{GX}(c,c') \leq d_{GX}(g_tc,g_tc') \leq e^{2|t|}d_{GX}(c,c').
\end{equation}
Since $g_t$ is $\Gamma$-equivariant, it descends to a well-defined flow on $GX_0$.

Let $\partial_\infty X$ denote the boundary at infinity of $X$, defined to be
the collection of equivalence classes of geodesic rays, where two rays are equivalent if they stay within
bounded distance of each other. Given $c \in GX$, we write $c(+\infty)$ (resp. $c(-\infty)$) for the point in $\partial_\infty
X$ determined by the positive (resp. negative) geodesic ray defined by $c$. 

Let $\partial_\infty^2 X = (\partial_\infty X \times \partial_\infty X) \setminus \Delta$,
where $\Delta$ is the diagonal. For a $\CAT(-1)$ space $X$,
the space of geodesic lines $GX$ can be identified with $\partial_\infty^2 X \times \mathbb{R}$ using a
Hopf parametrization. Under a Hopf parametrization, the topologies induced on $GX$ by $d_{GX}$ and on $\partial_\infty^2 X \times \mathbb{R}$ 
using the cone topology of $\partial_\infty X$ agree. In Hopf coordinates, the projection to the $\partial_\infty^2 X$ coordinate is the map
$c \mapsto (c(-\infty),c(+\infty))$, and translation on the $\RR$ coordinate corresponds to the geodesic flow on $GX$.

The action of $\Gamma$ preserves equivalence classes of geodesic rays, and thus $\Gamma$ acts on $\partial_\infty X$.
Let $\Lambda \subseteq \partial_\infty X$ denote the limit set of $\Gamma$,
defined to be the set of limit points of $\{a x\}_{a \in \Gamma}$ in $\partial_\infty  X$.
This definition is independent of $x$. The set $\Lambda$ is compact and $\Gamma$-invariant. Since $\Gamma$ is discrete and non-elementary, then $\Lambda$ is uncountable,
and $\Gamma$ acts minimally on $\Lambda$.

Let $\Omega X$ denote the set of geodesic lines $c \in GX$ such that
$c(-\infty), c(+\infty) \in \Lambda$,
and let $\Omega X_0 = \Gamma \backslash \Omega X$. The set  $\Omega X_0$ is the non-wandering set of the geodesic flow on $GX_0$. Since $\Gamma$ acts minimally on $\Lambda$, the 
geodesic flow is transitive on $\Omega X_0$.  It is proved in \cite[Theorem 1.4]{tR03} that if the length spectrum is non-arithmetic, then the geodesic flow on $\Omega X_0$ is topologically mixing. If $X$ is a manifold with negative curvature, this is conjectured to always hold true.
 %The Poincar\'e series of $\Gamma$ with respect to $x$ and $y$ is defined by
%\[
%P(s;x,y):=\sum_{\gamma\in\Gamma} e^{-s d(x,\gamma y)}.
%\]
%We define the \emph{critical exponent} of $\Gamma$ to be
%\[
%\delta_\Gamma
%:=
%\inf\left\{
%s>0:
%\sum_{\gamma\in\Gamma} e^{-s d(x,\gamma y)}
%\text{ converges}
%\right\}
%\]
%The critical exponent is independent of
%$x,y\in X$ by the triangle inequality. We assume that $\delta_\Gamma$ is finite. We say that $\Gamma$ is divergent-type if $P(\delta_\Gamma;x,y) = \infty$. If $\Gamma$ is divergence-type and the length spectrum is non-arithmetic, then the geodesic flow on $\Omega X_0$ is topologically mixing. This can be seen as a consequence of the fact that the Bowen-Margulis measure is supported on $\Omega X_0$ and is mixing \cite{tR03}. This statement and argument also holds for $\CAT(0)$ spaces for which every geodesic is rank one by \cite{gL20}.

For $\xi \in \bX$, $x, y \in X$, the Busemann functions are given by
\[
\beta_\xi(x,y) = \lim_{z \to \xi} (d(x, z)-d(y,z)).
\]

On $GX$, we define the (upstairs) unstable and stable sets by
\[
	W^\mathrm{u}(c) = \{c' \in GX \colon c'(-\infty) = c(-\infty)\}, \quad
	W^\mathrm{s}(c) = \{c' \in GX \colon c'(+\infty) = c(+\infty)\}
.\] 
Within each $W^\mathrm{u}(c)$ and each $W^\mathrm{s}(c)$, we use the Busemann function to specify the strong (upstairs) stable and unstable sets:
\begin{align*}
	W^\mathrm{uu}(c) &= \{c' \in W^\mathrm{u}(c) \colon \beta_{c(-\infty)}(c'(0),c(0)) = 0\},\\
	W^\mathrm{ss}(c) &= \{c' \in W^\mathrm{s}(c) \colon \beta_{c(+\infty)}(c'(0),c(0)) = 0\}.
\end{align*}

It is immediate from the definition that $W^*(ac) = aW^*(c)$ when $a \in \Gamma$ and $c
\in GX$ for $* \in \{ \mathrm{uu}, \mathrm{ss}\}$. For $c\in GX$,  $\zeta>0$ and $* \in \{ \mathrm{u},  \mathrm{uu}, \mathrm{s}, \mathrm{ss}\}$ we define the local versions of these sets by

\[
W^*(c, \zeta) = \{c' \in W^*(c) \colon d_{GX}(c, c') \leq \zeta\}.
\]

We define `downstairs' local horospherical stable and unstable sets $W_{\mathrm H}^*(c, \zeta) \subseteq GX_0$ for $c \in GX_0$ and $* \in \{ \mathrm{u}, \mathrm{s}, \mathrm{uu}, \mathrm{ss}\}$ by pushing them down from $GX$.  That is, for $c_0 \in GX_0$ with $c_0= \pi_{GX}c$ and $\zeta>0$, we define
\[
W_{\mathrm H}^*(c_0, \zeta) = \pi_{GX}(W^*(c, \zeta)).
\]

\begin{remark} \label{rem:horospherenonexpansive}
For $c\in GX_0$, we can define the metric local strong stable set with respect to $(\dbow_t)$ to be
\[
\WBow^{\mathrm{ss}} (c, \zeta) := \{c' \in GX_0 \colon \lim_{t\to+\infty}d_{GX_0}(g_tc,g_tc') = 0, ~ d_{GX_0} (c, c') \leq \zeta \},
\]
and we define $\WBow^{\mathrm{uu}} (c, \zeta)$ analogously. We have $W_H^\mathrm{ss}(c, \zeta) \subseteq  \WBow^{ss}(c, \zeta)$. If $X_0$ has a cusp, this inclusion is strict for a large (in particular dense) set of points in $\Omega X_0$.  The set $\WBow^{\mathrm{ss}}(c, \zeta)$ can contain many horospherical strong stable sets, corresponding to different ways to wind around the cusp. In other words,  $\WBow^{ss}(c, \zeta)$ can contain many sets of the form $W_H^\mathrm{ss}(a c, \zeta')$, where $a$ is a parabolic isometry. This was demonstrated by Burniol Clotet and Dal'Bo in \cite{CD26}. %We consider horospherical stable and unstable sets to avoid the overly complex behavior of the sets $\WBow^{\mathrm{ss}}(c, \zeta)$.
\end{remark}

\subsection{The family $\DDDH$ and its metric properties} \label{sec:metricpropertiesgeodesicflow}

We define a family of metrics $\DDDH =(\dH_t)_{t\geq0}$ by
\begin{equation}\label{eq:metricfamily}
	\dH_t(c_1,c_2) = \inf_{\tilde{c}_1, \tilde{c}_2} \sup_{s \in [0,t]} d_{G X}(g_s \tilde{c}_1,g_s \tilde{c}_2).
\end{equation}
We emphasize that $\DDDH$ is distinct from the family of Bowen dynamical metrics $(\dbow_t)$ with respect to $d_{GX}$, which is
\[
\dbow_t(c_1,c_2) = \sup_{s \in [0,t]} \inf_{\tilde{c}_1, \tilde{c}_2}  d_{G X}(g_s \tilde{c}_1,g_s \tilde{c}_2). 
\]
Let $B_t(c, \zeta)$ denote the Bowen ball in the metric $\dH_t$, that is 
\[
B_t(c, \zeta) = \{ c' \in GX_0 : \dH_t(c, c') < \zeta\}.
\]
We define the local injectivity radius at $x \in X_0$ to be
 \[
\operatorname{inj}(x)
:=
\frac12 \inf_{\gamma \in \Gamma \setminus \{e\}}
d(\tilde{x},\gamma \tilde{x}),
\]
where $\tilde{x} \in X$ is any lift of $x$. For $A \in \RRR$, we define $\operatorname{inj}(A)= \inf_{c \in A} \operatorname{inj}(c(0))$.

If the injectivity radius at $c_1(\tau)$ is at least $4 \epsilon$ for some $\tau \in [0, t]$, and $c_2 \in B_t(c_1, \eps)$, and we choose a lift $\tilde c_1$ of $c_1$, then there is a unique lift $\tilde c_2$ with $d_{G X}(g_\tau \tilde{c}_1, g_\tau \tilde{c}_2) < \eps$. It follows that
\begin{equation} \label{eqn:metricfamilygeodesicequation}
\dH_t(c_1,c_2) = \sup_{s \in [0,t]} d_{G X}(g_s \tilde{c}_1,g_s \tilde{c}_2),
\end{equation}
and we thus have $d_{G X}(g_s \tilde{c}_1, g_s \tilde{c}_2) < \eps$ for all $s \in[0, t]$. This is the key property that allows us to show the following.
\begin{lemma}
The family $\DDDH$ is coherent in the sense of Definition \ref{def:coherent-metrics}.
\end{lemma}
\begin{proof}
Let $s,t\geq 0$ and $c_1, c_2 \in GX_0$. Let $ \tilde c_1, \tilde c_2$ satisfy
\[
\dH_{s+t}(c_1,c_2) = \sup_{r \in [0,s+t]} d_{G X}(g_r \tilde{c}_1,g_r \tilde{c}_2).
\]
Then $ \sup_{r \in [0,s]} d_{G X}(g_r \tilde{c}_1,g_r \tilde{c}_2) \geq \dH_s(c_1,c_2) $, and $\sup_{r \in [0,t]} d_{G X}(g_r g_s \tilde{c}_1,g_r g_s\tilde{c}_2) \geq \dH_t(g_s c_1, g_s c_2)$, and thus $\dH_{s+t}(c_1,c_2) \geq \max ( \dH_s(c_1,c_2) , \dH_t(g_s c_1, g_s c_2))$.

Let $A \in \RRR$. %Recall that $\operatorname{inj}(A)= \inf_{c \in A} \operatorname{inj}(c(0))$. 
Let $\rmet(A) = \operatorname{inj}(A)/4$. Let $s,t\geq 0$, $c_1 \in g_{-s}(A)$, and $c_2\in GX_0$ be such that $\dH_s(c_1,c_2) \leq \rmet$ and $\dH_t(g_s c_1, g_s c_2) \leq \rmet$. We choose a lift $\tilde c_1$ of $c_1$. There is a unique lift  $\tilde c_2$ of $c_2$ with  $d_{GX}(g_s \tilde{c}_1,g_s \tilde{c}_2) < \rmet$. It follows from \eqref{eqn:metricfamilygeodesicequation} that 
\[
\dH_s(c_1,c_2) = \sup_{\tau \in [0,s]} d_{G X}(g_\tau \tilde{c}_1,g_\tau \tilde{c}_2),
\]
\[
\dH_t(g_s c_1, g_s c_2)  = \sup_{\tau \in [0,t]} d_{G X}(g_\tau g_s \tilde{c}_1,g_\tau g_s\tilde{c}_2),
\]
and also that
\[
\dH_{s+t}(c_1,c_2) = \sup_{\tau \in [0,s+t]} d_{G X}(g_\tau \tilde{c}_1,g_\tau \tilde{c}_2).
\]
It is thus immediate that $\dH_{s+t}(c_1,c_2)  = \max \big \{ \dH_s(c_1,c_2), \dH_{s+t}(c_1,c_2) \big \}.$
\end{proof}
We equip $GX_0$ with the family of metrics $\DDDH$ and we now have Condition \ref{cond:B}. We show that $(g_t)$ is expansive for the family of metrics $\DDDH$ with respect to $\RRR$. For $\theta>0$ and $c \in GX_0$,  the bi-infinite Bowen ball for $\DDDH$  is the set

\[
\Gamma_\theta (c; \DDDH) =\bigcap_{t>0} B_{2t}(g_{-t} c, \theta).
\]
\begin{lemma} \label{lem:expansivitymaintool}
%Let $S \subset X$ be a bounded set. 
Let $c \in GX_0$ and write $c(0) =x$. Let $\theta \in(0, \operatorname{inj}(x)/4)$. Then
\[
\Gamma_\theta(c;\DDDH) \subset g_{[-\theta,\theta]}(c).
\]
\end{lemma}
\begin{proof}
Let $c' \in \Gamma_\theta(c; \DDDH)$. We fix a lift $\tilde c$ of $c$. We fix $t>0$. Since $c' \in B_{2t}(g_{-t} c, \theta)$ and $\theta < \operatorname{inj}(x)/4$, there is a unique lift $\tilde c'$ so that
\[
\dH_{2t}(c, c') = \sup_{s \in [0, 2t]} d_{GX}(g_s g_{-t}\tilde c, g_s g_{-t} \tilde c').
\]
The lift $\tilde c'$ is the unique lift so that $d_{GX}(\tilde c(0), \tilde c'(0)) < \theta$, and is thus chosen independent of $t$. We thus conclude that 
\[
d_{GX}(g_{-t} \tilde c, g_{-t} \tilde c') < \theta \mbox{ and } d_{GX}(g_{t} \tilde c, g_{t} \tilde c')  < \theta.
\]
Since this holds for any $t>0$, it follows that $\tilde c(+\infty)= \tilde c'(+\infty)$ and $\tilde c(-\infty)= \tilde c'(-\infty)$. Thus, up to parametrization,  $\tilde c = \tilde c'$. The result follows.
\end{proof}

\begin{lemma}\label{lem:expansive}
The flow $(GX_0, (g_t))$ is expansive with respect to $\RRR$ for the family of metrics $\DDDH$ in the sense of Condition \ref{cond:E}.
\end{lemma}
\begin{proof}
Let $\rexp(A)=\rmet(A) = \operatorname{inj}(A)/4:= \alpha$. If $c\in A$ and $s_k, t_k\to\infty$ and
\[
\dH_{s_k}(f_{-s_k}c,f_{-s_k}c') \leq \rexp(A)
\quad\text{and}\quad
\dH_{t_k}(c,c') \leq \rexp(A)
\quad\text{for all }k,
\]
and we see that $c' \in \Gamma_{\alpha}(c;\DDDH)$. Thus by Lemma \ref{lem:expansivitymaintool}, 
 there exists $t\in [-\alpha,\alpha]$ such that $c' = g_t c$.
\end{proof}
\begin{remark} \label{expansivityremark}
The bi-infinite Bowen ball for $c_0 \in GX_0$ with respect to $d_{GX_0}$ is 
\begin{align*}
\Gamma_\theta (c; d_{GX_0}) &= \{c' \in GX_0 : d_{GX_0}(g_t c, g_tc')< \theta \mbox{ for all } t \in \RR \} \\
&= \WBow^{ss} (c, \zeta) \cap \WBow^{uu} (c, \zeta).
\end{align*}
As described in Remark \ref{rem:horospherenonexpansive}, if $X_0$ has a cusp, the sets $\WBow^{ss}(c, \zeta)$ (resp. $\WBow^{uu}(c, \zeta)$) will often contain many horospherical strong stable (resp. unstable) sets. %This occurs for a large (in particular dense) set of points in $\Omega X_0$ whenever $X_0$ has a cusp.  
Thus, the set $\Gamma_\theta (c; d_{GX_0})$ will often contain elements which are not of the form $g_{[-\alpha, \alpha]}c$. The geodesic flow is thus not expansive in the family of Bowen metrics $(\dbow_t)$ for $d_{GX_0}$. For pinched negative curvature manifolds, expansivity also fails in other natural metrics such as the Sasaki or Knieper metric since they are equivalent to $d_{GX_0}$. %This motivates our introduction of the expansivity property for a coherent families of metrics and 
%This is why it is important that we consider the family $\DDDH$ rather than a dynamical family of metrics $(\dbow_t)$.

\end{remark}

\subsection{Local product structure and specification property of geodesic flow}  We turn our attention to the specification property. This follows from the local product structure of the geodesic flow, which we now discuss. There exists a constant $K>1$ so that for any $c\in GX$ and any sufficiently small $\zeta>0$, we have
\begin{itemize}
\item if $c' \in W^\mathrm{uu}(c, \zeta)$ and $t>0$, then $d_{GX}(g_tc, g_tc') \leq C d_{GX}(c, c')e^{-t}$.
\item if $c'' \in W^\mathrm{ss}(c, \zeta)$ and $t<0$, then $d_{GX}(g_tc, g_tc'') \leq C d_{GX}(c, c'')e^{-t}$
\end{itemize}
The proof is based on using estimates on geodesics in $\mathbb H^2$ and a comparison argument, see \cite[Lemma 2.8]{CLT20} for details. For a metric space $(Y, d)$, and $\zeta>0$, we write
\[
(Y \times Y)_\zeta := \{ (x, y) \in Y : d(x,y) < \zeta\}.
\]
We have local product structure upstairs. Fix $\zeta_0>0$. There exists $\kappa>1$ so that for  $c, c' \in (GX \times GX)_{\zeta_0}$, we have that
\[
\langle c, c' \rangle:=  W^\mathrm{u}(c, \kappa \zeta_0) \cap W^\mathrm{ss}(c', \kappa  \zeta_0)
\]
defines a point in $GX$. Furthermore,
\[
d_{GX}(c, \langle c, c' \rangle) < \kappa d_{GX}(c, c') \mbox{ and } d_{GX}(c', \langle c, c' \rangle) < \kappa d_{GX}(c, c').
\]
Some details are given in \cite[Theorem 3.4]{CLT20}.  This structure descends to give local product structure downstairs, although we need to take care with scales when the local injectivity radius is not bounded away from $0$. Let $A \in \RRR$. Let $\zeta=\zeta(A) = \min \{ \zeta_0, \operatorname{inj}(A)/4 \lambda \}$. For  $c, c' \in (A \times A)_\zeta$, we define 
\[
\langle c, c' \rangle:=  W_H^\mathrm{u}(c, \kappa \zeta) \cap W_H^\mathrm{ss}(c', \kappa  \zeta).
\]
This operation agrees with taking the local product upstairs and pushing downstairs.

\begin{lemma} \label{lem:specforgeodesicflow}
The geodesic flow $(\Omega X_0, (g_t))$ has specification at all scales w.r.t. $\RRR$ and the metric family $\DDDH$.
\end{lemma}
\begin{proof}
Let $A\in \RRR$. Let $\zeta = \zeta(A)$. We have a well-defined bracket operation on $(A\times A)_\zeta$.  The details of the proof are analogous to \cite[Theorem 4.1]{BCFT} or \cite[\S1.2.4.3]{CT21} to obtain the specification property at scales $\eps \in (0, \zeta)$ for families of orbit segments $\{(c_i, s_i)\}$ with $c_i \in A_{s_i}$. Using compactness of $A$ and the mixing property of the geodesic flow on $\Omega X_0$, then for any $\delta>0$, we can find a family of orbit segments $\{(z_i, \tau)\}$ so that for any $c, c'\in A$, there exists $i$ so that $z_i \in B(c, \delta)$ and $g_\tau z_i \in B(c', \delta)$. We use local product structure and a recursive argument to use these orbit segments to `close up' to a single orbit segment which shadows each specified orbit segment $(c_i, s_i)$ in the metric $\dH_{s_i}$, using the orbit segments $\{(z_i, \tau)\}$ to transition. For the recursive argument, we require that there is a uniform amount of distance contraction along each $(z_i, \tau)$ (i.e. distance contracts by a definite amount for two points in the local unstable set for $g_\tau z_i$ under the action of $g_{-\tau}$). This holds true since if $g_\tau z' \in W_H^\mathrm{uu}(g_\tau z_i, \zeta)$, then $d_{GX}(z_i, z') \leq Ce^{-\tau} d_{GX}(g_\tau z^i, g_\tau z')$.
\end{proof}

\begin{remark} The weak specification property (i.e. transition times on each $A$ are bounded above, rather than exact) can be established using transitivity of the geodesic flow on $\Omega X_0$ in place of mixing. Consider a point whose forward orbit is dense in $\Omega X_0$. It takes a uniform amount of time to transition between any two $\rho$-balls in a fixed compact part of the space.  We obtain that for any $\eps>0$, we can find a family of orbit segments $\{(z_i, t_i)\}$ with $t_i \leq \tau$ so that for any $x, y \in A$, there exists $i$ so that $z_i \in B(x, \epsilon)$ and $g_{t_i} z_i \in B(y, \epsilon)$. We then argue analogously to the proof in which all $t_i= \tau$. This would allow us to additionally handle the geodesic flow for $\CAT(-1)$ spaces which are transitive but not mixing.
\end{remark}

\subsection{Extending the proof to Hadamard manifolds with negative curvature}\label{sec:only-neg}

Let $X$ be a Hadamard manifold with everywhere negative sectional curvature. These spaces introduce additional difficulties from the point of view of topology and geometric group theory. For example, there are examples which are not tame and which are not visibility manifolds \cite{nP11}. However, these topological phenomena have no impact on our analysis. See also \cite{dB16} for the theory of constructing hyperbolic surfaces and \cite{iB16} for a survey which includes discussion of manifolds with curvature approaching $0$ and $-\infty$. We show how to verify our expansivity and specification properties in this class. 

%We emphasize that when the curvature is bounded above by $-a^2$, there is a canonical identification between $T^1X$ and $GX$ using the Hopf coordinates. Another natural metric is the Sasaki metric $d_S$ on $T^1X$. By identifying $T^1X$ and $GX$, these metrics can also be considered on $GX$. We continue to work in the family of metrics $\DDDH$ on $GX_0$ and we discuss alternate metrics in the next section.

Since the curvature is negative, there are no flat strips and it holds true that if $c, c' \in GX$ and satisfy $\tilde c(+\infty)= \tilde c'(+\infty)$ and $\tilde c(-\infty)= \tilde c'(-\infty)$, then up to parametrization,  $\tilde c = \tilde c'$. With this property in hand, the proofs of Lemmas \ref{lem:expansivitymaintool} and \ref{lem:expansive} go through verbatim, and we conclude that the flow $(GX_0, (g_t))$ is expansive with respect to $\RRR$ for the family of metrics $\DDDH$.

 The specification proof is analogous to the proof for $\CAT(-1)$ spaces. We justify this claim. Horospheres are well-defined for Hadamard manifolds \cite{pE01, wB95}, and this can be used to define strong stable and unstable manifolds. Because the curvature is everywhere negative, the strong stable manifolds intersect the unstable manifolds transversally everywhere. However, the angle of intersection may be arbitrarily small. This makes the constants in the local product structure non-uniform globally, see \cite[\S4]{BCFT}, but we still get uniform estimates on any compact set.
 
For $A \subset \RRR$, the curvature is bounded above by a constant. The curvature upper bound gives a lower bound on the angle of intersection of the horospheres. This gives a uniform local product structure for nearby points in $A$ in the following sense: for every $A\in \RRR$, there exists $\zeta=\zeta(A)$ and $\kappa=\kappa(A)$ so that for all  $c, c' \in (A \times A)_\zeta$, there exists a unique point  
\[
\langle c, c' \rangle:=  W_H^\mathrm{u}(c, \kappa \zeta) \cap W_H^\mathrm{ss}(c', \kappa  \zeta),
\]
 and furthermore 
 \[
d_{GX}(c, \langle c, c' \rangle) < \kappa d_{GX}(c, c') \mbox{ and } d_{GX}(c', \langle c, c' \rangle) < \kappa d_{GX}(c, c').
\]
 For a geodesic $c$ with $c(t)\in A$, the upper curvature bound on A gives a definite amount of contraction in backwards time for points in $W_H^\mathrm{u}(g_t c, \kappa \zeta)$. This gives us all the ingredients to run the proof of Lemma \ref{lem:specforgeodesicflow} verbatim. We obtain our specification property.  We conclude that Theorem \ref{thm:CAT}  holds when $X$ is a Hadamard manifold with everywhere negative sectional curvature. 

\subsection{Criteria for finite entropy} \label{sec:finiteentropy}
We establish a simple criterion to check that a flow satisfies our hypothesis of having finite Gurevich-Sarig entropy, then apply it to our geodesic flows. %By our variational principle, this implies that all measures have finite measure-theoretic entropy. 
\begin{definition} \label{def:Lipschitzflow}
Let $\FFF = (f_t)$ be a continuous flow on a metric space $(Y,d)$. We say that the flow $\FFF$ is Lipschitz (with respect to $d$) if for all $t\in \RR$, we have
% there exists $t>0$ so that the time-$t$ map is Lipschitz in the global sense that there exists a constant $C \geq 1$ so that $d(f_t x, f_t y) < C d(x,y)$ for all $x, y \in X$.
\begin{equation} \label{Lipschitzflow}
d(f_t x, f_t y) \leq c e^{\lambda |t|} d(x,y).
\end{equation}
\end{definition}
For $A \subset X$ be compact, let $\Spanset(A, \delta)$ be the smallest number of closed balls of radius $r$ required to cover $A$ and let $\Sepset(A, \delta)$ be the maximum number of disjoint open balls with centers in $A$, and observe that $\Sepset(A, \delta) \leq \Spanset(A, \delta)$. We recall the definition of the upper box dimension.
\begin{definition}
Let $(Y,d)$ be a metric space. For $A \subset X$ compact, we define
\[
\dimb(A, d) = \limsup_{\delta \to 0} - \frac{\log \Spanset(A, \delta)}{\log \delta}.
\]
We define $\dimb(Y, d) = \sup \{\dimb(A, d) : A \subset Y \text{ is compact}\}$.
\end{definition}
 
\begin{lemma} \label{Lipschitzimpliesfiniteentropy}
Let $\FFF$ be a continuous flow on a metric space $(Y,d)$ satisfying Condition \ref{cond:basic}. Suppose that the flow $\FFF$ is Lipschitz in the sense of \eqref{Lipschitzflow}. Then for any compact $A\subset Y$ with interior, we have $h^A <\lambda \dimb(A, d)$. Suppose additionally that $\dimb (Y, d) <\infty$. Then the Gurevich-Sarig entropy satisfies $h_{GS} <\infty$.
  \end{lemma}
\begin{proof}
Let $A \subset Y$ be compact with interior. Let $D= D(A)= \dimb(A, d)$. For all $\delta>0$ sufficiently small, it follows that 
\begin{equation} \label{eqn:dimbound}
\frac{\log \Sepset(A, 2\delta)}{- \log \delta}\leq \frac{\log \Spanset(A, \delta)}{- \log \delta}\leq 2 D.
\end{equation}
\
We fix $\delta >0$ small enough so that \eqref{eqn:dimbound} applies and we let $T>0$. Let $E_T$ be a $(\delta,T)$-separated set for $A_T$. For each $x_1,x_2\in E_T$, there exists $t \in[0, T]$ so that $d(f_t x_1, f_t x_2) >\delta$. By \eqref{Lipschitzflow}, this is only possible if
\[
d(x_1,x_2)>c^{-1} e^{-\lambda t} \delta \geq c^{-1} e^{-\lambda T} \delta.
\]
Thus, we have $\# E_T \leq  \Sepset(A, c^{-1} e^{-\lambda T} \delta)$. Thus, we have
\begin{align*}
\limsup_{T\to \infty}\frac{\log \# E_T}{T} & \leq \limsup_{T\to \infty}\frac{\log \Sepset(A, c^{-1} e^{-\lambda T} \delta)}{T} \\
& =\limsup_{T\to \infty} \frac{\log \Sepset(A, c^{-1} e^{-\lambda T} \delta)}{- \log (c^{-1} e^{-\lambda T} \delta))}\frac{- \log (c^{-1} e^{-\lambda T} \delta))}{T}\leq 2\lambda D. 
\end{align*}
Thus, for all small $\delta >0$, we have $P^A(0, \delta) \leq \lambda D$, and thus $h^A\leq \lambda D$. Now suppose that $\dimb (Y, d) <\infty$. Then the above argument applies to every reference set $A \in \RRR$. Thus $h_{GS} \leq 2 \lambda \sup \{D(A): A\in \RRR\} \leq \dimb (Y, d) < \infty$.
\end{proof}

We now discuss how to apply Lemma \ref{Lipschitzimpliesfiniteentropy} in the context of geodesic flows.  We equip $GX$ with the metric $d_{GX}$. By \eqref{eq:bddgrowth},for any $c_1,c_2\in GX$, we have
\[
d_{GX}(g_tc_1,g_tc_2)<e^{2|t|}d_{GX}(c_1,c_2).
\]
Thus, by Lemma \ref{Lipschitzimpliesfiniteentropy}, if $GX$ has finite box dimension in d$_{GX}$, then the entropy of the geodesic flow is finite. This can be verified from mild conditions on the metric space $(X, d)$. Bonk and Schramm proved in \cite[Theorem 9.2]{BS} that if $(X,d)$ has `bounded growth at some scale' (that is, there exist $0 <r < R$ and $N \in \NN$ such that any ball of radius $R$ in $X$ can be covered using $N$ balls of radius $r$), then the visual
boundary $\partial_\infty X$ has finite Assouad dimension, and hence finite upper box dimension with respect to a visual metric. It follows from Corollary 5.1.5 and Lemma 5.1.6 of \cite{cD24} that $\Omega X$ has finite box dimension in $d_{GX}$ if and only if $\partial_\infty X$ has finite box dimension with respect to a visual metric. Thus if $X$ has bounded growth, the entropy of the geodesic flow is finite. 

We emphasize that if we can find \emph{any} metric for which the flow is H\"older and has finite box dimension, then we can conclude that the entropy is finite.  Finiteness of the box dimension is immediate in the Sasaki metric $d_S$ due to the finite dimension of $M$ as a smooth manifold. When the curvature is unbounded below, we do not expect that the geodesic flow is H\"older in $d_S$.  The Lipschitz property of $d_{GX}$ is an advantage of that choice of metric.

\subsection{H\"older potentials} \label{s.holderpotentials}
Let $(Y,d)$ be a metric space and $\varphi:Y \to \RR$ be continuous. We say that $\varphi$ is uniformly locally H\"older continuous (w.r.t $d$) if there exists $C,\alpha,\eps>0$ such that for any $x,y$ satisfying $d(x,y)\leq\eps$, we have
$$
|\varphi(x)-\varphi(y)|\leq Cd(x,y)^{\alpha}.
$$
For a geodesic metric space, we assume without loss of generality that $\eps=1$. This condition allows $\varphi$ to be unbounded, but implies that it has at most linear growth in the sense that for any sequence of points $(x_0, x_1, x_2, \ldots)$ with $d(x_n, x_{n+1}) =1$ we have $|\varphi(x_n)-\varphi(x_0)|\leq \sum_{i=1}^n|\varphi(x_i)-\varphi(x_{i-1})| \leq Cn$.

Let $X$ be a $\CAT(-1)$ space. Let $\varphi:GX_0 \to \RR$ be a potential whose lift $\tilde \varphi: GX \to \RR$ is uniformly locally H\"older continuous with respect to $d_{GX}$. This condition is a priori weaker than asking that $\varphi:GX_0 \to \RR$ is uniformly locally H\"older continuous with respect to $d_{GX_0}$. For $c_1, c_2 \in GX$, we let
\[
\tilde d_t (c_1, c_2) = \sup_{s \in [0, t]} d_{GX}(g_sc_1, g_S c_2).
\]
For $c \in GX$ and $\zeta>0$, let $\tilde B_T(c, \zeta)$ be the Bowen ball in the metric $\tilde d_T$, that is 
\[
\tilde B_T(c, \zeta) = \{ c' \in GX : \tilde d_T (c, c') < \zeta\}.
\]
By standard arguments, see e.g.\ \cite[Proposition 4.3]{CLT20}, then at sufficiently small scales, a function $\varphi$ which is uniformly locally H\"older continuous w.r.t $d_{GX}$ satisfies the following global Bowen property: there exists a scale $r>0$ and a constant $Q > 0$ such that for all $c \in GX$ and $T\geq 0$ and every %$y\in X$ such that $d(f_tx, f_ty) < \eps$ for all $t\in [0,T]$,
$c'\in \tilde B_T(c,r)$, we have $|\Phi(c,T) - \Phi(c',T)| \leq Q$, where $\Phi(c,T)$ is the ergodic integral of $\tilde \varphi$ along the upstairs orbit segment $(c, T)$. The same property trivially  holds for any smaller scale $r'\in (0, r)$.

Now fix $A\in \RRR$. Let $\zeta = \rbow(A) := \inf\{r, \operatorname{inj}(A)/10\}$. Let $c \in A \subset GX_0$ and let $\tilde c$ be a lift of $c$ to $GX$. By the definition of the metric $\dH_T$ and the assumption on the scale, we can identify $B_T(c, \zeta)$ with $\tilde B_T(\tilde c, \zeta)$ using the projection map $\pi_{GX}$. We thus have the following downstairs Bowen property: there is a constant $Q > 0$ so that for all $A\in \RRR$ there exists a scale $\rbow(A)>0$ such that for all $c \in A$ and $T\geq 0$ and every %$y\in X$ such that $d(f_tx, f_ty) < \eps$ for all $t\in [0,T]$,
$c'\in B_T(c,\rbow(A))$, we have $|\Phi(c,T) - \Phi(c',T)| \leq Q$. This property is stronger than the Bowen property of Condition \ref{cond:B}.
\begin{remark}
The above argument does not provide the Bowen property in the family of metrics $(\dbow_t)$. This is because the downstairs Bowen balls computed in $(\dbow_t)$ may be too big. This is another place in our arguments where it is crucial that we work in the family $(\dH_t)$.
\end{remark}

In \cite{PPS15}, when $X$ is a pinched negative curvature manifold, the class of potentials considered is those whose lift to the universal cover of the unit tangent bundle are uniformly locally H\"older continuous. They show that in the pinched curvature case, the natural metrics are H\"older equivalent.  The metrics they consider are the Sasaki metric, the Knieper metric, and a metric on geodesic lines which is like our metric \eqref{eq:dgx} but uses a quadratic term $e^{-s^2}$ instead of $e^{2|s|}$. It can be checked that $d_{GX}$ is H\"older equivalent to all these metric in the pinched negative curvature case. Thus, our class of potentials strictly generalizes the class considered in \cite{PPS15}, and includes the `natural' class of H\"older potentials. In the $\CAT(-1)$ case, our class also contains the potentials considered in \cite{DT25}, where potentials are asked to satisfy the Bowen property globally `upstairs'.  The class of Bowen potentials with respect to $\RRR$ is a much more general class than the class of potentials which are globally Bowen since we additionally allow the Bowen constant $Q$ to depend on $A$, and we expect that this additional generality can be useful in applications.

\section*{Acknowledgments}
No AI tools were used in any aspect of this paper.
%\thanks{%
VC: This material is based upon work supported by the National Science Foundation under Award No.\ DMS-1554794, DMS-2154378, and DMS-2453314, and by a grant from the Simons Foundation (915869, Climenhaga).
DT: This  work is partially supported by the National Science Foundation under Award No.\ DMS-1954463 and DMS-2349915. TW is supported by National Key R\&D Program of China No.\ 2021YFA1003204, NSFC Nos.\ 12401239, and STCSM No.\ 24ZR1437200.%}

\bibliographystyle{amsalpha}
\bibliography{non-cpt-spec}

%%%%%%%%%%%%%%%%%%%%%%%

\end{document}